\documentclass{amsart}

\usepackage{etoolbox}
\patchcmd{\section}{\scshape}{\bfseries}{}{}
\patchcmd{\subsection}{-.5em}{.5em}{}{}
\makeatletter
\renewcommand{\@secnumfont}{\bfseries}
\makeatother

\usepackage{amssymb}
\usepackage{enumerate}
\usepackage{url} 
\usepackage{flagderiv}
\usepackage{cmll}

\usepackage[top=3cm, left=2.3cm, bottom=4.5cm, right=2.2cm]{geometry}

\usepackage{booktabs}

\usepackage{xcolor}

\newtheorem{theorem}{Theorem}[section]

\newtheorem{lemma}[theorem]{Lemma}

\newtheorem{proposition}[theorem]{Proposition}

\newtheorem{corollary}[theorem]{Corollary}

\newtheorem{definition}[theorem]{Definition}

\newtheorem{example}[theorem]{Example}

\begin{document}
\mathchardef\mhyphen="2D

\title[Formalizing Groups in Type Theory]{Formalizing Groups in Type Theory}

\author[F. Kachapova]{Farida Kachapova}
\address{Department of Mathematical Sciences
\\Auckland University of Technology
\\New Zealand}

\email{farida.kachapova@aut.ac.nz}

\date{} 

\begin{abstract} In this paper we formalize some foundation concepts and theorems of group theory in a variant of type theory called the Calculus of Constructions with Definitions. In this theory we introduce definition of a group, which is both general and simple enough to use in formal proofs. Based on this definition, we formalize the concepts of subgroup, coset, conjugate, normal subgroup, and quotient group, and formally derive some related theorems. We aim to keep these formalizations transparent and concise, and as close as possible to the standard mathematical theory. The results can be implemented in proof assistants that are based on calculus of constructions. 
\end{abstract}



\maketitle

\section{Introduction}
Formalization of mathematics in type theory has a long and well known history, from B. Russell \cite{Russ96}
and A. Church \cite{Chur40} to P. Martin-L{\"o}f \cite{Mart85} and T. Coquand \cite{Coq88}. Type theory is used as foundation for proofs assistants, which are tools for semi-automatic development and checking of mathematical proofs. 
Current proof assistants are not very user-friendly, it takes a lot of time to master any of them and the results are not easily transferred from one of them to another. No doubt, they will improve in future. Meanwhile it seems advantageous to do formalizations in one of the underlying formal theories: proofs there are more transparent, easier to understand and possibly transfer to several proof assistants. 
For our formalizations we choose so called Calculus of Constructions with Definitions $\lambda D$ developed in \cite{Ned14}. It has a clear and expressive language and useful features that include decidability of type checking and proof checking, and strongly normalizing property (when no additional axioms are present). As any calculus of constructions, $\lambda D$ benefits from PAT interpretation (proofs-as-terms and propositions-as-types).The lack of inductive types in $\lambda D$ is compensated by its higher order logic and axiomatic approach to mathematical definitions.
  
In Section 2 we give a brief overview of the theory $\lambda D$, including its derived logic rules, equality, and iota descriptor used for introducing new objects with given properties. This section also gives formalization of a few basic concepts relating to sets and binary relations. Based on material in \cite{Ned14}, this section introduces some useful derived rules and abbreviations in order to streamline formal proofs.

R. Nederpelt and H. Geuvers started formalization of mathematics in $\lambda D$ by formalizaing natural numbers and integer arithmetic and providing a formal proof of B\'{e}zout's lemma \cite{Ned14}. We continued this line of research in paper \cite{Kach20} on binary relations. In the current paper we start formalizing group theory in $\lambda D$. The first important step is to choose definition of a group that is general and works well in formal proofs. We do that in Section 3, where we consider different approaches to formally defining a group in $\lambda D$. First we define a group as a type. It is easy to do but not general enough, e.g. when we need to consider subgroups. So we introduce a formal definition of a group as a set, that is subset of a type. Since type theory is designed to deal mostly with functions/ operations, sets are not naturally formalized in it. As any set in $\lambda D$, a group is regarded a subset of some type $S$ and is defined as a predicate on $S$. Then multiplication of elements $x$ and $y$ of group $G$ has to take as extra arguments the proofs that $x$ and $y$ belong to $G$. That makes formal proofs long and unclear. The situation is similar with the group operation of taking inverse. 

In order to avoid these complications we modify the definition of a group $G$, so that it is a subset of a type $S$ but the group operations of multiplication and taking inverse are defined on the entire $S$, with specific group axioms required only for elements of $G$. Using classical logic, we formally prove in $\lambda D$ that operations on a subset of type $S$ can be extended to the entire $S$ (Theorems \ref{theorem:extend_function_1} and \ref{theorem:extend_function_2}
). This means that the modified definition covers the case of the original, seemingly more general definition. In the rest of the paper we work with the modified definition of a group. 
In Section 3 we also derive usual consequences of group axioms and formally prove that the set of all permutations on a type with operation of composition is a group.

In Section 4 we formally define subgroups, set products and cosets, and derive their properties, including a theorem about permutable groups. We prove that if $H$ is a subgroup of a group $G$, then the relation \\$R_H:=\lambda x,y.(x^{-1}\cdot y)\;\varepsilon\;H$ is an equivalence relation on $G$, and we study its connection to cosets of $H$.

In Section 5 we formalize the concepts of normal subgroup and quotient group $G/H$ and formally derive the following facts:
\begin{itemize}
\item being conjugate is an equivalence relation on elements of the group and on its subsets;
\item criteria of a normal subgroup;
\item normal subgroups of a fixed group are permutable; 
\item a product of normal subgroups is also a normal subgroup;
\item a quotient group is indeed a group;
\item correspondence between subgroups of $G/H$ and  subgroups of $G$ containing $H$ (Correspondence Theorem).
\end{itemize}

 In our formalizations we aim to keep the language and theorems as close as possible to the ones of the standard mathematics. We follow the standard mathematical approach to group theory (as for example, in textbook \cite{Hum96}).
 
We construct formal definitions and proofs in flag format introduced in \cite{Jas67} and \cite{Fitch52} and explained more in \cite{Ned11}. In this format each "flag" contains a variable declaration or an assumption; its scope is determined by the "flag pole" where definitions and proofs are developed. Several flags are often combined to make formal deductions more concise. 
Each formal proof generates a proof term; in proof of theorem, proposition or lemma with number $n$ the resulting term is denoted with $\boldsymbol{term_n}$ for future references.
All proofs are given in full in Appendix, with only proof sketches in the main part of the paper.

\section{Type Theory $\lambda D$}

Details of the language and derivation rules of the theory $\lambda D$ can be found in \cite{Ned14} and \cite{Kach20}. Judgments are formally derived in $\lambda D$ using the derivation rules. Here we only list the features that will be extensively used in this paper.
\medskip

The following flag diagrams show three main derivation rules of $\lambda D$:  (\textit{form}),  (\textit{abst}), and (\textit{appl}).

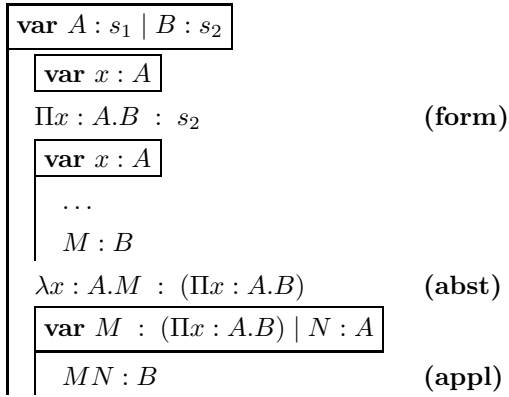
\setlength{\derivskip}{4pt}
\begin{flagderiv}
\introduce*{}{A:s_1\;|\;B:s_2}{}

\introduce*{}{x : A}{}

\conclude*{}{\Pi x:A.B\;:\;s_2}{\textbf{(form)}}

\introduce*{}{x : A}{}
\skipsteps*{\dots}{}
\step*{}{M:B}{}
\conclude*{}{\lambda x:A.M\;:\;(\Pi x:A.B)}{\textbf{(abst)}}

\introduce*{}{M\;:\;(\Pi x:A.B)\;|\;N:A}{}

\step*{}{MN:B}{\textbf{(appl)}}
\end{flagderiv}

Letters $s, s_1,s_2,\ldots$ are variables for sorts $*$ or $\square$. 
There is only one type $*$ in $\lambda D$. But informally we often use $*_p$ for propositions and $*_s$ for sets to make proofs more readable.

Arrow type $A\rightarrow B$ is a particular case of the dependent product $\Pi x:A.B$, when $x$ is not a free variable in $B$. It has the following derived rules:

\begin{flagderiv}
\introduce*{}{A:s_1\;|\;B:s_2}{}

\step*{}{A\rightarrow B:s_2}{}

\introduce*{}{x : A}{}
\skipsteps*{\dots}{}

\step*{}{M:B}{}

\conclude*{}{\lambda x:A.M\;:\;A\rightarrow B}{}

\assume*{}{u:A\rightarrow B\;|\;v:A}{}

\step*{}{uv:B}{}
\end{flagderiv}

In $\lambda D$ arrows are right associative, that is $A\rightarrow B\rightarrow C$ is a shorthand for $A\rightarrow (B\rightarrow C)$.

\subsection{Logic in $\lambda D$ }
\label{sect_logic}

As shown in \cite{Ned14}, intuitionistic logic is naturally formalized in $\lambda D$; we briefly describe that here. 

\subsubsection{}
\textit{Implication} $A\Rightarrow B$ is defined as the arrow type $A\rightarrow B$ and it has the same derived rules. 

\subsubsection{}
\textit{Falsity} $\bot$ is defined by:
\[\bot:=\Pi A:*_p.A\;:\;*_p\]

and has the derived rule:

\begin{flagderiv}
\introduce*{}{B:*_p\;|\;u:\bot}{}
\step*{}{uB:B}{}
\end{flagderiv}

\subsubsection{}
\textit{Negation} is defined by: 
\[\neg A:= A\Rightarrow \bot.\]

Next we list derived rules for other logical connectives and quantifiers.
We abbreviate the proof terms for some of the connectives and quantifiers to avoid too many repetitions in formal derivations; these abbreviations do not introduce any ambiguity.

\subsubsection{}
\textit{Conjunction} $\wedge$ has the following rules: 

\begin{flagderiv}
\introduce*{}{A,B:*_p}{}
\assume*{}{u:A\;|\;v:B}{}
\step*{}{\textbf{-in}(A,B,u,v)  \;:\;A\wedge B}{This is shortened to $\boldsymbol{\wedge(u,v):A\wedge B}$}
\done
\assume*{}{w:A\wedge B}{}
\step*{}{\boldsymbol{\wedge}\textbf{-el}_1(A,B,w)  \;:\;A}{This is shortened to $\boldsymbol{\wedge_1(w):A}$}
\step*{}{\boldsymbol{\wedge}\textbf{-el}_2(A,B,w)  \;:\;B}{This is shortened to $\boldsymbol{\wedge_2(w):B}$}
\end{flagderiv}

\newpage
\subsubsection{}

\textit{Disjunction} $\vee$ has the following rules:
  
\begin{flagderiv}
\introduce*{}{A,B:*_p}{}
\assume*{}{u:A}{}
\step*{}{\boldsymbol{\vee}\textbf{-in}_1(A,B,u)  \;:\;A\vee B}{This is shortened to $\boldsymbol{\vee_1(u):A\vee B}$}
\done
\assume*{}{u:B}{}
\step*{}{\boldsymbol{\vee}\textbf{-in}_2(A,B,u)  \;:\;A\vee B}{This is shortened to $\boldsymbol{\vee_2(u):A\vee B}$}
\done
\assume*{}{C:*_p}{}
\assume*{}{u:A\vee B\;|\;v:A\Rightarrow C\;|\;w:B\Rightarrow C}{}
\step*{}{\boldsymbol{\vee}\textbf{-el}(A,B,C,u,v,w)  \;:\;C}{This is shortened to $\boldsymbol{\vee(u,v,w):C}$}
\end{flagderiv}

\subsubsection{}

\textit{Bi-implication} $\Leftrightarrow$ is defined by: 
\[(A\Leftrightarrow B):=(A\Rightarrow B)\wedge (B\Rightarrow A).\]

\subsubsection{}
\textit{Universal quantifier} $\forall$ is defined by: 

\begin{flagderiv}
\introduce*{}{S:*_s\;|\;P:S\rightarrow *_p}{}
\step*{}{\text{Definition  }\forall (S,P)\;:=\Pi x:S.Px\;:\;*_p}{}
\step*{} {\text{Notation}: \mathbf{(\forall x:S.Px)} \text{ for }\forall (S,P)}{}
\end{flagderiv}

\subsubsection{}

\textit{Existential Quantifier} $\exists$ has the following rules:
\begin{flagderiv}
\introduce*{}{S:*_s\;|\;P:S\rightarrow *_p}{}
\introduce*{}{y:S\;|\;u:Py}{}
\step*{}{\boldsymbol{\exists}\textbf{-in}(S,P,y,u)  \;:\;(\exists x:S.Px)}{This is shortened to $\boldsymbol{\exists_1(P,y,u):(\exists x:S.Px)}$}
\done
\assume*{}{C:*_p}{}
\assume*{}{u:(\exists x:S.Px)\;|\;v:(\forall x:S.(Px\Rightarrow C))}{}
\step*{}{\boldsymbol{\exists}\textbf{-el}(S,P,u,C,v)  \;:\;C}{This is shortened to $\boldsymbol{\exists_2(u,v):C}$}
\end{flagderiv}

Here $x$ is not a free variable in $C$. 

The latter of these rules is usually used in this form:
\begin{flagderiv}
\introduce*{}{S:*_s\;|\;P:S\rightarrow *_p\;|\;C:*_p}{}
\skipsteps*{\dots}{}{}
\step*{}{u:(\exists x:S.Px)}{}
\assume*{}{x:S\;|\;w:Px}{}
\skipsteps*{\dots}{}{}
\step*{}{b:C}{}
\conclude*{}{v:=\lambda x:S.\lambda w:Px.b\;:\;(\forall x:S.(Px\Rightarrow C))}{}
\step*{}{\exists_2(u,v)\;:\;C}{}
\end{flagderiv}

We shorten it to the following:
\begin{flagderiv}
\introduce*{}{S:*_s\;|\;P:S\rightarrow *_p\;|\;C:*_p}{}
\skipsteps*{\dots}{}{}
\step*{}{u:(\exists x:S.Px)}{}
\assume*{}{x:S\;|\;w:Px}{}
\skipsteps*{\dots}{}{}
\step*{}{b:C}{}

\conclude*{}{\boldsymbol{ \exists_3(u,b)}\;:\;C}{}
\end{flagderiv}

\subsubsection{Classical Logic}
We aim to use intuitionistic logic in most proofs. In some cases classical logic is needed (e.g., in subsection \ref{sect_extensions}, 
where a function is extended from a subset of a type to the entire type). In these cases we use $\lambda D$ with addition of the following \textbf{Axiom of Excluded Third}:

\begin{flagderiv}
\introduce*{}{A:*_p}{}
\step*{}{\textbf{exc-thrd}(A):=\Bot \;:\;A\vee\neg A}{}
\end{flagderiv}

The axiom is introduced by a primitive definition with the symbol $\Bot$ replacing a non-existing proof term.

\subsection{Intensional Equality in $\lambda D$ }
\label{sect_int_equality}
Intensional equality is introduced in 
\cite{Ned14} as follows: 
\begin{flagderiv}
\introduce*{}{S:*\;|\;x,y: S}{}

\step*{}{eq(S,x,y):=\Pi P:S\rightarrow *_p. (Px\Rightarrow Py):*_p}{}
\step*{} {\text{Notation}: \boldsymbol{x=_Sy} \text{ for } eq(S,x,y)}{\textbf{Intensional equality}}
\end{flagderiv}
We will call the intensional equality just equality. 
Next we list derived rules of equality from \cite{Ned14}. We add some obvious rules that help to shorten formal proofs when symmetry property is necessary.

\subsubsection{}
\textit{Reflexivity}:

\begin{flagderiv}
\introduce*{}{S:*\;|\;x: S}{}
\step*{}{\boldsymbol{eq\mhyphen refl}(S,x):x=_Sx}
{This is shortened to $\boldsymbol{eq\mhyphen refl:x=_Sx}$}
\end{flagderiv}

\subsubsection{}
\textit{Symmetry}:

\begin{flagderiv}
\introduce*{}{S:*\;|\;x,y: S\;|\;u:x=_Sy}{}

\step*{}{\boldsymbol{eq\mhyphen sym
}(S,x,y,u):y=_Sx}{This is shortened to $\boldsymbol{eq\mhyphen sym(u):y=_Sx}$}
\end{flagderiv}

\subsubsection{}
\textit{Transitivity}:

\begin{flagderiv}
\introduce*{}{S:*\;|\;x,y,z: S}{}
\introduce*{}{u:x=_Sy\;|\;v:y=_Sz}{}
\step*{}{\boldsymbol{eq\mhyphen trans_1}
(S,x,y,z,u,v):x=_Sz}{This is shortened to $\boldsymbol{eq\mhyphen trans_1(u,v):x=_Sz}$}
\done

\introduce*{}{u:x=_Sy\;|\;v:z=_Sy}{}
\step*{}{a_1:=eq\mhyphen sym(v):y=_Sz}{}
\step*{}{\boldsymbol{eq\mhyphen trans_2}
(S,x,y,z,u,v):=trans_1(u,a_1)
\;:\;x=_Sz}{This is shortened to $\boldsymbol{eq\mhyphen trans_2(u,v):x=_Sz}$}
\done

\introduce*{}{u:y=_Sx\;|\;v:y=_Sz}{}
\step*{}{a_2:=eq\mhyphen sym(u):x=_Sy}{}
\step*{}{\boldsymbol{eq\mhyphen trans_3}
(S,x,y,z,u,v):=trans_1(a_2,v)
\;:\;x=_Sz}{This is shortened to $\boldsymbol{eq\mhyphen trans_3(u,v):x=_Sz}$}
\end{flagderiv}

\subsubsection{}
\textit{Substitutivity}:

\begin{flagderiv}
\introduce*{}{S:*\;|\;P: S\rightarrow *_p\;|\;x,y: S\;|\;u:x=_Sy}{}
\introduce*{}{v:Px}{}
\step*{}{\boldsymbol{eq\mhyphen subs_1}
(S,P,x,y,u,v):Py}{This is shortened to $\boldsymbol{eq\mhyphen subs_1(P,u,v):Py}$}
\done
\introduce*{}{v:Py}{}
\step*{}{a_1:=eq\mhyphen sym(u):y=_Sx}{}
\step*{}{\boldsymbol{eq\mhyphen subs_2}
(S,P,x,y,u,v):=eq\mhyphen subs_1(P,a_1,v)\;:\;
Px}{This is shortened to $\boldsymbol{eq\mhyphen subs_2(P,u,v):Px}$}
\end{flagderiv}

\subsubsection{}
\textit{Congruence}:

\begin{flagderiv}
\introduce*{}{Q,S:*\;|\;f: Q\rightarrow S\;|\;x,y: Q\;|\;u:x=_Qy}{}

\step*{}{\boldsymbol{eq\mhyphen cong_1}
(Q,S,f,x,y,u):fx=_Sfy
\\
\hspace{5cm}\text{This is shortened to }\boldsymbol{eq\mhyphen cong_1(f,u):fx=_Sfy}}{}

\step*{}{a_1:=eq\mhyphen sym(u):y=_Qx}{}

\step*{}{\boldsymbol{eq\mhyphen cong_2}
(Q,S,f,x,y,u):=eq\mhyphen cong_1(f,a_1)\;:\;
fy=_Sfx
\\
\hspace{5cm}\text{This is shortened to }\boldsymbol{eq\mhyphen cong_2(f,u):fy=_Sfx}}{}
\end{flagderiv}

We will usually omit an index in $=$.

\bigskip
\subsection{Sets in $\lambda D$}
\label{section:sets}
Sets are introduced in \cite{Ned14} as follows. In particular, a subset of type $S$ is defined as a predicate on $S$.

\begin{flagderiv}
\introduce*{}{S:*_s}{}
\step*{}{\boldsymbol{ps(S)}:=S\rightarrow *_p\;:\;\square
\quad\quad \textbf{Power set of S}}{}
\introduce*{}{V:ps(S)}{}
\step*{}{\text{Notation: }
\boldsymbol{\{x:S\;|\;Vx\}}\text{ for }\lambda x:S.Vx
}{}
\introduce*{}{x: S}{}
\step*{}{\boldsymbol{element}(S,x,V):=Vx:*_p}{}
\step*{}{\text{Notation: }x\varepsilon_S V \text{ or }x\varepsilon V\text{ for }element(S,x,V)}{}
\done[2]
\step*{}{\varnothing(S):=\{x:S\;|\;\bot\}:ps(S)}{}
\step*{}{\text{Notation: } \varnothing_S\text{ or } \varnothing\text{ for } \varnothing(S)}{}
\step*{}{\boldsymbol{full\mhyphen set}(S)
:=\{x:S\;|\;\neg\bot\}:ps(S)}{}

\introduce*{}{X,Y:ps(S)}{}

\step*{}{\text{Definition  }\cap(S,X,Y)\;:=\{x:S\;|\;x\varepsilon X\wedge x\varepsilon Y\}\;:\;ps(S)}{}
\step*{} {\text{Notation}: \boldsymbol{X\cap Y} \text{ for } \cap(S,X,Y)}{\textbf{Intersection of two sets}}

\step*{}{\text{Definition  }\cup(S,X,Y)\;:=\{x:S\;|\;x\varepsilon X\vee x\varepsilon Y\}\;:\;ps(S)}{}
\step*{} {\text{Notation}: \boldsymbol{X\cup Y} \text{ for } \cup(S,X,Y)}{\textbf{Union of two sets}}
\end{flagderiv}

Next we define intersection and union of a collection of sets.
\begin{flagderiv}
\introduce*{}{S: *_s\;|\;U:ps(ps(S))}{}
\step*{}{\text{Definition  }\cap(S,U)\;:=\{x:S\;|\;\forall Z:ps(S).(Z\varepsilon U\Rightarrow x\varepsilon Z)\}\;:\;ps(S)}{}
\step*{} {\text{Notation}: \boldsymbol{\cap U} \text{ for } \cap(S,U)}{\textbf{Intersection of collection of sets}}
\step*{}{\text{Definition  }\cup(S,U)\;:=\{x:S\;|\;\exists Z:ps(S).(Z\varepsilon U\wedge x\varepsilon Z)\}\;:\;ps(S)}{}
\step*{} {\text{Notation}: \boldsymbol{\cup U} \text{ for } \cup(S,U)}{\textbf{Union of collection of sets}}
\end{flagderiv}

\textit{Remark}: According to this definition, the union of an empty collection of sets is $\varnothing$ and its intersection is $full\mhyphen set(S)$.

Next diagram defines extensional equality of sets.

\begin{flagderiv}
\introduce*{}{S: *_s\;|\;X,Y:ps(S)}{}
\step*{}{\text{Definition  }\subseteq(S,X,Y)\;:=(\forall x:S.(x\varepsilon X\Rightarrow x\varepsilon Y))\;:\;*_p}{}
\step*{} {\text{Notation}: \boldsymbol{X\subseteq Y} \text{ for } \subseteq(S,X,Y)}{}
\step*{}{\text{Definition  }Ex\mhyphen eq(S,X,Y)\;:=X\subseteq Y\wedge Y\subseteq X\;:\;*_p}{}
\step*{} {\text{Notation}: \boldsymbol{X=_{ext}Y} \text{ for } Ex\mhyphen eq(S,X,Y)}{\textbf{ Extensional equality}}
\end{flagderiv}

Corresponding axiom is added to $\lambda D$:

\begin{flagderiv}
\introduce*{}{S: *_s\;|\;X,Y:ps(S);|\;u:X=_{ext}Y}{}

\step*{}{\boldsymbol{ext\mhyphen axiom}(S,X,Y,u)\;:=\Bot\;:\; X=_{ps(S)}Y}{\textbf{Extensionality Axiom}}
\end{flagderiv}

The Extensionality Axiom identifies the two types of set equality. 

We also add to $\lambda D$ the following axiom of extensionality for functions:
\begin{flagderiv}
\introduce*{}{Q,S: *_s\;|\;f,g:Q\rightarrow S\;|\;u:(\forall x:Q.fx=_Sgx)}{}

\step*{}{\boldsymbol{ext\mhyphen axiom\mhyphen fun}(Q,S,f,g,u)\;:=\Bot\;:\; f=_{Q\rightarrow S}g}{\textbf{Extensionality Axiom for Functions}}
\end{flagderiv}

We will usually omit an index in = and we will not elaborate on details of applying either Extensionality Axiom  when converting between different types of equality.

\subsection{Binary Relations}

In \cite{Ned14} a binary relation on a type $S$ is regarded as an object of type $S\rightarrow S\rightarrow  *_p$. In \cite{Kach20} we formalized in $\lambda D$ the standard theory of binary relations on a type. Next we define a binary relation on a set.

\begin{definition}
Here we define a binary relation $R$ on a subset $G$ of a type $S$ and a property $\textsf{consistent}_0(S,G,R)$, which means that the value of $R$ on $x,y\varepsilon G$ does not depend on the proofs of $x\varepsilon G$ and $y\varepsilon G$.

\begin{flagderiv}
\introduce*{}{S: *_s\;|\;G:ps(S)}{}
\introduce*{}{R:\Pi x:S.[(x\varepsilon G)\rightarrow \Pi y:S.((y\varepsilon G)\rightarrow *_p)]}{}
\step*{}{ consistent_0(S,G,R):=
\forall x:S.\Pi p,q: x\varepsilon G.\forall y:S.\Pi u,v: y\varepsilon G.(Rxpyu\Leftrightarrow  Rxqyv)\;:\;*_p}{}
\end{flagderiv}
\label{def:relation-0}
\end{definition}

\begin{lemma}
Here we construct an extension $Q$ of the binary relation $R$ to the type $S$.
\label{lemma:relation-0}
\end{lemma}
\begin{proof}
This is a sketch of the proof: 
\begin{flagderiv}
\introduce*{}{S: *_s\;|\;G:ps(S)\;|\;R:\Pi x:S.[(x\varepsilon G)\rightarrow \Pi y:S.((y\varepsilon G)\rightarrow *_p)]}{}

\skipsteps*{\dots}{}{}
\step*{}{Q:=\ldots\;:\;S\rightarrow S\rightarrow *_p}{}
\step*{}{\boldsymbol{ term_{ \ref{lemma:relation-0}.1}}(S,G,R):=Q\;:\;S\rightarrow S\rightarrow *_p\quad\quad\quad\quad\textbf{Extension of $R$ to $S$}}{}

\assume*{}{u:consistent_0(S,G,R)}{}

\step*{}{\boldsymbol{ term_{ \ref{lemma:relation-0}.2}}(S,G,R,u):=
\ldots
\;:\;\boldsymbol{
\forall x,y:S.(x\varepsilon G\Rightarrow y\varepsilon G
\Rightarrow \Pi p:x\varepsilon G.\Pi q:y\varepsilon G.(Qxy\Leftrightarrow Rxpyq))}}{}
\end{flagderiv}
The full proof is in the Appendix, pg. 32.
\end{proof}

Due to this lemma, we can replace the previous definition of a binary relation on  a set, without loss of generality, by the following more convenient definition. 

\begin{definition}
Definition of a binary relation $R$ on a subset $G$ of a type $S$, an equivalence relation on $G$, and a partition of $G$. 

\begin{flagderiv}
\introduce*{}{S: *_s\;|\;G:ps(S)}{}
\introduce*{}{R:S\rightarrow S\rightarrow *_p}{\textbf{R is a binary relation}}
\step*{}{\text{Definition  }\boldsymbol{refl}(S,G,R):=\forall x:S.(x\varepsilon G\Rightarrow Rxx):*_p}{\textbf{R is reflexive on G}}

\step*{}{\text{Definition  }\boldsymbol{sym}(S,G,R):=\forall x:S.[x\varepsilon G\Rightarrow\forall y:S.(y\varepsilon G\Rightarrow Rxy\Rightarrow Ryx)]:*_p}{\textbf{R is symmetric on G}}

\step*{}{\text{Definition  }\boldsymbol{trans}(S,G,R):=\forall x:S.[x\varepsilon G\Rightarrow\forall y:S.(y\varepsilon G\Rightarrow\forall z:S.(z\varepsilon G
\\\quad\quad
\Rightarrow
 Rxy\Rightarrow Ryz\Rightarrow Rxz))]:*_p}{\textbf{R is transitive on G}}

\step*{}{\text{Definition  }\boldsymbol{equiv\mhyphen rel}(S,G,R):=refl(S,G,R)\wedge sym(S,G,R)\wedge trans(S,G,R)}{}

\step*{}{\text{Definition  }\boldsymbol{partition}(S,G,R):=\forall x:S.(x\varepsilon G\Rightarrow
x\varepsilon Rx)
\wedge
\forall x:S.[x\varepsilon G
\\\quad\quad
\Rightarrow\forall y:S.(y\varepsilon G\Rightarrow\forall z:S.(z\varepsilon G
\Rightarrow
z\varepsilon Rx\Rightarrow z\varepsilon Ry\Rightarrow (Rx)\cap G=(Ry)\cap G))]:*_p}{\textbf{R is a partition of G}}
\end{flagderiv}
\label{definition:partition}
\end{definition}

In this definition the binary relation $R$ has type $S\rightarrow S\rightarrow *_p$, i.e. $R$ is defined on the whole type $S$, not just the subset $G$. But all the properties of $R$ are defined only for elements of $G$ because we are interested in the behavior of $R$ on $G$ only.
We will call a binary relation just a relation.

\begin{proposition}
Any equivalence relation on a set $G$ is a partition and vice versa.
\label{lemma:partition}
\end{proposition}
\begin{proof}
This is a sketch of the proof: 
\begin{flagderiv}
\introduce*{}{S: *_s\;|\;G:ps(S)\;|\;
R: S\rightarrow S\rightarrow *_p}{}
\skipsteps*{\dots}{}{}
\step*{}{\boldsymbol{ term_{ \ref{lemma:partition}}}(S,G,R):=\ldots\;:\;
\boldsymbol{equiv\mhyphen relation(S,G,R)\Leftrightarrow
partition(S,G,R)}}{}
\end{flagderiv}

The full proof is in the Appendix, pg. 32.
\end{proof}

\bigskip 
\subsection{The Iota Descriptor}
\label{sect_iota}

The following diagram defines existence with uniqueness in the language of $\lambda D$.

\begin{flagderiv}
\introduce*{}{S:*\;|\;P:S\rightarrow *_p}{}
\step*{}{\boldsymbol{\exists ^1x:S.Px}:=\exists x:S.Px\wedge \forall x,y:S.(Px\Rightarrow Py\Rightarrow x=y)}{}
\end{flagderiv}

Iota descriptor is introduced in $\lambda D$ (\cite{Ned14}, pg. 272 - 273) as a primitive definition:
\begin{flagderiv}
\introduce*{}{S:*\;|\;P:S\rightarrow *_p\;|\;u:(\exists ^1x:S.Px)}{}
\step*{}{\boldsymbol{\iota(S,P,u)}:=\Bot\;:\;S}{}
\step*{}{\boldsymbol{\iota\mhyphen prop(S,P,u)}:=\Bot\;:\;P(\iota(S,P,u))}{}
\end{flagderiv}

The following property of the iota descriptor is proven in \cite{Ned14}:

\begin{flagderiv}
\introduce*{}{S:*\;|\;P:S\rightarrow *_p\;|\;u:\exists ^1x:S.Px}{}
\skipsteps*{\dots}{}{}

\step*{}{\boldsymbol{\iota\mhyphen unique(S,P,u)}:=\ldots\;:\;\forall y:S.(Py\;\Rightarrow \;
y=\iota(S,P,u))}{}
\end{flagderiv}

In particular, this implies its proof irrelevance: the term $\iota(S,P,u)$ depends only on the existence of the proof $u$, not its content.

\bigskip
\subsection{Applications of the Iota Descriptor}
\label{sect_extensions}
In the following two theorems we use the iota descriptor to extend a function defined on a subset $G$ of a type $S$, to the entire $S$.

\begin{definition}
Here we introduce some preliminary predicates.

For a function $f$ from a set $G$ to type $T$:
\begin{itemize}
\item $\boldsymbol{consistent}_1(S,T,G,f)$ means that the value of $f$ on $x\varepsilon G$ does not depend on the proof of $x\varepsilon G$;
\item $\boldsymbol{Ext\mhyphen prop}_1(S,T,G,f,b,g)$ means that $g$ is an extension of $f$ to the type $S$, with gx=b when x is outside $G$.
\end{itemize}
\medskip

Similar predicates $\boldsymbol{consistent}_2(S,T,G,h)$ and $\boldsymbol{Ext\mhyphen prop}_2(S,T,G,h,b,g)$
are defined for a function $h$ of two variables from a set $G$ to type $T$.

\begin{flagderiv}
\introduce*{}{S,T: *_s\;|\;
G:ps(S)}{}
\introduce*{}{f:\Pi x:S.((x\varepsilon G)\rightarrow T)}{}
\step*{}{\text{Definition  }\boldsymbol{consistent}_1(S,T,G,f):=\forall x:S.\Pi p,q: x\varepsilon G.(fxp=fxq)\;:\;*_p}{}
\introduce*{}{b:T\;|\;g:S\rightarrow T}{}

\step*{}{\text{Definition  }\boldsymbol{Ext\mhyphen prop}_1(S,T,G,f,b,g):=
\forall x:S.[\Pi p:x\varepsilon G.\;(gx=fxp)\wedge (\neg(x\varepsilon G)\Rightarrow gx=b)]\;:\;*_p}{}
\done[2]

\introduce*{}{h:\Pi x:S.[(x\varepsilon G)\rightarrow \Pi y:S.((y\varepsilon G)\rightarrow T)]}{}
\step*{}{\text{Definition  }\boldsymbol{consistent}_2(S,T,G,h):=\forall x:S.\Pi p,q: x\varepsilon G.\forall y:S.\Pi u,v: y\varepsilon G.(hxpyu=hxqyv)\;:\;*_p}{}
\introduce*{}{b:T\;|\;g:S\rightarrow S\rightarrow T}{}

\step*{}{\text{Definition  }\boldsymbol{Ext\mhyphen prop}_2(S,T,G,h,b,g):=
\forall x,y:S.[\Pi p:x\varepsilon G.\Pi q:y\varepsilon G.\;(gxy=hxpyq)
\\\quad
\wedge (\neg(x\varepsilon G\wedge y\varepsilon G)\Rightarrow gxy=b)]\;:\;*_p}{}
\end{flagderiv}
\label{def:preliminary}
\end{definition}

\begin{theorem}
Suppose $S$ and $T$ are types, $G$ is a subset of $S$, $b:T$, and $f$ is a function from $G$ to $T$. Here we construct an extension $f^*:S\rightarrow T$ such that 
\[f^*x=\begin{cases}
fx &\text{ if}\quad x\varepsilon G,\\
b &\text{ if}\quad\neg (x\varepsilon G).
\end{cases}\]

\label
{theorem:extend_function_1}
\end{theorem}
\begin{proof}
The construction uses classical logic and the iota descriptor. It has the form:
\begin{flagderiv}
\introduce*{}{S,T: *_s\;|\;G:ps(S)\;|\;f:\Pi x:S.((x\varepsilon G)\rightarrow T)}{}

\assume*{}{b:T\;|\;u:consistent_1(S,T,G,f)}{}
\skipsteps*{\dots}{}

\step*{}{f^*:=\ldots
\;:\;S\rightarrow T}{}
\step*{}{ \boldsymbol{Ext}_1(S,T,G,f,b,u):=f^*
\;:\;\boldsymbol{S\rightarrow T}}{}

\step*{}{\boldsymbol{Ext\mhyphen proof}_1(S,T,G,f,b,u):=\ldots
\;:\;\boldsymbol{Ext\mhyphen prop_1}(S,T,G,f,b,f^*)}{}
\end{flagderiv}

The full proof is in the Appendix, pg. 34. 
\end{proof}

The following is a similar theorem for a function of two variables.
\begin{theorem}
Suppose $S$ and $T$ are types, $G$ is a subset of $S$, $b:T$, and $f$ is a function of two variables on $G$. Here we construct an extension $f^*:S\rightarrow S\rightarrow T$ such that 
\[f^*xy=\begin{cases}
fxy &\text{ if}\quad x\varepsilon G\wedge y\varepsilon G,\\
b &\text{ if}\quad\neg (x\varepsilon G\wedge y\varepsilon G).
\end{cases}\]

\label
{theorem:extend_function_2}
\end{theorem}
\begin{proof}
The construction is similar to the construction of the previous theorem and has the form:

\begin{flagderiv}
\introduce*{}{S,T: *_s\;|\;G:ps(S)\;|\;f:\Pi x:S.[(x\varepsilon G)\rightarrow \Pi y:S.((y\varepsilon G)\rightarrow T)]}{}

\assume*{}{b:T\;|\;u:consistent_2(S,T,G,f)}{}
\skipsteps*{\dots}{}

\step*{}{f^*:=\ldots
\;:\;S\rightarrow S\rightarrow T}{}

\step*{}{ \boldsymbol{Ext}_2(S,T,G,f,b,u):=f^*
\;:\;\boldsymbol{S\rightarrow S\rightarrow T}}{}

\step*{}{ \boldsymbol{Ext\mhyphen proof}_2(S,T,G,f,b,u):=\ldots
\;:\;\boldsymbol{Ext\mhyphen prop_2}(S,T,G,f,b, f^*)}{}
\end{flagderiv}

The full proof is in the Appendix, pg. 36.  
\end{proof}

\section{Definition of a Group. Properties and Examples\label{section:group}}
Here we introduce in $\lambda D$ and compare three versions of definition of a group. Then  we choose the best one for formalizing group theory.

\subsection{Group as a Type}

\begin{definition}
Definition of a group and other corresponding definitions are given in the following flag diagram.

\setlength{\derivskip}{4pt}
\begin{flagderiv}
\introduce*{}{S: *_s\;|\;\cdot: S\rightarrow S\rightarrow S}{}
\step*{}{\text{Notation } \boldsymbol{x\cdot y}\;\; for\; \cdot xy}{}

\step*{}{\text{Definition  }\boldsymbol{assoc}(S,\cdot)\;:=\forall x,y,z:S.((x\cdot y)\cdot z=x\cdot (y\cdot z))\;:\;*_p}{}

\step*{}{\text{Definition  }\boldsymbol{commut}(S,\cdot)\;:=\forall x,y:S.(x\cdot y=y\cdot x)\;:\;*_p}{}

\step*{}{\text{Definition  }\boldsymbol{semi\mhyphen group}(S,\cdot)\;:=assoc(S,\cdot)\;:\;*_p}{}

\introduce*{}{e:S}{}
\step*{}{\text{Definition  }\boldsymbol{identity}(S,\cdot,e)\;:=\forall x:S.(x\cdot e=x\wedge e\cdot x=x)\;:\;*_p}{}
\step*{}{\text{Definition  }\boldsymbol{monoid}(S,\cdot,e)\;:=semi\mhyphen group(S,\cdot)\wedge identity(S,\cdot,e)\;:\;*_p}{}
\introduce*{}{x,y:S}{}
\step*{}{\text{Definition  }\boldsymbol{inverse}(S,\cdot,e,x,y)\;:=(x\cdot y=e\wedge y\cdot x=e)\;:\;*_p}{}
\done

\introduce*{}{x:S}{}
\step*{}{\text{Definition  }\boldsymbol{invertible}(S,\cdot,e,x)\;:=\exists y:S.inverse(S,\cdot,e,x,y)\;:\;*_p}{}
\done

\introduce*{}{^{-1}:S\rightarrow S}{}
\step*{}{\text{Notation } \boldsymbol{x^{-1}}\;\; for\; ^{-1}x}{}
\step*{}{\text{Definition  }\boldsymbol{group}(S,\cdot,e,^{-1})\;:=monoid(S,\cdot,e)\wedge \forall x:S.inverse(S,\cdot,e,x,x^{-1})\;:\;*_p}{}
\step*{}{\text{Definition  }\boldsymbol{abelian\mhyphen group}(S,\cdot,e,^{-1})\;:=group(S,\cdot,e,^{-1})\wedge commut(S,\cdot)\;:\;*_p}{}
\end{flagderiv}
\label{def:group_type}
\end{definition}

\begin{proposition}
1) In a type with a binary operation the identity element (if exists) is unique.

2) In a monoid the inverse of any element (if exists) is unique.
\label{theorem:unique}
\end{proposition}
\begin{proof}
These are sketches of the proofs.

1)
\setlength{\derivskip}{4pt}
\begin{flagderiv}
\introduce*{}{S: *_s\;|\;\cdot: S\rightarrow S\rightarrow S}{}
\introduce*{}{e,d:S\;|\;u:identity(S,\cdot,e)\;|\;v:identity(S,\cdot,d)}{}
\skipsteps*{\dots}{}{}
\step*{}{\boldsymbol{term_{\ref{theorem:unique}.1}}(S,\cdot,e,d,u,v):=\ldots\;:\;\boldsymbol{e=d}}{}
\end{flagderiv}

2) 
\vspace{-0.2cm}
\begin{flagderiv}
\introduce*{}{S: *_s\;|\;\cdot: S\rightarrow S\rightarrow S\;|\;e:S\;|\;u: monoid(S,\cdot,e)}{}

\introduce*{}{b,x,y:S\;|\;v:inverse (S,\cdot,e, b, x)\;|\;w:inverse(S,\cdot, e, b, y)}{}

\skipsteps*{\dots}{}{} \step*{}{\boldsymbol{term_{\ref{theorem:unique}.2}}(S,\cdot,e,u,b,x,y, v,w):=\ldots\;:\;\boldsymbol{x=y}}{}
\end{flagderiv}

The full proofs are in the Appendix, pg. 38.
\end{proof}

\begin{example} Suppose $M$ is a type.

1) The power set $ps(M)$ with operation of intersection is a commutative monoid.

2) The power set $ps(M)$ with operation of union is a commutative monoid.
\label{example:monoid}
\end{example}
\begin{proof}
These are sketches of the proofs.

\begin{flagderiv}
\introduce*{}{M: *_s}{}
\step*{}{S:=ps(M)\;:\;\square}{}

\step*{}{\cap:=\lambda X,Y:S.X\cap Y\;:\; S\rightarrow S\rightarrow S}{}

\step*{}{\cup:=\lambda X,Y:S.X\cup Y\;:\;S\rightarrow S\rightarrow S}{}

\step*{}{E:=full\mhyphen set(M)\;:\;S}{}
\step*{}{e:=\varnothing\;:\;S}{}
\skipsteps*{\dots}{}{}
\step*{}{ \boldsymbol{ term_{ \ref{example:monoid}.1}}(M):=\ldots\;:\;\boldsymbol{monoid(S,\cap,E)\wedge commut(S,\cap)}}{}
\step*{}{ \boldsymbol{ term_{ \ref{example:monoid}.2}}(M):=\ldots\;:\;\boldsymbol{monoid(S,\cup,e)\wedge commut(S,\cup)}}{}
\end{flagderiv}

The full proofs are in the Appendix, pg. 38.
\end{proof}

\begin{definition} 
Here we define the identity function on a type $M$ and composition of functions:
\begin{flagderiv}
\introduce*{}{M: *_s}{}
\step*{}{\boldsymbol{id_M}:=\lambda x:M.x\;:\;M\rightarrow M}{\textbf{Identity function}}
\step*{}{\boldsymbol{\circ}:=\lambda f,g:M\rightarrow M.\lambda x:M.f(gx)\;:\;(M\rightarrow M)\rightarrow  M\rightarrow M}{\textbf{Operation  of composition}}
\end{flagderiv}
\label{def:composition}
\end{definition}

\begin{example} Suppose $M$ is a type.
Then $M\rightarrow M$ with operation of composition is a monoid.
\label{example: monoid_functions}
\end{example}
\begin{proof}
This is a sketch of the proof: 

\begin{flagderiv}
\introduce*{}{M: *_s}{}
\skipsteps*{\dots}{}{}
\step*{}{\boldsymbol{ term_{\ref{example: monoid_functions}}}(M):=\ldots
\;:\;\boldsymbol{monoid(M\rightarrow M,\circ,id_M)}}{}
\end{flagderiv}

The full proof is in the Appendix, pg. 42.
\end{proof}

\begin{example}
The type $\mathbb{Z}$ of all integers with operation of addition is an abelian group.
\label{example:Z}
\end{example}
\begin{proof}
In \cite{Ned14} the type $\mathbb{Z}$ and operations of addition and multiplication on $\mathbb{Z}$ are introduced axiomatically in $\lambda D$; also the integer arithmetic is developed in $\lambda D$. Using that it is easy to prove that $\mathbb{Z}$ with operation of addition is an abelian group, where 0 is the identity and for any $x$, $-x$ is regarded as its inverse. We skip the formal proof.
\end{proof}

\bigskip
\subsection{Group as a Set}
\label{section:group_subset}

\begin{definition}
This definition is for a group when it is a subset of a type. We use the formalization of a function on a subset from \cite{Ned14}.

\begin{flagderiv}
\introduce*{}{S: *_s\;|\;
G:ps(S)}{}
\introduce*{}{inv:\Pi x:S.((x\varepsilon G)\rightarrow S)}{}

\step*{}{\text{Definition  }\boldsymbol{closure}_1(S,G,inv):=\forall x:S.\Pi p: x\varepsilon G.((inv\; xp)\,\varepsilon \;G)\;:\;*_p}{}

\done
\introduce*{}{mult:\Pi x:S.[(x\varepsilon G)\rightarrow \Pi y:S.((y\varepsilon G)\rightarrow S)]}{}

\step*{}{\text{Definition  }\boldsymbol{closure}_2(S,G,mult):=\forall x:S.\Pi p: x\varepsilon G.\forall y:S.\Pi q: y\varepsilon G.((mult\; xpyq)\,\varepsilon \;G)\;:\;*_p}{}
\step*{}{\text{Definition  }\boldsymbol{assoc}(S,G,mult):=\forall x:S.\Pi p:x\varepsilon G.\forall y:S.\Pi q:y\varepsilon G.\Pi u: (mult\; xpyq)\,\varepsilon G.
\\\quad\quad
\forall z:S.\Pi r:z\varepsilon G.
\Pi v: (mult\; yqzr)\,\varepsilon G.\;
mult(mult\; xpyq)uzr=mult\;xp(mult\; yqzr)v
\;:\;*_p}{}
\step*{}{\text{Definition  }\boldsymbol{semi\mhyphen group}(S,G,mult):=
consistent_2(S,S,G,mult)
\\\quad\quad
\wedge\; closure_2(S,G,mult)
\wedge assoc(S,G,mult)\;:\;*_p}{}

\introduce*{}{e:S}{}
\step*{}{\text{Definition  }\boldsymbol{identity}(S,G,mult,e):=e\varepsilon G
\\\quad\quad
\wedge\forall x:S.\Pi p: x\varepsilon G.\Pi q: e\varepsilon G.(mult\; xpeq=x\wedge mult\; eqxp=x)\;:\;*_p}{}
\step*{}{\text{Definition  }\boldsymbol{monoid}(S,G,mult,e):=semi\mhyphen group(S,G,mult)\wedge identity(S,G,mult,e)\;:\;*_p}{}

\introduce*{}{inv:\Pi x:S.((x\varepsilon G)\rightarrow S)}{}
\step*{}{\text{Definition  }\boldsymbol{inverse_1}(S,G,mult,e,inv):=\forall x:S.\Pi p: x\varepsilon G.\Pi q: (inv\, xp)\;\varepsilon \;G.
[mult\; xp(inv\,xp)q=e
\\\quad\quad
\wedge\; mult\; (inv\,xp)qxp=e]
\;:\;*_p}{}

\step*{}{\text{Definition  }\boldsymbol{inverse\mhyphen prop}(S,G,mult,e,inv):=
consistent_1(S,S,G,inv)
\wedge closure_1(S,G,inv)
\\\quad\quad
\wedge \;inverse_1(S,G,mult,e, inv)\;:\;*_p}{}

\step*{}{\text{Definition  }\boldsymbol{group}(S,G,mult,e,inv):
\\\quad\quad=
monoid(S,G,mult,e)
\wedge inverse\mhyphen prop(S,G,mult,e,inv)
\;:\;*_p}{}

\end{flagderiv}
\label{def:group_subset_0}
\end{definition}

The conditions $consistent_1$ and $consistent_2$ were defined in Definition \ref{def:preliminary}. Here the condition $consistent_1$ is added for the operation $inv$ of taking inverse; this condition states that the value of 
$inv$ on any element $x$ from $G$ depends only on the existence of proof of $(x\varepsilon G)$, not on its content. The similar condition $consistent_2$ is added for the operation $mult$ of multiplication on $G$. These are particular cases of \textit{proof irrelevance}. As noted in \cite{Ned14}, proof irrelevance is desirable often but not always, therefore it is not added to $\lambda D$ as a general principle.

Definition \ref{def:group_subset_0} is more general than Definition \ref{def:group_type} but it makes proofs quite tedious. Instead we use the following more convenient definition, when a group is a set but the operations of multiplication and taking inverse are defined on the entire type $S$, with group axioms specified only for the set. 
We will show that the new definition covers the cases of the previous two definitions.

\bigskip
\begin{definition}
\textbf{Main Definition of Group}.

The following diagram defines that a subset $G$ of a type $S$ is a group.

\setlength{\derivskip}{4pt}
\begin{flagderiv}
\introduce*{}{S: *_s}{}

\introduce*{}{^{-1}:S\rightarrow S}{}
\step*{}{\text{Notation } \boldsymbol{x^{-1}}\;\; for\; ^{-1}x}{}

\introduce*{}{G:ps(S)}{}

\step*{}{\text{Definition  }\boldsymbol{Closure_1}(S,G,^{-1}):=\forall x:S.(x\varepsilon G\Rightarrow x^{-1}\varepsilon G)\;:\;*_p}{}
\done[2]

\introduce*{}{\cdot: S\rightarrow S\rightarrow S}{}

\step*{}{\text{Notation } \boldsymbol{x\cdot y}\;\; for\; \cdot xy}{}

\introduce*{}{e,x,y:S}{}

\step*{}{\text{Definition  }\boldsymbol{Inverse}_0(S,\cdot,e,x,y):=(x\cdot y=e\wedge y\cdot x=e)\;:\;*_p}{}
\done

\introduce*{}{G:ps(S)}{}

\step*{}{\text{Definition  }\boldsymbol{Closure_2}(S,G,\cdot):=\forall x:S.[x\varepsilon G\Rightarrow \forall y:S.(y\varepsilon G\Rightarrow (x\cdot y)\;\varepsilon\; G)]\;:\;*_p}{}

\step*{}{\text{Definition  }\boldsymbol{Assoc}(S,G,\cdot):=\forall x:S.[x\varepsilon G\Rightarrow \forall y:S.[y\varepsilon G
\Rightarrow \forall z:S.(z\varepsilon G\Rightarrow(x\cdot y)\cdot z=x\cdot (y\cdot z))]]\;:\;*_p}{}

\step*{}{\text{Definition  }\boldsymbol{Commut}(S,G,\cdot):=\forall x:S.[x\varepsilon G\Rightarrow \forall y:S.(y\varepsilon G\Rightarrow x\cdot y=y\cdot x)]\;:\;*_p}{}

\step*{}{\text{Definition  }\boldsymbol{Semi\mhyphen group}(S,G,\cdot):=Closure_2(S,G,\cdot)\wedge Assoc(S,G,\cdot)\;:\;*_p}{}

\introduce*{}{e:S}{}
\step*{}{\text{Definition  }\boldsymbol{Identity}(S,G,\cdot,e):=e\varepsilon G\wedge\forall x:S.[x\varepsilon G\Rightarrow (x\cdot e=x\wedge e\cdot x=x)]\;:\;*_p}{}
\step*{}{\text{Definition  }\boldsymbol{Monoid}(S,G,\cdot,e):=Semi\mhyphen group(S,G,\cdot)\wedge Identity(S,G,\cdot,e)\;:\;*_p}{}

\introduce*{}{^{-1}:S\rightarrow S}{}
\step*{}{\text{Definition  }\boldsymbol{Inverse}(S,G,\cdot,e,^{-1}):=\forall x:S.(x\varepsilon G\Rightarrow Inverse_0(S,\cdot,e,x,x^{-1}))
\;:\;*_p}{}
\step*{}{\text{Definition  }\boldsymbol{Inverse\mhyphen prop}(S,G,\cdot,e,^{-1}):=Closure_1(S,G,^{-1})\wedge Inverse(S,G,\cdot,e,^{-1})
\;:\;*_p}{}

\step*{}{\text{Definition  }\boldsymbol{Group}(S,G,\cdot,e,^{-1}):=Monoid(S,G,\cdot,e)\wedge 
Inverse\mhyphen prop(S,G,\cdot,e,^{-1})\;:\;*_p}{}

\step*{}{\text{Definition  }\boldsymbol{Abelian\mhyphen group}(S,G,\cdot,e,^{-1}):=Group(S,G,\cdot,e,^{-1})\wedge Commut(S,G,\cdot)\;:\;*_p}{}
\end{flagderiv}
\label{def:group_subset}
\end{definition}
\medskip

We use the usual shorthand for multiplication: $x\cdot y\cdot z$ means $(x\cdot y)\cdot z$.

\begin{example}
Suppose a type $S$ is a monoid. Then the set of all its invertible elements is a group.
\label{example:invert-monoid}
\end{example}

\begin{proof}
This is a sketch of the proof.
\setlength{\derivskip}{4pt}
\begin{flagderiv}
\introduce*{}{S: *_s\;|\;\cdot: S\rightarrow S\rightarrow S\;|\;e:S}{}

\step*{}{\text{Definition }\boldsymbol{Inv\mhyphen set}(S,\cdot,e):=\{x:S\;|\;invertible(S,\cdot,e,x)\}\;:\;ps(S)}{}
\step*{}{\text{Notation }G:=Inv\mhyphen set(S,\cdot,e)\;:\;ps(S)}{}

\assume*{}{u:monoid(S,\cdot,e)}{}

\skipsteps*{\dots}{}{}

\step*{}{\text{Definition }\boldsymbol{Inv}(S,\cdot,e,u):=
\ldots
\;:\;S\rightarrow S}{}
\step*{}{\text{Notation } ^{-1} \text{ for }Inv(S,\cdot,e,u)}{}

\skipsteps*{\dots}{}{}
\step*{}{\boldsymbol{term_{\ref{example:invert-monoid}}}(S,\cdot,e,u):=
\ldots
\;:\;Group(S,G,\cdot,e,^{-1})}{}
\end{flagderiv}

The full proof is in the Appendix, pg. 42.
\end{proof}

\begin{lemma} Definition \ref{def:group_type} is a particular case of  Definition \ref {def:group_subset}. 
\label{lemma:definitions}
\end{lemma}
\begin{proof}
This is a sketch of the proof.

\begin{flagderiv}
\introduce*{}{S: *_s\;|\;\;\cdot: S\rightarrow S\rightarrow S\;|\;e:S\;|\;^{-1}:S\rightarrow S\;|\;u:group(S,\cdot,e,^{-1})}{}
\step*{}{G:=full\mhyphen set(S)\;:\;ps(S)}{}
\skipsteps*{\dots}{}{}
\step*{}{\boldsymbol {term_{\ref {lemma:definitions}}}(S,\cdot,e,^{-1},u):=\ldots\;:\;Group(S,G,\cdot,e,^{-1})}{}
\end{flagderiv}

The full proof is in the Appendix, pg. 45.
\end{proof}

\begin{lemma} Definition \ref{def:group_subset_0} reduces to Definition \ref {def:group_subset}: 
if subset $G$ of type $S$ is a group in terms of Definition \ref{def:group_subset_0}, then the group operations of $G$ can be extended to $S$ and $G$ becomes a group in terms of Definition \ref{def:group_subset}. 
\label{lemma:two-definitions}
\end{lemma}
\begin{proof}
The statement seems obvious (when using Theorems \ref{theorem:extend_function_1} 
 and \ref{theorem:extend_function_2}) but the formal proof is quite long.
This is a sketch of the proof:
\setlength{\derivskip}{4pt}
\begin{flagderiv}
\introduce*{}{S: *_s\;|\;G:ps(S)\;|\;mult:\Pi x:S.[(x\varepsilon G)\rightarrow \Pi y:S.((y\varepsilon G)\rightarrow S)]\;|\;inv:\Pi x:S.((x\varepsilon G)\rightarrow S)}{}
\assume*{}{e:S\;|\;u:group(S,G,mult,e,inv)}{}
\skipsteps*{\dots}{}

\step*{}{c_1:=\ldots
\;:\; consistent_2 (S,S,G,mult)}{}
\step*{}{c_2:=\ldots
\;:\; consistent_1( S,S,G,inv)}{}

\skipsteps*{\dots}{}
\step*{}{\text{Definition }\; ^{-1}:=
Ext_1(S,S,G,inv,e,c_2)\;:\;S\rightarrow S}{}
\step*{}{\text{Definition }\; \cdot:=
Ext_2(S,S,G,mult,e,c_1)\;:\;S\rightarrow S \rightarrow S}{}
\skipsteps*{\dots}{}
\step*{}{c_3:=\ldots\;:\;Group(S,G,\cdot,e,^{-1})}{}
\end{flagderiv}

The full proof is in the Appendix, pg. 46.
\end{proof}

According to Lemmas \ref{lemma:definitions} and \ref{lemma:two-definitions}, we do not need Definitions \ref{def:group_subset_0} and \ref {def:group_type}. In the rest of the paper we will use only Definition \ref  {def:group_subset}, the Main Definition of Group.

\subsection{Example: Group of Permutations}
\begin{definition}
We use the definition of bijection from \cite{Ned14}.

\setlength{\derivskip}{4pt}
\begin{flagderiv}
\introduce*{}{S,T: *_s\;|\;f: S\rightarrow T}{}
\step*{}{\text{Definition  }\boldsymbol{inj}(S,T,f):=\forall x_1,x_2:S.(fx_1=fx_2\;\Rightarrow\; x_1=x_2)\;:\;*_p}{\textbf{Injective function}}

\step*{}{\text{Definition  }\boldsymbol{surj}(S,T,f):=\forall y:T.\exists x:S.(fx=y)\;:\;*_p}{\textbf{Surjective function}}

\step*{}{\text{Definition  }\boldsymbol{bij}(S,T,f):=
inj(S,T,f)\wedge surj(S,T,f)\;:\;*_p}{\textbf{Bijective function}}
\end{flagderiv}
\label{def:bijection}
\end{definition}

\begin{definition}
Here we define \textbf{permutation} on a type $M$ as a bijection of $M$ to itself.

\setlength{\derivskip}{4pt}
\begin{flagderiv}
\introduce*{}{M: *_s}{}
\introduce*{}{f: M\rightarrow M}{}
\step*{}{\text{Definition  }\boldsymbol{permutation}(M,f):=bij(M,M,f)\;:\;*_p}{\textbf{Permutation on M}}
\done
\step*{}{\text{Definition  }\boldsymbol{Perm(M)}:=\{f: M\rightarrow M\;|\; permutation(M,f)\}
\;:\;ps(M\rightarrow M)}{}
\end{flagderiv}
\label{def:permutation}
\end{definition}

\begin{lemma}
For each permutation $f$ on the type $M$ we construct its inverse $\textsf{invrs(M,f)}$. This is a sketch of the construction.
\setlength{\derivskip}{4pt}
\begin{flagderiv}
\introduce*{}{M: *_s\;|\;f: M\rightarrow M\;|\;u:f\varepsilon Perm(M)}{}
\skipsteps*{\dots}{}{}

\step*{}{\boldsymbol{ invrs}(M ,f, u):=\ldots\;:\;M\rightarrow M\hspace{2cm} \textbf{Inverse of permutation f}}{}

\step*{}{\text{Notation }\; g:=invrs(M ,f, u)\;:\;M\rightarrow M}{}
\skipsteps*{\dots}{}{}

\step*{}{ \boldsymbol{term_{ \ref{def:inverse}}}(M,f,u):=\ldots
\;:\;\boldsymbol{inverse(M\rightarrow M,\circ,id_M, f,g)}}{}
\end{flagderiv}
\label{def:inverse}

The full construction is in the Appendix, pg. 48.
\end{lemma}

\begin{proposition} Suppose $f$ is an element of the monoid $(M\rightarrow M,\circ,id_M)$ from Example \ref{example: monoid_functions}. Then
\[f \text{ is invertible }\Leftrightarrow f\text{ is a permuation on }M.\]
\label{lemma:permutations}
\end{proposition}
\begin{proof}
This is a sketch of the proof.

\setlength{\derivskip}{4pt}
\begin{flagderiv}
\introduce*{}{M: *_s\;|\;f: M\rightarrow M}{}

\skipsteps*{\dots}{}{}

\step*{}{ \boldsymbol{term_{ \ref{lemma:permutations}}}(M,f):=\ldots
\;:\;\boldsymbol{invertible(M\rightarrow M,\circ,id_M,f)\Leftrightarrow
f\varepsilon Perm(M)}}{}
\end{flagderiv}

The full proof is in the Appendix, pg. 49.
\end{proof}

\begin{theorem} The set of all permutations on a type $M$ with operation of composition is a group.
\label{theorem:permutations}
\end{theorem}
\begin{proof}
This is a sketch of the proof.
\setlength{\derivskip}{4pt}
\begin{flagderiv}
\introduce*{}{M: *_s}{}
\step*{}{a_1:=term_{\ref{example: monoid_functions}}(M)\;:\;monoid(M\rightarrow M,\circ,id_M)}{}

\step*{}{^{-1}:=Inv(M\rightarrow M, \circ,id_M,a_1)\;:\;(M\rightarrow M)\rightarrow M\rightarrow M}{}

\skipsteps*{\dots}{}{}
\step*{}{ \boldsymbol{term_{\ref{theorem:permutations}}}(M):=\ldots
\;:\;\boldsymbol{ Group(M\rightarrow M,Perm(M), \circ,id_M,^{-1})}}{}
\end{flagderiv}

The full proof is in the Appendix, pg. 50.
\end{proof}

\bigskip
\subsection{Basic Properties of Groups}

\begin{lemma}
The flag diagram in this technical lemma produces proof terms that will be used in future proofs.

\setlength{\derivskip}{4pt}
\begin{flagderiv}
\introduce*{}{S: *_s\;|\;
G:ps(S)\;|\;
\cdot: S\rightarrow S\rightarrow S\;|\;e:S\;|\; ^{-1}:S\rightarrow S\;|\;u:Group(S,G,\cdot,e,^{-1})}{}
\step*{}{a_1:=\wedge_1(\wedge_1(u))
\;:\;Semi\mhyphen group(S,G,\cdot)}{}
\step*{}{a_2:=\wedge_2(\wedge_1(u))
\;:\;Identity(S,G,\cdot,e)}{}
\step*{}{a_3:=\wedge_1(\wedge_2(u))
\;:\;Closure_1(S,G,^{-1})}{}
\step*{}{a_4:=\wedge_2(\wedge_2(u))
\;:\;Inverse(S,G,\cdot,e,^{-1})}{}
\step*{}{a_5:=\wedge_1(a_1)
\;:\;Closure_2(S,G,\cdot)}{}

\step*{}{\boldsymbol{Gr_1}(S,G,\cdot,e,^{-1}, u):=
\wedge_1(a_2)\;:\;
\boldsymbol{e\varepsilon G}}{}

\step*{}{\boldsymbol{Gr_2}(S,G,\cdot,e,^{-1}, u):=\wedge_2(a_1)\;:\;\boldsymbol{Assoc}(S,G,\cdot)}{}

\introduce*{}{x:S\;|\;v:x\varepsilon G}{}

\step*{}{a_6:=\wedge_2(a_2)xv
\;:\;x\cdot e=x\wedge e\cdot x=x}{}

\step*{}{\boldsymbol{Gr_3}(S,G,\cdot,e,^{-1}, u,x,v):=a_3xv\;:\;\boldsymbol{x^{-1}\varepsilon G}}{}
\step*{}{\boldsymbol{Gr_4}(S,G,\cdot,e,^{-1},u, x,v):=\wedge_1(a_6)\;:\;\boldsymbol{x\cdot e=x}}{}
\step*{}{\boldsymbol{Gr_5}(S,G,\cdot,e,^{-1},u, x,v):=\wedge_2(a_6)\;:\;\boldsymbol{e\cdot x=x}}{}
\step*{}{\boldsymbol{Gr_6}(S,G,\cdot,e,^{-1},u, x,v):=a_4xv\;:\;\boldsymbol{Inverse}_0(S,\cdot,e,x,x^{-1})}{}
\step*{}{\boldsymbol{Gr_7}(S,G,\cdot,e,^{-1},u,x,v):=
\wedge_1(a_4xv)\;:\;\boldsymbol{x\cdot x^{-1}=e}}{}
\step*{}{\boldsymbol{Gr_8}(S,G,\cdot,e,^{-1},u, x,v):=\wedge_2(a_4xv)\;:\; \boldsymbol{x^{-1}\cdot x=e}}{}
\introduce*{}{y:S\;|\;w:y\varepsilon G}{}
\step*{}{\boldsymbol{Gr_9}(S,G,\cdot,e,^{-1},u,x, y,v,w):=a_5xvyw \;:\; \boldsymbol{(x\cdot y)\;\varepsilon G}}{}
\end{flagderiv}
\label{lemma:technical_1}
\end{lemma}

\begin{proposition}
For any elements $x,y,z$ of a group $G$ with identity $e$:

1) $x\cdot y=z\;\Rightarrow\; y=x^{-1}\cdot z$,
\medskip

2) $y\cdot x=z\;\Rightarrow\; y=z\cdot x^{-1}$,
\medskip

3) $x\cdot y=e\;\Rightarrow\; y=x^{-1}$,
\medskip

4) $y\cdot x=e\;\Rightarrow\; y=x^{-1}$,
\medskip

5) $(x\cdot y)^{-1}=y^{-1}\cdot x^{-1}$,
\medskip

6) $(x^{-1})^{-1}=x$,
\medskip

7) $x\cdot x=x\;\Rightarrow\; x=e$,
\medskip

8) $e^{-1}=e$.
\label{theorem:axiom_corollary}
\end{proposition}
\begin{proof}
This is a sketch of the proof for all parts.
\begin{flagderiv}
\introduce*{}{S: *_s\;|\;G:ps(S)\;|\;\cdot: S\rightarrow S\rightarrow S\;|\;e:S\;|\;^{-1}:S\rightarrow S\;|\;u:Group(S,G,\cdot,e,^{-1})}{}

\introduce*{}{x:S\;|\;v:x\varepsilon G}{}

\introduce*{}{y:S\;|\;w:y\varepsilon G}{}

\introduce*{}{z:S\;|\;r:z\varepsilon G}{}

\skipsteps*{\dots}{}
\step*{}{\boldsymbol{term_{\ref{theorem:axiom_corollary}.1}}(S,G,\cdot,e,^{-1},u,
x,y,z,v,w,r):=\ldots\;:\;\boldsymbol{(x\cdot y=z\Rightarrow y=x^{-1}\cdot z)}}{}
\skipsteps*{\dots}{}
\step*{}{\boldsymbol{term_{\ref{theorem:axiom_corollary}.2}}(S,G,\cdot,e,^{-1},u,
x,y,z,v,w,r):=\ldots\;:\;\boldsymbol{(y\cdot x=z\Rightarrow y=z\cdot x^{-1})}}{}

\done

\skipsteps*{\dots}{}

\step*{}{\boldsymbol{term_{\ref{theorem:axiom_corollary}.3}}(S,G,\cdot,e,^{-1},u,
x,y,v,w):=\ldots\;:\;\boldsymbol{(x\cdot y=e\Rightarrow y=x^{-1})}}{}

\skipsteps*{\dots}{}

\step*{}{\boldsymbol{term_{\ref{theorem:axiom_corollary}.4}}(S,G,\cdot,e,^{-1},u,
x,y,v,w):=\ldots\;:\;\boldsymbol{(y\cdot x=e\Rightarrow y=x^{-1})}}{}

\skipsteps*{\dots}{}

\step*{}{\boldsymbol{ term_{\ref{theorem:axiom_corollary}.5}}(S,G,\cdot,e,^{-1},u,x,y,v,w):=
\ldots
\;:\;\boldsymbol{(x\cdot y)^{-1}=y^{-1}\cdot x^{-1}}}{}

\done

\skipsteps*{\dots}{}

\step*{}{\boldsymbol{term_{\ref{theorem:axiom_corollary}.6}}(S,G,\cdot,e,^{-1},u,x,v):=\ldots\;:\;\boldsymbol{(x^{-1})^{-1}= x}}{}

\skipsteps*{\dots}{}

\step*{}{\boldsymbol{term_{\ref{theorem:axiom_corollary}.7}}(S,G,\cdot,e,^{-1},u,x,v):=\ldots
\;:\;\boldsymbol{(x\cdot x=x\Rightarrow x=e)}}{}

\done

\skipsteps*{\dots}{}

\step*{}{\boldsymbol{ term_{ \ref{theorem:axiom_corollary}.8}}(S,G,\cdot,e, ^{-1},u):=\ldots\;:\;e^{-1}=e}{}
\end{flagderiv}

The full proof is in the Appendix, pg. 50. 
\end{proof}

\bigskip
\section{Subgroups}

\subsection{Definitions and Examples}
\begin{definition}
Here we define a subgroup $H$ of a group $G$ and a corresponding relation $R_H$ on $G$.
\begin{flagderiv}
\introduce*{}{S: *_s\;|\;
\cdot: S\rightarrow S\rightarrow S\;|\;G:ps(S)\;|\;^{-1}:S\rightarrow S}{}
\introduce*{}{H:ps(S)}{}

\step*{}{\text{Definition }\boldsymbol{R}(S,\cdot,^{-1}, H):=\lambda x,y:S.(x^{-1}\cdot y)\;\varepsilon \,H\;:\,S\rightarrow S\rightarrow *_p}{}

\step*{}{\text{Notation }
\boldsymbol{R_H}\text{ for } R(S,\cdot,^{-1},H)}{}

\introduce*{}{e:S}{}

\step*{}{\text{Definition }\boldsymbol{Subgroup}(S,G,\cdot,e,^{-1},H):=H\subseteq G\wedge
e\varepsilon H\wedge Closure_1(S,H,^{-1})\wedge Closure_2(S,H,\cdot)\;:\;*_p}{}
\step*{}{\text{Notation: }
\boldsymbol{H\leqslant G}\text{ for } Subgroup(S,G,\cdot,e,^{-1},H)}{}

\end{flagderiv}
\label{def:subgroup}
\end{definition}

Here we used the predicates \textit{Closure}$_1$ and \textit{Closure}$_2$ introduced in Definition \ref{def:group_subset}.

\begin{proposition}
1) A subgroup is a group itself.

2) If $B$ and $C$ are subgroups of a group $G$ and $C\subseteq B$, then $C$ is a subgroup of $B$.
\label{lemma:subgroup}
\end{proposition}
\begin{proof}
These are sketches of the proofs.

1)
\begin{flagderiv}
\introduce*{}{S: *_s\;|\;G:ps(S)\;|\;\cdot: S\rightarrow S\rightarrow S\;|\;e:S\;|\;^{-1}:S\rightarrow S\;|\;u:Group(S,G,\cdot,e,^{-1})}{}
\introduce*{}{H:ps(S)\;|\;v:H\leqslant G}{}
\skipsteps*{\dots}{}{}
\step*{}{\boldsymbol {term_{\ref{lemma:subgroup}.1}}(S,G,\cdot,e,^{-1},u,H,v):=\ldots\;:\;\boldsymbol{ Group(S,H,\cdot,e,^{-1})}}{}
\end{flagderiv}

\newpage
2)
\begin{flagderiv}
\introduce*{}{S: *_s\;|\;G:ps(S)\;|\;\cdot: S\rightarrow S\rightarrow S\;|\;e:S\;|\;^{-1}:S\rightarrow S\;|\;u:Group(S,G,\cdot,e,^{-1})}{}
\introduce*{}{B,C:ps(S)\;|\;v_1:B\leqslant G\;|\;v_2:C\leqslant G\;|\;w:C\subseteq B}{}

\skipsteps*{\dots}{}{}

\step*{}{\boldsymbol {term_{\ref{lemma:subgroup}.2}}(S,G,\cdot,e,^{-1},u,B,C,v_1,v_2,w):=\ldots\;:\;\boldsymbol{C\leqslant B}}{}

\end{flagderiv}

The full proofs are in the Appendix, pg. 52.
\end{proof}

\begin{proposition}
If $G$ is a group with identity $e$, then the following sets are its subgroups:

1) $G$;

2) $\{e\}$;

3) intersection of any two subgroups of $G$;

4) intersection of any non-empty collection of subgroups of $G$.
\label{lemma:triv-subgroups}
\end{proposition}
\begin{proof}
These are sketches of the proofs.

1)
\begin{flagderiv}
\introduce*{}{S: *_s\;|\;G:ps(S)\;|\;\cdot: S\rightarrow S\rightarrow S\;|\;e:S\;|\;^{-1}:S\rightarrow S\;|\;u:Group(S,G,\cdot,e,^{-1})}{}
\skipsteps*{\dots}{}{}
\step*{}{\boldsymbol{term_{\ref{lemma:triv-subgroups}.1}}(S,G,\cdot, e, ^{-1},u):=\ldots\;:\;\boldsymbol{G\leqslant G}}{}
\end{flagderiv}

2) 
\begin{flagderiv}
\introduce*{}{S: *_s\;|\;G:ps(S)\;|\;\cdot: S\rightarrow S\rightarrow S\;|\;e:S\;|\;^{-1}:S\rightarrow S\;|\;u:Group(S,G,\cdot,e,^{-1})}{}
\step*{}{\text{Notation }H:=\{x:S\;|\;x=e\}
\;:\;ps(S)}{}
\skipsteps*{\dots}{}{}
\step*{}{\boldsymbol{term_{\ref{lemma:triv-subgroups}.2}}(S,G, \cdot,e,^{-1},u):=\ldots\;:\;\boldsymbol{H\leqslant G}}{}
\end{flagderiv}

3) 
\begin{flagderiv}
\introduce*{}{S: *_s\;|\;G:ps(S)\;|\;\cdot: S\rightarrow S\rightarrow S\;|\;e:S\;|\;^{-1}:S\rightarrow S\;|\;u:Group(S,G,\cdot,e,^{-1})}{}

\introduce*{}{B,C:ps(S)\;|\;v:B\leqslant G\;|\;w:C\leqslant G}{}

\skipsteps*{\dots}{}{}
\step*{}{\boldsymbol{term_{\ref{lemma:triv-subgroups}.3}}(S,G,\cdot, e, ^{-1},u,B,C,v,w):=\ldots\;:\;\boldsymbol{B\cap C\leqslant G}}{}
\end{flagderiv}

4) 
\begin{flagderiv}
\introduce*{}{S: *_s\;|\;G:ps(S)\;|\;\cdot: S\rightarrow S\rightarrow S\;|\;e:S\;|\;^{-1}:S\rightarrow S\;|\;u:Group(S,G,\cdot,e,^{-1})}{}

\introduce*{}{U:ps(ps(S))\;|\;v:(\exists X:ps(S).X\varepsilon U)\;|\;w:[\forall X:ps(S).(X\varepsilon U\Rightarrow X\leqslant G)]}{}
\skipsteps*{\dots}{}{}
\step*{}{\boldsymbol{term_{\ref{lemma:triv-subgroups}.4}}(S,G, \cdot,e,^{-1},u,U,v,w):=\ldots\;:\;\boldsymbol{\cap U\leqslant G}}{}
\end{flagderiv}

The full proofs are in the Appendix, pg. 53.  
\end{proof}

\begin{example}
Suppose $m$ is a positive integer. In the abelian additive group $\mathbb{Z}$ from Example \ref{example:Z} the set $m\mathbb{Z}$ of all multiples of $m$ is a subgroup.
\begin{proof}
We define formally
\[m\mathbb{Z}:=\{x:\mathbb{Z}\;|\;\exists n:\mathbb{Z}.(x=mn)\}.\]
As in Example \ref{example:Z}, we refer to the formalization of integer arithmetic in \cite{Ned14} and skip the formal proof.
\end{proof}
\label{example:mZ}
\end{example}

\begin{proposition}
For any subgroup $H$ of a group $G$ the predicate $R_H$ is an equivalence relation on $G$.
\label{lemma:equiv}
\end{proposition}
\begin{proof}
This is a sketch of the proof:
\begin{flagderiv}
\introduce*{}{S: *_s\;|\;G:ps(S)\;|\;\cdot: S\rightarrow S\rightarrow S\;|\;e:S\;|\;^{-1}:S\rightarrow S\;|\;u:Group(S,G,\cdot,e,^{-1})}{}
\introduce*{}{H:ps(S)\;|\;v:H\leqslant G}{}
\skipsteps*{\dots}{}{}
\step*{}{\boldsymbol{term_{\ref{lemma:equiv}}}(S,G,\cdot,e,^{-1},u, H, v):=\ldots\;:\;\boldsymbol{equiv\mhyphen rel(S,G,R_H)}}{}
\end{flagderiv}

The full proof is in the Appendix, pg. 56.
\end{proof}

\begin{proposition}
Suppose $H$ is a subgroup of a group $G$. Then for any elements $x,y$ of $G$:
\[R_Hxy\Leftrightarrow (R_Hx\cap G=R_Hy\cap G).\]
\label{lemma:equiv-1}
\end{proposition}

\begin{proof}
This is a sketch of the proof:
\begin{flagderiv}
\introduce*{}{S: *_s\;|\;G:ps(S)\;|\;\cdot: S\rightarrow S\rightarrow S\;|\;e:S\;|\;^{-1}:S\rightarrow S\;|\;u:Group(S,G,\cdot,e,^{-1})}{}
\introduce*{}{H:ps(S)\;|\;v:H\leqslant G}{}
\introduce*{}{x,y:S\;|\; w_1:x\varepsilon G\;|\; w_2:y\varepsilon G}{}
\skipsteps*{\dots}{}{}
\step*{}{\boldsymbol {term_{\ref{lemma:equiv-1}}}(S,G,\cdot,e,^{-1},u, H, v,x,y,w_1,w_2):=\ldots\;:\;\boldsymbol{R_Hxy\Leftrightarrow (R_Hx\cap G=R_Hy\cap G)}}{}
\end{flagderiv}

The full proof is in the Appendix, pg. 57.
\end{proof}

\bigskip
\subsection{Set Products and Cosets}
\begin{definition}
In the following diagram we define products $g\cdot B$, $B\cdot g$, $B\cdot C$, and $B^{-1}$, where $g$ is an element of type $S$, and $B$ and $C$ are subsets of $S$.
\begin{flagderiv}
\introduce*{}{S: *_s}{}
\introduce*{}{\cdot: S\rightarrow S\rightarrow S}{}
\introduce *{}{B:ps(S)\;|\;g:S}{}
\step*{}{\text{Definition  }mt_1(S,\cdot,B,g):=
\{x:S\;|\;\exists b:S.(b\varepsilon B\wedge x=g\cdot b)\}:ps(S)}{}

\step*{}{\text{Notation  \;}\boldsymbol{g\cdot B} \;\text{ for }mt_1(S,\cdot,B,g)}{}

\step*{}{\text{Definition  }mt_2(S,\cdot,B,g):=
\{x:S\;|\;\exists b:S.(b\varepsilon B\wedge x=b\cdot g)\}:ps(S)}{}

\step*{}{\text{Notation  \;}\boldsymbol{B\cdot g} \;\text{ for }mt_2(S,\cdot,B,g)}{}
\done
\introduce *{}{B,C :ps(S)}{}

\step*{}{\text{Definition  }Mt_1(S,\cdot,B,C):=
\{x:S\;|\;\exists c:S.(c\varepsilon C\wedge x\varepsilon (B\cdot c)\}:ps(S)}{}
\step*{}{\text{Notation  \;}\boldsymbol{B\cdot C} \;\text{ for }Mt_1(S,\cdot,B,C)}{}
\step*{}{\text{Definition  }Mt_2(S,\cdot,B,C):=
\{x:S\;|\;\exists b:S.(b\varepsilon B\wedge x\varepsilon (b\cdot C)\}:ps(S)}{}
\done[2]

\introduce*{}{^{-1}:S\rightarrow S\;|\;B:ps(S)}{}

\step*{}{\text{Definition  } Iv(S,\,^{-1},B):=
\{x:S\;|\;\exists b:S.(b\varepsilon B\wedge x=b^{-1})\}:ps(S)}{}

\step*{}{\text{Notation  \;}\boldsymbol{B^{-1}} \;\text{ for }Iv(S,\,^{-1},B)}{}
\end{flagderiv}

If $H$ is a subgroup of a group $G$, then sets of the form $(x\cdot H)$ are called \textbf{left cosets} and of the form $(H\cdot x)$ - \textbf{right cosets}.
\label{def:set-product}
\end{definition}

\textit{Remark:} the multiplication on $S$ and the dots in $g\cdot C$, $B\cdot g$, $B\cdot C$ are all different operations but for brevity we use the same symbol for all of them. We use a similar convention for set inverse.

\begin{lemma}
Under conditions of Definition \ref{def:set-product}:
\[Mt_1(S,\cdot,B,C)=Mt_2(S,\cdot,B,C).\]
\label{lemma:set-mult}
\end{lemma}

\vspace{-0.7cm}

\begin{proof}
This is a sketch of the proof: 
\begin{flagderiv}
\introduce*{}{S: *_s\;|\;\cdot: S\rightarrow S\rightarrow S\;|\;B,C:ps(S)}{}
\skipsteps*{\dots}{}{}

\step*{}{\boldsymbol {term_{\ref{lemma:set-mult}}}(S,\cdot, B,C) := \ldots\;:\;\boldsymbol{Mt_1(S,\cdot,B,C)= Mt_2(S,\cdot,B,C)}}{}
\end{flagderiv}

The full proof is in the Appendix, pg. 58. 
\end{proof}

Due to this lemma, we can use the notation $B\cdot C$ for both $Mt_1(S,\cdot,B,C)$ and $Mt_2(S,\cdot,B,C)$. We will do that for brevity and we will not elaborate on details of substituting one of them with the other.

\begin{lemma}
Suppose $B$ and $C$ are subsets of type $S$, $b\varepsilon B$, $c\varepsilon C$, and $g$ is an element of $S$. Then the following hold.
\medskip

1) $b^{-1}\varepsilon \;B^{-1}$.
\medskip

2) $(g\cdot b)\;\varepsilon \;(g\cdot B)$.
\medskip

3) $(b\cdot g)\;\varepsilon \;(B\cdot g)$.
\medskip

4) $(b\cdot c)\;\varepsilon \;(B\cdot C)$.
\label{lemma:mult-triv}
\end{lemma}

\begin{proof}
This is a sketch of the proof for all parts.

\begin{flagderiv}
\introduce*{}{S: *_s}{}

\introduce*{}{^{-1}:S\rightarrow S\;|\;B:ps(S)\;|\;b:S\;|\;u:b\varepsilon B}{}

\skipsteps*{\dots}{}{}

\step*{}{\boldsymbol {term_{\ref{lemma:mult-triv}.1}}(S,^{-1},B,b,u):=\ldots
\;:\;\boldsymbol{b^{-1}\varepsilon B^{-1}}}{}

\done

\introduce*{}{\cdot: S\rightarrow S\rightarrow S\;|\; B:ps(S)\;|\;g,b:S\;|\;u:b\varepsilon B}{}

\skipsteps*{\dots}{}{}

\step*{}{\boldsymbol {term_{\ref{lemma:mult-triv}.2}}(S,\cdot,B,g,b,u):=\ldots
\;:\;\boldsymbol{(g\cdot b)\;\varepsilon \;(g\cdot B)}}{}

\step*{}{\boldsymbol {term_{\ref{lemma:mult-triv}.3}}(S,\cdot,B,g,b,u):=\ldots
\;:\;\boldsymbol{(b\cdot g)\;\varepsilon \;(B\cdot g)}}{}

\done

\introduce*{}{\cdot: S\rightarrow S\rightarrow S\;|\; B,C:ps(S)\;|\;b,c:S\;|\;u:b\varepsilon B\;|\;v:c\varepsilon C}{}

\skipsteps*{\dots}{}{}

\step*{}{\boldsymbol {term_{\ref{lemma:mult-triv}.4}}(S,\cdot,B,C,b,c,u,v):=\ldots\;:\; \boldsymbol{(b\cdot c)\;\varepsilon \;(B\cdot C)}}{}
\end{flagderiv}

The full proof is in the Appendix, pg. 59. 
\end{proof}

\begin{proposition}
Suppose $H$ is a subgroup of a group $G$ and $x,y$ are elements of $G$. Then the following hold.
\medskip

1) $x\cdot H=R_Hx\cap G$;
\medskip

2) $x\cdot H=y\cdot H\;\Leftrightarrow\;R_Hxy$.
\label{lemma:eq-class}
\end{proposition}

\begin{proof}
This is a sketch of the proof.

\begin{flagderiv}
\introduce*{}{S: *_s\;|\;G:ps(S)\;|\;\cdot: S\rightarrow S\rightarrow S\;|\;e:S\;|\;^{-1}:S\rightarrow S\;|\;u:Group(S,G,\cdot,e,^{-1})}{}

\introduce*{}{H:ps(S)\;|\;v:H\leqslant G\;|\;x:S\;|\;w:x\varepsilon G}{}
\skipsteps*{\dots}{}{}
\step*{}{\boldsymbol {term_{\ref{lemma:eq-class}.1}}(S,G,\cdot,e,^{-1},u, H, v,x,w):=\ldots
\;:\;\boldsymbol{x\cdot H=R_Hx\cap G}}{}
\introduce*{}{y:S\;|\;r:y\varepsilon G}{}
\skipsteps*{\dots}{}{}

\step*{}{\boldsymbol {term_{\ref{lemma:eq-class}.2}}(S,G,\cdot,e,^{-1},u, H,v,x,y,w,r):=\ldots
\;:\;\boldsymbol{x\cdot H=y\cdot H\;\Leftrightarrow \;R_Hxy}}{}
\end{flagderiv}

The full proof is in the Appendix, pg. 59. 
\end{proof}

\begin{lemma}
Suppose $G$ is a group with identity $e$. Then for any element $g$ of $G$ and subsets $B,C$ of $G$ the following hold.

1) $e\cdot B=B$.
\medskip

2) $B\cdot e=B$.
\medskip

3) $B^{-1}\subseteq G$.
\medskip

4) $(g\cdot B)\subseteq G$.
\medskip

5) $(B\cdot g)\subseteq G$.
\medskip

6) $(g\cdot B)^{-1}=B^{-1}\cdot g^{-1}$.
\medskip

7) $(B\cdot g)^{-1}=g^{-1}\cdot B^{-1}$.
\medskip

8) $(B\cdot C)\subseteq G$.
\medskip

9) $(B\cdot C)^{-1}=C^{-1}\cdot B^{-1}$.
\label{lemma:mult}
\end{lemma}

\begin{proof}
These are sketches of the proofs.

1) - 3)
\vspace{-0.2cm}
\begin{flagderiv}
\introduce*{}{S: *_s\;|\;G:ps(S)\;|\;\cdot: S\rightarrow S\rightarrow S\;|\;e:S\;|\;^{-1}:S\rightarrow S\;|\;u:Group(S,G,\cdot,e,^{-1})}{}

\introduce*{}{B:ps(S)\;|\;v:B\subseteq G}{}

\skipsteps*{\dots}{}{}

\step*{}{\boldsymbol {term_{ \ref{lemma:mult}.1}}(S,G,\cdot,e,^{-1},u, B,v):=\ldots 
\;:\;\boldsymbol{e\cdot B=B}}{}
\skipsteps*{\dots}{}{}
\step*{}{\boldsymbol{ term_{\ref{lemma:mult}.2}}(S,G,\cdot,e,^{-1},u, B,v):=\ldots
\;:\;\boldsymbol{B\cdot e=B}}{}
\skipsteps*{\dots}{}{}
\step*{}{\boldsymbol{ term_{\ref{lemma:mult}.3}}(S,G,\cdot,e,^{-1},u,B,v):=
\ldots
\;:\;\boldsymbol{B^{-1}\subseteq G}}{}
\end{flagderiv}

4) - 7) 
\begin{flagderiv}
\introduce*{}{S: *_s\;|\;G:ps(S)\;|\;\cdot: S\rightarrow S\rightarrow S\;|\;e:S\;|\;^{-1}:S\rightarrow S\;|\;u:Group(S,G,\cdot,e,^{-1})}{}

\introduce*{}{B:ps(S)\;|\;v:B\subseteq G\;|\;g:S\;|\;w:g\varepsilon G}{}
\skipsteps*{\dots}{}{}

\step*{}{\boldsymbol {term_{ \ref{lemma:mult}.4}}(S,G,\cdot,e,^{-1},u, B,v,g,w):=\ldots
\;:\;\boldsymbol{(g\cdot B)\subseteq G}}{}
\skipsteps*{\dots}{}{}

\step*{}{\boldsymbol {term_{ \ref{lemma:mult}.5}}(S,G,\cdot,e,^{-1},u, B,v,g,w):=\ldots
\;:\;\boldsymbol{(B\cdot g)\subseteq G}}{}
\skipsteps*{\dots}{}{}

\step*{}{\boldsymbol {term_{ \ref{lemma:mult}.6}}(S,G,\cdot,e,^{-1},u, B,v,g,w):=\ldots
\;:\;\boldsymbol{(g\cdot B)^{-1}=B^{-1}\cdot g^{-1}}}{}
\skipsteps*{\dots}{}{}

\step*{}{\boldsymbol {term_{ \ref{lemma:mult}.7}}(S,G,\cdot,e,^{-1},u, B,v,g,w):=\ldots
\;:\;\boldsymbol{(B\cdot g)^{-1}=g^{-1}\cdot B^{-1}}}{}
\end{flagderiv}

8) - 9)  
\begin{flagderiv}
\introduce*{}{S: *_s\;|\;G:ps(S)\;|\;\cdot: S\rightarrow S\rightarrow S\;|\;e:S\;|\;^{-1}:S\rightarrow S\;|\;u:Group(S,G,\cdot,e,^{-1})}{}

\introduce*{}{B,C:ps(S)\;|\;v:B\subseteq G\;|\;w:B\subseteq G}{}
\skipsteps*{\dots}{}{}
\step*{}{\boldsymbol {term_{ \ref{lemma:mult}.8}}(S,G,\cdot,e,^{-1},u, B,C,v,w):=\ldots
\;:\;\boldsymbol{(B\cdot C)\subseteq G}}{}
\skipsteps*{\dots}{}{}
\step*{}{\boldsymbol {term_{ \ref{lemma:mult}.9}}(S,G,\cdot,e,^{-1},u, B,C,v,w):=\ldots
\;:\;\boldsymbol{(B\cdot C)^{-1}=C^{-1}\cdot B^{-1}}}{}
\end{flagderiv}

The full proofs are in the Appendix, pg. 61. 
\end{proof}

\begin{proposition}
For any elements $g,h$ of a group $G$ and subsets $B,C,D$ of $G$ the following hold.
\medskip

1) $(B\cdot g)\cdot h=B\cdot (g\cdot h)$;
\medskip

2) $(g\cdot B)\cdot h=g\cdot (B\cdot h)$;
\medskip

3) $(g\cdot h)\cdot B=g\cdot (h\cdot B)$;
\medskip

4) $(B\cdot C)\cdot g=B\cdot (C\cdot g)$;
\medskip

5) $(B\cdot g)\cdot C=B\cdot (g\cdot C)$;
\medskip

6) $(g\cdot B)\cdot C=g\cdot (B\cdot C)$;
\medskip

7) $(B\cdot C)\cdot D=B\cdot (C\cdot D)$.
\label{lemma:set-assoc}
\end{proposition}

\newpage
\begin{proof}
These are sketches of the proofs.

1) - 3) 
\begin{flagderiv}
\introduce*{}{S: *_s\;|\;G:ps(S)\;|\;\cdot: S\rightarrow S\rightarrow S\;|\;e:S\;|\;^{-1}:S\rightarrow S\;|\;u:Group(S,G,\cdot,e,^{-1})}{}

\introduce*{}{B:ps(S)\;|\;v:B\subseteq G\;|\;g,h:S\;|\;w_1:g\varepsilon G\;|\;w_2:h\varepsilon G}{}

\skipsteps*{\dots}{}{}
\step*{}{\boldsymbol {term_{ \ref{lemma:set-assoc}.1}}(S,G,\cdot,e,^{-1},u, B,v,g,h,w_1,w_2):=\ldots
\;:\;\boldsymbol{(B\cdot g)\cdot h=B\cdot(g\cdot h)}}{}

\skipsteps*{\dots}{}{}
\step*{}{\boldsymbol {term_{ \ref{lemma:set-assoc}.2}}(S,G,\cdot,e,^{-1},u, B,v,g,h,w_1,w_2):=\ldots
\;:\;\boldsymbol{(g\cdot B)\cdot h=g\cdot(B\cdot h)}}{}
\skipsteps*{\dots}{}{}
\step*{}{\boldsymbol {term_{ \ref{lemma:set-assoc}.3}}(S,G,\cdot,e, ^{-1},u, B,v,g,h,w_1, w_2):=\ldots
\;:\;\boldsymbol{(g\cdot h)\cdot B=g\cdot(h\cdot B)}}{}
\end{flagderiv}

4) - 6) 
 
\begin{flagderiv}
\introduce*{}{S: *_s\;|\;G:ps(S)\;|\;\cdot: S\rightarrow S\rightarrow S\;|\;e:S\;|\;^{-1}:S\rightarrow S\;|\;u:Group(S,G,\cdot,e,^{-1})}{}
\introduce*{}{B,C:ps(S)\;|\;v_1:B\subseteq G\;|\;v_2:C\subseteq G\;|\;g:S\;|\;w:g\varepsilon G}{}

\skipsteps*{\dots}{}{}
\step*{}{\boldsymbol {term_{ \ref{lemma:set-assoc}.4}}(S,G,\cdot,e, ^{-1},u, B,C,v_1, v_2,g, w):=\ldots
\;:\;\boldsymbol{(B\cdot C)\cdot g=B\cdot(C\cdot g)}}{}
\skipsteps*{\dots}{}{}
\step*{}{\boldsymbol {term_{ \ref{lemma:set-assoc}.5}}(S,G,\cdot,e, ^{-1},u, B,C,v_1, v_2,g, w):=\ldots
\;:\;\boldsymbol{(B\cdot g)\cdot C=B\cdot(g\cdot C)}}{}
\skipsteps*{\dots}{}{}
\step*{}{\boldsymbol {term_{ \ref{lemma:set-assoc}.6}}(S,G,\cdot,e, ^{-1},u, B,C,v_1, v_2,g, w):=\ldots
\;:\;\boldsymbol{(g\cdot B)\cdot C=g\cdot(B\cdot C)}}{}
\end{flagderiv}

7) 
\begin{flagderiv}
\introduce*{}{S: *_s\;|\;G:ps(S)\;|\;\cdot: S\rightarrow S\rightarrow S\;|\;e:S\;|\;^{-1}:S\rightarrow S\;|\;u:Group(S,G,\cdot,e,^{-1})}{}
\introduce*{}{B,C,D:ps(S)\;|\;v_1:B\subseteq G\;|\;v_2:C\subseteq G\;|\;v_3:D\subseteq G}{}
\skipsteps*{\dots}{}{}
\step*{}{\boldsymbol {term_{ \ref{lemma:set-assoc}.7}}(S,G,\cdot,e, ^{-1},u, B,C,D,v_1, v_2,v_3):=\ldots
\;:\;\boldsymbol{(B\cdot C)\cdot D=B\cdot(C\cdot D)}}{}
\end{flagderiv}

The full proofs are in the Appendix, pg. 67.
\end{proof}

\begin{proposition}
Suppose $H$ and $D$ are subgroups of a group $G$ and $g$ is an element of $G$. Then the following hold.
\medskip

1) $H^{-1}=H$.
\medskip

2) $H\cdot H=H$.
\medskip

3) $(H\cdot D)^{-1}=D\cdot H$.
\medskip

4) $(g\cdot H)^{-1}=H\cdot g^{-1}$.
\medskip

5) $(H\cdot g)^{-1}=g^{-1}\cdot H$.
\medskip

6) $(g^{-1}\cdot H\cdot g)$ is a subgroup of $G$.
\medskip

7) $g\cdot H=H\;\Leftrightarrow\;g\varepsilon H$.
\label{lemma:coset}
\end{proposition}

\begin{proof}
These are sketches of the proofs.

1) - 3)
\begin{flagderiv}
\introduce*{}{S: *_s\;|\;G:ps(S)\;|\;\cdot: S\rightarrow S\rightarrow S\;|\;e:S\;|\;^{-1}:S\rightarrow S\;|\;u:Group(S,G,\cdot,e,^{-1})}{}
\introduce*{}{H:ps(S)\;|\;v:H\leqslant G}{}

\skipsteps*{\dots}{}{}

\step*{}{\boldsymbol {term_{\ref{lemma:coset}.1}}(S,G,\cdot,e,^{-1},u, H, v):=\ldots\;:\;\boldsymbol{H^{-1}=H}}{}

\step*{}{\boldsymbol {term_{\ref{lemma:coset}.2}}(S,G,\cdot,e,^{-1},u, H, v):=\ldots\;:\;\boldsymbol{H\cdot H=H}}{}

\introduce*{}{D:ps(S)\;|\;w:D\leqslant G}{}
\skipsteps*{\dots}{}{}
\step*{}{\boldsymbol {term_{\ref{lemma:coset}.3}}(S,G,\cdot,e,^{-1},u,H,D,v, w):=\ldots
\;:\;\boldsymbol{(H\cdot D)^{-1}=D\cdot H}}{}
\end{flagderiv}

4) - 6) 

\begin{flagderiv}
\introduce*{}{S: *_s\;|\;G:ps(S)\;|\;\cdot: S\rightarrow S\rightarrow S\;|\;e:S\;|\;^{-1}:S\rightarrow S\;|\;u:Group(S,G,\cdot,e,^{-1})}{}
\introduce*{}{H:ps(S)\;|\;v:H\leqslant G\;|\;g:S\;|\;w:g\varepsilon G}{}
\skipsteps*{\dots}{}{}
\step*{}{\boldsymbol {term_{\ref{lemma:coset}.4}}(S,G,\cdot,e,^{-1},u, H, v,g,w):=\ldots
\;:\;\boldsymbol{(g\cdot H)^{-1}=H\cdot g^{-1}}}{}

\step*{}{\boldsymbol {term_{\ref{lemma:coset}.5}}(S,G,\cdot,e,^{-1},u, H, v,g,w):=\ldots
\;:\;\boldsymbol{(H\cdot g)^{-1}=g^{-1}\cdot H}}{}
\skipsteps*{\dots}{}{}

\step*{}{\boldsymbol{
term_{\ref{lemma:coset}.6}}(S,G,\cdot, e,^{-1},u,H,v,g,w) 
:=\ldots
\;:\;\boldsymbol{
(g^{-1} \cdot H\cdot g)\leqslant G} }{}
\end{flagderiv}

7) 

\begin{flagderiv}
\introduce*{}{S: *_s\;|\;G:ps(S)\;|\;\cdot: S\rightarrow S\rightarrow S\;|\;e:S\;|\;^{-1}:S\rightarrow S\;|\;u:Group(S,G,\cdot,e,^{-1})}{}

\introduce*{}{H:ps(S)\;|\;v:H\leqslant G\;|\;g:S\;|\;w:g\varepsilon G}{}

\skipsteps*{\dots}{}{}

\step*{}{\boldsymbol{
term_{\ref{lemma:coset}.7}}(S,G,\cdot, e,^{-1},u,H,v,g,w) 
:=\ldots
\;:\;\boldsymbol{
(g\cdot H=H)\;\Leftrightarrow\;g\varepsilon H}}{}
\end{flagderiv}

The full proofs are in the Appendix, pg. 74.
\end{proof}

\begin{theorem}
Let $B$ and $C$ be subgroups of a group $G$. Then $(B\cdot C)$ is a subgroup of $G$ if and only if $B\cdot C=C\cdot B$. In this case $B$ and $C$ are called \textbf{permutable}.

\label{theorem:permutable}
\end{theorem}
\begin{proof}
This is a sketch of the proof:
\begin{flagderiv}
\introduce*{}{S: *_s\;|\;G:ps(S)\;|\;\cdot: S\rightarrow S\rightarrow S\;|\;e:S\;|\;^{-1}:S\rightarrow S\;|\;u:Group(S,G,\cdot,e,^{-1})}{}
\introduce*{}{B,C:ps(S)\;|\;v:B\leqslant G\;|\;w:C\leqslant G}{}

\skipsteps*{\dots}{}{}

\step*{}{\boldsymbol {term_{\ref{theorem:permutable}}}(S,G,\cdot,e,^{-1},u, B,C,v,w):=\ldots
\;:\;\boldsymbol{(B\cdot C)\leqslant G
\Leftrightarrow B\cdot C=C\cdot B}}{}
\end{flagderiv}

The full proof is in the Appendix, pg. 78.
\end{proof}

\section{Normal Subgroups and Quotient Groups}
\subsection{Normal Subgroups}

\begin{definition}
Let $B$ be a subset of a group $G$ and $x,g$ be elements of $G$. 

1) $(g^{-1}\cdot x\cdot g)$ is called a \textbf{conjugate} of $x$.

2) $(g^{-1}\cdot B\cdot g)$ is called a \textbf{conjugate} of $B$.
\label{def:conjugate}
\end{definition}

\begin{proposition}
1) Being conjugate is an equivalence relation on a fixed group.

2) Being conjugate is an equivalence relation on subsets of a fixed group.

3) Being conjugate is an equivalence relation on subgroups of a fixed group.
\label{prop:conjugate}
\end{proposition}

\begin{proof}
This is a sketch of the proof for all parts.
\begin{flagderiv}
\introduce*{}{S: *_s\;|\;G:ps(S)}{}

\step*{}{\text{Definition  }Subs(S,G):=
\{B:ps(S)|B\subseteq G\}
\;:\;ps(ps((S))}{}

\step*{}{\text{Notation }M\;\text{ for }Subs(S,G)}{}

\introduce*{}{\cdot: S\rightarrow S\rightarrow S\;|\;^{-1}:S\rightarrow S}{}

\step*{}{\text{Definition  }R_c:=\lambda x,y:S.\exists g:S.(g\varepsilon G\wedge y=g^{-1}\cdot x\cdot g)\;:\;S\rightarrow S\rightarrow *_p}{}

\step*{}{\text{Definition  }R_s:=\lambda B,C:ps(S).\exists g:S.(g\varepsilon G\wedge C=g^{-1}\cdot B\cdot g)\;:\;ps(S)\rightarrow ps(S)\rightarrow *_p}{}

\introduce*{}{e:S}{}

\step*{}{\text{Definition  }Subg(S,G,\cdot,e,^{-1}):=
\{B:ps(S)\;|\;B\leqslant G\}
\;:\;ps(ps((S))}{}

\step*{}{\text{Notation }K\;\text{ for }Subg(S,G,\cdot,e,^{-1})}{}

\assume*{}{u:Group(S,G,\cdot,e,^{-1})}{}

\skipsteps*{\dots}{}{}

\step*{}{\boldsymbol {term_{\ref{prop:conjugate}.1}}(S,G,\cdot,e,^{-1},u):= 
\ldots
\;:\;\boldsymbol{equiv\mhyphen rel(S,G,R_c)}}{}

\skipsteps*{\dots}{}{}

\step*{}{\boldsymbol {term_{\ref{prop:conjugate}.2}}(S,G,\cdot,e,^{-1},u):= 
\ldots
\;:\;\boldsymbol{equiv\mhyphen rel(ps(S),M,R_s)}}{}

\skipsteps*{\dots}{}{}

\step*{}{\boldsymbol {term_{\ref{prop:conjugate}.3}}(S,G,\cdot,e,^{-1},u):= 
\ldots
\;:\;\boldsymbol{equiv\mhyphen rel(ps(S),K,R_s)}}{}
\end{flagderiv}

The full proof is in the Appendix, pg. 79.  
\end{proof}

\begin{definition}
\textbf{Normal subgroup.}
\begin{flagderiv}
\introduce*{}{S: *_s\;|\;
\cdot: S\rightarrow S\rightarrow S\;|\;G:ps(S)\;|\;e:S\;|\;^{-1}:S\rightarrow S}{}
\introduce*{}{H:ps(S)}{}
\step*{}{\text{Definition }\boldsymbol{Normal\mhyphen subgroup}(S,G,\cdot,e,^{-1},H):=H\leqslant G
\\\quad\quad
\wedge
\forall g:S.[g\varepsilon G\Rightarrow\forall h:S.(h\varepsilon H\Rightarrow (g^{-1}\cdot h\cdot g) \;\varepsilon\;H)]
\;:\;*_p}{}

\step*{}{\text{Notation: }
\boldsymbol{H\triangleleft G}\text{ for } Normal\mhyphen subgroup(S,G,\cdot,e,^{-1},H)}{}
\end{flagderiv}
\label{def:normal}
\end{definition}

\begin{proposition}
Let $G$ be a group with identity $e$. Then

1) $G$ is a normal subgroup of $G$;

2) $\{e\}$ is a normal subgroup of $G$.
\label{prop:normal-triv}
\end{proposition}

\begin{proof}
These are sketches of the proofs.

1)
\begin{flagderiv}
\introduce*{}{S: *_s\;|\;G:ps(S)\;|\;\cdot: S\rightarrow S\rightarrow S\;|\;e:S\;|\;^{-1}:S\rightarrow S\;|\;u:Group(S,G,\cdot,e,^{-1})}{}

\skipsteps*{\dots}{}{}

\step*{}{\boldsymbol {term_{\ref{prop:normal-triv}.1}}(S,G,\cdot,e,^{-1},u):= 
\ldots
\;:\;\boldsymbol{G\triangleleft G}}{}
\end{flagderiv}

2)
\begin{flagderiv}
\introduce*{}{S: *_s\;|\;G:ps(S)\;|\;\cdot: S\rightarrow S\rightarrow S\;|\;e:S\;|\;^{-1}:S\rightarrow S\;|\;u:Group(S,G,\cdot,e,^{-1})}{}

\step*{}{\text{Notation }H:=\{x:S\;|\;x=e\}
\;:\;ps(S)}{}

\skipsteps*{\dots}{}{}

\step*{}{\boldsymbol {term_{\ref{prop:normal-triv}.2}}(S,G,\cdot,e,^{-1},u):= 
\ldots
\;:\;\boldsymbol{H\triangleleft G}}{}
\end{flagderiv}

The full proofs are in the Appendix, pg. 84. 
\end{proof}

\begin{proposition}
Any subgroup of an abelian group is normal.
\label{prop:normal-abel}
\end{proposition}

\begin{proof}
This is a sketch of the proof: 

\begin{flagderiv}
\introduce*{}{S: *_s\;|\;G:ps(S)\;|\;\cdot: S\rightarrow S\rightarrow S\;|\;e:S\;|\;^{-1}:S\rightarrow S\;|\;u:Abelian\mhyphen group(S,G,\cdot,e,^{-1})}{}

\introduce*{}{H:ps(S)\;|\;v:H\leqslant G}{}

\skipsteps*{\dots}{}{}

\step*{}{\boldsymbol {term_{\ref{prop:normal-abel}}}(S,G,\cdot,e,^{-1},u,H,v):= 
\ldots
\;:\;\boldsymbol{H\triangleleft G}}{}
\end{flagderiv}

The full proof is in the Appendix, pg. 85. 
\end{proof}

\begin{theorem}
Suppose $H$ is a subgroup of a group $G$. Then intersection of all sets of the form $(x^{-1}\cdot H\cdot x)$, $x\varepsilon G$, is a normal subgroup of $G$.
\label{prop:normal-inters}
\end{theorem}

\begin{proof}
This is a sketch of the proof: 

\begin{flagderiv}
\introduce*{}{S: *_s\;|\;G:ps(S)\;|\;\cdot: S\rightarrow S\rightarrow S\;|\;e:S\;|\;^{-1}:S\rightarrow S\;|\;u:Group(S,G,\cdot,e,^{-1})}{}

\introduce*{}{H:ps(S)\;|\;v:H\leqslant G}{}

\step*{}{U:=\{Z:ps(S)\;|\;\exists x:S.(x\varepsilon G\wedge Z=x^{-1}\cdot H\cdot x)\}\;:\;ps(ps(S))}{}

\step*{}{N:=\cap U: ps(S)}{}

\skipsteps*{\dots}{}{}

\step*{}{\boldsymbol {term_{\ref{prop:normal-inters}}}(S,G,\cdot,e,^{-1},u,H,v):= 
\ldots
\;:\;\boldsymbol{N\triangleleft G}}{}
\end{flagderiv}

The full proof is in the Appendix, pg. 86. 
\end{proof}

\begin{proposition}
Suppose $H$ is a subgroup of a group $G$. Then the following conditions are equivalent.

1) $H\triangleleft G$; 

2) for any element $g$ of $G$: $\quad(g^{-1}\cdot H\cdot g)=H$;
\medskip

3) for any element $g$ of $G$: $\quad(g\cdot H\cdot g^{-1})=H$;
\medskip

4) for any element $g$ of $G$: $\quad g\cdot H=H\cdot g$;
\medskip

5) for any element $g$ of $G$: $\quad(g^{-1}\cdot H\cdot g)\subseteq H$.
\label{prop:normal-criteria}
\end{proposition}

\begin{proof}
It is sufficient to prove the following implications:
\[1) \Rightarrow 2) \Rightarrow 3) \Rightarrow 4) \Rightarrow 5) \Rightarrow 1).\]

This is a sketch of the proof: 

\begin{flagderiv}
\introduce*{}{S: *_s\;|\;G:ps(S)\;|\;\cdot: S\rightarrow S\rightarrow S\;|\;e:S\;|\;^{-1}:S\rightarrow S\;|\;u:Group(S,G,\cdot,e,^{-1})}{}

\introduce*{}{H:ps(S)\;|\;v:H\leqslant G}{}

\step*{}{\text{Notation }\;
A:=\forall g:S.(g\varepsilon G
\Rightarrow (g^{-1}\cdot H\cdot g)=H)\;:\;*_p}{}

\step*{}{\text{Notation }\;
B:=\forall g:S.(g\varepsilon G
\Rightarrow (g\cdot H\cdot g^{-1})=H)\;:\;*_p}{}

\step*{}{\text{Notation }\;
C:=\forall g:S.(g\varepsilon G
\Rightarrow g\cdot H=H\cdot g)\;:\;*_p}{}

\step*{}{\text{Notation }\;
D:=\forall g:S.(g\varepsilon G
\Rightarrow (g^{-1}\cdot H\cdot g)\subseteq H)\;:\;*_p}{}

\skipsteps*{\dots}{}{}

\step*{}{\boldsymbol {term_{\ref{prop:normal-criteria}.1}}(S,G,\cdot,e,^{-1},u,H,v):= 
\ldots\;:\;\boldsymbol {H\triangleleft G\Rightarrow A}}{}

\skipsteps*{\dots}{}{}

\step*{}{\boldsymbol {term_{\ref{prop:normal-criteria}.2}}(S,G,\cdot,e,^{-1},u,H,v):= 
\ldots\;:\;\boldsymbol {A\Rightarrow B}}{}

\skipsteps*{\dots}{}{}

\step*{}{\boldsymbol {term_{\ref{prop:normal-criteria}.3}}(S,G,\cdot,e,^{-1},u,H,v):= 
\ldots\;:\;\boldsymbol {B\Rightarrow C}}{}

\skipsteps*{\dots}{}{}

\step*{}{\boldsymbol {term_{\ref{prop:normal-criteria}.4}}(S,G,\cdot,e,^{-1},u,H,v):= 
\ldots\;:\;\boldsymbol {C\Rightarrow D}}{}

\skipsteps*{\dots}{}{}

\step*{}{\boldsymbol {term_{\ref{prop:normal-criteria}.5}}(S,G,\cdot,e,^{-1},u,H,v):= 
\ldots\;:\;\boldsymbol {D\Rightarrow H\triangleleft G}
}{}
\end{flagderiv}

The full proof is in the Appendix, pg. 88. 
\end{proof}

\begin{corollary}
Suppose $H$ is a subgroup of a group $G$. Then 
\[H\triangleleft G\quad\Leftrightarrow\quad
\forall g:S.(g\varepsilon G\;\Rightarrow g\;\cdot H=H\cdot g).\]
\label{corol:normal}
\end{corollary}

\vspace{-0.6cm}
\begin{proof}
This is a sketch of the proof: 

\begin{flagderiv}
\introduce*{}{S: *_s\;|\;G:ps(S)\;|\;\cdot: S\rightarrow S\rightarrow S\;|\;e:S\;|\;^{-1}:S\rightarrow S\;|\;u:Group(S,G,\cdot,e,^{-1})}{}

\introduce*{}{H:ps(S)}{}

\assume*{}{v:H\triangleleft G}{}

\skipsteps*{\dots}{}{}

\step*{}{\boldsymbol {term_{\ref{corol:normal}.1}}(S,G,\cdot,e,^{-1},u,H,v):=
\ldots\;:\;
\boldsymbol {[\forall g:S.(g\varepsilon G
\Rightarrow g\cdot H=H\cdot g)]}}{}

\done

\assume*{}{v:H\leqslant G\;|\;w:[\forall g:S.(g\varepsilon G\Rightarrow g\cdot H=H\cdot g)]}{}

\skipsteps*{\dots}{}{}

\step*{}{\boldsymbol {term_{\ref{corol:normal}.2}}(S,G,\cdot,e,^{-1},u,H,v,w):=
\ldots\;:\;\boldsymbol {H\triangleleft G}}{}
\end{flagderiv}

The full proof is in the Appendix, pg. 90. 
\end{proof}

\begin{lemma}
For any normal subgroup $H$ of a group $G$ and elements $x,y$ of $G$:
\medskip

1) $(x\cdot H)^{-1}=x^{-1}\cdot H$;
\medskip

2) $(x\cdot H)\cdot(y\cdot H)=(x\cdot y)\cdot H$.
\label{lemma:normal-subgroup}
\end{lemma}
\begin{proof}
This is a sketch of the proof for both parts.

\begin{flagderiv}
\introduce*{}{S: *_s\;|\;G:ps(S)\;|\;\cdot: S\rightarrow S\rightarrow S\;|\;e:S\;|\;^{-1}:S\rightarrow S\;|\;u:Group(S,G,\cdot,e,^{-1})}{}

\introduce*{}{H:ps(S)\;|\;v:H\triangleleft G}{}

\introduce*{}{x:S\;|\;w:x\varepsilon G}{}

\skipsteps*{\dots}{}{}

\step*{}{\boldsymbol {term_{\ref{lemma:normal-subgroup}.1}}(S,G,\cdot,e,^{-1},u,H,v,x,w):= 
\ldots\;:\;\boldsymbol {(x\cdot H)^{-1}=x^{-1}\cdot H}}{}

\done

\introduce*{}{x,y:S\;|\;w_1:x\varepsilon G\;|\;w_2:y\varepsilon G}{}

\step*{}{\boldsymbol {term_{\ref{lemma:normal-subgroup}.2}}(S,G,\cdot,e,^{-1},u,H,v,x,y,w_1,w_2):=\ldots
\;:\;\boldsymbol {(x\cdot H)\cdot(y\cdot H)=x\cdot y\cdot H}}{}
\end{flagderiv}

The full proof is in the Appendix, pg. 91. 
\end{proof}

\begin{theorem}
Suppose $B$ is a subgroup of a group $G$ and $H$ is a  normal subgroup of $G$. Then the following hold.

1) $B\cap H$ is a normal subgroup of $B$.
\medskip

2) $B\cdot H=H\cdot B$.
\medskip

3) $(B\cdot H)$ is a subgroup of $G$.
\label{theorem:product-subnormal}
\end{theorem}

\begin{proof}
This is a sketch of the proof for all parts.

\begin{flagderiv}
\introduce*{}{S: *_s\;|\;G:ps(S)\;|\;\cdot: S\rightarrow S\rightarrow S\;|\;e:S\;|\;^{-1}:S\rightarrow S\;|\;u:Group(S,G,\cdot,e,^{-1})}{}

\introduce*{}{B,H:ps(S)\;|\;v:B\leqslant G\;|\;w:H\triangleleft G}{}

\skipsteps*{\dots}{}{}

\step*{}{\boldsymbol {term_{\ref{theorem:product-subnormal}.1}}(S,G,\cdot,e,^{-1},u,B,H,v,w):= 
\ldots\;:\;\boldsymbol {(B\cap H)\triangleleft B}}{}

\skipsteps*{\dots}{}{}

\step*{}{\boldsymbol {term_{\ref{theorem:product-subnormal}.2}}(S,G,\cdot,e,^{-1},u,B,H,v,w):= 
\ldots\;:\;\boldsymbol {B\cdot H=H\cdot B}}{}

\skipsteps*{\dots}{}{}

\step*{}{\boldsymbol {term_{\ref{theorem:product-subnormal}.3}}(S,G,\cdot,e, ^{-1},u, B,H,v,w):= 
\;:\;\boldsymbol {(B\cdot H)\leqslant G}}{}
\end{flagderiv}

The full proof is in the Appendix, pg. 92. 
\end{proof}

\begin{theorem}
Suppose $B$ and $C$ are normal subgroups of a group $G$. Then the following hold.

1) $B\cdot C$ is a normal subgroup of $G$.
\medskip

2) If $B\cap C=\{e\}$, then for any $b\varepsilon B$ and $c\varepsilon C$:
\[b\cdot c=c\cdot b.\]
\label{theorem:product-normal}
\end{theorem}

\newpage
\begin{proof}
This is a sketch of the proof for both parts.
\begin{flagderiv}
\introduce*{}{S: *_s\;|\;G:ps(S)\;|\;\cdot: S\rightarrow S\rightarrow S\;|\;e:S\;|\;^{-1}:S\rightarrow S\;|\;u:Group(S,G,\cdot,e,^{-1})}{}

\introduce*{}{B,C:ps(S)\;|\;v:B\triangleleft G\;|\;w:C\triangleleft G}{}

\skipsteps*{\dots}{}{}

\step*{}{\boldsymbol {term_{\ref{theorem:product-normal}.1}}(S,G,\cdot,e,^{-1},u,B,C,v,w):= 
\ldots\;:\;\boldsymbol {(B\cdot C)\triangleleft G}}{}\

\assume*{}{r_1:[\forall g:S.(
g\varepsilon B\Rightarrow g\varepsilon C\Rightarrow g=e)]}{}

\introduce*{}{b:S\;|\;r_2:b\varepsilon B\;|\;c:S\;|\;r_3:c\varepsilon C}{}

\skipsteps*{\dots}{}{}

\step*{}{\boldsymbol {term_{\ref{theorem:product-normal}.2}}(S,G,\cdot,e,^{-1},u,B,C,v,w,r_1, b,c,r_2,r_3):= 
\ldots
\;:\;\boldsymbol {b\cdot c=c\cdot b}}{}
\end{flagderiv}

The full proof is in the Appendix, pg. 93.  
\end{proof}

\bigskip
\subsection{Quotient Groups}

\begin{definition}
\textbf{Quotient group} $G/H$ is defined for a normal subgroup $H$ of a group $G$.
\begin{flagderiv}
\introduce*{}{S: *_s\;|\;
\cdot: S\rightarrow S\rightarrow S\;|\;G:ps(S)\;|\;e:S\;|\;^{-1}:S\rightarrow S}{}
\introduce*{}{H:ps(S)}{}

\step*{}{\text{Definition }\;\boldsymbol{Quotient\mhyphen group}(S,G,\cdot,e,^{-1},
H):=\{X:ps(S)\;|\;\exists x:S.(x\varepsilon G\wedge X=x\cdot H)\}\;:\;ps(ps(S))}{}

\step*{}{\text{Notation }\;
\boldsymbol{G/H} \text{ for }Quotient\mhyphen group(S,G,\cdot,e,^{-1},H)}{}

\step*{}{\text{Notation }\;
\boldsymbol{E}:=(e\cdot H)}{}
\end{flagderiv}
\label{def:quotient}
\end{definition}

\begin{theorem}
If $H$ is a normal subgroup of a group $G$, then the set $G/H$ is a group with identity $E$, where the operations of set multiplication and taking set inverse are from Definition \ref{def:set-product}.
\label{theorem:quotient}
\end{theorem}

\begin{proof}
This is a sketch of the proof: 

\begin{flagderiv}
\introduce*{}{S: *_s\;|\;G:ps(S)\;|\;\cdot: S\rightarrow S\rightarrow S\;|\;e:S\;|\;^{-1}:S\rightarrow S\;|\;u:Group(S,G,\cdot,e,^{-1})}{}

\introduce*{}{H:ps(S)\;|\;v:H\triangleleft G}{}

\skipsteps*{\dots}{}{}

\step*{}{\boldsymbol {term_{\ref{theorem:quotient}}}(S,G,\cdot,e,^{-1},u,H,v):= 
\ldots
\;:\;\boldsymbol{Group(ps(S),G/H,\cdot,E,^{-1})}}{}
\end{flagderiv}

The full proof is in the Appendix, pg. 95.  
\end{proof}

\begin{example}
In Example \ref{example:Z}
we considered the abelian additive group $\mathbb{Z}$
and in Example  \ref{example:mZ} its subgroup $m\mathbb{Z}$, where $m$ is an arbitrary positive integer. 

By Proposition \ref{prop:normal-abel}, $m\mathbb{Z}$ is a normal subgroup of $\mathbb{Z}$ and by Theorem \ref{theorem:quotient} the quotient group $\mathbb{Z}/(m\mathbb{Z})$ is a group; it is denoted $\mathbb{Z}_m$.
\end{example}

\begin{theorem}
\textbf{Correspondence Theorem}. Suppose $H$ is a normal subgroup of a group $G$. Then the following hold.

1) If $B$ is a subgroup of $G$ containing $H$, then $H\triangleleft B$ and $B/H\leqslant G/H$.

2) If $C$ is a subgroup of $G/H$, then there is a subgroup $B$ of $G$ such that $H\triangleleft B$ and $C=B/H$.

\label{theorem:correspondence}
\end{theorem}

\begin{proof}
These are sketches of the proofs.

1) 
\begin{flagderiv}
\introduce*{}{S: *_s\;|\;G:ps(S)\;|\;\cdot: S\rightarrow S\rightarrow S\;|\;e:S\;|\;^{-1}:S\rightarrow S\;|\;u:Group(S,G,\cdot,e,^{-1})}{}

\introduce*{}{H,B:ps(S)\;|\;v_1:H\triangleleft G\;|\;
v_2:B\leqslant G
\;|\;w:H\subseteq B}{}

\skipsteps*{\dots}{}{}

\step*{}{\boldsymbol {term_{ \ref{theorem:correspondence}.1}}(S,G,\cdot,e,^{-1},u,H,B,v_1, v_2,w):= 
\ldots
\;:\;\boldsymbol{(H\triangleleft B)
\wedge (B/H\leqslant G/H)}}{}
\end{flagderiv}

2)
\begin{flagderiv}
\introduce*{}{S: *_s\;|\;G:ps(S)\;|\;\cdot: S\rightarrow S\rightarrow S\;|\;e:S\;|\;^{-1}:S\rightarrow S\;|\;u:Group(S,G,\cdot,e,^{-1})}{}

\introduce*{}{H:ps(S)\;|\;v:H\triangleleft G\;|\;
C:ps(ps(S))
\;|\;w:C\leqslant G/H
}{}

\step*{}{\text{Notation }\;
B:=\{x:S\;|\;x\varepsilon G\wedge x\cdot H\;\varepsilon\;C\}\;:\;ps(S)}{}
\skipsteps*{\dots}{}{}

\step*{}{\boldsymbol {term_{ \ref{theorem:correspondence}.2}}(S,G,\cdot,e,^{-1},u,H,C,v,w)
:=\ldots
\;:\;\boldsymbol{B\leqslant G\wedge (H\triangleleft B)
\wedge (C=B/H)}}{}
\end{flagderiv}

The full proofs are in the Appendix, pg. 98. 
\end{proof}

\section{Conclusion}
Using the approach from \cite{Ned14} to formalizing sets in type theory as predicates on a type, we formalize some foundation concepts of group theory in $\lambda D$, the Calculus of Constructions with Definitions. The theory $\lambda D$ is chosen because of its beneficial features such as decidability of type checking and proof checking, and proofs-as-terms interpretation.

In type theory it is not easy to deal with sets which are not proper types. For groups we overcame this difficulty as follows. For a group $G$, which is a subset of a type $S$ we assume that its operations of multiplication and taking inverse are defined on the entire type $S$, while specific groups axioms hold only on $G$. We show that this causes no loss of generality if classical logic is used. This approach simplifies formal derivations in $\lambda D$ relating to groups.

We formalize in $\lambda D$ the concepts of group, 
subgroup, coset, conjugate, normal subgroup, and quotient group, and derive some related theorems. In particular, we formally derive in $\lambda D$:
\begin{itemize}
\item criteria for a normal subgroup;
\item the necessary and sufficient condition for a product of subgroups to be a subgroup;
\item the theorem about correspondence between subgroups of $G$ containing a normal subgroup $H$, and subgroups of the quotient group $G/H$.
\end{itemize}

The results can be implemented in proof assistants based on calculus of constructions. 
Due to the abstract nature of group theory, the results can be used for formalizing other parts of mathematics. 

We keep our formalizations as close as possible to the standard mathematical practice. In formal definitions and proofs we use the flag format. In $\lambda D$ we added some derived logical rules and abbreviations for existing derived rules to streamline formal proofs
and make them more concise and readable.

Next we plan to continue formalizing group theory in $\lambda D$ including finite groups, homomorphisms, generators, and related theorems, and to formalize other algebraic structures in $\lambda D$.

\newpage
\section*{Appendix}

\subsection*{Proof of Lemma \ref{lemma:relation-0}}

\begin{flagderiv}
\introduce*{}{S: *_s\;|\;G:ps(S)\;|\;R:\Pi x:S.[(x\varepsilon G)\rightarrow \Pi y:S.((y\varepsilon G)\rightarrow *_p)]}{}

\step*{}{\text{Definition }Q:=\lambda x,y: S.(x\varepsilon G\wedge y\varepsilon G\wedge \Pi p:x\varepsilon G.\Pi q:y\varepsilon G.Rxpyq)\;:\;S\rightarrow S\rightarrow *_p}{}
\step*{}{\boldsymbol{ term_{ \ref{lemma:relation-0}.1}}(S,G,R):=Q\;:\;S\rightarrow S\rightarrow *_p\quad\quad\quad\quad\textbf{Extension of $R$ to $S$}}{}

\assume*{}{u:consistent_0(S,G,R)}{}

\introduce*{}{x,y:S\;|\;v_1:x\varepsilon G\;|\;v_2:y\varepsilon G}{}
\introduce*{}{p:x\varepsilon G\;|\;q:y\varepsilon G}{}
\assume*{}{w:Qxy}{}
\step*{}{a_1:=\wedge_2(w)pq\;:\;Rxpyq}{}
\conclude*{}{a_2:=\lambda w:Qxy.a_1\;:\;Qxy\Rightarrow Rxpyq}{}
\assume*{}{w:Rxpyq}{}
\introduce*{}{p_1:x\varepsilon G\;|\;q_1:y\varepsilon G}{}
\step*{}{a_3:= uxpp_1yqq_1 \;:\; (Rxpyq\Leftrightarrow Rxp_1yq_1)}{}
\step*{}{a_4:=\wedge_1(a_3)w\;:\; Rxp_1yq_1}{}
\conclude*{}{a_5:=\lambda p_1:x\varepsilon G.\lambda q_1:y\varepsilon G.a_4 \;:\;(\Pi p_1:x\varepsilon G.\Pi q_1:y\varepsilon G.Rxp_1yq_1)}{}

\step*{}{a_6:=\wedge(\wedge(v_1,v_2),a_5)\;:\;Qxy}{}
\conclude*{}{a_7:=\lambda w:Rxpyq.a_6\;:\;Rxpyq\Rightarrow Qxy}{}
\step*{}{a_8:=\wedge(a_2,a_7)\;:\; Qxy\Leftrightarrow Rxpyq}{}

\conclude*[2]{}{\boldsymbol{ term_{ \ref{lemma:relation-0}.2}}(S,G,R,u):=
\lambda x,y:S.v_1:x\varepsilon G.\lambda v_2:y\varepsilon G.
\lambda p:x\varepsilon G.\lambda q:y\varepsilon G.a_8
\\\quad\quad
\;:\;\boldsymbol{
\forall x,y:S.(
x\varepsilon G\Rightarrow y\varepsilon G
\Rightarrow \Pi p:x\varepsilon G.\Pi q:y\varepsilon G.(Qxy\Leftrightarrow Rxpyq))}}{}
\end{flagderiv}

\subsection*{Proof of Proposition \ref{lemma:partition}}

\begin{flagderiv}
\introduce*{}{S: *_s\;|\;G:ps(S)\;|\;
R: S\rightarrow S\rightarrow *_p}{}
\step*{}{\text{Notation }A:=\forall x:S.(x\varepsilon G\Rightarrow
x\varepsilon Rx)\;:\;*_p}{}
\step*{}{\text{Notation }B:=\forall x:S.[x\varepsilon G
\Rightarrow\forall y:S.(y\varepsilon G\Rightarrow\forall z:S.(z\varepsilon G
\Rightarrow
z\varepsilon Rx\Rightarrow z\varepsilon Ry\Rightarrow (Rx)\cap G=(Ry)\cap G))]\;:\;*_p}{}
\assume*{}{u:equiv\mhyphen rel(S,G,R)}{}
\step*{}{a_1:=\wedge_1 (\wedge_1(u))\;:\;refl(S,G,R)}{}
\step*{}{a_2:=\wedge_2(\wedge_1(u))\;:\;sym(S,G,R)}{}
\step*{}{a_3:=\wedge_2(u)\;:\;trans(S,G,R)}{}

\introduce*{}{x:S\;|\;v:x\varepsilon G}{}
\step*{}{a_4:=a_1xv:Rxx}{}
\step*{}{a_4:(x\varepsilon Rx)}{}
\conclude*{}{a_5:=\lambda x:S.\lambda v:x\varepsilon G.a_4\;:\;A}{}

\introduce*{}{x:S\;|\;v_1:x\varepsilon G\;|\;y:S\;|\;v_2:y\varepsilon G\;|\;z:S\;|\;v_3:z\varepsilon G\;|\;w_1:z\varepsilon Rx\;|\;w_2:z\varepsilon Ry}{}
\step*{}{w_1:Rxz}{}
\step*{}{w_2:Ryz}{}
\step*{}{a_6:= a_2xv_1zv_3w_1 :Rzx}{}
\step*{}{a_7:= a_3yv_2zv_3xv_1w_2a_6 :Ryx}{}
\introduce*{}{t:S\;|\; w_4:t\varepsilon (Rx)\cap G}{}

\step*{}{a_8:=\wedge_1(w_4)\;:\;t\varepsilon Rx}{}
\step*{}{a_9:=\wedge_2(w_4)\;:\;t\varepsilon G}{}
\step*{}{a_8:Rxt}{}
\step*{}{a_{10}:= a_3yv_2xv_1ta_9a_7a_8 :Ryt}{}
\step*{}{a_{10}:t\varepsilon Ry}{}
\step*{}{a_{11}:=\wedge(a_{10},a_9)\;:\;t\varepsilon (Ry)\cap G}{}
\conclude*{}{a_{12}(x,v_1,y,v_2,z,v_3,w_1, w_2):=\lambda t:S. \lambda w_4:t\varepsilon (Rx)\cap G.a_{11}
\;:\;(Rx)\cap G\subseteq (Ry)\cap G}{}
\step*{}{a_{13}:=a_{12}(x,v_1,y,v_2,z,v_3,w_1, w_2)\;:\;(Rx)\cap G\subseteq (Ry)\cap G}{}
\step*{}{a_{14}:=a_{12}(y,v_2,x,v_1,z,v_3,w_2, w_1)\;:\;(Ry)\cap G\subseteq (Rx)\cap G}{}
\step*{}{a_{15}:=\wedge( a_{13},a_{14})
\;:\;(Rx)\cap G=(Ry)\cap G}{}

\conclude*{}{a_{16}:=\lambda x:S.\lambda v_1:x\varepsilon G.\lambda y:S.\lambda v_2:y\varepsilon G.\lambda z:S.\lambda v_3:z\varepsilon G.
\lambda w_1:z\varepsilon Rx.
\lambda w_2:z\varepsilon Ry.a_{15}\;:\;B}{}
\step*{}{a_{17}:=\wedge(a_5, a_{16})\;:\;partition(S,G,R)}{}
\conclude*{}{a_{18}:=
\lambda u:equiv\mhyphen rel(S,G,R).a_{17}
\;:\;
equiv\mhyphen rel(S,G,R)\Rightarrow partition(S,G,R)}{}

\assume*{}{u:partition(S,G,R)}{}
\step*{}{u:A\wedge B}{}
\step*{}{a_{19}:=\wedge_1(u)\;:\;A}{}
\step*{}{a_{20}:=\wedge_2(u)\;:\;B}{}
\introduce*{}{x:S\;|\;v:x\varepsilon G}{}
\step*{}{a_{21}(x,v):=a_{19}xv:x\varepsilon Rx}{}
\step*{}{a_{22}:=a_{21}(x,v)\;:\;Rxx}{}
\conclude*{}{a_{23}:=\lambda x:S.\lambda v:x\varepsilon G.a_{22}\;:\;refl(S,G,R)}{}

\introduce*{}{x:S\;|\;v_1:x\varepsilon G\;|\;y:S\;|\;v_2:y\varepsilon G\;|\;w:Rxy}{}

\step*{}{w:y\varepsilon Rx}{}
\step*{}{a_{24}:=a_{21}(y,v_2)
\;:\;y\varepsilon Ry}{}
\step*{}{a_{25}:=a_{20}xv_1yv_2yv_2wa_{24}\;:\;(Rx)\cap G=(Ry)\cap G}{}
\step*{}{a_{26}:=a_{21}(x,v_1)
\;:\;x\varepsilon Rx}{}
\step*{}{a_{27}:=\wedge(a_{26},v_1)\;:\;x\varepsilon (Rx)\cap G}{}

\step*{}{a_{28}:=\wedge_1( a_{25})xa_{27}\;:\;x\varepsilon (Ry)\cap G}{}
\step*{}{a_{29}:=\wedge_1(a_{28})\;:\;
x\varepsilon Ry}{}

\step*{}{a_{29}:Ryx}{}

\conclude*{}{a_{30}:=\lambda x:S.\lambda v_1:x\varepsilon G.\lambda y:S.\lambda v_2:y\varepsilon G.\lambda w:Rxy.
a_{29}\;:\;sym(S,G,R)}{}

\introduce*{}{x:S\;|\;v_1:x\varepsilon G\;|\;y:S\;|\;v_2:y\varepsilon G\;|\;z:S\;|\;v_3:z\varepsilon G\;|\;w_1:Rxy\;|\;w_2:Ryz}{}

\step*{}{w_1:y\varepsilon Rx}{}
\step*{}{a_{31}:=a_{30}yv_2zv_3w_2
\;:\;Rzy}{}
\step*{}{a_{31}:y\varepsilon Rz}{}

\step*{}{a_{32}:=a_{20}zv_3xv_1yv_2a_{31}w_1\;:\;(Rz)\cap G=(Rx)\cap G}{}
\step*{}{a_{33}:=a_{21}(z,v_3)\;:\; z\varepsilon Rz}{}
\step*{}{a_{34}:=\wedge(a_{33},v_3)\;:\;z\varepsilon (Rz)\cap G}{}
\step*{}{a_{35}:=\wedge_1(a_{32})za_{34}\;:\;z\varepsilon (Rx)\cap G}{}
\step*{}{a_{36}:=\wedge_1(a_{35})\;:\;z\varepsilon Rx}{}

\step*{}{a_{36}:Rxz}{}
\conclude*{}{a_{37}:=\lambda x:S.\lambda v_1:x\varepsilon G.\lambda y:S.\lambda v_2:y\varepsilon G.\lambda z:S.\lambda v_3:z\varepsilon G.
\lambda w_1:Rxy.\lambda w_2:Ryz.
a_{36}\;:\;trans(S,G,R)}{}
\step*{}{a_{38}:=\wedge(\wedge(a_{23}, a_{30}), a_{37})\;:\;equiv\mhyphen rel(S,G,R)}{}
\conclude*{}{a_{39}:=
\lambda u:partition(S,G,R).a_{38}
\;:\;partition(S,G,R)\Rightarrow equiv\mhyphen rel(S,G,R)}{}
\step*{}{\boldsymbol{ term_{ \ref{lemma:partition}}}(S,G,R):=\wedge(a_{18},a_{39})
\;:\;
\boldsymbol{equiv\mhyphen relation(S,G,R)\Leftrightarrow
partition(S,G,R)}}{}
\end{flagderiv}

\subsection*{Proof of Theorem \ref{theorem:extend_function_1}}

We use the classical Axiom of Excluded Third with the proof term $exc\mhyphen thrd$ and the iota descriptor.

\begin{flagderiv}
\introduce*{}{S,T: *_s\;|\;G:ps(S)\;|\;f:\Pi x:S.((x\varepsilon G)\rightarrow T)}{}

\introduce*{}{b:T\;|\;u:consistent_1(S,T,G,f)}{}

\step*{}{\text{Notation } P:=\lambda x:S.\lambda y:T.[\Pi p:x\varepsilon G.\;(y=fxp)\wedge (\neg(x\varepsilon G)\Rightarrow y=b)]\;:\;S\rightarrow T\rightarrow *_p}{}

\introduce*{}{x:S}{}
\step*{}{a_1:=exc\mhyphen thrd(x\varepsilon G)\;:\;x\varepsilon G\vee \neg(x\varepsilon G)}{}

\assume*{}{v:\neg(x\varepsilon G)}{}
\step*{}{\text{Notation } y:=b\;:\;T}{}
\step*{}{a_2:=eq\mhyphen refl\;:\;y=b}{}

\step*{}{a_3:=\lambda w:\neg(x\varepsilon G).a_2
\;:\;(\neg(x\varepsilon G)\Rightarrow y=b)}{}
\introduce*{}{p:x\varepsilon G}{}
\step*{}{a_4:=vp\;:\;\bot}{}
\step*{}{a_5:=a_4(y=fxp)\;:\;y=fxp}{}
\conclude*{}{a_6:=\lambda p:x\varepsilon G.a_5
\;:\;(\Pi p:x\varepsilon G.\;(y=fxp))}{}
\step*{}{a_7:=\wedge(a_6,a_3)\;:\;Pxy}{}
\step*{}{a_8:=\exists_1 (Px,y,a_7)\;:\;(\exists y:T.Pxy)}{}
\conclude*{}{a_9:=\lambda v:\neg(x\varepsilon G).a_8\;:\;(\neg(x\varepsilon G)\Rightarrow \exists y:T.Pxy)}{}
\assume*{}{v:x\varepsilon G}{}
\step*{}{\text{Notation } y:=fxv\;:\;T}{}
\assume*{}{w:\neg(x\varepsilon G)}{}
\step*{}{a_{10}:=wv\;:\;\bot}{}
\step*{}{a_{11}:=a_{10}(y=b)\;:\;y=b}{}
\conclude*{}{a_{12}:=\lambda w:\neg(x\varepsilon G).a_{11}\;:\;(\neg(x\varepsilon G)\Rightarrow y=b)}{}

\introduce*{}{p:x\varepsilon G}{}
\step*{}{a_{13}:=uxvp\;:\;fxv=fxp}{}
\step*{}{a_{13}\;:\;y=fxp}{}
\conclude*{}{a_{14}:=\lambda p:x\varepsilon G.a_{13}
\;:\;(\Pi p:x\varepsilon G.\;y=fxp)}{}
\step*{}{a_{15}:=\wedge(a_{14},a_{12})\;:\;Pxy}{}
\step*{}{a_{16}:=\exists_1 (Px,y,a_{15})\;:\;(\exists y:T.Pxy)}{}
\conclude*{}{a_{17}:=\lambda v:x\varepsilon G.a_{16}\;:\;(x\varepsilon G\Rightarrow \exists y:T.Pxy)}{}
\step*{}{a_{18}:=\vee(a_1,a_{17},a_9)\;:\;(\exists y:T.Pxy)}{}

\introduce*{}{y,z:T\;|\;v_1:Pxy\;|\;v_2:Pxz}{}
\step*{}{a_{19}:=\wedge_1(v_1)\;:\;(\Pi p:x\varepsilon G.\;y=fxp)}{}
\step*{}{a_{20}:=\wedge_2(v_1)\;:\;(\neg(x\varepsilon G)\Rightarrow y=b)}{}

\step*{}{a_{21}:=\wedge_1(v_2)\;:\;(\Pi p:x\varepsilon G.\;z=fxp)}{}
\step*{}{a_{22}:=\wedge_2(v_2)\;:\;(\neg(x\varepsilon G)\Rightarrow z=b)}{}

\assume*{}{w:\neg(x\varepsilon G)}{}
\step*{}{a_{23}:=a_{20}w\;:\;y=b}{}
\step*{}{a_{24}:=a_{22}w\;:\;z=b}{}

\step*{}{a_{25}:=eq\mhyphen trans_2(a_{23},a_{24})
\;:\;y=z}{}
\conclude*{}{a_{26}:=\lambda w:\neg(x\varepsilon G).a_{25}\;:\;(\neg(x\varepsilon G)\Rightarrow y=z)}{}

\assume*{}{w:x\varepsilon G}{}
\step*{}{a_{27}:=a_{19}w\;:\;y=fxw}{}
\step*{}{a_{28}:=a_{21}w\;:\;z=fxw}{}

\step*{}{a_{29}:=eq\mhyphen trans_2(a_{27},a_{28})\;:\;y=z}{}

\conclude*{}{a_{30}:=\lambda w:x\varepsilon G.a_{29}\;:\;(x\varepsilon G\Rightarrow y=z)}{}
\step*{}{a_{31}:=\vee(a_1,a_{30},a_{26})\;:\;y=z}{}
\conclude*{}{a_{32}:=\lambda y,z:T.\lambda v_1:Pxy.\lambda  v_2:Pxz.a_{31}\;:\;\forall y,z:T.(Pxy\Rightarrow Pxz\Rightarrow y=z)}{}
\step*{}{a_{33}:=\wedge(a_{18},a_{32})\;:\;(\exists^1 y:T.Pxy)}{}

\step*{}{c(x):=\iota(T,Px,a_{33})\;:\;T}{}

\step*{}{d(x):=\iota\mhyphen prop(T,Px,a_{33})\;:\;Pxc(x)}{}
\conclude*{}{f^*:=\lambda x:S.c(x)\;:\;S\rightarrow T}{}

\step*{}{ \boldsymbol{Ext}_1(S,T,G,f,b,u):=f^*
\;:\;S\rightarrow T}{}

\introduce*{}{x:S}{}
\step*{}{d(x):Pxc(x)}{}
\step*{}{d(x):Px(f^*x)}{}

\conclude*{}{Ext\mhyphen proof_1(S,T,G,f,b,u):=\lambda x:S.d(x)\;:\;[\forall x:S.Px(f^*x)]}{}

\step*{}{ \boldsymbol{Ext\mhyphen proof}_1(S,T,G,f,b,u)
\;:\;\boldsymbol{ Ext\mhyphen prop_1}(S,T,G,f,b,f^*)}{}
\end{flagderiv}

\subsection*{Proof of Theorem \ref{theorem:extend_function_2}
}

\begin{flagderiv}
\introduce*{}{S,T: *_s\;|\;G:ps(S)\;|\;f:\Pi x:S.[(x\varepsilon G)\rightarrow \Pi y:S.((y\varepsilon G)\rightarrow T)]}{}

\assume*{}{b:T\;|\;u:consistent_2(S,T,G,f)}{}

\step*{}{\text{Notation } P:=\lambda x,y:S.\lambda z:T.[\Pi p:x\varepsilon G.\Pi q:y\varepsilon G.\;(z=fxpyq)
\\\quad\quad
\wedge (\neg(x\varepsilon G\wedge y\varepsilon G)\Rightarrow z=b)]\;:\;S\rightarrow S\rightarrow T\rightarrow *_p}{}

\introduce*{}{x,y:S}{}
\step*{}{a_1:=exc\mhyphen thrd(x\varepsilon G\wedge y\varepsilon G)\;:\;(x\varepsilon G\wedge y\varepsilon G)\vee \neg(x\varepsilon G\wedge y\varepsilon G)}{}

\assume*{}{v:x\varepsilon G\wedge y\varepsilon G}{}
\step*{}{a_2:=\wedge_1 (v)\;:\;x\varepsilon G}{}
\step*{}{a_3:=\wedge_2 (v)\;:\;y\varepsilon G}{}
\step*{}{\text{Notation } z:=fxa_2ya_3\;:\;T}{}
\assume*{}{w:\neg(x\varepsilon G\wedge y\varepsilon G)}{}
\step*{}{a_4:=wv\;:\;\bot}{}
\step*{}{a_5:=a_4(z=b)\;:\;z=b}{}
\conclude*{}{a_6:=\lambda w:\neg(x\varepsilon G\wedge y\varepsilon G).a_5\;:\;(\neg(x\varepsilon G\wedge y\varepsilon G)\Rightarrow z=b)}{}
\introduce*{}{p:x\varepsilon G\;|\;q:y\varepsilon G}{}
\step*{}{a_7:=uxa_2pya_3q\;:\;fxa_2ya_3=fxpyq}{}
\step*{}{a_7\;:\;z=fxpyq}{}

\conclude*{}{a_8:=\lambda p:x\varepsilon G.\lambda q:y\varepsilon G.a_7
\;:\;(\Pi p:x\varepsilon G.\Pi q:y\varepsilon G.\;z=fxpyq)}{}
\step*{}{a_9:=\wedge(a_8,a_6)\;:\;Pxyz}{}
\step*{}{a_{10}:=\exists_1(Pxy,z,a_9)\;:\;(\exists z:T.Pxyz)}{}
\conclude*{}{a_{11}:=\lambda v:(x\varepsilon G.\wedge y\varepsilon G).a_{10}\;:\;(x\varepsilon G\wedge y\varepsilon G)\Rightarrow \exists z:T.Pxyz}{}

\assume*{}{v:\neg(x\varepsilon G\wedge y\varepsilon G)}{}
\step*{}{\text{Notation } z:=b\;:\;T}{}
\step*{}{a_{12}:=eq\mhyphen refl\;:\;z=b}{}

\step*{}{a_{13}:=\lambda w:\neg(x\varepsilon G\wedge y\varepsilon G).a_{12}
\;:\;(\neg(x\varepsilon G\wedge y\varepsilon G)\Rightarrow z=b)}{}
\introduce*{}{p:x\varepsilon G\;|\;q: y\varepsilon G}{}
\step*{}{a_{14}:=\wedge(p,q)\;:\;(x\varepsilon G\wedge y\varepsilon G)}{}
\step*{}{a_{15}:=va_{14}\;:\;\bot}{}
\step*{}{a_{16}:=a_{15}(z=fxpyq)\;:\;z=fxpyq}{}
\conclude*{}{a_{17}:=\lambda p:x\varepsilon G.\lambda q:y\varepsilon G.a_{16}
\;:\;(\Pi p:x\varepsilon G.\Pi q:y\varepsilon G.\;z=fxpyq)}{}
\step*{}{a_{18}:=\wedge(a_{17},a_{13})\;:\;Pxyz}{}
\step*{}{a_{19}:=\exists_1(Pxy,z,a_{18})\;:\;(\exists z:T.Pxyz)}{}
\conclude*{}{a_{20}:=\lambda v:\neg(x\varepsilon G\wedge y\varepsilon G).a_{19}\;:\;(\neg(x\varepsilon G\wedge y\varepsilon G)\Rightarrow \exists z:T.Pxyz)}{}
\step*{}{a_{21}:=\vee(a_1,a_{11},a_{20})\;:\;(\exists z:T.Pxyz)}{}

\introduce*{}{z_1,z_2:T\;|\;v_1:Pxyz_1\;|\;v_2:Pxyz_2}{}
\step*{}{a_{22}:=\wedge_1(v_1)
\;:\;(\Pi p:x\varepsilon G.\Pi q:y\varepsilon G.\;z_1=fxpyq)}{}
\step*{}{a_{23}:=\wedge_2(v_1)
\;:\;\neg(x\varepsilon G\wedge y\varepsilon G)\Rightarrow z_1=b}{}

\step*{}{a_{24}:=\wedge_1(v_2)
\;:\;(\Pi p:x\varepsilon G.\Pi q:y\varepsilon G.\;z_2=fxpyq)}{}
\step*{}{a_{25}:=\wedge_2(v_2)
\;:\;\neg(x\varepsilon G\wedge y\varepsilon G)\Rightarrow z_2=b}{}

\assume*{}{w:x\varepsilon G\wedge y\varepsilon G}{}
\step*{}{a_{26}:=\wedge_1(w)\;:\;x\varepsilon G}{}
\step*{}{a_{27}:=\wedge_2(w)\;:\;y\varepsilon G}{}
\step*{}{a_{28}:=a_{22}a_{26}a_{27}\;:\;z_1=fxa_{26}ya_{27}}{}
\step*{}{a_{29}:=a_{24}a_{26}a_{27}\;:\;z_2=fxa_{26}ya_{27}}{}

\step*{}{a_{30}:=eq\mhyphen trans_2(a_{28},a_{29})\;:\;z_1=z_2}{}
\conclude*{}{a_{31}:=\lambda w:( x\varepsilon G\wedge y\varepsilon G).a_{30}\;:\;(x\varepsilon G\wedge y\varepsilon G\Rightarrow z_1=z_2)}{}

\assume*{}{w:\neg(x\varepsilon G\wedge y\varepsilon G)}{}
\step*{}{a_{32}:=a_{23}w\;:\;z_1=b}{}
\step*{}{a_{33}:=a_{25}w\;:\;z_2=b}{}

\step*{}{a_{34}:=eq\mhyphen trans_2(a_{32},a_{33})
\;:\;z_1=z_2}{}
\conclude*{}{a_{35}:=\lambda w:\neg(x\varepsilon G\wedge y\varepsilon G).a_{34}\;:\;(\neg(x\varepsilon G\wedge y\varepsilon G)\Rightarrow z_1=z_2)}{}

\step*{}{a_{36}:=\vee(a_1,a_{31},
a_{35})\;:\;z_1=z_2}{}
\conclude*{}{a_{37}:=\lambda z_1,z_2:T.\lambda v_1:Pxyz_1.\lambda  v_2:Pxyz_2.a_{36}\;:\;[\forall z_1,z_2:T.(Pxyz_1\Rightarrow Pxyz_2\Rightarrow z_1=z_2)]}{}
\step*{}{a_{38}:=\wedge(a_{21},a_{37})\;:\;(\exists^1 z:T.Pxyz)}{}

\step*{}{c(x,y):=\iota(T,Pxy,a_{38})\;:\;T}{}

\step*{}{d(x,y):=\iota\mhyphen prop(T,Pxy,a_{38})\;:\;Pxyc(x,y)}{}

\conclude*{}{f^*:=\lambda x,y:S.c(x,y)\;:\;S\rightarrow S\rightarrow T}{}

\step*{}{ \boldsymbol{Ext}_2(S,T,G,f,b,u):=f^*
\;:\;S\rightarrow S\rightarrow T}{}

\introduce*{}{x,y:S}{}
\step*{}{d(x,y):Pxyc(x,y)}{}
\step*{}{d(x,y):Pxy(f^*xy)}{}

\conclude*{}{Ext\mhyphen proof_2(S,T,G,f,b,u):=\lambda x,y:S.d(x,y)
\;:\;[\forall x,y:S.Pxy(f^*xy)]}{}

\step*{}{ \boldsymbol{Ext\mhyphen proof}_2(S,T,G,f,b,u)
\;:\;\boldsymbol{ Ext\mhyphen prop_2}(S,T,G,f,b,f^*)}{}
\end{flagderiv}

\subsection*{Proof of Proposition \ref{theorem:unique}}

1)
\setlength{\derivskip}{4pt}
\begin{flagderiv}
\introduce*{}{S: *_s\;|\;\cdot: S\rightarrow S\rightarrow S}{}
\introduce*{}{e,d:S\;|\;u:identity(S,\cdot,e)\;|\;v:identity(S,\cdot,d)}{}
\step*{}{a_1:=\wedge_1(ud)\;:\;d\cdot e=d}{}
\step*{}{a_2:=\wedge_2(ve)\;:\;d\cdot e=e}{}

\step*{}{\boldsymbol{term_{\ref{theorem:unique}.1}}(S,\cdot,e,d,u,v):=eq\mhyphen trans_3(a_2,a_1)\;:\;\boldsymbol{e=d}}{}
\end{flagderiv}

2) \begin{flagderiv}
\introduce*{}{S: *_s\;|\;\cdot: S\rightarrow S\rightarrow S\;|\;e:S\;|\;u: monoid(S,\cdot,e)}{}
\step*{}{a_1:=\wedge_1(u)
\;:\;assoc(S,\cdot)}{}
\step*{}{a_2:=\wedge_2(u)
\;:\;identity(S,\cdot,e)}{}
\introduce*{}{b,x,y:S}{}

\step*{}{a_3:=\wedge_2(a_2x)\;:\;e\cdot x=x}{}

\step*{}{a_4:=\wedge_1(a_2y)\;:\;y\cdot e=y}{}
\assume*{}{v:inverse(S,\cdot,e,b,x)\;|\;w:inverse(S,\cdot,e,b,y)}{}
\step*{}{a_5:=\wedge_1(v)
\;:\;b\cdot x=e}{}
\step*{}{a_6:=\wedge_2(w)\;:\;y\cdot b=e}{}

\step*{}{a_7:= a_1ybx\;:\;(y\cdot b)\cdot x=y\cdot (b\cdot x)}{}

\step*{}{a_8:=eq\mhyphen cong_2(\lambda z:S.(z\cdot x),a_6)\;:\;e\cdot x=(y\cdot b)\cdot x}{}

\step*{}{a_9:=eq\mhyphen cong_1(\lambda z:S.(y\cdot z),a_5)\;:\;y\cdot (b\cdot x)=y\cdot e}{}

\step*{}{\boldsymbol{term_{\ref{theorem:unique}.2}}(S,\cdot,e,u,b,x,y, v,w):=eq\mhyphen trans_1(eq\mhyphen trans_1(eq\mhyphen trans_1(eq\mhyphen trans_3(a_3,a_8),a_7), a_9),a_4)\;:\;\boldsymbol{x=y}}{}
\end{flagderiv}

\subsection*{Proof of Example \ref{example:monoid}}

\begin{flagderiv}
\introduce*{}{M: *_s}{}
\step*{}{S:=ps(M)\;:\;\square}{}
\step*{}{\cap:=\lambda X,Y:S.X\cap Y
\;:\;S\rightarrow S\rightarrow S}{}
\step*{}{\cup:=\lambda X,Y:S.X\cup Y\;:\;S\rightarrow S\rightarrow S}{}
\step*{}{E:=full\mhyphen set(M)\;:\;S}{}
\step*{}{e:=\varnothing\;:\;S}{}
\introduce*{}{X,Y:S}{}
\introduce*{}{x:M}{}
\assume*{}{u:x\;\varepsilon\;(X\cap Y)}{}
\step*{}{a_1:=\wedge_1(u)\;:\;x\varepsilon X}{}
\step*{}{a_2:=\wedge_2(u)\;:\;x\varepsilon Y}{}
\step*{}{a_3:=\wedge(a_2,a_1)\;:\;x\varepsilon Y\wedge x\varepsilon X}{}
\conclude*{}{a_4:=\lambda u:x\;\varepsilon\;(X\cap Y).a_3\;:\;(x\;\varepsilon\;(X\cap Y)\Rightarrow x\;\varepsilon \;(Y\cap X))}{}

\assume*{}{u:x\;\varepsilon\;(X\cup Y)}{}
\step*{}{u\;:\;x\varepsilon X\vee x\varepsilon Y}{}
\assume*{}{v:x\varepsilon X}{}

\step*{}{a_5:=\vee_2(v)\;:\;x\;\varepsilon\;(Y\cup X)}{}

\conclude*{}{a_6:=\lambda v:x\varepsilon X.a_5\;:\;(x\varepsilon X\Rightarrow x\;\varepsilon\;(Y\cup X))}{}

\assume*{}{v:x\varepsilon Y}{}

\step*{}{a_7:=\vee_1(v)\;:\;x\;\varepsilon\;(Y\cup X)}{}

\conclude*{}{a_8:=\lambda v:x\varepsilon Y.a_7\;:\;(x\varepsilon Y\Rightarrow x\;\varepsilon\;(Y\cup X))}{}

\step*{}{a_9:=\vee(u,a_6,a_8)\;:\;x\;\varepsilon\;(Y\cup X)}{}
\conclude*{}{a_{10}:=\lambda u:x\;\varepsilon\;(X\cup Y).a_9\;:\;(x\;\varepsilon\;(X\cup Y)\Rightarrow x\;\varepsilon\;(Y\cup X))}{}
\conclude*{}{a_{11}(X,Y):=\lambda x:M.a_4\;:\;X\cap Y\subseteq Y\cap X}{}
\step*{}{a_{12}(X,Y):=\lambda x:M.a_{10}\;:\;X\cup Y\subseteq Y\cup X}{}

\step*{}{a_{13}:=\wedge(a_{11}(X,Y),
a_{11}(Y,X))\;:\;X\cap Y= Y\cap X}{}
\step*{}{a_{14}:=\wedge(a_{12}(X,Y),
a_{12}(Y,X))\;:\;X\cup Y= Y\cup X}{}

\conclude*{}{a_{15}:=\lambda X,Y:S.a_{13}\;:\;commut(S,\cap)}{}
\step*{}{a_{16}:=\lambda X,Y:S.a_{14}\;:\;commut(S,\cup)}{}

\introduce*{}{X:S}{}
\introduce*{}{x:M}{}
\assume*{}{u:x\varepsilon X}{}
\step*{}{a_{17}:=\lambda v:\bot.v\;:\;\bot\rightarrow \bot}{}
\step*{}{a_{17}\;:\;\neg \bot}{}
\step*{}{a_{18}:=
\wedge(u,a_{17})\;:\;x\;\varepsilon\;(X\cap E)}{}
\step*{}{a_{19}:=
\vee_1(u)\;:\;x\;\varepsilon\;(X\cup e)}{}
\conclude*{}{a_{20}:=\lambda u:x\varepsilon X.a_{18}\;:\;(x\varepsilon X\Rightarrow x\;\varepsilon\;(X\cap E))}{}
\step*{}{a_{21}:=\lambda u:x\varepsilon X.a_{19}\;:\;(x\varepsilon X\Rightarrow x\;\varepsilon\;(X\cup e))}{}

\step*{}{a_{22}:=\lambda u:x\;\varepsilon\;(X\cap E).\wedge_1(u)\;:\;(x\;\varepsilon\;(X\cap E)\Rightarrow x\varepsilon X)}{}

\assume*{}{u:x\;\varepsilon\;(X\cup e)}{}
\step*{}{u:x\varepsilon X\vee \bot}{}
\step*{}{a_{23}:=\lambda v:x\varepsilon X.v\;:\;(x\varepsilon X\Rightarrow x\varepsilon X)}{}
\step*{}{a_{24}:=\lambda v:\bot.v(x\varepsilon X)\;:\;(\bot\Rightarrow x\varepsilon X)}{}
\step*{}{a_{25}:=\vee(u,a_{23},a_{24})\;:\;x\varepsilon X}{}
\conclude*{}{a_{26}:=\lambda u:x\;\varepsilon\;(X\cup e).a_{25}\;:\;(x\;\varepsilon \;(X\cup e)\Rightarrow x\varepsilon X)}{}
\conclude*{}{a_{27}:=\lambda x:M.a_{20}\;:\;X\subseteq X\cap E}{}

\step*{}{a_{28}:=\lambda x:M.a_{21}\;:\;X\subseteq X\cup e}{}
\step*{}{a_{29}:=\lambda x:M.a_{22}\;:\;X\cap E\subseteq X}{}
\step*{}{a_{30}:=\lambda x:M.a_{26}\;:\;X\cup e\subseteq X}{}

\step*{}{a_{31}:=\wedge(a_{29},a_{27})\;:\;X\cap E= X}{}
\step*{}{a_{32}:=\wedge(a_{30},a_{28})\;:\;X\cup e= X}{}
\step*{}{a_{33}:=a_{15}EX\;:\;E\cap X= X\cap E}{}
\step*{}{a_{34}:=a_{16}eX\;:\;e\cup X= X\cup e}{}
\step*{}{a_{35}:=eq\mhyphen trans_1(a_{33},a_{31})\;:\;E\cap X= X}{}
\step*{}{a_{36}:=eq\mhyphen trans_1(a_{34},a_{32})\;:\;e\cup X= X}{}
\step*{}{a_{37}:=\wedge(a_{31},a_{35})\;:\;(X\cap E= X)\wedge (E\cap X= X)}{}
\step*{}{a_{38}:=\wedge(a_{32},a_{36})\;:\;(X\cup e= X)\wedge (e\cup X= X)}{}
\conclude*{}{a_{39}:=\lambda X:S.a_{37}\;:\;identity(S,\cap,E)}{}
\step*{}{a_{40}:=\lambda X:S.a_{38}\;:\;identity(S,\cup,e)}{}

\introduce*{}{X,Y,Z:S}{}
\introduce*{}{x:M}{}
\assume*{}{u:x\;\varepsilon\;((X\cap Y)\cap Z)}{}
\step*{}{a_{41}:=\wedge_1(u)\;:\;x\;\varepsilon\;(X\cap Y)}{}
\step*{}{a_{42}:=\wedge_1(a_{41})\;:\;x\varepsilon X}{}
\step*{}{a_{43}:=\wedge_2(a_{41})\;:\;x\varepsilon Y}{}
\step*{}{a_{44}:=\wedge_2(u)\;:\;x\varepsilon Z}{}
\step*{}{a_{45}:=\wedge(a_{42},\wedge(a_{43},a_{44}))\;:\;x\;\varepsilon\;(X\cap(Y\cap Z))}{}
\conclude*{}{a_{46}:=\lambda 
u:x\;\varepsilon\;((X\cap Y)\cap Z).a_{45}\;:\;[x\;\varepsilon\;((X\cap Y)\cap Z)\Rightarrow
x\;\varepsilon\;(X\cap(Y\cap Z))]}{}

\assume*{}{u:x\;\varepsilon\;(X\cap (Y\cap Z))}{}
\step*{}{a_{47}:=\wedge_1(u)\;:\;x\varepsilon X}{}
\step*{}{a_{48}:=\wedge_2(u)\;:\;x\;\varepsilon\;(Y\cap Z)}{}
\step*{}{a_{49}:=\wedge_1(a_{48})\;:\;x\varepsilon Y}{}
\step*{}{a_{50}:=\wedge_2(a_{48})\;:\;x\varepsilon Z}{}

\step*{}{a_{51}:=\wedge(\wedge(a_{47}, a_{49}),a_{50})\;:\;x\;\varepsilon\;((X\cap Y)\cap Z)}{}
\conclude*{}{a_{52}:=\lambda 
u:x\;\varepsilon\;(X\cap (Y\cap Z)).a_{51}\;:\;[x\;\varepsilon\;(X\cap(Y\cap Z))\Rightarrow
x\;\varepsilon\;((X\cap Y)\cap Z)]}{}
\assume*{}{u:x\;\varepsilon\;((X\cup Y)\cup Z)}{}
\step*{}{u:x\;\varepsilon\;(X\cup Y)\vee x\varepsilon Z}{}
\assume*{}{v:x\;\varepsilon\;(X\cup Y)}{}
\step*{}{v:x\varepsilon X\vee x\varepsilon Y}{}

\step*{}{a_{53}:=\lambda w:x\varepsilon X.\vee_1(w)\;:\;(x\varepsilon X\Rightarrow x\;\varepsilon\;(X\cup(Y\cup Z)))}{}
\assume*{}{w:x\varepsilon Y}{}
\step*{}{a_{54}:=\vee_1(w)\;:\;x\;\varepsilon\;(Y\cup Z)}{}

\step*{}{a_{55}:=\vee_2(a_{54})\;:\;x\;\varepsilon\;(X\cup(Y\cup Z))}{}
\conclude*{}{a_{56}:=\lambda w:x\varepsilon Y.a_{55}\;:\;[x\varepsilon Y\Rightarrow x\;\varepsilon\;(X\cup(Y\cup Z))]}{}
\step*{}{a_{57}:=\vee(v,a_{53},a_{56})\;:\;x\;\varepsilon\;(X\cup(Y\cup Z))}{}
\conclude*{}{a_{58}:=\lambda v:x\;\varepsilon \;(X\cup Y).a_{57}\;:\;[x\;\varepsilon \;(X\cup Y)\Rightarrow x\;\varepsilon\;(X\cup(Y\cup Z))]}{}

\assume*{}{v:x\varepsilon Z}{}
\step*{}{a_{59}:=\vee_2(v)\;:\;x\;\varepsilon\;(Y\cup Z)}{}
\step*{}{a_{60}:=\vee_2(a_{59})\;:\;x\;\varepsilon\;(X\cup(Y\cup Z))}{}
\conclude*{}{a_{61}:=\lambda v:x\varepsilon Z.a_{60}\;:\;[x\varepsilon Z\Rightarrow x\;\varepsilon\;(X\cup(Y\cup Z))]}{}
\step*{}{a_{62}:=\vee(u,a_{58},a_{61})\;:\;x\;\varepsilon\;(X\cup(Y\cup Z))}{}
\conclude*{}{a_{63}:=\lambda u:x\varepsilon (X\cup Y)\cup Z.a_{62}\;:\;[x\;\varepsilon\;((X\cup Y)\cup Z)\Rightarrow x\;\varepsilon\;(X\cup(Y\cup Z))]}{}
\conclude*{}{a_{64}:=\lambda x:M.a_{46}\;:\;(X\cap Y)\cap Z\subseteq X\cap(Y\cap Z)}{}
\step*{}{a_{65}:=\lambda x:M.a_{52}\;:\;
X\cap(Y\cap Z)\subseteq 
(X\cap Y)\cap Z}{}

\step*{}{a_{66}(X,Y,Z):=\lambda x:M.a_{63}\;:\;(X\cup Y)\cup Z\subseteq X\cup(Y\cup Z)}{}

\step*{}{a_{67}:=
\wedge(a_{64},a_{65})
\;:\;(X\cap Y)\cap Z= X\cap(Y\cap Z)}{}

\step*{}{a_{68}:=a_{66}(Z,Y,X)\;:\;(Z\cup Y)\cup X\subseteq Z\cup(Y\cup X)}{}
\step*{}{a_{69}:=a_{16}ZY\;:\;Z\cup Y= Y\cup Z}{}
\step*{}{a_{70}:=a_{16}YX\;:\;Y\cup X= X\cup Y}{}
\step*{}{a_{71}:=eq\mhyphen subs_1(\lambda D:S.(D\cup X\subseteq(Z\cup Y)\cup X),a_{69},a_{68})\;:\;
(Y\cup Z)\cup X\subseteq Z\cup(Y\cup X)}{}
\step*{}{a_{72}:=eq\mhyphen subs_1(\lambda D:S.((Y\cup Z)\cup X\subseteq Z\cup D),a_{70}, a_{71}) \;:\;
(Y\cup Z)\cup X\subseteq Z\cup
(X\cup Y)}{}
\step*{}{a_{73}:=a_{16}(Y\cup Z)X\;:\;(Y\cup Z)\cup X= 
X\cup(Y\cup Z)}{}
\step*{}{a_{74}:=a_{16}Z(X\cup Y)\;:\;Z\cup(X\cup Y)=
(X\cup Y)\cup Z}{}
\step*{}{a_{75}:=eq\mhyphen subs_1(\lambda D:S.(D\subseteq Z\cup(X\cup Y)), a_{73}, a_{72}) \;:\;
X\cup(Y\cup Z)\subseteq 
Z\cup(X\cup Y)}{}
\step*{}{a_{76}:=eq\mhyphen subs_1(\lambda D:S.(X\cup(Y\cup Z)\subseteq D),a_{74}, a_{75}) \;:\;
X\cup(Y\cup Z)\subseteq 
(X\cup Y)\cup Z}{}
\step*{}{a_{77}:=\wedge(a_{66}(X,Y,Z),a_{76})\;:\;(X\cup Y)\cup Z=X\cup(Y\cup Z)
}{}
\conclude*{}{a_{78}:=\lambda  X,Y,Z:S.a_{67} \;:\; assoc(S,\cap)}{}
\step*{}{a_{79}:=\lambda  X,Y,Z:S.a_{77} \;:\; assoc(S,\cup)}{}
\step*{}{a_{80}:=\wedge(a_{78},a_{39})\;:\; monoid(S,\cap,E)}{}
\step*{}{a_{81}:=\wedge(a_{79},a_{40})\;:\; monoid(S,\cup,e)}{}

\step*{}{ \boldsymbol{ term_{ \ref{example:monoid}.1}(M)}:=\wedge(a_{80},a_{15})\;:\;\boldsymbol{monoid(S,\cap,E)\wedge commut(S,\cap)}\quad\quad\quad\textsf{1) is proven}}{}
\step*{}{ \boldsymbol{ term_{ \ref{example:monoid}.2}(M)}:=\wedge(a_{81},a_{16})\;:\;\boldsymbol{monoid(S,\cup,e)\wedge commut(S,\cup)}\quad\quad\quad\;\textsf{2) is proven}}{}
\end{flagderiv}

\newpage
\subsection*{Proof of Example \ref{example: monoid_functions}
}
Here we implicitly use the Axiom of Extensionality for functions.

\begin{flagderiv}
\introduce*{}{M: *_s}{}
\introduce*{}{f,g,h:M\rightarrow M}{}
\introduce*{}{x:M}{}
\step*{}{a_1:=eq\mhyphen refl\;:\;((f\circ g)\circ h)x=((f\circ g)\circ h)x
}{}
\step*{}{a_1:((f\circ g)\circ h)x=f(g(hx))}{}
\step*{}{a_2:=eq\mhyphen refl\;:\;(f\circ(g\circ h))x=(f\circ(g\circ h))x
}{}
\step*{}{a_2:(f\circ(g\circ h))x=f(g(hx))}{}

\step*{}{a_3:=eq\mhyphen trans_2(a_1,a_2)\;:\;((f\circ g)\circ h)x=(f\circ(g\circ h))x}{}
\conclude*{}{a_4:=\lambda x: M.a_3\;:\;(f\circ g)\circ h=f\circ(g\circ h)}{}
\conclude*{}{a_5:=\lambda f,g,h:M\rightarrow M.a_4\;:\;assoc(M\rightarrow M,\circ)}{}

\introduce*{}{f:M\rightarrow M}{}
\introduce*{}{x:M}{}
\step*{}{a_6:=eq\mhyphen refl\;:\;(f\circ id_M) x=(f\circ id_M) x}{}
\step*{}{a_6\;:\;(f\circ id_M) x=f(id_M x)}{}
\step*{}{a_7(x):=eq\mhyphen refl\;:\;id_M  x=x}{}
\step*{}{a_8:=eq\mhyphen cong_1(f,a_7(x))\;:\;f(id_M  x)=fx}{}
\step*{}{a_9:=eq\mhyphen trans_1(a_6,a_8)\;:\;(f\circ id_M)x=fx}{}
\step*{}{a_{10}:=eq\mhyphen refl\;:\;(id_M\circ f) x=(id_M\circ f) x}{}
\step*{}{a_{10}\;:\;(id_M\circ f) x=id_M(fx)}{}
\step*{}{a_{11}:=a_7(fx)\;:\;id_M(fx)=fx}{}
\step*{}{a_{12}:=eq\mhyphen trans_1(a_{10},a_{11})\;:\;(id_M\circ f) x=fx}{}
\conclude*{}{a_{13}:=\lambda x:M.a_9\;:\;f\circ id_M=f}{}
\step*{}{a_{14}:=\lambda x:M.a_{12}\;:\;id_M\circ f=f}{}
\step*{}{a_{15}:=\wedge(a_{13},a_{14})\;:\; f\circ id_M=f\wedge id_M\circ f=f}{}
\conclude*{}{a_{16}:=\lambda f:M\rightarrow M.a_{15}
\;:\;identity(M\rightarrow M,\circ,id_M)}{}

\step*{}{\boldsymbol{ term_{\ref{example: monoid_functions}}}(M):=\wedge(a_5,a_{16})
\;:\;\boldsymbol{monoid(M\rightarrow M,\circ,id_M)}}{}
\end{flagderiv}

\subsection*{Proof of Example \ref{example:invert-monoid}}

For the monoid $S$ we use Definition 
\ref{def:group_type}, since $S$  is a type here. To prove that the set of all its invertible elements is a group we show that this set satisfies the conditions of Definition \ref{def:group_subset}, the Main Definition  of Group. In the following proof we use the terms derived in Theorem \ref{theorem:extend_function_1}.

\begin{flagderiv}
\introduce*{}{S: *_s\;|\;\cdot: S\rightarrow S\rightarrow S\;|\;e:S}{}

\step*{}{\text{Definition }\boldsymbol{Inv\mhyphen set}(S,\cdot,e):=\{x:S\;|\;invertible(S,\cdot,e,x)\}\;:\;ps(S)}{}
\step*{}{\text{Notation }G:=Inv\mhyphen set(S,\cdot,e)\;:\;ps(S)}{}

\assume*{}{u:monoid(S,\cdot,e)}{}

\step*{}{a_1:=\wedge_1(u)
\;:\;assoc(S,\cdot)}{}

\step*{}{a_2:=\wedge_2(u)
\;:\;identity(S,\cdot,e)}{}

\introduce*{}{x:S\;|\;v_1:x\varepsilon G\;|\;y:S\;|\;v_2:y\varepsilon G\;|\;z:S\;|\;v_3:z\varepsilon G}{}
\step*{}{a_3:=a_1xyz\;:\;(x\cdot y)\cdot z=x\cdot(y\cdot z)}{}
\conclude*{}{a_4:=\lambda x:S. \lambda v_1:x\varepsilon G.\lambda y:S. \lambda v_2:y\varepsilon G.\lambda z:S. \lambda v_3:z\varepsilon G.
a_3\;:\;Assoc(S,G,\cdot)}{}
\step*{}{a_5:=a_2e\;:\;e\cdot e=e\wedge e\cdot e=e}{}
\step*{}{a_5:inverse(S,\cdot,e,e,e)}{}
\step*{}{a_6:=\exists_1(\lambda y:S.inverse(S,\cdot,e,e,y),e, a_5)\;:\;invertible(S,\cdot,e,e)}{}

\step*{}{a_6\;:\;e\varepsilon G}{}

\step*{}{a_7:=\lambda x:S.\lambda v: x\varepsilon G.a_2x\;:\;[\forall x:S.(x\varepsilon G\Rightarrow x\cdot e=x\wedge e\cdot x=x)]}{}
\step*{}{a_8:=\wedge(a_6,a_7)\;:\;Identity(S,G,\cdot,e)}{}
\introduce*{}{x:S\;|\;v:x\varepsilon G}{}
\step*{}{v:(\exists z:S.(x\cdot z=e\wedge z\cdot x=e))}{}
\introduce*{}{z_1:S\;|\;r_1:(x\cdot z_1=e\wedge z_1\cdot x=e)}{}
\step*{}{a_9:=\wedge_1(r_1)\;:\;x\cdot z_1=e}{}
\step*{}{a_{10}:=\wedge_2(r_1)\;:\;z_1\cdot x=e}{}
\introduce*{}{y:S\;|\;w:y\varepsilon G}{}
\step*{}{w:(\exists z:S.(y\cdot z=e\wedge z\cdot y=e))}{}
\introduce*{}{z_2:S\;|\;r_2:(y\cdot z_2=e\wedge z_2\cdot y=e)}{}
\step*{}{a_{11}:=\wedge_1(r_2)\;:\;y\cdot z_2=e}{}
\step*{}{a_{12}:=\wedge_2(r_2)\;:\;z_2\cdot y=e}{}

\step*{}{a_{13}:=eq\mhyphen cong_1(\lambda t:S.(x\cdot t),a_{11})\;:\; x\cdot(y\cdot z_2)=x\cdot e}{}

\step*{}{a_{14}:=\wedge_1(a_2x)\;:\;x\cdot e=x}{}

\step*{}{a_{15}:= a_1xyz_2\;:\;x\cdot y\cdot z_2=x\cdot (y\cdot z_2)}{}

\step*{}{a_{16}:=eq\mhyphen trans_1(eq\mhyphen trans_1(a_{15},a_{13}),a_{14})\;:\; x\cdot y\cdot z_2=x}{}

\step*{}{a_{17}:=eq\mhyphen cong_1(\lambda t:S.(t\cdot z_1),a_{16})\;:\; x\cdot y\cdot z_2\cdot z_1=x\cdot z_1}{}

\step*{}{a_{18}:= a_1(x\cdot y)z_2z_1\;:\;x\cdot y\cdot z_2\cdot z_1=x\cdot y\cdot  (z_2\cdot z_1)}{}

\step*{}{a_{19}:=eq\mhyphen trans_1(eq\mhyphen trans_3(a_{18},a_{17}),a_9)\;:\;x\cdot y \cdot (z_2\cdot z_1)=e}{}

\step*{}{a_{20}:=eq\mhyphen cong_1(\lambda t:S.(t\cdot y),a_{10})\;:\; z_1\cdot x\cdot y=e\cdot y}{}

\step*{}{a_{21}:=\wedge_2(a_2y)\;:\;e\cdot y=y}{}

\step*{}{a_{22}:= a_1z_1xy\;:\;z_1\cdot x\cdot y=z_1\cdot  (x\cdot y)}{}

\step*{}{a_{23}:=eq\mhyphen trans_1(eq\mhyphen trans_3(a_{22},a_{20}),a_{21})\;:\;z_1\cdot (x\cdot y)=y}{}

\step*{}{a_{24}:=eq\mhyphen cong_1(\lambda t:S.(z_2\cdot t),a_{23})\;:\; z_2\cdot (z_1\cdot(x\cdot y))=z_2\cdot y}{}

\step*{}{a_{25}:= a_1z_2z_1(x\cdot y)\;:\;z_2\cdot z_1\cdot(x\cdot y)=z_2\cdot (z_1\cdot(x\cdot y))}{}

\step*{}{a_{26}:=eq\mhyphen trans_1(eq\mhyphen trans_1(a_{25},a_{24}),a_{12})\;:\;z_2\cdot z_1\cdot(x\cdot y)=e}{}

\step*{}{a_{27}:=\wedge(a_{19},a_{26})\;:\; (x\cdot y)\cdot (z_2\cdot z_1)=e\wedge (z_2\cdot z_1)\cdot (x\cdot y)=e}{}

\step*{}{a_{28}:=\exists_1(\lambda t:S.((x\cdot y)\cdot t=e\wedge t\cdot (x\cdot y)=e),(z_2\cdot z_1), a_{27})\;:\;(x\cdot y)\;\varepsilon \;G}{}
\conclude*{}{a_{29}:=\exists_3(w,a_{28})\;:\;(x\cdot y)\;\varepsilon \;G}{}
\conclude*{}{a_{30}:=\lambda y:S.\lambda w:y\;\varepsilon \;G.a_{29}
\;:\;[\forall y:S.(y\varepsilon G\Rightarrow (x\cdot y)\;\varepsilon \;G)]}{}
\conclude*{}{a_{31}:=\exists_3(v,a_{30})\;:\;[\forall y:S.(y\varepsilon G\Rightarrow (x\cdot y)\;\varepsilon \;G)]}{}

\conclude*{}{a_{32}:=\lambda x:S.\lambda v:x\varepsilon G.a_{31}
\;:\;Closure_2(S,G,\cdot)}{}
\step*{}{a_{33}:=\wedge(\wedge(a_{32},a_4), a_8)\;:\;Monoid(S,G,\cdot)}{}
\step*{}{\text{Notation }P:=\lambda x,y:S.inverse(S,\cdot,e,x,y)\;:\;S\rightarrow S\rightarrow *_p}{}
\introduce*{}{x:S\;|\;v:x\varepsilon G}{}
\step*{}{v:(\exists y:S.Pxy)}{}
\introduce*{}{y_1,y_2:S\;|\; w_1:Pxy_1\;|\;w_2:Pxy_2}{}
\step*{}{w_1:inverse(S,\cdot,e,x,y_1)}{}
\step*{}{w_2:inverse(S,\cdot,e,x,y_2)}{}
\step*{}{a_{34}:=term_{ \ref{theorem:unique}.2}(S,\cdot,e,u,x,y_1,y_2,w_1,w_2)\;:\;y_1=y_2}{}
\conclude*{}{a_{35}:=\lambda y_1,y_2:S.\lambda w_1:Pxy_1.\lambda w_2:Pxy_2.a_{34}
\;:\;[\forall y_1,y_2:S.(Pxy_1\Rightarrow Pxy_2\Rightarrow y_1=y_2)]}{}
\step*{}{a_{36}:=\wedge(v,a_{35})\;:\;(\exists^1y:S.Pxy)}{}
\step*{}{ivs(x,v):=\iota (S,Px,a_{36})\;:\;S}{}
\step*{}{a_{37}(x,v):=\iota \mhyphen prop(S,Px,a_{36})\;:\;Px(ivs(x,v))}{}
\introduce*{}{w:x\varepsilon G}{}
\step*{}{a_{38}:=a_{37}(x,w)\;:\;Px(ivs(x,w))}{}
\step*{}{a_{39} :=\iota \mhyphen unique(S,Px,a_{36})(ivs(x,w))a_{38}\;:\; ivs(x,w)=ivs(x,v)}{}
\step*{}{a_{40} :=eq \mhyphen sym(a_{39})\;:\; ivs(x,v)=ivs(x,w)}{}

\conclude*[2]{}{a_{41}:=\lambda x:S.\lambda v,w:x\varepsilon G.a_{40}
\;:\;[\forall x:S.\Pi v,w:x\varepsilon G.(ivs(x,v)=ivs(x,w))]}{}

\step*{}{Ivs:=
\lambda x:S.\lambda v:x\varepsilon G.ivs(x,v)
\;:\;\Pi x:S.((x\varepsilon G)\rightarrow S)}{}
\step*{}{a_{41}\;:\;consistent_1(S,S,G,Ivs)}{}

\step*{}{\text{Definition }\boldsymbol{Inv}(S,\cdot,e,u):=
Ext_1(S,S,G,Ivs,e,a_{41})
\;:\;S\rightarrow S}{}

\step*{}{\text{Notation } ^{-1} \text{ for }Inv(S,\cdot,e,u)}{}

\step*{}{\text{Notation: }x^{-1}\text{ for }^{-1}x}{}
\step*{}{a_{42}:=
Ext\mhyphen proof_1(S,S,G,Ivs,e,a_{41})
\;:\;
\forall x:s.[\Pi p:x\varepsilon G.(x^{-1}=Ivs\;xp)\wedge (\neg(x\varepsilon G)\Rightarrow x^{-1}=e)]}{}
\introduce*{}{x:S\;|\;v:x\varepsilon G}{}
\step*{}{a_{43}:=\wedge_1(a_{42}x)v\;:\;x^{-1}=Ivs\;xv}{}
\step*{}{a_{43}\;:\;x^{-1}=ivs(x,v)}{}

\step*{}{a_{44}:=a_{37}(x,v)\;:\;Px(ivs(x,v))}{}
\step*{}{a_{45}:=eq\mhyphen subs_2(Px,a_{43},a_{44})
\;:\;Px(x^{-1})}{}

\step*{}{a_{45}\;:\; Inverse_0(S,\cdot, e,x,x^{-1})}{}
\step*{}{a_{46}:=\wedge_1(a_{45})\;:\;x\cdot x^{-1}=e}{}
\step*{}{a_{47}:=\wedge_2(a_{45})\;:\;x^{-1}\cdot x=e}{}
\step*{}{a_{48}:=\wedge(a_{47},a_{46})\;:\;x^{-1}\cdot x=e\wedge x\cdot x^{-1}=e}{}
\step*{}{a_{48}\;:\; inverse(S,\cdot, e,x^{-1},x)}{}
\step*{}{a_{49}:=
\exists_1(\lambda y:S.inverse(S,\cdot, e,x^{-1},y),x,a_{48})
\;:\; x^{-1}\varepsilon G}{}
\conclude*{}{a_{50}:=\lambda x:S.\lambda v:x\varepsilon G.a_{49}\;:\;Closure_1(S,G, ^{-1})}{}
\step*{}{a_{51}:=\lambda x:S.\lambda v:x\varepsilon G.a_{45}\;:\;Inverse(S,G,\cdot,e, ^{-1})}{}
\step*{}{\boldsymbol{term_{\ref{example:invert-monoid}}}(S,\cdot,e,u):=
\wedge(a_{33},\wedge(a_{50},a_{51}))
\;:\;\boldsymbol{Group(S,G,\cdot,e,^{-1})}}{}
\end{flagderiv}

\subsection*{Proof of Lemma \ref{lemma:definitions}
\label{section:definitions}
}
\begin{flagderiv}
\introduce*{}{S: *_s\;|\;\;\cdot: S\rightarrow S\rightarrow S\;|\;e:S\;|\;^{-1}:S\rightarrow S\;|\;u:group(S,\cdot,e,^{-1})}{}

\step*{}{G:=full\mhyphen set(S)\;:\;ps(S)}{}

\step*{}{a_1:=\wedge_1(u)\;:\;monoid(S,\cdot,e)}{}
\step*{}{a_2:=\wedge_2(u)\;:\;(\forall x:S.inverse
(S,\cdot,e, x,x^{-1}))}{}

\step*{}{a_3:=\wedge_1(a_1)\;:\;assoc(S,\cdot)}{}
\step*{}{a_4:=\wedge_2(a_1)\;:\;identity(S,\cdot,e)}{}

\step*{}{a_5:=\lambda v:\bot.v\;:\;\neg\bot}{}
\step*{}{a_5\;:\;e\varepsilon G}{}

\introduce*{}{x:S\;|\;v:x\varepsilon G}{}
\step*{}{a_5\;:\;x^{-1}\varepsilon G}{}
\step*{}{a_6:=a_4x\;:\;x\cdot e=x\wedge e\cdot x=x}{}
\step*{}{a_7:=a_2x\;:\;x\cdot x^{-1}=e\wedge x^{-1}\cdot x=e}{}

\introduce*{}{y:S.\;|\;w:y\varepsilon G}{}
\step*{}{a_5\;:\;(x\cdot y)\;\varepsilon \;G}{}

\introduce*{}{z:S.\;|\;r:z\varepsilon G}{}
\step*{}{a_8:=a_3xyz\;:\;(x\cdot y)\cdot z=x\cdot (y\cdot z)}{}

\conclude*[2]{}{a_9:=\lambda y:S.\lambda w:y\varepsilon G.\lambda z:S.\lambda r:z\varepsilon G.a_8
\;:\;[\forall y:S.[y\varepsilon G\Rightarrow 
\forall z:S.(z\varepsilon G\Rightarrow 
(x\cdot y)\cdot z=x\cdot (y\cdot z))]}{}
\step*{}{a_{10}:=
\lambda y:S.\lambda w:y\varepsilon G.a_5\;:\;[\forall y:S.(y\varepsilon G\Rightarrow (x\cdot y)\;\varepsilon \;G)]}{}

\conclude*{}{a_{11}:=
\lambda x:S.\lambda v:x\varepsilon G.a_{10}\;:\;Closure_2(S,G,\cdot)}{}

\step*{}{a_{12}:=
\lambda x:S.\lambda v:x\varepsilon G.a_9\;:\;Assoc(S,G,\cdot)}{}

\step*{}{a_{13}:=
\lambda x:S.\lambda v:x\varepsilon G.a_7\;:\;Inverse(S,G,\cdot,e,^{-1})}{}

\step*{}{a_{14}:=
\lambda x:S.\lambda v:x\varepsilon G.a_5\;:\;Closure_1(S,G,\cdot,^{-1})}{}

\step*{}{a_{15}:=
\lambda x:S.\lambda v:x\varepsilon G.a_6\;:\;
[\forall x:S.(x\varepsilon G\Rightarrow x\cdot e=x\wedge e\cdot x=x)]}{}

\step*{}{a_{16}:=\wedge (a_5,a_{15})\;:\;Identity(S,G,\cdot,e,^{-1})}{}

\step*{}{a_{17}:=\wedge (a_{14},a_{13})\;:\;Inverse\mhyphen prop(S,G,\cdot,e,^{-1})}{}

\step*{}{a_{18}:=\wedge (a_{11},a_{12})\;:\;Semi\mhyphen group(S,G,\cdot)}{}

\step*{}{\boldsymbol {term_{\ref {lemma:definitions}}}(S,\cdot,e,^{-1},u):=\wedge (\wedge (a_{18},a_{16}),a_{17})\;:\;\boldsymbol{Group(S,G,\cdot,e,^{-1})}}{}
\end{flagderiv}

\subsection*{Proof of Lemma \ref{lemma:two-definitions}}

In this proof we use the terms derived in Theorems \ref{theorem:extend_function_1} and \ref{theorem:extend_function_2}.

\begin{flagderiv}
\introduce*{}{S: *_s\;|\;G:ps(S)\;|\;mult:\Pi x:S.[(x\varepsilon G)\rightarrow \Pi y:S.((y\varepsilon G)\rightarrow S)]\;|\;inv:\Pi x:S.((x\varepsilon G)\rightarrow S)}{}

\assume*{}{e:S\;|\;u:group(S,G,mult,e,inv)}{}
\step*{}{a_1:=\wedge_1(\wedge_1(\wedge_1( \wedge_1(u))))
\;:\; consistent_2(S, S,G,mult)}{}
\step*{}{a_2:=\wedge_2(\wedge_1(\wedge_1( \wedge_1(u))))
\;:\;closure_2(S,G,mult)}{}
\step*{}{a_3:=\wedge_2(\wedge_1( \wedge_1(u)))
\;:\;assoc(S,G,mult)}{}
\step*{}{a_4:=\wedge_1(\wedge_2( \wedge_1(u)))
\;:\;e\varepsilon G}{}
\step*{}{a_5:=\wedge_2(\wedge_2( \wedge_1(u)))
\;:\;[\forall x:S.\Pi p: x\varepsilon G.\Pi q: e\varepsilon G.(mult\; xpeq=x\wedge mult\; eqxp=x)]}{}

\step*{}{a_6:=\wedge_1(\wedge_1( \wedge_2(u)))
\;:\;consistent_1(S,S,G,inv)}{}
\step*{}{a_7:=\wedge_2(\wedge_1( \wedge_2(u)))
\;:\;closure_1(S,G,inv)}{}
\step*{}{a_8:=\wedge_2( \wedge_2(u))
\;:\;inverse_1(S, G,mult,e,inv)}{}

\step*{}{\text{Definition }\; ^{-1}:=
Ext_1(S,S,G,inv,e,a_6)\;:\;S\rightarrow S}{}

\step*{}{\text{Notation }\; x^{-1} \text{ for }^{-1}x}{}

\step*{}{a_9:=Ext\mhyphen proof_1(S,S,G,inv,e,a_6)
\;:\; Ext\mhyphen prop_1(S,S,G,inv,e,^{-1})}{}

\step*{}{\text{Definition }\; \cdot:=
Ext_2(S,S,G,mult,e,a_1)\;:\;S\rightarrow S \rightarrow S}{}
\step*{}{\text{Notation }\; x\cdot y \text{ for }\cdot xy}{}

\step*{}{a_{10}:=Ext\mhyphen proof_2(S,S,G,mult,e,a_1)
\;:\; Ext\mhyphen prop_2(S,S,G,mult,e,\cdot)}{}

\introduce*{}{x:S\;|\;v:x\varepsilon G}{}
\step*{}{a_{11}:=a_7xv\;:\;(inv\;xv)\;\varepsilon G}{}
\step*{}{a_{12}:=\wedge_1( a_9x)\;:\;(\Pi p: x\varepsilon G.x^{-1}= inv\;xp)}{}
\step*{}{a_{13}:=a_{12}v\;:\;x^{-1}= inv\;xv}{}

\step*{}{a_{14}:=eq\mhyphen subs_2(\lambda z:S.z\varepsilon G, a_{13},a_{11})\;:\;x^{-1}\varepsilon G}{}
\step*{}{a_{15}:=a_5xva_4\;:\;
mult\; xvea_4=x\wedge mult\; ea_4xv=x
}{}
\step*{}{a_{16}:=\wedge_1( a_{15})\;:\;mult\;xvea_4=x}{}
\step*{}{a_{17}:=\wedge_2( a_{15})\;:\;mult\;ea_4xv=x}{}
\step*{}{a_{18}:=a_{10}xe\;:\;[\Pi p:x\varepsilon G.\Pi q:e\varepsilon G.(x\cdot e=mult\;xpeq)\wedge (\neg(x\varepsilon G\wedge e\varepsilon G)\Rightarrow x\cdot e=e)]}{}
\step*{}{a_{19}:=\wedge_1( a_{18})\;:\;[\Pi p:x\varepsilon G.\Pi q:e\varepsilon G.(x\cdot e=mult\;xpeq)]}{}
\step*{}{a_{20}:=a_{19}va_4\;:\;x\cdot e=mult\;xvea_4}{}
\step*{}{a_{21}:=eq\mhyphen trans_1(a_{20},a_{16})
\;:\;x\cdot e=x}{}

\step*{}{a_{22}:=a_{10}ex\;:\;[\Pi q:e\varepsilon G.\Pi p:x\varepsilon G.(e\cdot x=mult\;eqxp)\wedge (\neg(e\varepsilon G\wedge x\varepsilon G)\Rightarrow e\cdot x=e)]}{}
\step*{}{a_{23}:=\wedge_1( a_{22})\;:\;\;:\;[\Pi q:e\varepsilon G.\Pi p:x\varepsilon G.(e\cdot x=mult\;eqxp)]}{}
\step*{}{a_{24}:=a_{23}a_4v\;:\;e\cdot x=mult\;ea_4xv}{}
\step*{}{a_{25}:=eq\mhyphen trans_1(a_{24},a_{17})
\;:\;e\cdot x=x}{}
\step*{}{a_{26}:=\wedge(a_{21},a_{25})\;:\;x\cdot e=x\wedge e\cdot x=x}{}
\step*{}{a_{27}:=
a_8xva_{11}\;:\;
mult\;xv(inv\;xv)a_{11}=e\wedge
mult\;(inv\;xv)a_{11}xv=e}{}
\step*{}{a_{28}:=
\wedge_1(a_{27})\;:\;
mult\;xv(inv\;xv)a_{11}=e}{}
\step*{}{a_{29}:=
\wedge_2(a_{27})\;:\;
mult\;(inv\;xv)a_{11}xv=e}{}
\step*{}{a_{30}:=a_{10}x(inv\;xv)\;:\;
[\Pi p:x\varepsilon G.\Pi q:(inv\;xv)\varepsilon G.(x\cdot (inv\;xv)=mult\;xp(inv\;xv)q)
\\\quad\quad
\wedge (\neg(x\varepsilon G\wedge (inv\;xv)\;\varepsilon \;G)\Rightarrow x\cdot (inv\;xv)=e)]
}{}
\step*{}{a_{31}:=\wedge_1(
a_{30})va_{11}\;:\;x\cdot (inv\;xv)=mult\;xv(inv\;xv)a_{11}
}{}
\step*{}{a_{32}:=eq\mhyphen trans_1(a_{31},a_{28})
\;:\;x\cdot (inv\;xv)=e}{}
\step*{}{a_{33}:=eq\mhyphen subs_2(\lambda z:S.(x\cdot z=e),a_{13},a_{32})\;:\;x\cdot x^{-1}=e}{}

\step*{}{a_{34}:=a_{10}(inv\;xv)x\;:\;
[\Pi q:(inv\;xv)\varepsilon G.\Pi p:x\varepsilon G.((inv\;xv)\cdot x=mult\;(inv\;xv)qxp)
\\\quad\quad
\wedge (\neg((inv\;xv)\;\varepsilon \;G\wedge x\varepsilon G)\Rightarrow (inv\;xv)\cdot x=e)]
}{}
\step*{}{a_{35}:=\wedge_1(
a_{34})a_{11}v\;:\;(inv\;xv)\cdot x=mult\;(inv\;xv)a_{11}xv
}{}
\step*{}{a_{36}:=eq\mhyphen trans_1(a_{35},a_{29})
\;:\;(inv\;xv)\cdot x=e}{}

\step*{}{a_{37}:=eq\mhyphen subs_2(\lambda z:S.(z\cdot x=e),a_{13},a_{36})
\;:\;x^{-1}\cdot  x=e}{}

\step*{}{a_{38}:=\wedge (a_{33},a_{37})\;:\;Inverse_0(S,\cdot,e,x, x^{-1})}{}

\introduce*{}{y:S\;|\;w:y\varepsilon G}{}
\step*{}{a_{39}:=a_2xvyw\;:\;(mult\;xvyw)\;\varepsilon \;G}{}

\step*{}{a_{40}:=a_{10}xy\;:\;
[\Pi p:x\varepsilon G.\Pi q:y\varepsilon G.(x\cdot y=mult\;xpyq)
\wedge (\neg(x\varepsilon G\wedge y\varepsilon G)\Rightarrow x\cdot y=e)]
}{}
\step*{}{a_{41}:=\wedge_1(a_{40})vw\;:\;
x\cdot y=mult\;xvyw}{}

\step*{}{a_{42}:=eq\mhyphen subs_2(\lambda z:S.z\varepsilon G,a_{42},a_{39})
\;:\;(x\cdot y)\;\varepsilon\; G}{}

\introduce*{}{z:S\;|\;t:z\varepsilon G}{}

\step*{}{a_{43}:=a_2ywzt\;:\;(mult\;ywzt)\;\varepsilon\;G}{}

\step*{}{a_{44}:=a_3xvywa_{39}zta_{43}\;:\;
mult(mult\;xvyw)a_{39}zt=
mult\;xv(mult\;ywzt)a_{43}
}{}
\step*{}{a_{45}:=a_{10}yz\;:\;
[\Pi q:y\varepsilon G.\Pi r:z\varepsilon G.(y\cdot z=mult\;yqzr)
\wedge (\neg(y\varepsilon G\wedge z\varepsilon G)\Rightarrow y\cdot z=e)]
}{}
\step*{}{a_{46}:=\wedge_1(a_{45})wt\;:\;
y\cdot z=mult\;ywzt}{}

\step*{}{a_{47}:=a_{10}(mult\;xvyw)z\;:\;
[\Pi p: (mult\;xvyw)\varepsilon G.\Pi r:z\varepsilon G.
[(mult\;xvyw)\cdot z=mult(mult\;xvyw)pzr]
\\\quad\quad
\wedge (\neg((mult\;xvyw)\varepsilon G\wedge z\varepsilon G)\Rightarrow (mult\;xvyw)\cdot z=e)]}{}

\step*{}{a_{48}:=\wedge_1(a_{47})a_{39}t\;:\;
(mult\;xvyw)\cdot z=mult(mult\;xvyw)a_{39}zt}{}

\step*{}{a_{49}:=a_{10}x(mult\;ywzt)\;:\;
[\Pi p:x\varepsilon G.
\Pi q: (mult\;ywzt)\varepsilon G.
[x\cdot (mult\;ywzt)
\\\quad\quad
=mult\;xp(mult\;ywzt)q)]
\wedge (\neg(x\varepsilon G\wedge (mult\;ywzt)\varepsilon G)\Rightarrow x\cdot (mult\;ywzt)=e)]}{}

\step*{}{a_{50}:=\wedge_1(a_{49})va_{43}\;:\;x\cdot (mult\;ywzt)=mult\;xv(mult\;ywzt)a_{43}}{}

\step*{}{a_{51}:=eq\mhyphen trans_2(eq\mhyphen trans_1(a_{48},a_{44}),a_{50})
\;:\;(mult\;xvyw)\cdot z=x\cdot (mult\;ywzt)}{}

\step*{}{a_{52}:=eq\mhyphen subs_2(\lambda m:S.(m\cdot z=
x\cdot (mult\;ywzt)),
a_{41},a_{51})
\;:\;(x\cdot y)\cdot z=x\cdot (mult\;ywzt)}{}

\step*{}{a_{53}:=eq\mhyphen subs_2(\lambda m:S.((x\cdot y)\cdot z=x\cdot m),a_{46}, a_{52})\;:\;(x\cdot y)\cdot z=x\cdot (y\cdot z)}{}

\conclude*[2]{}{a_{54}:=\lambda y:S.\lambda w:y\varepsilon G.\lambda z:S.\lambda t:z\varepsilon G.a_{53}\;:\;\forall y:S[y\varepsilon G\Rightarrow \forall z:S(z\varepsilon G\Rightarrow(x\cdot y)\cdot z=x\cdot (y\cdot z))]}{}

\step*{}{a_{55}:=\lambda y:S.\lambda w:y\varepsilon G.a_{42}\;:\;\forall y:S(y\varepsilon G\Rightarrow (x\cdot y)\;\varepsilon \;G)}{}

\conclude*{}{a_{56}:=\lambda x:S.\lambda v:x\varepsilon G.a_{55}\;:\;Closure_2(S,G,\cdot)}{}
\step*{}{a_{57}:=\lambda x:S.\lambda v:x\varepsilon G.a_{54}\;:\;Assoc(S,G,\cdot)}{}
\step*{}{a_{58}:=\lambda x:S.\lambda v:x\varepsilon G.a_{14}\;:\;Closure_1(S,G,^{-1})}{}
\step*{}{a_{59}:=\lambda x:S.\lambda v:x\varepsilon G.a_{38}\;:\;Inverse(S,G,\cdot,e,^{-1})}{}

\step*{}{a_{60}:=\lambda x:S.\lambda v:x\varepsilon G.a_{26}\;:\;
\forall x:S.[x\varepsilon G\Rightarrow (x\cdot e=x\wedge e\cdot x=x)]}{}

\step*{}{a_{61}:=\wedge(a_{58},a_{59})\;:\;Inverse\mhyphen prop(S,G,\cdot,e,^{-1})}{}

\step*{}{a_{62}:=\wedge(a_4,a_{60})\;:\;Identity(S,G,\cdot,e)}{}
\step*{}{a_{63}:=\wedge(a_{56},a_{57})\;:\;Semi\mhyphen group(S,G,\cdot)}{}
\step*{}{a_{64}:=\wedge(a_{63},a_{62})\;:\;Monoid(S,G,\cdot,e)}{}
\step*{}{a_{65}:=\wedge(a_{64},a_{61})\;:\;\boldsymbol{Group(S,G,\cdot,e,^{-1})}}{}
\end{flagderiv}

\subsection*{Proof of Lemma \ref{def:inverse}}

\begin{flagderiv}
\introduce*{}{M: *_s\;|\;f: M\rightarrow M\;|\;u:f\varepsilon Perm(M)}{}

\step*{}{\text{Notation }P:=\lambda y,x:M.fx=y\;:\;M\rightarrow M\rightarrow *_p}{}

\step*{}{u:bij(M,M,f)}{}
\step*{}{a_1:=\wedge_1(u)\;:\;inj(M,M,f)}{}
\step*{}{a_2:=\wedge_2(u)\;:\;surj(M,M,f)}{}
\introduce*{}{y:M}{}
\step*{}{a_3:=a_2y\;:\;(\exists x:M.fx=y)}{}
\step*{}{a_3\;:\;(\exists x:M.Pyx)}{}
\introduce*{}{x_1,x_2:M\;|\;v:Pyx_1\;|\;w :Pyx_2}{}

\step*{}{v:fx_1=y}{}
\step*{}{a_4:=eq\mhyphen sym(w)\;:\;y=fx_2}{}
\step*{}{a_5:=eq\mhyphen trans_1(v,a_4)\;:\;fx_1=fx_2}{}
\step*{}{a_6:= a_1x_1x_2a_5\;:\;x_1=x_2}{}
\conclude*{}{a_7:=\lambda x_1,x_2:M.\lambda v:Pyx_1.\lambda w:Pyx_2.a_6\;:\;[\forall x_1,x_2:M.(Pyx_1\Rightarrow Pyx_2 \Rightarrow (x_1=x_2))]}{}
\step*{}{a_8:=\wedge(a_3,a_7)\;:\;(\exists^1 x:M.Pyx)}{}

\step*{}{c(y):=
\iota(M,Py,a_8)\;:\;M}{}

\step*{}{d(y):=
\iota\mhyphen prop(M,Py,a_8)\;:\;Pyc(y)}{}

\conclude*{}{\boldsymbol{ invrs}(M ,f, u):=\lambda y:M.c(y)\;:\;M\rightarrow M\hspace{2cm} \textbf{Inverse of permutation f}}{}

\step*{}{\text{Notation }\; g:=invrs(M ,f, u)\;:\;M\rightarrow M}{}

\introduce*{}{y:M}{}

\step*{}{a_9:=d(y)\;:\;Pyc(y)}{}

\step*{}{a_9:Py(gy)}{}

\step*{}{a_9:f(gy)=y}{}
\step*{}{a_9:(f\circ g)y=id_My}{}

\conclude*{}{a_{10}:=\lambda y:M.a_9\;:\;f\circ g=id_M}{}

\introduce*{}{x:M}{}
\step*{}{\text{Notation }y:=fx}{}
\step*{}{a_{11}:=a_{10}y\;:\;f(gy)=y}{}

\step*{}{a_{11}\;:\;f(gy)=fx}{}

\step*{}{a_{12}:=a_1(gy)xa_{11}\;:\;gy=x}{}
\step*{}{a_{12}\;:\;g(fx)=x}{}

\step*{}{a_{12}\;:\;(g\circ f)x=id_Mx}{}
\conclude*{}{a_{13}:=\lambda x:M.a_{12}
\;:\;g\circ f=id_M}{}
\step*{}{ \boldsymbol{term_{ \ref{def:inverse}}}(M,f,u):=\wedge(a_{10},a_{13})
\;:\;\boldsymbol{inverse(M\rightarrow M,\circ,id_M, f,g)}}{}
\end{flagderiv}

\subsection*{Proof of Proposition \ref{lemma:permutations}}
Here we use the term $invrs$ from Lemma \ref{def:inverse}.

\begin{flagderiv}
\introduce*{}{M: *_s\;|\;f: M\rightarrow M}{}

\step*{}{\text{Notation }S:=
M\rightarrow M\;:\;*}{}

\assume*{}{u:invertible(S,\circ,id_M,f)}{}

\step*{}{u\;:\;(\exists g:S.(f\circ g=id_M\wedge g\circ f=id_M))}{}

\introduce*{}{g:S\;|\;v:(f\circ g=id_M\wedge g\circ f=id_M)}{}

\step*{}{a_1:=\wedge_1(v)\;:\;f\circ g=id_M}{}
\step*{}{a_2:=\wedge_2(v)\;:\;g\circ f=id_M}{}
\introduce*{}{x_1,x_2:M\;|\;w:fx_1=fx_2}{}

\step*{}{a_3:=
eq\mhyphen cong_1(\lambda h:S.hx_1,a_2)\;:\;(g\circ f)x_1=id_Mx_1}{}

\step*{}{a_4:=
eq\mhyphen cong_1(\lambda h:S.hx_2,a_2)\;:\;(g\circ f)x_2=id_Mx_2}{}

\step*{}{a_3\;:\; g(fx_1)=x_1}{}

\step*{}{a_4\;:\; g(fx_2)=x_2}{}

\step*{}{a_5:=
eq\mhyphen cong_1(g,w)\;:\;g(fx_1)=g(fx_2)}{}

\step*{}{a_6:=
eq\mhyphen trans_1(eq\mhyphen trans_3(a_3,a_5),a_4)\;:\; x_1=x_2}{}

\conclude*{}{a_7:=\lambda x_1,x_2:M.\lambda w:(fx_1= fx_2).a_6\;:\;inj(M,M,f)}{}

\conclude*{}{a_8:=\exists(u,a_7)\;:\;inj(M,M,f)}{}
\introduce*{}{y:M}{}

\step*{}{a_9:=
eq\mhyphen cong_1(\lambda h:S.hy,a_1)\;:\;(f\circ g)y=id_My}{}
\step*{}{a_9\;:\;f(gy)=y}{}

\step*{}{a_{10}:=
\exists_1(\lambda x:M.(fx=y), gy,a_9) \;:\;(\exists x:M.fx=y)}{}

\conclude*{}{a_{11}:=\lambda y:M.a_{10}\;:\;surj(M,M,f)}{}

\step*{}{a_{12}:=
\wedge(a_8,a_{11}) \;:\;bij(M,M,f)}{}

\step*{}{a_{12} \;:\;f\varepsilon Perm(M)}{}

\conclude*{}{a_{13}:=\lambda u:invertible(S,\circ,id_M,f). a_{12}\;:\;invertible(S,\circ,id_M,f)\Rightarrow f
\varepsilon Perm(M)}{}

\assume*{}{u:f
\varepsilon Perm(M)}{}

\step*{}{\text{Notation }Q:=
\lambda h:S.inverse(S,\circ,id_M,f,h)\;:\;
S\rightarrow *_p}{}

\step*{}{\text{Notation }g:=
invrs(M,f,u)\;:\;S}{}

\step*{}{a_{14}:=term_{\ref{def:inverse}}
(M,f,u)\;:\;Qg}{}

\step*{}{a_{15}:=
\exists_1(Q,g,a_{14}) \;:\;invertible(S,\circ,id_M,f)}{}

\conclude*{}{a_{16}:=\lambda u:f\varepsilon Perm(M).a_{15}\;:\;f\varepsilon Perm(M)\Rightarrow invertible(S,\circ,id_M,f)}{}

\step*{}{ \boldsymbol{term_{ \ref{lemma:permutations}}}(M,f):=\wedge(a_{13},a_{16})
\;:\;\boldsymbol{invertible(M\rightarrow M,\circ,id_M,f)\Leftrightarrow f\varepsilon Perm(M)}}{}
\end{flagderiv}

\subsection*{Proof of Theorem \ref{theorem:permutations}}

Here we use the terms $Inv$ and $Inv\mhyphen set$ from Example \ref{example:invert-monoid}.

\begin{flagderiv}
\introduce*{}{M: *_s}{}

\step*{}{G:=Inv\mhyphen set
(M\rightarrow M,\circ,id_M)\;:\;ps(M\rightarrow M)}{}

\step*{}{a_1:= term_{\ref{example: monoid_functions}}(M)\;:\;monoid(M\rightarrow M,\circ,id_M)}{}

\step*{}{^{-1}:=Inv(M\rightarrow M, \circ,id_M,a_1)\;:\;(M\rightarrow M)\rightarrow M\rightarrow M}{}

\step*{}{a_2:= term_{\ref{example:invert-monoid}}(M\rightarrow M,\circ,id_M,a_1)\;:\;Group(M\rightarrow M,G,\circ,id_M,^{-1})}{}

\introduce*{}{f: M\rightarrow M}{}
\step*{}{a_3:=term_{\ref{lemma:permutations}}(M,f)\;:\;invertible(M\rightarrow M, \circ,id_M,f)\Leftrightarrow f\varepsilon Perm(M)}{}

\step*{}{a_3\;:\;f\varepsilon G\Leftrightarrow f\varepsilon Perm(M)}{}

\step*{}{a_4:=\wedge_1(a_3) \;:\;f\varepsilon G\Rightarrow f\varepsilon Perm(M)}{}

\step*{}{a_5:=\wedge_2(a_3) \;:\;f\varepsilon Perm(M)\Rightarrow f\varepsilon G}{}

\conclude*{}{a_6:=\lambda f:M\rightarrow M.a_4\;:\;G\subseteq Perm(M)}{}

\step*{}{a_7:=\lambda f:M\rightarrow M.a_5\;:\;Perm(M)\subseteq G}{}

\step*{}{a_8:=\wedge(a_6,a_7) \;:\;G=Perm(M)}{}

\step*{}{\text{Notation }P:=\lambda Z:ps(M\rightarrow M).Group(M\rightarrow M,Z,\circ,id_M,^{-1})\;:\;ps(ps(M\rightarrow M))\rightarrow *_p}{}

\step*{}{ \boldsymbol{term_{\ref{theorem:permutations}}}(M):=eq\mhyphen subs_1(P,a_8,a_2)
\;:\;\boldsymbol{ Group(M\rightarrow M,Perm(M), \circ,id_M,^{-1})}}{}
\end{flagderiv}

\subsection*{Proof of Proposition \ref{theorem:axiom_corollary}}

\begin{flagderiv}
\introduce*{}{S: *_s\;|\;G:ps(S)\;|\;\cdot: S\rightarrow S\rightarrow S\;|\;e:S\;|\;^{-1}:S\rightarrow S\;|\;u:Group(S,G,\cdot,e,^{-1})}{}

\step*{}{a_1:=Gr_1(S,G,\cdot,e,^{-1},u)
\;:\;e\varepsilon G}{}

\step*{}{a_2:=Gr_2(S,G,\cdot,e,^{-1},u)
\;:\;Assoc(S,G,\cdot)}{}

\introduce*{}{x:S\;|\;v:x\varepsilon G}{}

\step*{}{a_3:=Gr_3(S,G,\cdot,e,^{-1},u,x,v)\;:\;x^{-1}\varepsilon G}{}

\step*{}{a_4:=Gr_7(S,G,\cdot,e,^{-1},u,x,v)
\;:\;x\cdot x^{-1}=e}{}

\step*{}{a_5:=Gr_8(S,G,\cdot,e,^{-1},u,x,v)
\;:\;x^{-1}\cdot x=e}{}

\step*{}{a_6:=Gr_4(S,G,\cdot,e,^{-1},u,(x^{-1}),a_3)
\;:\;x^{-1}\cdot e=x^{-1}}{}

\step*{}{a_7:=Gr_5(S,G,\cdot,e,^{-1},u,(x^{-1}),a_3)
\;:\;e\cdot x^{-1}=x^{-1}}{}

\introduce*{}{y:S\;|\;w:y\varepsilon G}{}

\step*{}{a_8:=Gr_3(S,G,\cdot,e,^{-1},u,y,w)\;:\;y^{-1}\varepsilon G}{}

\step*{}{a_9:=Gr_8(S,G,\cdot,e,^{-1},u,y,w)
\;:\;y^{-1}\cdot y=e}{}

\step*{}{a_{10}:= Gr_4(S,G,\cdot,e, ^{-1},u,y,w)
\;:\;y\cdot e=y}{}

\step*{}{a_{11}:= Gr_5(S,G,\cdot,e,^{-1} ,u,y,w)
\;:\;e\cdot y=y}{}

\step*{}{a_{12}:= Gr_9(S,G,\cdot,e,^{-1}, u,x,y,v,w)\;:\;(x\cdot y)\;\varepsilon \;G}{}

\step*{}{a_{13}:= a_2(x^{-1})a_3xvyw
\;:\;x^{-1}\cdot x\cdot y=x^{-1}\cdot (x\cdot y)}{}

\step*{}{a_{14}:=eq\mhyphen cong_2(\lambda t:S.(t\cdot y), a_5)\;:\;e\cdot y=x^{-1}\cdot x\cdot y}{}

\step*{}{a_{15}:=eq\mhyphen trans_3(eq\mhyphen trans_1(a_{14}, a_{13}),a_{11})\;:\;x^{-1}\cdot(x\cdot y)=y}{}

\introduce*{}{z:S\;|\;r:z\varepsilon G}{}

\assume*{}{p:(x\cdot y=z)}{} 
\step*{}{a_{16}:=eq\mhyphen cong_1(\lambda t:S.(x^{-1}\cdot t),p)\;:\;x^{-1}\cdot (x\cdot y)=x^{-1}\cdot z}{}

\step*{}{a_{17}:=eq\mhyphen trans_3( a_{15}, a_{16})\;:\; y=x^{-1}\cdot z}{}

\conclude*{}{\boldsymbol{term_{\ref{theorem:axiom_corollary}.1}}(S,G,\cdot,e,^{-1},u,
x,y,z,v,w,r):=\lambda p:(x\cdot y=z).a_{17}\;
\\\quad\quad\quad
:\;\boldsymbol{(x\cdot y=z\Rightarrow y=x^{-1}\cdot z)}\quad\quad \textsf{1) is proven}{}}

\assume*{}{p:(y\cdot x=z)}{} 

\step*{}{a_{18}:=eq\mhyphen cong_1(\lambda t:S.(t\cdot x^{-1}), p)\;:\;y\cdot x\cdot x^{-1}=z\cdot x^{-1}}{}

\step*{}{a_{19}:=eq\mhyphen cong_1(\lambda t:S.(y\cdot t),a_4)\;:\;y\cdot(x\cdot x^{-1} )=y\cdot e}{}

\step*{}{a_{20}:= a_2ywxvx^{-1}a_3
\;:\;y\cdot x\cdot x^{-1}=y\cdot(x\cdot x^{-1} )}{}

\step*{}{a_{21}:=eq\mhyphen trans_3(eq\mhyphen trans_1(eq\mhyphen trans_1(a_{20}, a_{19}),a_{10}),a_{18})\;:\; y=z\cdot x^{-1}}{}

\conclude*{}{\boldsymbol{term_{\ref{theorem:axiom_corollary}.2}}(S,G,\cdot,e,^{-1},u,
x,y,z,v,w,r):=\lambda p:(y\cdot x=z).a_{21}\;
\\\quad\quad\quad
:\;\boldsymbol{(y\cdot x=z\Rightarrow y=z\cdot x^{-1})}\quad\quad \textsf{2) is proven}}{}

\done

\assume*{}{p:(x\cdot y=e)}{}

\step*{}{a_{22}:=
term_{\ref{theorem:axiom_corollary}.1}(S,G,\cdot,e,^{-1},u,
x,y,e,v,w,a_1)p\;:\;y=x^{-1}\cdot e}{}

\step*{}{a_{23}:=eq\mhyphen trans_1(a_{22},a_6)\;:\;y=x^{-1}}{}

\conclude*{}{\boldsymbol{term_{\ref{theorem:axiom_corollary}.3}}(S,G,\cdot,e,^{-1},u,
x,y,v,w):=\lambda p:(x\cdot y=e).a_{23}
\;:\;\boldsymbol{(x\cdot y=e\Rightarrow y=x^{-1})}}{\textsf{3) is proven}}

\assume*{}{p:(y\cdot x=e)}{}

\step*{}{a_{24}:=
term_{\ref{theorem:axiom_corollary}.2}(S,G,\cdot,e,^{-1},u,
x,y,e,v,w,a_1)p\;:\;y=e\cdot x^{-1}}{}

\step*{}{a_{25}:=eq\mhyphen trans_1(a_{24}, a_7)\;:\;y=x^{-1}}{}

\conclude*{}{\boldsymbol{term_{\ref{theorem:axiom_corollary}.4}}(S,G,\cdot,e,^{-1},u,
x,y,v,w):=\lambda p:(y\cdot x=e).a_{25}
\;:\;\boldsymbol{(y\cdot x=e\Rightarrow y=x^{-1})}}{\textsf{4) is proven}}

\step*{}{a_{26}:=eq\mhyphen cong_1(\lambda t:S.(y^{-1}\cdot t), a_{15})\;:\;y^{-1}\cdot (x^{-1}\cdot (x\cdot y))=y^{-1}\cdot y}{}

\step*{}{a_{27}:= a_2(y^{-1})a_8(x^{-1}) a_3(x\cdot y)a_{12}
\;:\;y^{-1}\cdot x^{-1}\cdot (x\cdot y)=y^{-1}\cdot (x^{-1}\cdot (x\cdot y))}{}

\step*{}{a_{28}:=eq\mhyphen trans_1(eq\mhyphen trans_1( a_{27},a_{26}),a_9)\;:\;y^{-1}\cdot x^{-1}\cdot (x\cdot y)=e}{}

\step*{}{a_{29}:= Gr_9(S,G,\cdot,e,^{-1}, u,y^{-1},x^{-1},
a_8,a_3)\;:\;(y^{-1}\cdot x^{-1})\;\varepsilon \;G}{}

\step*{}{a_{30}:= term_{\ref{theorem:axiom_corollary}.4}(S,G,\cdot,e,^{-1},u,
(x\cdot y),(y^{-1}\cdot x^{-1}),a_{12},a_{29}) a_{28}\;:\;y^{-1}\cdot x^{-1}=(x\cdot y)^{-1}}{}

\step*{}{\boldsymbol{ term_{\ref{theorem:axiom_corollary}.5}}(S,G,\cdot,e,^{-1}, u,x,y,v,w):=
eq\mhyphen sym(a_{30})
\;:\;\boldsymbol{(x\cdot y)^{-1}=y^{-1}\cdot x^{-1}}}{\textsf{5) is proven}}

\done

\step*{}{a_{31}:=term_{\ref{theorem:axiom_corollary}.3}(S,G,\cdot,e,^{-1}, u,x^{-1},x,a_3,v)a_5\;:\;x=(x^{-1})^{-1}}{}

\step*{}{\boldsymbol{term_{\ref{theorem:axiom_corollary}.6}}(S,G,\cdot,e,^{-1}, u,x,v):=eq\mhyphen sym(a_{31})\;:\;\boldsymbol{(x^{-1})^{-1}= x}}{\textsf{6) is proven}}

\assume*{}{(p:x\cdot x=x)}{}

\step*{}{a_{32}:=term_{\ref{theorem:axiom_corollary}.1}(S,G,\cdot,e,^{-1}, u,x,x,x,v,v,v)p\;:\;x=x^{-1}\cdot x}{}

\step*{}{a_{33}:= eq\mhyphen trans_1(a_{32},a_5)\;:\;x=e}{}

\conclude*{}{\boldsymbol{term_{\ref{theorem:axiom_corollary}.7}}(S,G,\cdot,e,^{-1},u,x,v):=\lambda p:(x\cdot x=x).a_{33}
\;:\;\boldsymbol{(x\cdot x=x\Rightarrow x=e)}}{\textsf{7) is proven}}

\done

\step*{}{a_{34}:=Gr_4(S,G,\cdot,e,^{-1},u,e,a_1)\;:\;e\cdot e=e}{}

\step*{}{a_{35}:= term_{\ref{theorem:axiom_corollary}.3}(S,G,\cdot,e,^{-1}, u,e,e,a_1,a_1)a_{34}
\;:\;e=e^{-1}}{}

\step*{}{\boldsymbol{ term_{ \ref{theorem:axiom_corollary}.8}}(S,G,\cdot,e, ^{-1},u):=eq\mhyphen sym(a_{35})\;:\;e^{-1}=e}{\textsf{8) is proven}}
\end{flagderiv}

\subsection*{Proof of Proposition \ref{lemma:subgroup}}
\label{section:subgroup}

1)
\begin{flagderiv}
\introduce*{}{S: *_s\;|\;G:ps(S)\;|\;\cdot: S\rightarrow S\rightarrow S\;|\;e:S\;|\;^{-1}:S\rightarrow S\;|\;u:Group(S,G,\cdot,e,^{-1})}{}
\introduce*{}{H:ps(S)\;|\;v:H\leqslant G}{}
\step*{}{a_1:=\wedge_1(\wedge_1( \wedge_1(v)))\;:\;H\subseteq G}{}
\step*{}{a_2:=\wedge_2(\wedge_1( \wedge_1(v)))\;:\;e\varepsilon H}{}
\step*{}{a_3:=\wedge_2( \wedge_1(v))\;:\;Closure_1(S,H,^{-1})}{}
\step*{}{a_4:=\wedge_2(v)\;:\;Closure_2(S,H,\cdot)}{}
\introduce*{}{x:S\;|\;w:x\varepsilon H}{}
\step*{}{a_5(x,w):=a_1xw\;:\;x\varepsilon G}{}
\step*{}{a_6:=a_5(x,w)\;:\;x\varepsilon G}{}

\step*{}{a_7:=\wedge(Gr_4(S,G,\cdot,e, ^{-1},u,x,a_6),Gr_5(S,G,\cdot,e, ^{-1},u,x,a_6))\;:\;x\cdot e=x\wedge e\cdot x=x}{}
\step*{}{a_8:=Gr_6(S,G,\cdot,e, ^{-1},u,x,a_6)\;:\;Inverse_0(S,\cdot,e,x, x^{-1})}{}
\conclude*{}{a_9:=
\lambda x:S.\lambda w:x\;\varepsilon H.a_7
\;:\;\forall x:S.[x\varepsilon H\Rightarrow (x\cdot e=x\wedge e\cdot x=x)]}{}

\step*{}{a_{10}:=
\lambda x:S.\lambda w:x\;\varepsilon H.a_8
\;:\;Inverse(S,H,\cdot,e,^{-1})}{}

\step*{}{a_{11}:=
\wedge(a_2,a_9)
\;:\;Identity(S,H,\cdot,e)}{}

\step*{}{a_{12}:=
\wedge(a_3,a_{10})
\;:\;Inverse\mhyphen prop(S,H,\cdot,e,^{-1})}{}

\introduce*{}{x:S\;|\;w_1:x\varepsilon H\;|\;y:S\;|\;w_2:y\varepsilon H\;|\;z:S\;|\;w_3:z\varepsilon H}{}

\step*{}{a_{13}:= a_5(x,w_1)\;:\;x\varepsilon G}{}

\step*{}{a_{14}:= a_5(y,w_2)\;:\;y\varepsilon G}{}

\step*{}{a_{15}:= a_5(z,w_3)\;:\;z\varepsilon G}{}

\step*{}{a_{16}:=Gr_2(S,G,\cdot, e,^{-1},u)xa_{13}ya_{14}za_{15}
\;:\;
(x\cdot y)\cdot z=x\cdot(y\cdot z)}{}
\conclude*{}{a_{17}:=\lambda x:S.\lambda w_1:x\varepsilon H.\lambda y:S.\lambda w_2:y\varepsilon H.\lambda z:S.\lambda w_3:z\varepsilon H.a_{16}\;:\;Assoc(S,H,\cdot)}{}

\step*{}{\boldsymbol {term_{\ref{lemma:subgroup}.1}}(S,G,\cdot,e,^{-1},u,H,v):=\wedge(\wedge(\wedge(a_4, a_{17}), a_{11}),a_{12})\;:\;\boldsymbol{Group(S,H,\cdot,e,^{-1})}}{}
\end{flagderiv}

2)
\begin{flagderiv}
\introduce*{}{S: *_s\;|\;G:ps(S)\;|\;\cdot: S\rightarrow S\rightarrow S\;|\;e:S\;|\;^{-1}:S\rightarrow S\;|\;u:Group(S,G,\cdot,e,^{-1})}{}
\introduce*{}{B,C:ps(S)\;|\;v_1:B\leqslant G\;|\;v_2:C\leqslant G\;|\;w:C\subseteq B}{}

\step*{}{a_1:=\wedge_2(\wedge_1( \wedge_1(v_2)))\;:\;e\varepsilon C}{}

\step*{}{a_2:=\wedge_2( \wedge_1(v_2))\;:\;Closure_1(S,C,^{-1})}{}

\step*{}{a_3:=\wedge_2(v_2)\;:\;Closure_2(S,C,\cdot)}{}

\step*{}{\boldsymbol {term_{\ref{lemma:subgroup}.2}}(S,G,\cdot,e,^{-1},u,B,C,v_1,v_2,w):=\wedge(\wedge(\wedge(w,a_1), a_2),a_3)\;:\;\boldsymbol{C\leqslant B}}{}
\end{flagderiv}

\subsection*{Proof of Proposition \ref{lemma:triv-subgroups}}

1) 
\begin{flagderiv}
\introduce*{}{S: *_s\;|\;G:ps(S)\;|\;\cdot: S\rightarrow S\rightarrow S\;|\;e:S\;|\;^{-1}:S\rightarrow S\;|\;u:Group(S,G,\cdot,e,^{-1})}{}
\step*{}{a_1:=\wedge_1(\wedge_1( \wedge_1(u)))\;:\;Closure_2(S,G,\cdot)}{}

\step*{}{a_2:=\wedge_1(\wedge_2(u))\;:\;Closure_1(S,G,^{-1})}{}

\step*{}{a_3:=Gr_1(S,G,\cdot,e,^{-1},u)\;:\;e\varepsilon G}{}
\step*{}{a_4:=\lambda x:S.\lambda v:x\varepsilon G.v\;:\;G\subseteq G}{}

\step*{}{\boldsymbol{term_{\ref{lemma:triv-subgroups}.1}}(S,G,\cdot, e, ^{-1},u):=\wedge(\wedge(\wedge(a_4, a_3),a_2),a_1)\;:\;\boldsymbol{G\leqslant G}}{}
\end{flagderiv}

2)\begin{flagderiv}
\introduce*{}{S: *_s\;|\;G:ps(S)\;|\;\cdot: S\rightarrow S\rightarrow S\;|\;e:S\;|\;^{-1}:S\rightarrow S\;|\;u:Group(S,G,\cdot,e,^{-1})}{}
\step*{}{\text{Notation }H:=\{x:S\;|\;x=e\}
\;:\;ps(S)}{}
\step*{}{a_1:=eq\mhyphen refl\;:\;e=e}{}
\step*{}{a_1\;:\;e\varepsilon H}{}
\step*{}{a_2:=Gr_1(S,G,\cdot,e,^{-1},u)\;:\;e\varepsilon G}{}
\step*{}{a_3:=Gr_4(S,,G,\cdot,e,^{-1},u,e,a_2)\;:\;e\cdot e=e}{}
\step*{}{a_4:=term_{
\ref{theorem:axiom_corollary}.8}(S,,G,\cdot,e, ^{-1},u)\;:\;e^{-1}=e}{}

\introduce*{}{x:S\;|\;v:x\varepsilon H}{}
\step*{}{v:x=e}{}

\step*{}{a_5:=eq\mhyphen subs_2(\lambda z:S.z\varepsilon G,v,a_2)\;:\;x\varepsilon G}{}
\step*{}{a_6:=eq\mhyphen cong_1(^{-1},v)\;:\;x^{-1}=e^{-1}}{}
\step*{}{a_7:=eq\mhyphen trans_1(a_6,a_4)\;:\;x^{-1}=e}{}
\step*{}{a_7\;:\;x^{-1}\varepsilon H}{}
\conclude*{}{a_8:=\lambda x:S.\lambda v:x\varepsilon H.a_7\;:\;Closure_1(S,H,^{-1})}{}
\step*{}{a_9:=\lambda x:S.\lambda v:x\varepsilon H.a_5\;:\;H\subseteq G}{}

\introduce*{}{x:S\;|\; v:x\varepsilon H\;|\;y:S\;|\;w:y\varepsilon H}{}

\step*{}{v:x=e}{}
\step*{}{w:y=e}{}

\step*{}{a_{10}:=eq\mhyphen cong_1(\lambda z:S.(x\cdot z),w)\;:\;x\cdot y=x\cdot e}{}
\step*{}{a_{11}:=eq\mhyphen cong_1(\lambda z:S.(z\cdot e), v)\;:\;x\cdot e=e\cdot e}{}
\step*{}{a_{12}:=eq\mhyphen trans_1(eq\mhyphen trans_1(a_{10},a_{11}),a_3)\;:\;x\cdot y=e}{}
\step*{}{a_{12}\;:\;(x\cdot y)\;\varepsilon \;H}{}

\conclude*{}{a_{13}:=\lambda 
x:S.\lambda v:x\varepsilon H.\lambda y:S.\lambda w:y\varepsilon H.a_{12}
\;:\;Closure_2(S,H,\cdot)}{}
\step*{}{\boldsymbol{term_{\ref{lemma:triv-subgroups}.2}}(S,G, \cdot,e,^{-1},u):=\wedge(\wedge(\wedge(a_9,a_1), a_8),a_{13})\;:\;\boldsymbol{H\leqslant G}}{}
\end{flagderiv}

3) 
\begin{flagderiv}
\introduce*{}{S: *_s\;|\;G:ps(S)\;|\;\cdot: S\rightarrow S\rightarrow S\;|\;e:S\;|\;^{-1}:S\rightarrow S\;|\;u:Group(S,G,\cdot,e,^{-1})}{}

\introduce*{}{B,C:ps(S)\;|\;v:B\leqslant G\;|\;w:C\leqslant G}{}

\step*{}{a_1:=\wedge_1(\wedge_1( \wedge_1(v)))\;:\;B\subseteq G}{}
\step*{}{a_2:=\wedge_2(\wedge_1( \wedge_1(v)))\;:\;e\varepsilon B}{}

\step*{}{a_3:=\wedge_2( \wedge_1(v))\;:\; Closure_1(S,B,^{-1})}{}

\step*{}{a_4:=\wedge_2(v)\;:\;Closure_2(S,B,\cdot)}{}

\step*{}{a_5:=\wedge_2(\wedge_1( \wedge_1(w)))\;:\;e\varepsilon C}{}

\step*{}{a_6:=\wedge_2( \wedge_1(w))\;:\; Closure_1(S,C,^{-1})}{}

\step*{}{a_7:=\wedge_2(w)\;:\;Closure_2(S,C,\cdot)}{}

\step*{}{a_8:=\wedge( a_2,a_5)\;:\;e\;\varepsilon\; (B\cap C) }{}

\introduce*{}{x:S\;|\;r:x\;\varepsilon\;(B\cap C)}{}

\step*{}{a_9:=\wedge_1(r)\;:\;x\varepsilon B}{}

\step*{}{a_{10}:=\wedge_2(r)\;:\;x\varepsilon C}{}

\step*{}{a_{11}:=
a_1xa_9\;:\; x\varepsilon G}{}

\step*{}{a_{12}:=
a_3xa_9\;:\; x^{-1}\varepsilon B}{}

\step*{}{a_{13}:=
a_6xa_{10}\;:\; x^{-1}\varepsilon C}{}

\step*{}{a_{14}:=
\wedge (a_{12},a_{13})\;:\; x^{-1}\;\varepsilon\;(B\cap C)}{}

\conclude*{}{a_{15}:=\lambda x:S.\lambda r:(x\;\varepsilon\;(B\cap C)).a_{11}\;:\;B\cap C\subseteq G}{}

\step*{}{a_{16}:=\lambda x:S.\lambda r:(x\;\varepsilon\;(B\cap C)).a_{14}\;:\;Closure_1(S,(B\cap C),\;^{-1})}{}

\introduce*{}{x:S\;|\;r_1:x\;\varepsilon\;(B\cap C)\;|\;y:S\;|\;r_2:y\;\varepsilon\;(B\cap C)}{}

\step*{}{a_{17}:=\wedge_1(r_1)\;:\;x\varepsilon B}{}

\step*{}{a_{18}:=\wedge_2(r_1)\;:\;x\varepsilon C}{}

\step*{}{a_{19}:=\wedge_1(r_2)\;:\;y\varepsilon B}{}

\step*{}{a_{20}:=\wedge_2(r_2)\;:\;y\varepsilon C}{}

\step*{}{a_{21}:=
a_4xa_{17}ya_{19}\;:\; (x\cdot y)\;\varepsilon\;B}{}

\step*{}{a_{22}:=
a_7xa_{18}ya_{20}\;:\; (x\cdot y)\;\varepsilon\;C}{}

\step*{}{a_{23}:=
\wedge (a_{21},a_{22})\;:\; (x\cdot y)\;\varepsilon\;(B\cap C)}{}

\conclude*{}{a_{24}:=\lambda x:S.\lambda r_1:(x\;\varepsilon\;(B\cap C)).\lambda y:S.\lambda r_2:(y\;\varepsilon\;(B\cap C)).a_{23}\;:\;Closure_2(S,(B\cap C),\cdot)}{}

\step*{}{\boldsymbol{term_{\ref{lemma:triv-subgroups}.3}}(S,G,\cdot, e, ^{-1},u,B,C,v,w):=\wedge(\wedge(\wedge(a_{15}, a_8),a_{16}),a_{24})\;:\;\boldsymbol{B\cap C\leqslant G}}{}
\end{flagderiv}

4) 
\begin{flagderiv}
\introduce*{}{S: *_s\;|\;G:ps(S)\;|\;\cdot: S\rightarrow S\rightarrow S\;|\;e:S\;|\;^{-1}:S\rightarrow S\;|\;u:Group(S,G,\cdot,e,^{-1})}{}

\introduce*{}{U:ps(ps(S))\;|\;v:(\exists X:ps(S).X\varepsilon U)\;|\;w:[\forall X:ps(S).(X\varepsilon U\Rightarrow X\leqslant G)]}{}
\introduce*{}{X:ps(S)\;|\;r:X\varepsilon U}{}
\step*{}{a_1:=wXr:X\leqslant G}{}

\step*{}{a_2(X,r):=\wedge_1(\wedge_1( \wedge_1(a_1)))\;:\;X\subseteq G}{}

\step*{}{a_3(X,r):=\wedge_2(\wedge_1( \wedge_1(a_1)))\;:\;e\varepsilon X}{}

\step*{}{a_4(X,r):=\wedge_2( \wedge_1(a_1))\;:\;Closure_1(S,X,^{-1})}{}

\step*{}{a_5(X,r):=\wedge_2(a_1)\;:\;Closure_2(S,X,\cdot)}{}
\done

\introduce*{}{X:ps(S)\;|\;r:X\varepsilon U}{}

\step*{}{a_6:=a_3(X,r):e\varepsilon X}{}

\conclude*{}{a_7:=\lambda X:ps(S).\lambda r:X\varepsilon U.a_6\;:\; e\varepsilon (\cap U)}{}

\introduce*{}{x:S\;|\;r_1:x\varepsilon (\cap U)}{}
\introduce*{}{X:ps(S)\;|\;r_2:X\varepsilon U}{}

\step*{}{a_8:=r_1Xr_2:x\varepsilon X}{}

\step*{}{a_9:= a_2(X,r_2)xa_8\;:\;x\varepsilon G}{}

\step*{}{a_{10}:= a_4(X,r_2)xa_8\;:\;x^{-1}\varepsilon X}{}

\conclude*{}{a_{11}:=\exists_3(v,a_9)\;:\; x\varepsilon G}{}

\step*{}{a_{12}:=\lambda X:ps(S).\lambda r_2:X\varepsilon U.a_{10}\;:\; x^{-1}\varepsilon (\cap U)}{}

\conclude*{}{a_{13}:=
\lambda x:S.\lambda r_1:x\varepsilon (\cap U).a_{11}
\;:\;(\cap U) \subseteq G}{}

\step*{}{a_{14}:=\lambda x:S.\lambda r_1:x\varepsilon (\cap U).a_{12}\;:\; Closure_1(S,\cap U,\;^{-1})}{}

\introduce*{}{x:S\;|\; r_1:x\varepsilon (\cap U)\;|\;y:S\;|\; r_2:y\varepsilon (\cap U)}{}
\introduce*{}{X:ps(S)\;|\;r_3:X\varepsilon U}{}

\step*{}{a_{15}:= r_1Xr_3:x\;\varepsilon \;X}{}

\step*{}{a_{16}:= r_2Xr_3:y\;\varepsilon \;X}{}

\step*{}{a_{17}:=a_5(X,r_3)
xa_{15}ya_{16}:(x\cdot y)\;\varepsilon \;X}{}

\conclude*{}{a_{18}:=\lambda X:ps(S).\lambda r_3:X\varepsilon U.a_{17}\;:\; (x\cdot y)\;\varepsilon \;(\cap U)}{}
\conclude*{}{a_{19}:=\lambda x:S.\lambda r_1:x\varepsilon (\cap U).\lambda y:S.\lambda r_2:y\varepsilon (\cap U).a_{18}\;:\; Closure_2(S,\cap U,\cdot)}{}

\step*{}{\boldsymbol{term_{\ref{lemma:triv-subgroups}.4}}(S,G, \cdot,e,^{-1},u,U,v,w):=\wedge(\wedge( \wedge(a_{13},a_7), a_{14}),a_{19})\;:\;\boldsymbol{\cap U\leqslant G}}{}
\end{flagderiv}

\subsection*{Proof of Proposition \ref{lemma:equiv}}
\label{section:equiv}

\begin{flagderiv}
\introduce*{}{S: *_s\;|\;G:ps(S)\;|\;\cdot: S\rightarrow S\rightarrow S\;|\;e:S\;|\;^{-1}:S\rightarrow S\;|\;u:Group(S,G,\cdot,e,^{-1})}{}

\step*{}{a_1:=Gr_2(S,G, \cdot,e,^{-1},u)\;:\;Assoc(S,G,\cdot)}{}

\introduce*{}{H:ps(S)\;|\;v:H\leqslant G}{}
\step*{}{a_2:=\wedge_1(\wedge_1( \wedge_1(v)))\;:\;H\subseteq G}{}
\step*{}{a_3:=\wedge_2(\wedge_1( \wedge_1(v)))\;:\;e\varepsilon H}{}

\step*{}{a_4:=\wedge_2( \wedge_1(v))\;:\; Closure_1(S,H,^{-1})}{}

\step*{}{a_5:=\wedge_2(v)\;:\;Closure_2(S,H,\cdot)}{}

\introduce*{}{x:S\;|\;w:x\varepsilon G}{}

\step*{}{a_6:=Gr_8(S,G,\cdot, e,^{-1},u,x,w)\;:\;x^{-1}\cdot x=e}{}

\step*{}{a_7:=eq\mhyphen subs_2(\lambda z:S.z\varepsilon H,a_6,a_3)\;:\;(x^{-1}\cdot x)\;\varepsilon H}{}
\step*{}{a_7:R_Hxx}{}
\conclude*{}{a_8:=\lambda x:S.\lambda w:x\varepsilon G.a_7\;:\;refl(S,G,R_H)}{}

\introduce*{}{x:S\;|\; v_1:x\varepsilon G\;|\;y:S\;|\; v_2:y\varepsilon G\;|\;w:R_Hxy}{}
\step*{}{w:(x^{-1}\cdot y)\;\varepsilon H}{}
\step*{}{a_9(x,v_1):=Gr_3(S,G,\cdot,e,^{-1},u,x,v_1)\;:\;x^{-1}\varepsilon G}{}
\step*{}{a_{10}:= term_{\ref {theorem:axiom_corollary}.5}(S,\cdot,e,^{-1},u,x^{-1},y,
a_9(x,v_1),v_2)\;:\; (x^{-1}\cdot y)^{-1}=y^{-1}\cdot(x^{-1})^{-1}}{}
\step*{}{a_{11}:= term_{\ref {theorem:axiom_corollary}.6}(S,\cdot,e,^{-1},u,x,v_1)\;:\; (x^{-1})^{-1}=x}{}

\step*{}{a_{12}:=eq\mhyphen subs_1(\lambda z:S.((x^{-1}\cdot y)^{-1}=y^{-1}\cdot z),a_{11},a_{10})
\;:\;(x^{-1}\cdot y)^{-1}=y^{-1}\cdot x}{}

\step*{}{a_{13}:=a_4(x^{-1}\cdot y)w\;:\;(x^{-1}\cdot y)^{-1}\;\varepsilon\; H}{}

\step*{}{a_{14}:=eq\mhyphen subs_1(\lambda z:S.z\varepsilon H,a_{12},a_{13})
\;:\;(y^{-1}\cdot x)\;\varepsilon H}{}
\step*{}{a_{14}:R_Hyx}{}
\conclude*{}{a_{15}:=\lambda x:S.\lambda v_1:x\varepsilon G.\lambda y:S.\lambda v_2:y\varepsilon G.\lambda w:R_Hxy.a_{14}\;:\;sym(S,G,R_H)}{}

\introduce*{}{x:S\;|\;
v_1:x\varepsilon G\;|\;
y:S\;|\;v_2: y\varepsilon G\;|\;z:S\;|\;v_3:z\varepsilon G\;|\;w_1:R_Hxy\;|\;w_2:R_Hyz}{}
\step*{}{w_1:(x^{-1}\cdot y)\;\varepsilon H}{}
\step*{}{w_2:(y^{-1}\cdot z)\;\varepsilon H}{}

\step*{}{a_{16}:=a_2(y^{-1}\cdot z)w_2\;:\;(y^{-1}\cdot z)\;\varepsilon\;G}{}

\step*{}{a_{17}:=a_5(x^{-1}\cdot y)w_1(y^{-1}\cdot z)w_2\;:\;((x^{-1}\cdot y)\cdot (y^{-1}\cdot z))\;\varepsilon H}{}
\step*{}{a_{18}:=a_1x^{-1}a_9(x,v_1)yv_2(y^{-1}\cdot z)a_{16}\;:\;(x^{-1}\cdot  y)\cdot (y^{-1}\cdot z)=x^{-1}\cdot (y\cdot (y^{-1}\cdot z))}{}

\step*{}{a_{19}:=a_1yv_2y^{-1}a_9(y,v_2)zv_3\;:\;(y\cdot y^{-1})\cdot z=y\cdot (y^{-1}\cdot z)}{}

\step*{}{a_{20}:=Gr_7(S,G, \cdot,e,^{-1},u,y,v_2)\;:\;y\cdot y^{-1}=e}{}
\step*{}{a_{21}:=Gr_5(S,G, \cdot,e,^{-1},u,z,v_3)\;:\;e\cdot z=z}{}
\step*{}{a_{22}:=eq\mhyphen cong_1(\lambda t:S.(t\cdot z), a_{20})
\;:\;(y\cdot y^{-1})\cdot z=e\cdot z}{}

\step*{}{a_{23}:=eq\mhyphen trans_1(eq\mhyphen trans_3(a_{19},a_{22}),a_{21})\;:\;y\cdot (y^{-1}\cdot z)=z}{}
\step*{}{a_{24}:=eq\mhyphen subs_1(\lambda t:S.((x^{-1}\cdot y)\cdot(y^{-1}\cdot z)=x^{-1}\cdot t), a_{23},a_{18})
\;:\;(x^{-1}\cdot y)\cdot(y^{-1}\cdot z)=x^{-1}\cdot z}{}
\step*{}{a_{25}:=eq\mhyphen subs_1(\lambda t:S.t\varepsilon H, a_{24},a_{17})
\;:\;(x^{-1}\cdot z)\;\varepsilon H}{}
\step*{}{a_{25}:R_Hxz}{}
\conclude*{}{a_{26}:=\lambda x:S.\lambda v_1:x\varepsilon G.\lambda y:S.\lambda v_2:y\varepsilon G.\lambda z:S.\lambda v_3:z\varepsilon G.\lambda w_1:R_Hxy.\lambda w_2:R_Hyz.a_{25}\;:\;trans(S,G,R_H)}{}
\step*{}{\boldsymbol{term_{\ref{lemma:equiv}}}(S,G,\cdot,e,^{-1},u, H, v):=\wedge(\wedge (a_8,a_{15}),a_{26})\;:\;\boldsymbol{equiv\mhyphen rel(S,G,R_H)}}{}
\end{flagderiv}

\subsection*{Proof of Proposition \ref{lemma:equiv-1}}
\begin{flagderiv}
\introduce*{}{S: *_s\;|\;G:ps(S)\;|\;\cdot: S\rightarrow S\rightarrow S\;|\;e:S\;|\;^{-1}:S\rightarrow S\;|\;u:Group(S,G,\cdot,e,^{-1})}{}
\introduce*{}{H:ps(S)\;|\;v:H\leqslant G}{}

\step*{}{\text{Notation }R:=R_H\;:\; S\rightarrow S\rightarrow *_p}{}

\step*{}{a_1:= term_{\ref{lemma:equiv}}(S,G,\cdot,e,^{-1},u, H,v)\;:\; equiv\mhyphen rel(S,G,R)}{}

\step*{}{a_2:=\wedge_1(\wedge_1(a_1))\;:\;refl(S,G,R)}{}
\step*{}{a_3:=\wedge_2(\wedge_1(a_1))\;:\;sym(S,G,R)}{}
\step*{}{a_4:=\wedge_2(a_1)\;:\;trans(S,G,R)}{}
\introduce*{}{x,y:S\;|\; w_1:x\varepsilon G\;|\; w_2:y\varepsilon G}{}
\assume*{}{r_1:Rxy}{}
\step*{}{a_5:=a_3xw_1yw_2r_1\;:\;Ryx}{}
\introduce*{}{z:S\;|\; r_2:z\;\varepsilon\;(Rx\cap G)}{}
\step*{}{a_6:=\wedge_1(r_2)\;:\;z\varepsilon Rx}{}

\step*{}{a_7:=\wedge_2(r_2)\;:\;z\varepsilon G}{}
\step*{}{a_6:Rxz}{}

\step*{}{a_8:=a_4yw_2xw_1za_7a_5a_6\;:\;Ryz}{}
\step*{}{a_8:z\varepsilon Ry}{}

\step*{}{a_9:=\wedge( a_8,a_7)\;:\;z\;\varepsilon \;(Ry\cap G)}{}

\conclude*{}{a_{10}(x,y,w_1, w_2, r_1):=\lambda z:S.\lambda r_2:z\varepsilon (Rx\cap G).a_9\;:\;Rx\cap G\subseteq Ry\cap G}{}

\step*{}{a_{11}:=\wedge(a_{10}(x,y,w_1,w_2, r_1),a_{10}(y,x,w_2,w_1,a_5))\;:\;Rx\cap G=Ry\cap G}{}
\conclude*{}{a_{12}:=\lambda r_1:Rxy.a_{11}\;:\;(Rxy\Rightarrow (Rx\cap G=Ry\cap G))}{}
\assume*{}{r:Rx\cap G=Ry\cap G}{}

\step*{}{a_{13}:=a_2yw_2\;:\;Ryy}{}
\step*{}{a_{13}\;:\;y\varepsilon Ry}{}

\step*{}{a_{14}:=\wedge(a_{13}, w_2)\;:\;y\varepsilon (Ry\cap G)}{}

\step*{}{a_{15}:=\wedge_2(r)y
a_{14}\;:\;y\varepsilon (Rx\cap G)}{}
\step*{}{a_{16}:=\wedge_1(a_{15})\;:\;y\varepsilon Rx}{}
\step*{}{a_{16}\;:\;Rxy}{}

\conclude*{}{a_{17}:=\lambda r:(Rx\cap G=Ry\cap G).a_{16}\;:\;(Rx\cap G=Ry\cap G)\Rightarrow Rxy}{}

\step*{}{\boldsymbol {term_{\ref{lemma:equiv-1}}}(S,G,\cdot,e,^{-1},u, H, v,x,y,w_1,w_2):=\wedge (a_{12},a_{17})\;:\;\boldsymbol{R_Hxy\Leftrightarrow (R_Hx\cap G=R_Hy\cap G)}}{}
\end{flagderiv}

\subsection*{Proof of Lemma \ref{lemma:set-mult}}
\begin{flagderiv}
\introduce*{}{S: *_s\;|\;\cdot: S\rightarrow S\rightarrow S\;|\;B,C:ps(S)}{}

\introduce*{}{x:S\;|\;u:x\varepsilon Mt_1(S,\cdot,B,C)}{}
\step*{}{u:[\exists c:S.(c\varepsilon C\wedge x\;\varepsilon \;(B\cdot c))]}{}
\introduce*{}{c:S\;|\;v:c\varepsilon C\wedge x\;\varepsilon \;(B\cdot c)}{}
\step*{}{a_1:=\wedge_1(v)\;:\;c\varepsilon C}{}
\step*{}{a_2:=\wedge_2(v)\;:\;x\;\varepsilon \;(B\cdot c)}{}
\step*{}{a_2:(\exists b:S.(b\varepsilon B\wedge x=b\cdot c))}{}

\introduce*{}{b:S\;|\;w:b\varepsilon B\wedge x=b\cdot c}{}
\step*{}{a_3:=\wedge_1(w)\;:\;b\varepsilon B}{}
\step*{}{a_4:=\wedge_2(w)\;:\;x=b\cdot c}{}
\step*{}{a_5:=\wedge(a_1,a_4)\;:\;c\varepsilon C\wedge x=b\cdot c}{}
\step*{}{a_6:=
\exists_1(\lambda t:S.(t\varepsilon C\wedge x=b\cdot t), c,a_5)
\;:\;x\;\varepsilon\;(b\cdot C)}{}

\step*{}{a_7:=\wedge(a_3,a_6)\;:\;b\varepsilon B\wedge x\;\varepsilon\;(b\cdot C)}{}
\step*{}{a_8:=
\exists_1(\lambda t:S.(t\varepsilon B\wedge x\;\varepsilon\;(t\cdot C)),b,a_7)
\;:\;x\;\varepsilon\; Mt_2(S,\cdot,B,C)}{}
\conclude*{}{a_9:=\exists_3(a_2,a_8)\;:\;x\;\varepsilon\; Mt_2(S,\cdot,B,C)}{}
\conclude*{}{a_{10}:=\exists_3(u,a_9)\;:\;x\;\varepsilon\; Mt_2(S,\cdot,B,C)}{}
\conclude*{}{a_{11}:=
\lambda x:S.\lambda u:x\;\varepsilon\; Mt_1(S,\cdot,B,C).a_{10}
\;:\; Mt_1(S,\cdot,B,C)\subseteq Mt_2(S,\cdot,B,C)}{}

\introduce*{}{x:S\;|\;u:x\varepsilon Mt_2(S,\cdot,B,C)}{}
\step*{}{u:[\exists b:S.(b\varepsilon B\wedge x\;\varepsilon \;(b\cdot C))]}{}
\introduce*{}{b:S\;|\;v:b\varepsilon B\wedge x\;\varepsilon \;(b\cdot C)}{}
\step*{}{a_{12}:=\wedge_1(v)\;:\;b\varepsilon B}{}
\step*{}{a_{13}:=\wedge_2(v)\;:\;x\;\varepsilon \;(b\cdot C)}{}
\step*{}{a_{13}:(\exists c:S.(c\varepsilon C\wedge x=b\cdot c))}{}

\introduce*{}{c:S\;|\;w:c\varepsilon C\wedge x=b\cdot c}{}
\step*{}{a_{14}:=\wedge_1(w)\;:\;c\varepsilon C}{}
\step*{}{a_{15}:=\wedge_2(w)\;:\;x=b\cdot c}{}
\step*{}{a_{16}:=\wedge(a_{12},a_{15})\;:\;b\varepsilon B\wedge x=b\cdot c}{}
\step*{}{a_{17}:=
\exists_1(\lambda t:S.(t\varepsilon B\wedge x=t\cdot c),b,a_{16})
\;:\;x\;\varepsilon\;(B\cdot c)}{}

\step*{}{a_{18}:=\wedge(a_{14},a_{17})\;:\;c\varepsilon C\wedge x\;\varepsilon\;(B\cdot c)}{}
\step*{}{a_{19}:=
\exists_1(\lambda t:S.(t\varepsilon C\wedge x\;\varepsilon\;(B\cdot t)),c,a_{18})
\;:\;x\;\varepsilon\; Mt_1(S,\cdot,B,C)}{}
\conclude*{}{a_{20}:=\exists_3(a_{13},a_{19})\;:\;x\;\varepsilon\; Mt_1(S,\cdot,B,C)}{}
\conclude*{}{a_{21}:=\exists_3(u,a_{20})\;:\;x\;\varepsilon\; Mt_1(S,\cdot,B,C)}{}
\conclude*{}{a_{22}:=
\lambda x:S.\lambda u:x\;\varepsilon\; Mt_2(S,\cdot,B,C).a_{21}
\;:\; Mt_2(S,\cdot,B,C)\subseteq Mt_1(S,\cdot,B,C)}{}

\step*{}{\boldsymbol {term_{\ref{lemma:set-mult}}}(S,\cdot, B,C) :=\wedge(a_{11}, a_{22}) 
\;:\;\boldsymbol{Mt_1(S,\cdot,B,C)= Mt_2(S,\cdot,B,C)}}{}
\end{flagderiv}

\subsection*{Proof of Lemma \ref{lemma:mult-triv}}

\begin{flagderiv}
\introduce*{}{S: *_s}{}

\introduce*{}{^{-1}:S\rightarrow S\;|\;B:ps(S)\;|\;b:S\;|\;u:b\varepsilon B}{}

\step*{}{a_1:=eq\mhyphen refl\;:\;b^{-1}=b^{-1}}{}

\step*{}{a_2:=\wedge(u,a_1)
\;:\;b\varepsilon B\wedge b^{-1}=b^{-1}}{}

\step*{}{\boldsymbol {term_{\ref{lemma:mult-triv}.1}}(S,\;^{-1},B,b,u):=\exists_1(\lambda t:S.(t\varepsilon B\wedge b^{-1}=t^{-1}),b,a_2)
\;:\;\boldsymbol{b^{-1}\varepsilon B^{-1}}}{\textsf{1) is proven}}

\done

\introduce*{}{\cdot: S\rightarrow S\rightarrow S\;|\; B:ps(S)\;|\;g,b:S\;|\;u:b\varepsilon B}{}

\step*{}{a_3:=eq\mhyphen refl\;:\;g\cdot b=g\cdot b}{}

\step*{}{a_4:=eq\mhyphen refl\;:\;b\cdot g=b\cdot g}{}

\step*{}{a_5:=\wedge(u,a_3)
\;:\;b\varepsilon B\wedge g\cdot b=g\cdot b}{}

\step*{}{a_6:=\wedge(u,a_4)
\;:\;b\varepsilon B\wedge b\cdot g=b\cdot g}{}

\step*{}{\boldsymbol {term_{\ref{lemma:mult-triv}.2}}(S,\cdot,B,g,b,u):=\exists_1(\lambda t:S.(t\varepsilon B\wedge g\cdot b=g\cdot t),b,a_5)
\;:\;\boldsymbol{(g\cdot b)\;\varepsilon \;(g\cdot B)}}{\textsf{2) is proven}}

\step*{}{\boldsymbol {term_{\ref{lemma:mult-triv}.3}}(S,\cdot,B,g,b,u):=\exists_1(\lambda t:S.(t\varepsilon B\wedge b\cdot g=t\cdot g), b,a_6)
\;:\;\boldsymbol{(b\cdot g)\;\varepsilon \;(B\cdot g)}}{\textsf{3) is proven}}

\done

\introduce*{}{\cdot: S\rightarrow S\rightarrow S\;|\; B,C:ps(S)\;|\;b,c:S\;|\;u:b\varepsilon B\;|\;v:c\varepsilon C}{}

\step*{}{a_7:=term_{ \ref{lemma:mult-triv}.3}(S,\cdot,B,c,b,u)
\;:\;( b\cdot c)\;\varepsilon \;(B\cdot c)}{}

\step*{}{a_8:=\wedge(v,a_7)
\;:\;c\varepsilon C\wedge  (b\cdot c)\;\varepsilon \;(B\cdot c)}{}

\step*{}{\boldsymbol {term_{\ref{lemma:mult-triv}.4}}(S,\cdot,B,C,b,c,u,v):=\exists_1(\lambda t:S.(t\varepsilon C\wedge (b\cdot c)\;\varepsilon \;(B\cdot t)),c, a_8)
\;:\;\boldsymbol{(b\cdot c)\;\varepsilon \;(B\cdot C)}}{\textsf{4) is proven}}
\end{flagderiv}

\subsection*{Proof of Proposition \ref{lemma:eq-class}}

\begin{flagderiv}
\introduce*{}{S: *_s\;|\;G:ps(S)\;|\;\cdot: S\rightarrow S\rightarrow S\;|\;e:S\;|\;^{-1}:S\rightarrow S\;|\;u:Group(S,G,\cdot,e,^{-1})}{}
\step*{}{a_1:=Gr_2(S,G,\cdot, e,^{-1},u)\;:\;Assoc(S,G,\cdot)}{}

\introduce*{}{H:ps(S)\;|\;v:H\leqslant G\;|\;x:S\;|\;w:x\varepsilon G}{}
\step*{}{\text{Notation } R:=R_H\;:\;S\rightarrow S\rightarrow *_p}{}
\step*{}{a_2:=\wedge_1
\wedge_1(\wedge_1(v)))\;:\;H\subseteq G}{}

\step*{}{a_3:=Gr_3(S,G,\cdot, e,^{-1},u,x,w)\;:\;x^{-1}\varepsilon G}{}

\step*{}{a_4:=Gr_7(S,G,\cdot, e,^{-1},u,x,w)\;:\;x\cdot x^{-1}=e}{}
\introduce*{}{z:S\;|\;r_1:z\;\varepsilon\; (x\cdot H)}{}

\step*{}{r_1:(\exists h:S.(h\varepsilon H\wedge z=x\cdot h))}{}

\introduce*{}{h:S\;|\;r_2:h \varepsilon H\wedge z=x\cdot h}{}
\step*{}{a_5:=\wedge_1(r_2)\;:\;h\varepsilon H}{}
\step*{}{a_6:=\wedge_2(r_2)\;:\;z=x\cdot h}{}

\step*{}{a_7:=eq\mhyphen sym(a_6)\;:\;x\cdot h=z}{}

\step*{}{a_8:=a_2ha_5\;:\;h\varepsilon G}{}

\step*{}{a_9:=Gr_9(S,G,\cdot, e,^{-1},u,x,h, w,a_8)\;:\;(x\cdot h)\;\varepsilon \;G}{}

\step*{}{a_{10}:=eq\mhyphen subs_1(\lambda t:S.t\varepsilon G,a_7,a_9)
\;:\;z\varepsilon G}{}
\step*{}{a_{11}:=
term_{\ref{theorem:axiom_corollary}.1}(S,G,\cdot,e,^{-1},u, x,h,z,w,a_8,a_{10})a_7
\;:\;h=x^{-1}\cdot z}{}
\step*{}{a_{12}:=eq\mhyphen subs_1(\lambda t:S.t\varepsilon H,a_{11},a_5)
\;:\;(x^{-1}\cdot z)\;\varepsilon \;H}{}
\step*{}{a_{12}:Rxz}{}
\step*{}{a_{12}:z\varepsilon Rx}{}
\step*{}{a_{13}:=\wedge(a_{12},a_{10})\;:\;
z\;\varepsilon \;(Rx\cap G)}{}

\conclude*{}{a_{14}:=\exists_3(r_1,a_{13})\;:\;
z\;\varepsilon \;(Rx\cap G)}{}
\conclude*{}{a_{15}:=
\lambda z:S.\lambda
r_1:(z\;\varepsilon\; (x\cdot H)).a_{14}\;:\;(x\cdot H)\subseteq Rx\cap G}{}

\introduce*{}{z:S\;|\;r:z\;\varepsilon\; (Rx\cap G)}{}

\step*{}{a_{16}:= 
\wedge_1(r)\;:\;
z\varepsilon Rx}{}

\step*{}{a_{17}:= 
\wedge_1(r)\;:\;
z\varepsilon G}{}

\step*{}{a_{16}:Rxz}{}
\step*{}{a_{16}\;:\;
(x^{-1}\cdot z)\;\varepsilon \;H}{}

\step*{}{a_{18}:= a_1xw(x^{-1})a_3 za_{17}\;:\;
(x\cdot x^{-1})\cdot z=x\cdot (x^{-1}\cdot z)}{}

\step*{}{a_{19}:=eq\mhyphen cong_1(\lambda t:S.(t\cdot z), a_4)\;:\;(x\cdot x^{-1})\cdot z=e\cdot z}{}

\step*{}{a_{20}:=Gr_5(S,G,\cdot, e,^{-1},u,z,a_{17})\;:\;e\cdot z=z}{}

\step*{}{a_{21}:=eq\mhyphen trans_1(eq\mhyphen trans_3(a_{18},a_{19}),a_{20})
\;:\;x\cdot (x^{-1}\cdot z)=z}{}

\step*{}{a_{22}:= term_{\ref{lemma:mult-triv}.2}(S,\cdot,H,x,(x^{-1}\cdot z),a_{16})\;:\;(x\cdot (x^{-1}\cdot z))\;\varepsilon\;(x\cdot H)}{}

\step*{}{a_{23}:=eq\mhyphen subs_1(\lambda t:S.(t\;\varepsilon\;(x\cdot H)), a_{21},a_{22})\;:\;z\;\varepsilon\;(x\cdot H)}{}

\conclude*{}{a_{24}:=\lambda z:S.\lambda r:z\varepsilon (Rx\cap G).a_{23}\;:\;Rx\cap G\subseteq (x\cdot  H)}{}

\step*{}{\boldsymbol {term_{\ref{lemma:eq-class}.1}}(S,G,\cdot,e,^{-1},u, H, v,x,w):=\wedge(a_{15},a_{24})
\;:\;\boldsymbol{x\cdot H=R_Hx\cap G}}{\textsf{1) is proven}}

\introduce*{}{y:S\;|\;r:y\varepsilon G}{}

\step*{}{a_{25}:= term_{\ref{lemma:equiv-1}}(S,G,\cdot,e,^{-1},u, H, v,x,y,w,r)\;:\;Rxy\Leftrightarrow (Rx\cap G=Ry\cap G)}{}

\step*{}{a_{26}:= term_{\ref{lemma:eq-class}.1}(S,G,\cdot,e,^{-1},u, H, v,x,w)\;:\;x\cdot H=Rx\cap G}{}

\step*{}{a_{27}:= term_{\ref{lemma:eq-class}.1}(S,G,\cdot,e,^{-1},u, H, v,y,r)\;:\;y\cdot H=Ry\cap G}{}

\step*{}{a_{28}:=eq\mhyphen subs_2(\lambda Z:ps(S).\;(Rxy\Leftrightarrow Z=Ry\cap G),a_{26},a_{25})\;:\;Rxy\;\Leftrightarrow\; x\cdot H=Ry\cap G}{}

\step*{}{a_{29}:=eq\mhyphen subs_2(\lambda Z:ps(S).\;(Rxy\Leftrightarrow x\cdot H=Z),a_{27},a_{28})\;:\;Rxy\;\Leftrightarrow\; x\cdot H=y\cdot H}{}

\step*{}{\boldsymbol {term_{\ref{lemma:eq-class}.2}}(S,G,\cdot,e,^{-1},u, H,v,x,y,w,r):=eq\mhyphen sym(a_{29})
\\\quad\quad
\;:\;\boldsymbol{x\cdot H=y\cdot H\;\Leftrightarrow \;R_Hxy}}{\textsf{2) is proven}}
\end{flagderiv}

\newpage
\subsection*{Proof of Lemma \ref{lemma:mult}}

1) - 3) are proven in the following diagram.

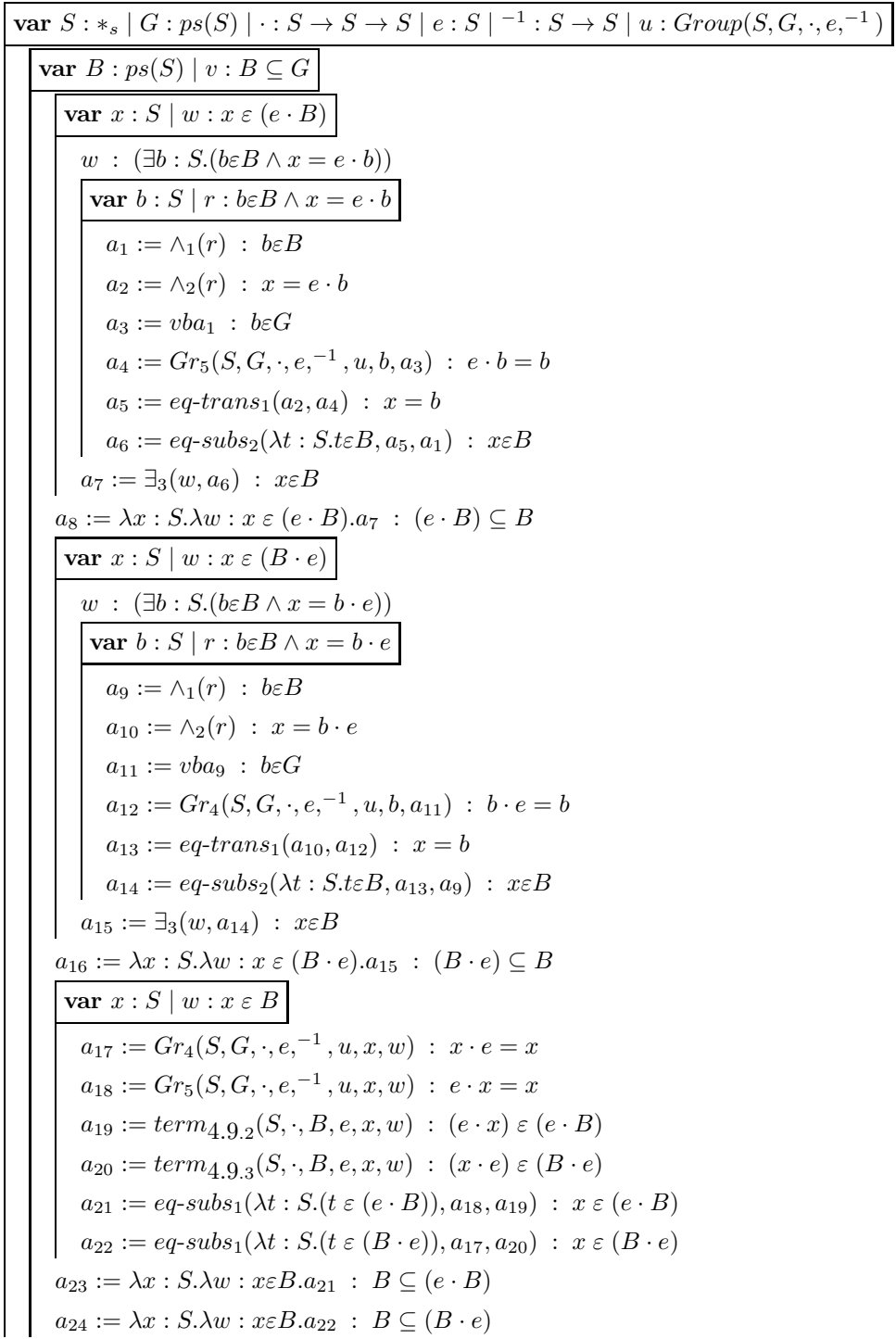
\begin{flagderiv}
\introduce*{}{S: *_s\;|\;G:ps(S)\;|\;\cdot: S\rightarrow S\rightarrow S\;|\;e:S\;|\;^{-1}:S\rightarrow S\;|\;u:Group(S,G,\cdot,e,^{-1})}{}

\introduce*{}{B:ps(S)\;|\;v:B\subseteq G}{}
\introduce*{}{x:S\;|\;w:x\;\varepsilon\; (e\cdot B)}{}

\step*{}{w\;:\;(\exists b:S.(b\varepsilon B\wedge x=e\cdot b))}{}
\introduce*{}{b:S\;|\;r:b\varepsilon B\wedge x=e\cdot b}{}
\step*{}{a_1:=\wedge_1(r)\;:\;b\varepsilon B}{}
\step*{}{a_2:=\wedge_2(r)\;:\;x=e\cdot b}{}
\step*{}{a_3:=vba_1\;:\;b\varepsilon G}{}
\step*{}{a_4:=
Gr_5(S,G, \cdot,e,^{-1},u, b,a_3)\;:\;e\cdot b=b}{}
\step*{}{a_5:=eq\mhyphen trans_1(a_2,a_4)\;:\;x=b}{}

\step*{}{a_6:=eq\mhyphen subs_2(\lambda t:S.t\varepsilon B,a_5,a_1)
\;:\;x\varepsilon B}{}
\conclude*{}{a_7:=\exists_3(w,a_6)\;:\;x\varepsilon B}{}
\conclude*{}{a_8:=
\lambda x:S.\lambda w:x\;\varepsilon \;(e\cdot B).a_7
\;:\;(e\cdot B)\subseteq B}{}

\introduce*{}{x:S\;|\;w:x\;\varepsilon\; (B\cdot e)}{}

\step*{}{w\;:\;(\exists b:S.(b\varepsilon B\wedge x=b\cdot e))}{}
\introduce*{}{b:S\;|\;r:b\varepsilon B\wedge x=b\cdot e}{}
\step*{}{a_9:=\wedge_1(r)\;:\;b\varepsilon B}{}
\step*{}{a_{10}:=\wedge_2(r)\;:\;x=b\cdot e}{}
\step*{}{a_{11}:=vba_9\;:\;b\varepsilon G}{}
\step*{}{a_{12}:=
Gr_4(S,G, \cdot,e,^{-1},u, b,a_{11})\;:\;b\cdot e=b}{}
\step*{}{a_{13}:=eq\mhyphen trans_1(a_{10},a_{12})\;:\;x=b}{}

\step*{}{a_{14}:=eq\mhyphen subs_2(\lambda t:S.t\varepsilon B,a_{13},a_9)
\;:\;x\varepsilon B}{}
\conclude*{}{a_{15}:=\exists_3(w,a_{14})\;:\;x\varepsilon B}{}
\conclude*{}{a_{16}:=
\lambda x:S.\lambda w:x\;\varepsilon\; (B\cdot e).a_{15}
\;:\;(B\cdot e)\subseteq B}{}

\introduce*{}{x:S\;|\;w:x\;\varepsilon\; B}{}
\step*{}{a_{17}:=
Gr_4(S,G, \cdot,e,^{-1},u, x,w)\;:\;x\cdot e=x}{}
\step*{}{a_{18}:=
Gr_5(S,G, \cdot,e,^{-1},u, x,w)\;:\;e\cdot x=x}{}

\step*{}{a_{19}:= term_{\ref{lemma:mult-triv}.2}(S,\cdot,B,e,x,w)\;:\;
(e\cdot x)\;\varepsilon \;(e\cdot B)}{}

\step*{}{a_{20}:= term_{\ref{lemma:mult-triv}.3}(S,\cdot,B,e,x,w)\;:\;(x\cdot e)\;\varepsilon\;(B\cdot e)}{}

\step*{}{a_{21}:=eq\mhyphen subs_1(\lambda t:S.(t\;\varepsilon \;(e\cdot B)),a_{18},a_{19})
\;:\;x\;\varepsilon \;(e\cdot B)}{}

\step*{}{a_{22}:=eq\mhyphen subs_1(\lambda t:S.(t\;\varepsilon \;(B\cdot e)),a_{17},a_{20})
\;:\;x\;\varepsilon\;(B\cdot e)}{}

\conclude*{}{a_{23}:=
\lambda x:S.\lambda w:x\varepsilon B.a_{21}
\;:\;B\subseteq (e\cdot B)}{}

\step*{}{a_{24}:=
\lambda x:S.\lambda w:x\varepsilon B.a_{22}
\;:\;B\subseteq (B\cdot e)}{}

\step*{}{\boldsymbol {term_{ \ref{lemma:mult}.1}}(S,G,\cdot,e,^{-1},u, B,v):=\wedge(a_8,a_{23}) 
\;:\;\boldsymbol{e\cdot B=B}}{\textsf{1) is proven}}

\step*{}{\boldsymbol{ term_{\ref{lemma:mult}.2}}(S,G,\cdot,e,^{-1},u, B,v):=\wedge(a_{16},a_{24}) 
\;:\;\boldsymbol{B\cdot e=B}}{\textsf{2) is proven}}

\introduce*{}{x:S\;|\;w:x\;\varepsilon\; B^{-1}}{}

\step*{}{w\;:\;(\exists b:S.(b\varepsilon B\wedge x=b^{-1}))}{}
\introduce*{}{b:S\;|\;r:b\varepsilon B\wedge x=b^{-1}}{}
\step*{}{a_{25}:=\wedge_1(r)\;:\;b\varepsilon B}{}
\step*{}{a_{26}:=\wedge_2(r)\;:\;x=b^{-1}}{}

\step*{}{a_{27}:=vba_{25}\;:\;b\varepsilon G}{}
\step*{}{a_{28}:=
Gr_3(S,G, \cdot,e,^{-1},u, b,a_{27})\;:\;b^{-1}\varepsilon G}{}

\step*{}{a_{29}:=eq\mhyphen subs_2(\lambda t:S.t\varepsilon G,a_{26},a_{28})
\;:\;x\varepsilon G}{}

\conclude*{}{a_{30}:=\exists_3(w,a_{29})\;:\;x\varepsilon G}{}

\conclude*{}{\boldsymbol{ term_{\ref{lemma:mult}.3}}(S,G,\cdot,e,^{-1},u,B,v):=
\lambda x:S.\lambda w:x\varepsilon B^{-1}.a_{30}
\;:\;\boldsymbol{B^{-1}\subseteq G}}{\textsf{3) is proven}}
\end{flagderiv}

4) - 7) are proven in the following diagram.
 
\begin{flagderiv}
\introduce*{}{S: *_s\;|\;G:ps(S)\;|\;\cdot: S\rightarrow S\rightarrow S\;|\;e:S\;|\;^{-1}:S\rightarrow S\;|\;u:Group(S,G,\cdot,e,^{-1})}{}

\introduce*{}{B:ps(S)\;|\;v:B\subseteq G\;|\;g:S\;|\;w:g\varepsilon G}{}

\introduce*{}{x:S\;|\;r_1:x\;\varepsilon\; (g\cdot B)}{}

\step*{}{r_1\;:\;(\exists b:S.(b\varepsilon B\wedge x=g\cdot b))}{}
\introduce*{}{b:S\;|\;r_2:b\varepsilon B\wedge x=g\cdot b}{}
\step*{}{a_1:=\wedge_1(r_2)\;:\;b\varepsilon B}{}
\step*{}{a_2:=\wedge_2(r_2)\;:\;x=g\cdot b}{}

\step*{}{a_3:=vba_1\;:\;b\varepsilon G}{}

\step*{}{a_4:=Gr_9(S,G, \cdot,e,^{-1},u,g, b,w,a_3)
\;:\;(g\cdot b)\;\varepsilon \;G}{}

\step*{}{a_5:=eq\mhyphen subs_2(\lambda t:S.t\varepsilon G, a_2, a_4)
\;:\;x\varepsilon G}{}

\conclude*{}{a_6:=\exists_3(r_1,a_5)\;:\;x\varepsilon G}{}
\conclude*{}{\boldsymbol {term_{ \ref{lemma:mult}.4}}(S,G,\cdot,e,^{-1},u, B,v,g,w):=\lambda x:S.\lambda r_1:x\;\varepsilon\;(g\cdot B).a_6
\;:\;\boldsymbol{(g\cdot B)\subseteq G}}{\textsf{4) is proven}}

\introduce*{}{x:S\;|\;r_1:x\;\varepsilon\; (B\cdot g)}{}

\step*{}{r_1\;:\;(\exists b:S.(b\varepsilon B\wedge x=b\cdot g))}{}
\introduce*{}{b:S\;|\;r_2:b\varepsilon B\wedge x=b\cdot g}{}
\step*{}{a_7:=\wedge_1(r_2)\;:\;b\varepsilon B}{}
\step*{}{a_8:=\wedge_2(r_2)\;:\;x=b\cdot g}{}

\step*{}{a_9:=vba_7\;:\;b\varepsilon G}{}
\step*{}{a_{10}:=
Gr_9(S,G, \cdot,e,^{-1},u, b,g,a_9,w)\;:\;(b\cdot g)\;\varepsilon \;G}{}
\step*{}{a_{11}:=eq\mhyphen subs_2(\lambda t:S.t\varepsilon G, a_8, a_{10})
\;:\;x\varepsilon G}{}

\conclude*{}{a_{12}:=\exists_3(r_1,a_{11})\;:\;x\varepsilon G}{}
\conclude*{}{\boldsymbol {term_{ \ref{lemma:mult}.5}}(S,G,\cdot,e,^{-1},u, B,v,g,w):=\lambda x:S.\lambda r_1:x\;\varepsilon\;(B\cdot g).a_{12}
\;:\;\boldsymbol{(B\cdot g)\subseteq G}}{\textsf{5) is proven}}

\introduce*{}{c,d:S\;|\;r_1:c\varepsilon G\;|\;r_2:d\varepsilon G}{}
\step*{}{a_{13}(c,d, r_1,r_2):= term_{\ref{theorem:axiom_corollary}.5}(S,G,\cdot, e,^{-1},u,c,d, r_1,r_2)\;:\;(c\cdot d)^{-1}=d^{-1}\cdot c^{-1}}{}

\introduce*{}{r_3:h=c\cdot d\;|\;r_4:x=h^{-1}}{}

\step*{}{a_{14}:= eq\mhyphen cong_1(^{-1},r_3)\;:\;h^{-1}=(c\cdot d)^{-1}}{}

\step*{}{a_{15}(c,d,h,x, r_1,r_2,r_3,r_4):= eq\mhyphen trans_1(eq\mhyphen trans_1(r_4, a_{14}),a_{13}(c,d, r_1,r_2)) \;:\;x=d^{-1}\cdot c^{-1}}{}

\done[2]

\introduce*{}{x:S\;|\;r_1:x\;\varepsilon\; (g\cdot B)^{-1}}{}
\step*{}{r_1:(\exists h:S.(h\;\varepsilon\; (g\cdot B)\wedge x=h^{-1})}{}

\introduce*{}{h:S\;|\;r_2:h\;\varepsilon \;(g\cdot B)\wedge x=h^{-1}}{}
\step*{}{a_{16}:=\wedge_1(r_2)\;:\;h\;\varepsilon\;(g\cdot B)}{}
\step*{}{a_{17}:=\wedge_2(r_2)\;:\;x=h^{-1}}{}
\step*{}{a_{16}:(\exists b:S.(b\varepsilon B\wedge h=g\cdot b))}{}

\introduce*{}{b:S\;|\;r_3:b\varepsilon B\wedge h=g\cdot b}{}

\step*{}{a_{18}:=\wedge_1(r_3)\;:\;b\varepsilon B}{}
\step*{}{a_{19}:=\wedge_2(r_3)\;:\;h=g\cdot b}{}

\step*{}{a_{20}:= vba_{18}\;:\;b\varepsilon G}{}

\step*{}{a_{21}:= a_{15}(g,b,h,x,w,a_{20},a_{19},a_{17})\;:\;x=b^{-1}\cdot g^{-1}}{}

\step*{}{a_{22}:= term_{\ref{lemma:mult-triv}.1}(S,\;^{-1},B,b,a_{18})\;:\;b^{-1}\;\varepsilon\;B^{-1}}{}

\step*{}{a_{23}:= term_{\ref{lemma:mult-triv}.3}(S,\cdot,B^{-1},g^{-1}, b^{-1},a_{22})\;:\;(b^{-1}\cdot g^{-1})\;\varepsilon\;(B^{-1}\cdot g^{-1})}{}

\step*{}{a_{24}:=eq\mhyphen subs_2(\lambda t:S.(t\;\varepsilon\;(B^{-1}\cdot g^{-1})),a_{21},a_{23})
\;:\;x\;\varepsilon\;(B^{-1}\cdot g^{-1})}{}

\conclude*{}{a_{25}:=\exists_3(a_{16},a_{24}) \;:\; x\;\varepsilon\;(B^{-1}\cdot g^{-1})}{}

\conclude*{}{a_{26}:=\exists_3(r_1,a_{25}) \;:\; x\;\varepsilon\;(B^{-1}\cdot g^{-1})}{}

\conclude*{}{a_{27}:=
\lambda x:S.\lambda r_1:
x\;\varepsilon\; (g\cdot B)^{-1}.a_{26}\;:\;
(g\cdot B)^{-1}\subseteq (B^{-1}\cdot g^{-1})}{}

\introduce*{}{x:S\;|\;r_1:x\;\varepsilon\; (B\cdot g)^{-1}}{}
\step*{}{r_1:[\exists h:S.(h\varepsilon (B\cdot g)\wedge x=h^{-1})]}{}

\introduce*{}{h:S\;|\;r_2:h\varepsilon (B\cdot g)\wedge x=h^{-1}}{}
\step*{}{a_{28}:=\wedge_1(r_2)\;:\;h\varepsilon (B\cdot g)}{}
\step*{}{a_{29}:=\wedge_2(r_2)\;:\;x=h^{-1}}{}
\step*{}{a_{28}:(\exists b:S.(b\varepsilon B\wedge h=b\cdot g))}{}

\introduce*{}{b:S\;|\;r_3:b\varepsilon B\wedge h=b\cdot g}{}

\step*{}{a_{30}:=\wedge_1(r_3)\;:\;b\varepsilon B}{}
\step*{}{a_{31}:=\wedge_2(r_3)\;:\;h=b\cdot g}{}

\step*{}{a_{32}:= vba_{30}\;:\;b\varepsilon G}{}

\step*{}{a_{33}:= a_{15}(b,g,h,x,a_{32},w,a_{31},a_{29})\;:\;x=g^{-1}\cdot b^{-1}}{}

\step*{}{a_{34}:= term_{\ref{lemma:mult-triv}.1}(S,\;^{-1},B,b,a_{30})\;:\;b^{-1}\;\varepsilon\;B^{-1}}{}

\step*{}{a_{35}:= term_{\ref{lemma:mult-triv}.2}(S,\cdot,B^{-1},g^{-1}, b^{-1},a_{34}) \;:\;(g^{-1}\cdot b^{-1})\;\varepsilon\;(g^{-1}\cdot B^{-1})}{}

\step*{}{a_{36}:=eq\mhyphen subs_2(\lambda t:S.(t\;\varepsilon\;(g^{-1}\cdot B^{-1})),a_{33},a_{35})
\;:\;x\;\varepsilon\;(g^{-1}\cdot B^{-1})}{}

\conclude*{}{a_{37}:=\exists_3(a_{28},a_{36}) \;:\; x\;\varepsilon\;(g^{-1}\cdot B^{-1})}{}

\conclude*{}{a_{38}:=\exists_3(r_1,a_{37}) \;:\; x\;\varepsilon\;(g^{-1}\cdot B^{-1})}{}

\conclude*{}{a_{39}:=
\lambda x:S.\lambda r_1:
x\;\varepsilon\; (B\cdot g)^{-1}.a_{38}\;:\;
(B\cdot g)^{-1}\subseteq (g^{-1}\cdot B^{-1})}{}

\introduce*{}{x:S\;|\;r_1:x\;\varepsilon\; (B^{-1}\cdot g^{-1})}{}
\step*{}{r_1:(\exists h:S.(h\varepsilon B^{-1}\wedge x=h\cdot g^{-1}))}{}

\introduce*{}{h:S\;|\;r_2:h\varepsilon B^{-1}\wedge x=h\cdot g^{-1}}{}
\step*{}{a_{40}:=\wedge_1(r_2)\;:\;h\varepsilon B^{-1}}{}
\step*{}{a_{41}:=\wedge_2(r_2)\;:\;x=h\cdot g^{-1}}{}
\step*{}{a_{40}:(\exists b:S.(b\varepsilon B\wedge h=b^{-1}))}{}

\introduce*{}{b:S\;|\;r_3:b\varepsilon B\wedge h=b^{-1}}{}

\step*{}{a_{42}:=\wedge_1(r_3)\;:\;b\varepsilon B}{}
\step*{}{a_{43}:=\wedge_2(r_3)\;:\;h=b^{-1}}{}

\step*{}{a_{44}:= vba_{42}\;:\;b\varepsilon G}{}

\step*{}{a_{45}:= a_{13}(g,b,w,a_{44})\;:\;(g\cdot b)^{-1}=b^{-1}\cdot g^{-1}}{}

\step*{}{a_{46}:=eq\mhyphen cong_1(\lambda t:S.(t\cdot g^{-1}), a_{43})\;:\;h\cdot g^{-1}=b^{-1}\cdot g^{-1}}{}

\step*{}{a_{47}:=eq\mhyphen trans_2(eq\mhyphen trans_1(a_{41},a_{46}),a_{45})
\;:\;x=(g\cdot b)^{-1}}{}

\step*{}{a_{48}:= term_{\ref{lemma:mult-triv}.2}(S,\cdot,B,g,b,a_{42})\;:\;(g\cdot b)\;\varepsilon\;(g\cdot B)}{}

\step*{}{a_{49}:= term_{\ref{lemma:mult-triv}.1}(S,\;^{-1},(g\cdot B),(g\cdot b) ,a_{48}) \;:\;(g\cdot b)^{-1}\;\varepsilon\;(g\cdot B)^{-1}}{}

\step*{}{a_{50}:=eq\mhyphen subs_2(\lambda t:S.(t\;\varepsilon\;(g\cdot B)^{-1}), a_{47}, a_{49})
\;:\;x\;\varepsilon\;(g\cdot B)^{-1}}{}

\conclude*{}{a_{51}:=\exists_3(a_{40},a_{50}) \;:\; x\;\varepsilon\;(g\cdot B)^{-1})}{}

\conclude*{}{a_{52}:=\exists_3(r_1,a_{51}) \;:\; x\;\varepsilon\;(g\cdot B)^{-1})}{}

\conclude*{}{a_{53}:=
\lambda x:S.\lambda r_1:
(x\;\varepsilon\;(B^{-1}\cdot g^{-1})).a_{52}\;:\;
(B^{-1}\cdot g^{-1})\subseteq (g\cdot B)^{-1}}{}

\step*{}{\boldsymbol {term_{ \ref{lemma:mult}.6}}(S,G,\cdot,e,^{-1},u, B,v,g,w):=\wedge(a_{27},a_{53})
\;:\;\boldsymbol{(g\cdot B)^{-1}=B^{-1}\cdot g^{-1}}}{\textsf{6) is proven}}

\introduce*{}{x:S\;|\;r_1:x\;\varepsilon\; (g^{-1}\cdot B^{-1})}{}
\step*{}{r_1:(\exists h:S.(h\varepsilon B^{-1}\wedge x=g^{-1}\cdot h))}{}

\introduce*{}{h:S\;|\;r_2:h\varepsilon B^{-1}\wedge x=g^{-1}\cdot h}{}
\step*{}{a_{54}:=\wedge_1(r_2)\;:\;h\varepsilon B^{-1}}{}
\step*{}{a_{55}:=\wedge_2(r_2)\;:\;x=g^{-1}\cdot h}{}
\step*{}{a_{54}:(\exists b:S.(b\varepsilon B\wedge h=b^{-1}))}{}

\introduce*{}{b:S\;|\;r_3:b\varepsilon B\wedge h=b^{-1}}{}

\step*{}{a_{56}:=\wedge_1(r_3)\;:\;b\varepsilon B}{}
\step*{}{a_{57}:=\wedge_2(r_3)\;:\;h=b^{-1}}{}

\step*{}{a_{58}:= vba_{56}\;:\;b\varepsilon G}{}

\step*{}{a_{59}:= a_{13}(b,g,a_{58},w)\;:\;(b\cdot g)^{-1}=g^{-1}\cdot b^{-1}}{}

\step*{}{a_{60}:=eq\mhyphen cong_1(\lambda t:S.(g^{-1}\cdot t), a_{57})\;:\;g^{-1}\cdot h=g^{-1}\cdot b^{-1}}{}

\step*{}{a_{61}:=eq\mhyphen trans_2(eq\mhyphen trans_1(a_{55},a_{60}),a_{59})
\;:\;x=(b\cdot g)^{-1}}{}

\step*{}{a_{62}:= term_{\ref{lemma:mult-triv}.3}(S,\cdot,B,g,b,a_{56})\;:\;(b\cdot g)\;\varepsilon\;(B\cdot g)}{}

\step*{}{a_{63}:= term_{\ref{lemma:mult-triv}.1}(S,\;^{-1},(B\cdot g),(b\cdot g), a_{62}) \;:\;(b\cdot g)^{-1}\;\varepsilon\;(B\cdot g)^{-1}}{}

\step*{}{a_{64}:=eq\mhyphen subs_2(\lambda t:S.(t\;\varepsilon\;(B\cdot g)^{-1}), a_{61}, a_{63})
\;:\;x\;\varepsilon\;(B\cdot g)^{-1}}{}

\conclude*{}{a_{65}:=\exists_3(a_{54},a_{64}) \;:\; x\;\varepsilon\;(B\cdot g)^{-1})}{}

\conclude*{}{a_{66}:=\exists_3(r_1,a_{65}) \;:\;x\;\varepsilon\;(B\cdot g)^{-1})}{}

\conclude*{}{a_{67}:=
\lambda x:S.\lambda r_1:
(x\;\varepsilon\;(g^{-1}\cdot B^{-1})).a_{66}\;:\;
(g^{-1}\cdot B^{-1})\subseteq (B\cdot g)^{-1}}{}

\step*{}{\boldsymbol {term_{ \ref{lemma:mult}.7}}(S,G,\cdot,e,^{-1},u, B,v,g,w):=\wedge(a_{39},a_{67})
\;:\;\boldsymbol{(B\cdot g)^{-1}=g^{-1}\cdot B^{-1}}}{\textsf{7) is proven}}
\end{flagderiv}

8) and 9) are proven in the following diagram.
 
\begin{flagderiv}
\introduce*{}{S: *_s\;|\;G:ps(S)\;|\;\cdot: S\rightarrow S\rightarrow S\;|\;e:S\;|\;^{-1}:S\rightarrow S\;|\;u:Group(S,G,\cdot,e,^{-1})}{}

\introduce*{}{B,C:ps(S)\;|\;v:B\subseteq G\;|\;w:C\subseteq G}{}
\introduce*{}{x:S\;|\;r_1:x\;\varepsilon\; (B\cdot C)}{}

\step*{}{r_1\;:\;[\exists c:S.(c\varepsilon C\wedge x\;\varepsilon\;(B\cdot c))]}{}
\introduce*{}{c:S\;|\;r_2:c\varepsilon C\wedge x\;\varepsilon\;(B\cdot c))}{}
\step*{}{a_1:=\wedge_1(r_2)\;:\;c\varepsilon C}{}
\step*{}{a_2:=\wedge_2(r_2)\;:\;x\;\varepsilon\;(B\cdot c)}{}
\step*{}{a_3:=wca_1\;:\;c\varepsilon G}{}
\step*{}{a_4:=
term_{\ref{lemma:mult}.5}(S,G, \cdot,e,^{-1},u, B,v,c,a_3)\;:\;(B\cdot c)\subseteq G}{}

\step*{}{a_5:=a_4xa_2
\;:\;x\varepsilon G}{}

\conclude*{}{a_6:=\exists_3(r_1,a_5)\;:\;x\varepsilon G}{}
\conclude*{}{\boldsymbol {term_{ \ref{lemma:mult}.8}}(S,G,\cdot,e,^{-1},u, B,C,v,w):=\lambda x:S.\lambda r_1:x\;\varepsilon\;(B\cdot C).a_6
\;:\;\boldsymbol{(B\cdot C)\subseteq G}}{\textsf{8) is proven}}

\introduce*{}{x:S\;|\;r_1:x\;\varepsilon\; (B\cdot C)^{-1}}{}

\step*{}{r_1\;:\;[\exists h:S.(h\;\varepsilon\;(B\cdot C)\wedge x=h^{-1})]}{}
\introduce*{}{h:S\;|\;r_2:h\;\varepsilon\;(B\cdot C)\wedge x=h^{-1}}{}
\step*{}{a_7:=\wedge_1(r_2)\;:\;h\;\varepsilon\;(B\cdot C)}{}
\step*{}{a_8:=\wedge_2(r_2)\;:\;x=h^{-1}}{}

\step*{}{a_7\;:\;[\exists c:S.(c\varepsilon C\wedge
h\;\varepsilon\;(B\cdot c))]}{}
\introduce*{}{c:S\;|\;r_3:c\varepsilon C\wedge
h\;\varepsilon\;(B\cdot c)}{}
\step*{}{a_9:=\wedge_1(r_3)\;:\;c\varepsilon C}{}
\step*{}{a_{10}:=\wedge_2(r_3)\;:\;h\;\varepsilon\;(B\cdot c)}{}
\step*{}{a_{11}:=wca_9
\;:\;c\varepsilon G}{}

\step*{}{a_{12}:= term_{\ref{lemma:mult-triv}.1}(S,\;^{-1},(B\cdot c),h,a_{10})\;:\;h^{-1}\;\varepsilon\;(B\cdot c)^{-1}}{}

\step*{}{a_{13}:=eq\mhyphen subs_2(\lambda t:S.(t\;\varepsilon\; (B\cdot c)^{-1}), a_8,a_{12})\;:\;x\;\varepsilon\;(B\cdot c)^{-1}}{}

\step*{}{a_{14}:=
term_{\ref{lemma:mult}.7}(S,G, \cdot,e,^{-1},u, B,v,c,a_{11})\;:\;(B\cdot c)^{-1}=c^{-1}\cdot B^{-1}}{}

\step*{}{a_{15}:=\wedge_1( a_{14})a_{13}\;:\;x\;\varepsilon\;(c^{-1}\cdot B^{-1})}{}

\step*{}{a_{16}:= term_{\ref{lemma:mult-triv}.1}(S,\;^{-1},C,c,a_9) \;:\;c^{-1}\varepsilon C^{-1}}{}

\step*{}{a_{17}:=\wedge(a_{16},a_{15})\;:\;c^{-1}\varepsilon C^{-1}\wedge x\;\varepsilon\;(c^{-1}\cdot B^{-1})}{}

\step*{}{a_{18}:=\exists_1(\lambda t:S.(t\varepsilon C^{-1}\wedge x\;\varepsilon\;(t\cdot B^{-1})), c^{-1} , a_{17}) \;:\; x\;\varepsilon\;(C^{-1}\cdot B^{-1})}{}

\conclude*{}{a_{19}:=\exists_3(a_7,a_{18})\;:\;x\;\varepsilon\;(C^{-1}\cdot B^{-1})}{}
\conclude*{}{a_{20}:=\exists_3(r_1,a_{19})\;:\;x\;\varepsilon\;(C^{-1}\cdot B^{-1})}{}

\conclude*{}{a_{21}:=\lambda x:S.\lambda r_1:x\;\varepsilon\;(B\cdot C)^{-1}.a_{20}
\;:\;(B\cdot C)^{-1}\subseteq C^{-1}\cdot B^{-1}}{}

\introduce*{}{x:S\;|\;r_1:x\;\varepsilon\; C^{-1}\cdot B^{-1}}{}

\step*{}{r_1\;:\;[\exists h:S.(h\;\varepsilon\;B^{-1}\wedge x\;\varepsilon\;(C^{-1}\cdot h))]}{}
\introduce*{}{h:S\;|\;r_2:h\;\varepsilon\;B^{-1}\wedge x\;\varepsilon\;(C^{-1}\cdot h)}{}
\step*{}{a_{22}:=\wedge_1(r_2)\;:\;h\;\varepsilon\;B^{-1}}{}
\step*{}{a_{23}:=\wedge_2(r_2)\;:\;x\;\varepsilon\;(C^{-1}\cdot h)}{}

\step*{}{a_{22}\;:\;(\exists b:S.(b\varepsilon B\wedge
h=b^{-1}))}{}
\introduce*{}{b:S\;|\;r_3:b\varepsilon B\wedge
h=b^{-1})}{}
\step*{}{a_{24}:=\wedge_1(r_3)\;:\;b\varepsilon B}{}
\step*{}{a_{25}:=\wedge_2(r_3)\;:\;h=b^{-1}}{}
\step*{}{a_{26}:=vba_{24}
\;:\;b\varepsilon G}{}

\step*{}{a_{27}:=eq\mhyphen cong_1(\lambda t:S.(C^{-1}\cdot t), a_{25})\;:\;C^{-1}\cdot h= C^{-1}\cdot b^{-1}}{}

\step*{}{a_{28}:=\wedge_1(a_{27})xa_{23}\;:\;x\;\varepsilon\;(C^{-1}\cdot b^{-1})}{}

\step*{}{a_{29}:=
term_{\ref{lemma:mult}.6}(S,G, \cdot,e,^{-1},u,C,w,b,a_{26})\;:\;(b\cdot C)^{-1}=C^{-1}\cdot b^{-1}}{}

\step*{}{a_{30}:=\wedge_2(a_{29})xa_{28}\;:\;x\;\varepsilon\;(b\cdot C)^{-1}}{}

\step*{}{a_{30}\;:\;
[\exists y:S.(y\;\varepsilon\;(b\cdot C)\wedge x=y^{-1})]}{}
\introduce*{}{y:S\;|\;r_4:y\;\varepsilon\;(b\cdot C)\wedge x=y^{-1}}{}
\step*{}{a_{31}:=\wedge_1(r_4)\;:\;y\;\varepsilon\;(b\cdot C)}{}
\step*{}{a_{32}:=\wedge_2(r_4)\;:\;x=y^{-1}}{}

\step*{}{a_{33}:=\wedge(a_{24},a_{31})\;:\;b\varepsilon B\wedge y\;\varepsilon\;(b\cdot C)}{}

\step*{}{a_{34}:=\exists_1(\lambda t:S.(t\varepsilon B\wedge y\;\varepsilon\;(t\cdot C)),b, a_{33}) \;:\;y\;\varepsilon\;(B\cdot C)}{}

\step*{}{a_{35}:= term_{\ref{lemma:mult-triv}.1}(S,\;^{-1},(B\cdot C),y,a_{34})\;:\;y^{-1}\;\varepsilon\;(B\cdot C)^{-1}}{}

\step*{}{a_{36}:=eq\mhyphen subs_2(\lambda t:S.(t\;\varepsilon\; (B\cdot C)^{-1}), a_{32},a_{35})\;:\;x\;\varepsilon\;(B\cdot C)^{-1}}{}

\conclude*{}{a_{37}:=\exists_3(a_{30},a_{36})\;:\;x\;\varepsilon\;(B\cdot C)^{-1}}{}
\conclude*{}{a_{38}:=\exists_3(a_{22},a_{37})\;:\;x\;\varepsilon\;(B\cdot C)^{-1}}{}
\conclude*{}{a_{39}:=\exists_3(r_1,a_{38})\;:\;x\;\varepsilon\;(B\cdot C)^{-1}}{}

\conclude*{}{a_{40}:=\lambda x:S.\lambda r_1:x\;\varepsilon\;(C^{-1}\cdot B^{-1}).a_{39}
\;:\;(C^{-1}\cdot B^{-1})\subseteq (B\cdot C)^{-1}}{}

\step*{}{\boldsymbol {term_{ \ref{lemma:mult}.9}}(S,G,\cdot,e,^{-1},u, B,C,v,w):=\wedge(a_{21},a_{40})
\;:\;\boldsymbol{(B\cdot C)^{-1}=C^{-1}\cdot B^{-1}}}{\textsf{9) is proven}}
\end{flagderiv}

\subsection*{Proof of Proposition \ref{lemma:set-assoc}}

1) - 3) are proven in the following diagram.

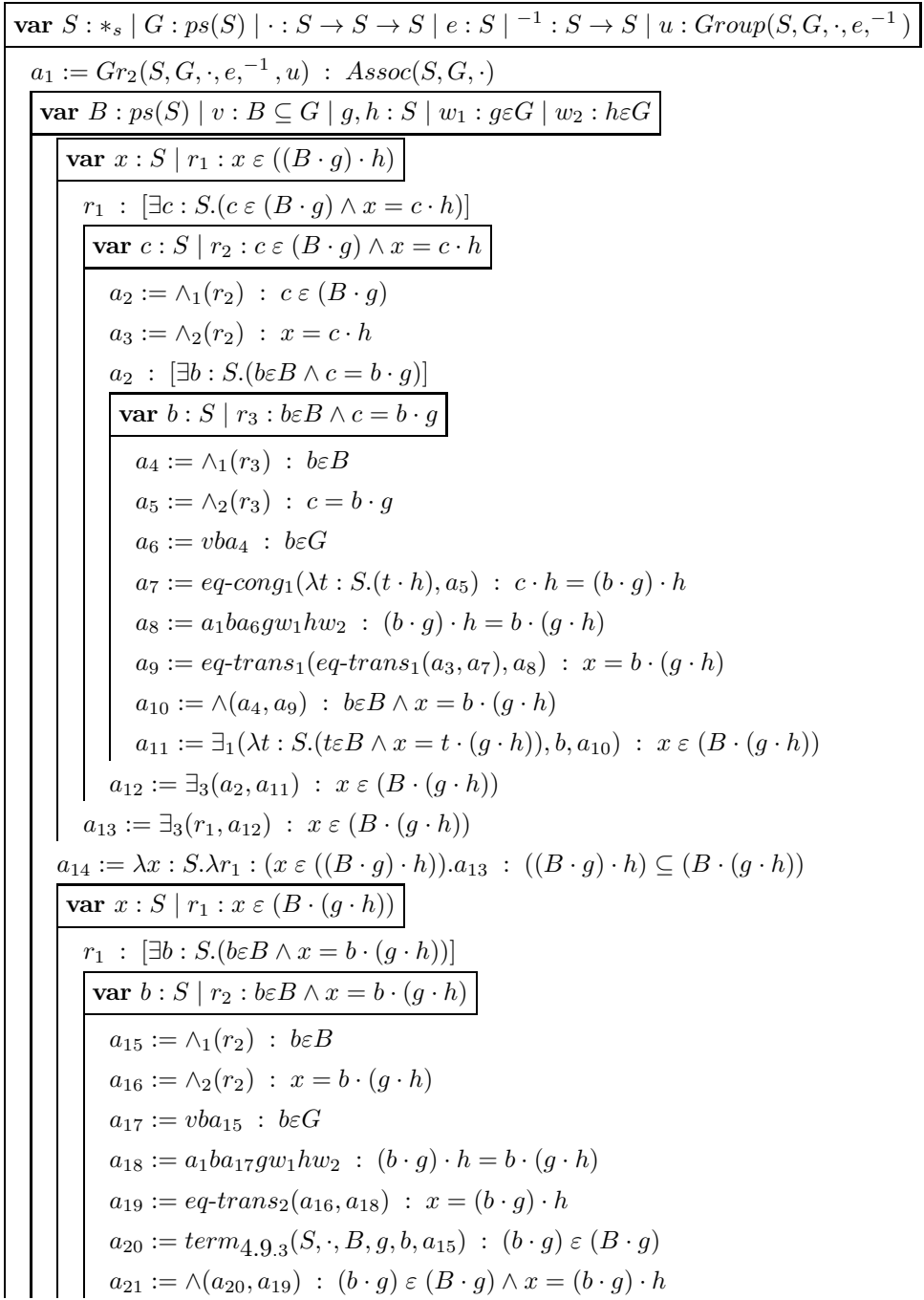
\begin{flagderiv}
\introduce*{}{S: *_s\;|\;G:ps(S)\;|\;\cdot: S\rightarrow S\rightarrow S\;|\;e:S\;|\;^{-1}:S\rightarrow S\;|\;u:Group(S,G,\cdot,e,^{-1})}{}
\step*{}{a_1:=Gr_2(S,G,\cdot,e,^{-1},u)\;:\;Assoc(S,G,\cdot)}{}
\introduce*{}{B:ps(S)\;|\;v:B\subseteq G\;|\;g,h:S\;|\;w_1:g\varepsilon G\;|\;w_2:h\varepsilon G}{}

\introduce*{}{x:S\;|\;r_1:x\;\varepsilon\; ((B\cdot g)\cdot h)}{}
\step*{}{r_1\;:\;[\exists c:S.(c\;\varepsilon \;(B\cdot g)\wedge x=c\cdot h)]}{}
\introduce*{}{c:S\;|\;r_2:c\;\varepsilon \;(B\cdot g)\wedge x=c\cdot h}{}
\step*{}{a_2:=\wedge_1(r_2)\;:\;c\;\varepsilon \;(B\cdot g)}{}
\step*{}{a_3:=\wedge_2(r_2)\;:\;x=c\cdot h}{}

\step*{}{a_2\;:\;[\exists b:S.(b\varepsilon B\wedge c=b\cdot g)]}{}
\introduce*{}{b:S\;|\;r_3:b\varepsilon B\wedge c=b\cdot g}{}
\step*{}{a_4:=\wedge_1(r_3)\;:\;b\varepsilon B}{}
\step*{}{a_5:=\wedge_2(r_3)\;:\;c=b\cdot g}{}
\step*{}{a_6:=vba_4\;:\;b\varepsilon G}{}
\step*{}{a_7:=eq\mhyphen cong_1(\lambda t:S.(t\cdot h), a_5)
\;:\;c\cdot h=(b\cdot g)\cdot h}{}
\step*{}{a_8:=
a_1ba_6gw_1hw_2
\;:\;(b\cdot g)\cdot h=b\cdot (g\cdot h)}{}
\step*{}{a_9:=eq\mhyphen trans_1(eq\mhyphen trans_1(a_3, a_7),a_8)
\;:\;x=b\cdot (g\cdot h)}{}
\step*{}{a_{10}:=\wedge(a_4, a_9)
\;:\;b\varepsilon B\wedge x=b\cdot (g\cdot h)}{}
\step*{}{a_{11}:=\exists_1(\lambda t:S.(t\varepsilon B\wedge x=t\cdot (g\cdot h)),b, a_{10})\;:\;x\;\varepsilon\;(B\cdot(g\cdot h))}{}
\conclude*{}{a_{12}:=\exists_3(a_2,a_{11})\;:\;x\;\varepsilon\;(B\cdot(g\cdot h))}{}
\conclude*{}{a_{13}:=\exists_3(r_1,a_{12})\;:\;x\;\varepsilon\;(B\cdot(g\cdot h))}{}
\conclude*{}{a_{14}:=
\lambda x:S.\lambda r_1:(x\;\varepsilon\;((B\cdot g)\cdot h)).a_{13}
\;:\;((B\cdot g)\cdot h)\subseteq (B\cdot(g\cdot h))}{}

\introduce*{}{x:S\;|\;r_1:x\;\varepsilon\; (B\cdot(g\cdot h))}{}
\step*{}{r_1\;:\;[\exists b:S.(b\varepsilon B\wedge x=b\cdot(g\cdot h))]}{}
\introduce*{}{b:S\;|\;r_2:b\varepsilon B\wedge x=b\cdot (g\cdot h)}{}
\step*{}{a_{15}:=\wedge_1(r_2)\;:\;b\varepsilon B}{}
\step*{}{a_{16}:=\wedge_2(r_2)\;:\;x=b\cdot (g\cdot h)}{}
\step*{}{a_{17}:= vba_{15}\;:\;b\varepsilon G}{}
\step*{}{a_{18}:= a_1ba_{17}gw_1hw_2
\;:\;(b\cdot g)\cdot h=b\cdot (g\cdot h)}{}

\step*{}{a_{19}:=eq\mhyphen trans_2(a_{16}, a_{18})\;:\;x= (b\cdot g)\cdot h}{}

\step*{}{a_{20}:= term_{\ref{lemma:mult-triv}.3}(S,\cdot,B,g,b, a_{15}) \;:\;(b\cdot g)\;\varepsilon\;(B\cdot g)}{}

\step*{}{a_{21}:=\wedge(a_{20}, a_{19})
\;:\;(b\cdot g)\;\varepsilon\;(B\cdot g)\wedge x=(b\cdot g)\cdot h}{}

\step*{}{a_{22}:=\exists_1(\lambda t:S.(t\;\varepsilon\;(B\cdot g)\wedge x=t\cdot h),(b\cdot g), a_{21})\;:\;x\;\varepsilon\;((B\cdot g)\cdot h)}{}

\conclude*{}{a_{23}:=\exists_3(r_1,a_{22})\;:\;x\;\varepsilon\;((B\cdot g)\cdot h)}{}

\conclude*{}{a_{24}:=
\lambda x:S.\lambda r_1:(x\;\varepsilon\;(B\cdot(g\cdot h))).a_{23}
\;:\;(B\cdot(g\cdot h))\subseteq ((B\cdot g)\cdot h)}{}

\step*{}{\boldsymbol {term_{ \ref{lemma:set-assoc}.1}}(S,G,\cdot,e,^{-1},u, B,v,g,h,w_1,w_2):=\wedge(a_{14},a_{24})
\;:\;\boldsymbol{(B\cdot g)\cdot h=B\cdot(g\cdot h)}}{\textsf{1) is proven}}

\introduce*{}{x:S\;|\;r_1:x\;\varepsilon\; ((g\cdot B)\cdot h)}{}
\step*{}{r_1\;:\;[\exists c:S.(c\;\varepsilon \;(g\cdot B)\wedge x=c\cdot h)]}{}
\introduce*{}{c:S\;|\;r_2:c\;\varepsilon \;(g\cdot B)\wedge x=c\cdot h}{}
\step*{}{a_{25}:=\wedge_1(r_2)\;:\;c\;\varepsilon \;(g\cdot B)}{}
\step*{}{a_{26}:=\wedge_2(r_2)\;:\;x=c\cdot h}{}

\step*{}{a_{25}\;:\;[\exists b:S.(b\varepsilon B\wedge c=g\cdot b)]}{}
\introduce*{}{b:S\;|\;r_3:b\varepsilon B\wedge c=g\cdot b}{}
\step*{}{a_{27}:=\wedge_1(r_3)\;:\;b\varepsilon B}{}
\step*{}{a_{28}:=\wedge_2(r_3)\;:\;c=g\cdot b}{}
\step*{}{a_{29}:=vba_{27}\;:\;b\varepsilon G}{}
\step*{}{a_{30}:=eq\mhyphen cong_1(\lambda t:S.(t\cdot h), a_{28})
\;:\;c\cdot h=(g\cdot b)\cdot h}{}
\step*{}{a_{31}:=
a_1gw_1ba_{29}hw_2
\;:\;(g\cdot b)\cdot h=g\cdot (b\cdot h)}{}
\step*{}{a_{32}:=eq\mhyphen trans_1(eq\mhyphen trans_1(a_{26}, a_{30}),a_{31})
\;:\;x=g\cdot (b\cdot h)}{}

\step*{}{a_{33}:= term_{\ref{lemma:mult-triv}.3}(S,\cdot,B,h,b,a_{27})\;:\;(b\cdot h)\;\varepsilon\;(B\cdot h)}{}

\step*{}{a_{34}:=\wedge(a_{33}, a_{32})
\;:\;(b\cdot h)\;\varepsilon\;(B\cdot h)\wedge x=g\cdot (b\cdot h)}{}

\step*{}{a_{35}:=\exists_1(\lambda t:S.(t\;\varepsilon\;(B\cdot h)\wedge x=g\cdot t),(b\cdot h), a_{34})\;:\;x\;\varepsilon\;(g\cdot (B\cdot h))}{}

\conclude*{}{a_{36}:=\exists_3(a_{25},a_{35})\;:\;x\;\varepsilon\;(g\cdot(B\cdot h))}{}
\conclude*{}{a_{37}:=\exists_3(r_1,a_{36})\;:\;x\;\varepsilon\;(g\cdot(B\cdot h))}{}
\conclude*{}{a_{38}:=
\lambda x:S.\lambda r_1:(x\;\varepsilon\;((g\cdot B)\cdot h)).a_{37}
\;:\;((g\cdot B)\cdot h)\subseteq (g\cdot(B\cdot h))}{}

\introduce*{}{x:S\;|\;r_1:x\;\varepsilon\; (g\cdot (B\cdot h))}{}
\step*{}{r_1\;:\;[\exists c:S.(c\;\varepsilon \;(B\cdot h)\wedge x=g\cdot c)]}{}
\introduce*{}{c:S\;|\;r_2:c\;\varepsilon \;(B\cdot h)\wedge x=g\cdot c}{}
\step*{}{a_{39}:=\wedge_1(r_2)\;:\;c\;\varepsilon \;(B\cdot h)}{}
\step*{}{a_{40}:=\wedge_2(r_2)\;:\;x=g\cdot c}{}

\step*{}{a_{39}\;:\;[\exists b:S.(b\varepsilon B\wedge c=b\cdot h)]}{}
\introduce*{}{b:S\;|\;r_3:b\varepsilon B\wedge c=b\cdot h}{}
\step*{}{a_{41}:=\wedge_1(r_3)\;:\;b\varepsilon B}{}
\step*{}{a_{42}:=\wedge_2(r_3)\;:\;c=b\cdot h}{}
\step*{}{a_{43}:=vba_{41}\;:\;b\varepsilon G}{}
\step*{}{a_{44}:=eq\mhyphen cong_1(\lambda t:S.(g\cdot t), a_{42})
\;:\;g\cdot c=g\cdot( b\cdot h)}{}
\step*{}{a_{45}:=
a_1gw_1ba_{43}hw_2
\;:\;(g\cdot b)\cdot h=g\cdot (b\cdot h)}{}

\step*{}{a_{46}:=eq\mhyphen trans_2(eq\mhyphen trans_1(a_{40}, a_{44}),a_{45})
\;:\;x=(g\cdot b)\cdot h}{}

\step*{}{a_{47}:= term_{\ref{lemma:mult-triv}.2}(S,\cdot,B,g,b,a_{41})\;:\;(g\cdot b)\;\varepsilon\;(g\cdot B)}{}

\step*{}{a_{48}:=\wedge(a_{47}, a_{46})
\;:\;(g\cdot b)\;\varepsilon\;(g\cdot B)\wedge x=(g\cdot b)\cdot h}{}

\step*{}{a_{49}:=\exists_1(\lambda t:S.(t\;\varepsilon\;(g\cdot B)\wedge x=t\cdot h),(g\cdot b), a_{48})\;:\;x\;\varepsilon\;((g\cdot B)\cdot h)}{}

\conclude*{}{a_{50}:=\exists_3(a_{39},a_{49})\;:\;x\;\varepsilon\;((g\cdot B)\cdot h)}{}
\conclude*{}{a_{51}:=\exists_3(r_1,a_{50})\;:\;x\;\varepsilon\;((g\cdot B)\cdot h)}{}
\conclude*{}{a_{52}:=
\lambda x:S.\lambda r_1:(x\;\varepsilon\;(g\cdot (B\cdot h))).a_{51}
\;:\;(g\cdot(B\cdot h))\subseteq ((g\cdot B)\cdot h)
}{}

\step*{}{\boldsymbol {term_{ \ref{lemma:set-assoc}.2}}(S,G,\cdot,e,^{-1},u, B,v,g,h,w_1,w_2):=\wedge(a_{38},a_{52})
\;:\;\boldsymbol{(g\cdot B)\cdot h=g\cdot(B\cdot h)}}{\textsf{2) is proven}}

\introduce*{}{x:S\;|\;r_1:x\;\varepsilon\; ((g\cdot h)\cdot B)}{}
\step*{}{r_1\;:\;[\exists b:S.(b\;\varepsilon B\wedge x=(g\cdot h)\cdot b)]}{}
\introduce*{}{b:S\;|\;r_2:b\;\varepsilon B\wedge x=(g\cdot h)\cdot b}{}
\step*{}{a_{53}:=\wedge_1(r_2)\;:\;b\;\varepsilon B}{}
\step*{}{a_{54}:=\wedge_2(r_2)\;:\;x=(g\cdot h)\cdot b}{}

\step*{}{a_{55}:=vba_{53}\;:\;b\varepsilon G}{}

\step*{}{a_{56}:=
a_1gw_1hw_2ba_{55}
\;:\;(g\cdot h)\cdot b=g\cdot (h\cdot b)}{}

\step*{}{a_{57}:=eq\mhyphen trans_1(a_{54},a_{56})
\;:\;x=g\cdot (h\cdot b)}{}

\step*{}{a_{58}:= term_{\ref{lemma:mult-triv}.2}(S,\cdot,B,h,b,a_{53})\;:\;(h\cdot b)\;\varepsilon\;(h\cdot B)}{}

\step*{}{a_{59}:=\wedge(a_{58}, a_{57})
\;:\;(h\cdot b)\;\varepsilon\;(h\cdot B)\wedge x=g\cdot (h\cdot b)}{}

\step*{}{a_{60}:=\exists_1(\lambda t:S.(t\;\varepsilon\;(h\cdot B)\wedge x=g\cdot t),(h\cdot b), a_{59})\;:\;x\;\varepsilon\;(g\cdot (h\cdot B))}{}

\conclude*{}{a_{61}:=\exists_3(r_1,a_{60})\;:\;x\;\varepsilon\;(g\cdot(h\cdot B))}{}

\conclude*{}{a_{62}:=
\lambda x:S.\lambda r_1:(x\;\varepsilon\;(g\cdot h)\cdot B)).a_{61}
\;:\;((g\cdot h)\cdot B)\subseteq (g\cdot(h\cdot B))}{}

\introduce*{}{x:S\;|\;r_1:x\;\varepsilon\; (g\cdot (h\cdot B))}{}
\step*{}{r_1\;:\;[\exists c:S.(c\;\varepsilon \;(h\cdot B)\wedge x=g\cdot c)]}{}
\introduce*{}{c:S\;|\;r_2:c\;\varepsilon \;(h\cdot B)\wedge x=g\cdot c}{}
\step*{}{a_{63}:=\wedge_1(r_2)\;:\;c\;\varepsilon \;(h\cdot B)}{}
\step*{}{a_{64}:=\wedge_2(r_2)\;:\;x=g\cdot c}{}

\step*{}{a_{63}\;:\;(\exists b:S.(b\varepsilon B\wedge c=h\cdot b))}{}
\introduce*{}{b:S\;|\;r_3:b\varepsilon B\wedge c=h\cdot b}{}
\step*{}{a_{65}:=\wedge_1(r_3)\;:\;b\varepsilon B}{}
\step*{}{a_{66}:=\wedge_2(r_3)\;:\;c=h\cdot b}{}
\step*{}{a_{67}:= vba_{65}\;:\;b\varepsilon G}{}
\step*{}{a_{68}:=eq\mhyphen cong_1(\lambda t:S.(g\cdot t), a_{66})
\;:\;g\cdot c=g\cdot( h\cdot b)}{}
\step*{}{a_{69}:=
a_1gw_1hw_2ba_{67}
\;:\;(g\cdot h)\cdot b=g\cdot (h\cdot b)}{}

\step*{}{a_{70}:=eq\mhyphen trans_2(eq\mhyphen trans_1(a_{64}, a_{68}),a_{69})
\;:\;x=(g\cdot h)\cdot b}{}

\step*{}{a_{71}:=\wedge(a_{65}, a_{70})
\;:\;b\varepsilon B\wedge x=(g\cdot h)\cdot b}{}
\step*{}{a_{72}:=\exists_1(\lambda t:S.(t\varepsilon B\wedge x=(g\cdot h)\cdot t),b, a_{71})\;:\;x\;\varepsilon\;((g\cdot h)\cdot B)}{}

\conclude*{}{a_{73}:=\exists_3(a_{63},a_{72})\;:\;x\;\varepsilon\;((g\cdot h)\cdot B)}{}
\conclude*{}{a_{74}:=\exists_3(r_1,a_{73})\;:\;x\;\varepsilon\;((g\cdot h)\cdot B)}{}
\conclude*{}{a_{75}:=
\lambda x:S.\lambda r_1:x\;\varepsilon\;(g\cdot (h\cdot B)).a_{74}
\;:\;(g\cdot(h\cdot B))\subseteq ((g\cdot h)\cdot B)}{}

\step*{}{\boldsymbol {term_{ \ref{lemma:set-assoc}.3}}(S,G,\cdot,e, ^{-1},u, B,v,g,h,w_1, w_2):=\wedge(a_{62},a_{75})
\;:\;\boldsymbol{(g\cdot h)\cdot B=g\cdot(h\cdot B)}}{\textsf{3) is proven}}
\end{flagderiv}

4) - 6) are proven in the following diagram.

\begin{flagderiv}
\introduce*{}{S: *_s\;|\;G:ps(S)\;|\;\cdot: S\rightarrow S\rightarrow S\;|\;e:S\;|\;^{-1}:S\rightarrow S\;|\;u:Group(S,G,\cdot,e,^{-1})}{}
\introduce*{}{B,C:ps(S)\;|\;v_1:B\subseteq G\;|\;v_2:C\subseteq G\;|\;g:S\;|\;w:g\varepsilon G}{}

\introduce*{}{x:S\;|\;r_1:x\;\varepsilon\; ((B\cdot C)\cdot g)}{}
\step*{}{r_1\;:\;[\exists f:S.(f\;\varepsilon \;(B\cdot C)\wedge x=f\cdot g)]}{}
\introduce*{}{f:S\;|\;r_2:f\;\varepsilon \;(B\cdot C)\wedge x=f\cdot g}{}
\step*{}{a_1:=\wedge_1(r_2)\;:\;f\;\varepsilon \;(B\cdot C)}{}
\step*{}{a_2:=\wedge_2(r_2)\;:\;x=f\cdot g}{}

\step*{}{a_1\;:\;(\exists c:S.(c\varepsilon C\wedge f\;\varepsilon\; (B\cdot c))}{}
\introduce*{}{c:S\;|\;r_3:c\varepsilon C\wedge f\;\varepsilon\; (B\cdot c)}{}
\step*{}{a_3:=\wedge_1(r_3)\;:\;c\varepsilon C}{}
\step*{}{a_4:=\wedge_2(r_3)\;:\;f\;\varepsilon\; (B\cdot c)}{}
\step*{}{a_5:=v_2ca_3\;:\;c\varepsilon G}{}

\step*{}{a_6:=\wedge(a_4, a_2)
\;:\;f\;\varepsilon\; (B\cdot c)\wedge x=f\cdot g}{}
\step*{}{a_7:=\exists_1(\lambda t:S.(t\;\varepsilon\; (B\cdot c)\wedge x=t\cdot g),f, a_6)\;:\;x\;\varepsilon\;((B\cdot c)\cdot g)}{}

\step*{}{a_8:=
term_{\ref{lemma:set-assoc}.1}(S,G,\cdot,e, ^{-1},u,B,v_1,c,g,a_5, w)\;:\;(B\cdot c)\cdot g=B\cdot(c\cdot g)}{}
\step*{}{a_9:=\wedge_1(a_8)xa_7
\;:\;x\;\varepsilon\; (B\cdot(c\cdot g))}{}

\step*{}{a_{10}:= term_{\ref{lemma:mult-triv}.3}(S,\cdot,C,g,c,a_3)\;:\;(c\cdot g)\;\varepsilon\;(C\cdot g)}{}

\step*{}{a_{11}:=\wedge(a_{10},a_9)
\;:\;(c\cdot g)\;\varepsilon\;(C\cdot g)\wedge x\;\varepsilon\; (B\cdot(c\cdot g))}{}

\step*{}{a_{12}:=\exists_1(\lambda t:S.(t\;\varepsilon\;(C\cdot g)\wedge x\;\varepsilon\; (B\cdot t)),(c\cdot g), a_{11})\;:\;x\;\varepsilon\; (B\cdot(C\cdot g))}{}

\conclude*{}{a_{13}:=\exists_3(a_1,a_{12})\;:\;x\;\varepsilon\; (B\cdot(C\cdot g))}{}

\conclude*{}{a_{14}:=\exists_3(r_1,a_{13})\;:\;x\;\varepsilon\; (B\cdot(C\cdot g))}{}

\conclude*{}{a_{15}:=
\lambda x:S.\lambda r_1:(x\;\varepsilon\;((B\cdot C)\cdot g)).a_{14}
\;:\;((B\cdot C)\cdot g)\subseteq (B\cdot(C\cdot g))}{}

\introduce*{}{x:S\;|\;r_1:x\;\varepsilon\; (B\cdot (C\cdot g))}{}
\step*{}{r_1\;:\;[\exists b:S.(b\varepsilon B\wedge x\;\varepsilon\; (b\cdot (C\cdot g)))]}{}
\introduce*{}{b:S\;|\;r_2:b\varepsilon B\wedge x\;\varepsilon\; (b\cdot (C\cdot g))}{}
\step*{}{a_{16}:=\wedge_1(r_2)\;:\;b\varepsilon B}{}
\step*{}{a_{17}:=\wedge_2(r_2)\;:\;x\;\varepsilon\; (b\cdot (C\cdot g))}{}

\step*{}{a_{18}:=v_1ba_{16}\;:\;b\varepsilon G}{}

\step*{}{a_{19}:=
term_{\ref{lemma:set-assoc}.2}(S,G,\cdot,e, ^{-1},u,C,v_2,b,g, a_{18}, w)\;:\;(b\cdot C)\cdot g=b\cdot(C\cdot g)}{}

\step*{}{a_{20}:=\wedge_2( a_{19})xa_{17}
\;:\;x\;\varepsilon\;((b\cdot C)\cdot g)}{}

\step*{}{a_{20}\;:\;[\exists f:S.(f\;\varepsilon\;(b\cdot C)\wedge x=f\cdot g)]}{}
\introduce*{}{f:S\;|\;r_3:f\;\varepsilon\;(b\cdot C)\wedge x=f\cdot g}{}
\step*{}{a_{21}:=\wedge_1(r_3)\;:\;f\;\varepsilon\;(b\cdot C)}{}
\step*{}{a_{22}:=\wedge_2(r_3)\;:\;x=f\cdot g}{}

\step*{}{a_{23}:=\wedge(a_{16}, a_{21})
\;:\;b\varepsilon B\wedge f\;\varepsilon\;(b\cdot C)}{}
\step*{}{a_{24}:=\exists_1(\lambda t:S.(t\varepsilon B\wedge f\;\varepsilon\;(t\cdot C)),b, a_{23})\;:\;f\;\varepsilon\;(B\cdot C)}{}

\step*{}{a_{25}:=\wedge(a_{24},a_{22})
\;:\;f\;\varepsilon\;(B\cdot C)\wedge x=f\cdot g}{}
\step*{}{a_{26}:=\exists_1(\lambda t:S.(t\;\varepsilon\;(B\cdot C)\wedge x=t\cdot g),f, a_{25})\;:\;x\;\varepsilon\;((B\cdot C)\cdot g)}{}

\conclude*{}{a_{27}:=\exists_3(a_{20},a_{26})\;:\;x\;\varepsilon\;((B\cdot C)\cdot g)}{}

\conclude*{}{a_{28}:=\exists_3(r_1,a_{27})\;:\;x\;\varepsilon\;((B\cdot C)\cdot g)}{}

\conclude*{}{a_{29}:=
\lambda x:S.\lambda r_1:(x\;\varepsilon\;(B\cdot (C\cdot g))). a_{28}
\;:\;(B\cdot(C\cdot g))\subseteq ((B\cdot C)\cdot g)}{}

\step*{}{\boldsymbol {term_{ \ref{lemma:set-assoc}.4}}(S,G,\cdot,e, ^{-1},u, B,C,v_1, v_2,g, w):=\wedge(a_{15}, a_{29})
\;:\;\boldsymbol{(B\cdot C)\cdot g=B\cdot(C\cdot g)}}{\textsf{4) is proven}}

\introduce*{}{x:S\;|\;r_1:x\;\varepsilon\; ((B\cdot g)\cdot C)}{}
\step*{}{r_1\;:\;[\exists c:S.(c\varepsilon C\wedge x\;\varepsilon\; ((B\cdot g)\cdot c))]}{}
\introduce*{}{c:S\;|\;r_2:c\varepsilon C\wedge x\;\varepsilon\; ((B\cdot g)\cdot c)}{}
\step*{}{a_{30}:=\wedge_1(r_2)\;:\;c\varepsilon C}{}
\step*{}{a_{31}:=\wedge_2(r_2)\;:\;x\;\varepsilon\; ((B\cdot g)\cdot c)}{}

\step*{}{a_{32}:=
v_2ca_{30}
\;:\;c\varepsilon G}{}

\step*{}{a_{33}:=
term_{\ref{lemma:set-assoc}.1}(S,G,\cdot,e, ^{-1},u,B, v_1,g,c,w,a_{32})
\;:\;(B\cdot g)\cdot c=B\cdot(g\cdot c)}{}

\step*{}{a_{34}:=
term_{\ref{lemma:mult-triv}.2}(S,\cdot,C,g,c,a_{30})\;:\;(g\cdot c)\;\varepsilon\;
(g\cdot C)}{}

\step*{}{a_{35}:=\wedge_1( a_{33})xa_{31}
\;:\;x\;\varepsilon\;(B\cdot (g\cdot c))}{}

\step*{}{a_{36}:=\wedge(a_{34},a_{35})
\;:\;(g\cdot c)\;\varepsilon\;(g\cdot C)\wedge x\;\varepsilon\;(B\cdot (g\cdot c))}{}

\step*{}{a_{37}:=\exists_1(\lambda t:S.(t\;\varepsilon\;(g\cdot C)\wedge x\;\varepsilon\;(B\cdot t)),(g\cdot c), a_{36})\;:\;x\;\varepsilon\; (B\cdot(g\cdot C))}{}

\conclude*{}{a_{38}:=\exists_3(r_1,a_{37})\;:\;x\;\varepsilon\; (B\cdot(g\cdot C))}{}

\conclude*{}{a_{39}:=
\lambda x:S.\lambda r_1:(x\;\varepsilon\;((B\cdot g)\cdot C)).a_{38}
\;:\;((B\cdot g)\cdot C)\subseteq (B\cdot(g\cdot C))}{}

\introduce*{}{x:S\;|\;r_1:x\;\varepsilon\; (B\cdot (g\cdot C))}{}
\step*{}{r_1\;:\;[\exists b:S.(b\varepsilon B\wedge x\;\varepsilon\;(b\cdot (g\cdot C))]}{}
\introduce*{}{b:S\;|\;r_2:b\varepsilon B\wedge x\;\varepsilon\;(b\cdot (g\cdot C)}{}
\step*{}{a_{40}:=\wedge_1(r_2)\;:\;b\varepsilon B}{}
\step*{}{a_{41}:=\wedge_2(r_2)\;:\;x\;\varepsilon\;(b\cdot (g\cdot C)}{}

\step*{}{a_{42}:=v_1ba_{40}\;:\;b\varepsilon G}{}

\step*{}{a_{43}:=
term_{\ref{lemma:set-assoc}.3}(S,G,\cdot,e, ^{-1},u,C,v_2,b,g,a_{42},w)\;:\;(b\cdot g)\cdot C=b\cdot(g\cdot C)}{}

\step*{}{a_{44}:=\wedge_2(a_{43})xa_{41}
\;:\;x\;\varepsilon\;((b\cdot g)\cdot C)}{}

\step*{}{a_{45}:=
term_{\ref{lemma:mult-triv}.3}(S,\cdot,B,g,b,a_{40})\;:\;(b\cdot g)\;\varepsilon\;(B\cdot g)}{}

\step*{}{a_{46}:=\wedge (a_{45},a_{44})\;:\;(b\cdot g)\;\varepsilon\;(B\cdot g)\wedge x\;\varepsilon\;((b\cdot g)\cdot C)}{}

\step*{}{a_{47}:=\exists_1(\lambda t:S.(t\varepsilon\;(B\cdot g)\wedge x\;\varepsilon\;(t\cdot C)),(b\cdot g), a_{46})\;:\;x\;\varepsilon\;((B\cdot g)\cdot C)}{}

\conclude*{}{a_{48}:=\exists_3(r_1,a_{47})\;:\;x\;\varepsilon\;((B\cdot g)\cdot C)}{}

\conclude*{}{a_{49}:=
\lambda x:S.\lambda r_1:(x\;\varepsilon\;(B\cdot (g\cdot C))). a_{48}
\;:\;(B\cdot(g\cdot C))\subseteq ((B\cdot g)\cdot C)}{}

\step*{}{\boldsymbol {term_{ \ref{lemma:set-assoc}.5}}(S,G,\cdot,e, ^{-1},u, B,C,v_1, v_2,g, w):=\wedge(a_{39}, a_{49})\;:\;\boldsymbol{(B\cdot g)\cdot C=B\cdot(g\cdot C)}}{\textsf{5) is proven}}

\introduce*{}{x:S\;|\;r_1:x\;\varepsilon\; ((g\cdot B)\cdot C)}{}
\step*{}{r_1\;:\;[\exists c:S.(c\varepsilon C\wedge x\;\varepsilon\; ((g\cdot B)\cdot c))]}{}
\introduce*{}{c:S\;|\;r_2:c\varepsilon C\wedge x\;\varepsilon\; ((g\cdot B)\cdot c)}{}
\step*{}{a_{50}:=\wedge_1(r_2)\;:\;c\varepsilon C}{}
\step*{}{a_{51}:=\wedge_2(r_2)\;:\;x\;\varepsilon\; ((g\cdot B)\cdot c)}{}

\step*{}{a_{52}:=
v_2ca_{50}
\;:\;c\varepsilon G}{}

\step*{}{a_{53}:=
term_{\ref{lemma:set-assoc}.2}(S,G,\cdot,e, ^{-1},u,B,v_1,g,c,w, a_{52})\;:\;(g\cdot B)\cdot c=g\cdot(B\cdot c)}{}
\step*{}{a_{54}:=\wedge_1(a_{53})xa_{51}
\;:\;x\;\varepsilon\;(g\cdot(B\cdot c))}{}

\step*{}{a_{54}\;:\;[\exists f:S.(f\;\varepsilon\; (B\cdot c)\wedge x=g\cdot f)]}{}
\introduce*{}{f:S\;|\;r_3:f\;\varepsilon\; (B\cdot c)\wedge x=g\cdot f}{}
\step*{}{a_{55}:=\wedge_1(r_3)\;:\;f\;\varepsilon\; (B\cdot c)}{}
\step*{}{a_{56}:=\wedge_2(r_3)\;:\;x=g\cdot f}{}

\step*{}{a_{57}:=\wedge(a_{50},a_{55})
\;:\;c\varepsilon C\wedge f\;\varepsilon\; (B\cdot c)}{}

\step*{}{a_{58}:=\exists_1(\lambda t:S.(t\varepsilon C\wedge f\;\varepsilon\; (B\cdot t)),c, a_{57})\;:\;f\;\varepsilon\; (B\cdot C)}{}

\step*{}{a_{59}:=\wedge(a_{58},a_{56})
\;:\;f\;\varepsilon\; (B\cdot C)\wedge x=g\cdot f}{}

\step*{}{a_{60}:=\exists_1(\lambda t:S.(t\;\varepsilon\; (B\cdot C)\wedge x=g\cdot t),f, a_{59})\;:\;x\;\varepsilon\;( g\cdot(B\cdot C))}{}

\conclude*{}{a_{61}:=\exists_3(a_{54},a_{60})\;:\;x\;\varepsilon\;( g\cdot(B\cdot C))}{}

\conclude*{}{a_{62}:=\exists_3(r_1,a_{61})\;:\;x\;\varepsilon\;( g\cdot(B\cdot C))}{}

\conclude*{}{a_{63}:=
\lambda x:S.\lambda r_1:(x\;\varepsilon\;((g\cdot B)\cdot C)).a_{62}
\;:\;((g\cdot B)\cdot C)\subseteq (g\cdot(B\cdot C))}{}

\introduce*{}{x:S\;|\;r_1:x\;\varepsilon\; (g\cdot(B\cdot C))}{}
\step*{}{r_1\;:\;[\exists f:S.(f\;\varepsilon \;(B\cdot C)\wedge x=g\cdot f)]}{}
\introduce*{}{f:S\;|\;r_2:f\;\varepsilon \;(B\cdot C)\wedge x=g\cdot f}{}
\step*{}{a_{64}:=\wedge_1(r_2)\;:\;f\;\varepsilon \;(B\cdot C)}{}
\step*{}{a_{65}:=\wedge_2(r_2)\;:\;x=g\cdot f}{}

\step*{}{a_{64}\;:\;[\exists c:S.(c\varepsilon  C\wedge f\;\varepsilon \;(B\cdot c))]}{}
\introduce*{}{c:S\;|\;r_3:c\varepsilon  C\wedge f\;\varepsilon \;(B\cdot c)}{}
\step*{}{a_{66}:=\wedge_1(r_3)\;:\;c\varepsilon  C}{}
\step*{}{a_{67}:=\wedge_2(r_3)\;:\;f\;\varepsilon \;(B\cdot c)}{}

\step*{}{a_{68}:= v_2ca_{66}\;:\;c\varepsilon G}{}

\step*{}{a_{69}:=\wedge(a_{67},a_{65})
\;:\;f\;\varepsilon \;(B\cdot c)\wedge x=g\cdot f}{}

\step*{}{a_{70}:=\exists_1(\lambda t:S.(t\;\varepsilon \;(B\cdot c)\wedge x=g\cdot t),f, a_{69})\;:\;x\;\varepsilon \;(g\cdot(B\cdot c))}{}

\step*{}{a_{71}:=
term_{\ref{lemma:set-assoc}.2}(S,G,\cdot,e, ^{-1},u,B,v_1,g,c,w, a_{68})\;:\;(g\cdot B)\cdot c=g\cdot(B\cdot c)}{}

\step*{}{a_{72}:=\wedge_2(a_{71})xa_{70}
\;:\;x\;\varepsilon\;((g\cdot B)\cdot c)}{}

\step*{}{a_{73}:=\wedge(a_{66},a_{72})
\;:\;c\varepsilon C\wedge x\;\varepsilon\;((g\cdot B)\cdot c)}{}

\step*{}{a_{74}:=\exists_1(\lambda t:S.(t\varepsilon C\wedge x\;\varepsilon\;((g\cdot B)\cdot t)),c, a_{73})\;:\;x\;\varepsilon\;((g\cdot B)\cdot C)}{}

\conclude*{}{a_{75}:=\exists_3(a_{64},a_{74})\;:\;x\;\varepsilon\;((g\cdot B)\cdot C)}{}

\conclude*{}{a_{76}:=\exists_3(r_1,a_{75})\;:\;x\;\varepsilon\;((g\cdot B)\cdot C)}{}

\conclude*{}{a_{77}:=
\lambda x:S.\lambda r_1:(x\;\varepsilon\;(g\cdot (B\cdot C))). a_{76}
\;:\;(g\cdot(B\cdot C))\subseteq ((g\cdot B)\cdot C)}{}

\step*{}{\boldsymbol {term_{ \ref{lemma:set-assoc}.6}}(S,G,\cdot,e, ^{-1},u, B,C,v_1, v_2,g, w):=\wedge(a_{63}, a_{77})
\;:\;\boldsymbol{(g\cdot B)\cdot C=g\cdot(B\cdot C)}}{\textsf{6) is proven}}
\end{flagderiv}
 
7) 
\vspace{-0.2cm}

\begin{flagderiv}
\introduce*{}{S: *_s\;|\;G:ps(S)\;|\;\cdot: S\rightarrow S\rightarrow S\;|\;e:S\;|\;^{-1}:S\rightarrow S\;|\;u:Group(S,G,\cdot,e,^{-1})}{}
\introduce*{}{B,C,D:ps(S)\;|\;v_1:B\subseteq G\;|\;v_2:C\subseteq G\;|\;v_3:D\subseteq G}{}

\introduce*{}{x:S\;|\;r_1:x\;\varepsilon\; ((B\cdot C)\cdot D)}{}
\step*{}{r_1\;:\;[\exists d:S.(d\varepsilon D\wedge x\;\varepsilon \;(B\cdot C)\cdot d)]}{}
\introduce*{}{d:S\;|\;r_2:d\varepsilon D\wedge x\;\varepsilon \;((B\cdot C)\cdot d)}{}
\step*{}{a_1:=\wedge_1(r_2)\;:\;d\varepsilon D}{}
\step*{}{a_2:=\wedge_1(r_2)\;:\;x\;\varepsilon \;((B\cdot C)\cdot d)}{}

\step*{}{a_3:=v_3da_1\;:\;d\varepsilon G}{}
\step*{}{a_4:=
term_{\ref{lemma:set-assoc}.4}(S,G,\cdot,e, ^{-1},u,B,C,v_1,v_2,d,a_3)\;:\;(B\cdot C)\cdot d=B\cdot(C\cdot d)}{}
\step*{}{a_5:=\wedge_1(a_4)xa_2
\;:\;x\;\varepsilon\;(B\cdot(C\cdot d))}{}
\step*{}{a_5\;:\;[\exists f:S.(f\;\varepsilon \;(C\cdot d)\wedge x\;\varepsilon\; (B\cdot f))}{}
\introduce*{}{f:S\;|\;r_3:f\;\varepsilon \;(C\cdot d)\wedge x\;\varepsilon\; (B\cdot f)}{}
\step*{}{a_6:=\wedge_1(r_3)\;:\;f\;\varepsilon \;(C\cdot d)}{}
\step*{}{a_7:=\wedge_2(r_3)\;:\;x\;\varepsilon\; (B\cdot f)}{}
\step*{}{a_8:=\wedge(a_1, a_6)
\;:\;d\varepsilon D\wedge f\;\varepsilon \;(C\cdot d)}{}

\step*{}{a_9:=\exists_1(\lambda t:S.(t\varepsilon D\wedge f\;\varepsilon \;(C\cdot t)),d, a_8)\;:\;f\;\varepsilon\;(C\cdot D)}{}

\step*{}{a_{10}:=\wedge(a_9, a_7)
\;:\;f\;\varepsilon\;(C\cdot D)\wedge x\;\varepsilon\; (B\cdot f)}{}

\step*{}{a_{11}:=\exists_1(\lambda t:S.(t\;\varepsilon\;(C\cdot D)\wedge x\;\varepsilon\; (B\cdot t)),f, a_{10})\;:\;x\;\varepsilon\;(B\cdot(C\cdot D))}{}

\conclude*{}{a_{12}:=\exists_3(a_5,a_{11})\;:\;x\;\varepsilon\;(B\cdot(C\cdot D))}{}

\conclude*{}{a_{13}:=\exists_3(r_1,a_{12})\;:\;x\;\varepsilon\;(B\cdot(C\cdot D))}{}

\conclude*{}{a_{14}:=
\lambda x:S.\lambda r_1:(x\;\varepsilon\;((B\cdot C)\cdot D)).a_{13}
\;:\;((B\cdot C)\cdot D)\subseteq (B\cdot(C\cdot D))}{}

\introduce*{}{x:S\;|\;r_1:x\;\varepsilon\; (B\cdot (C\cdot D))}{}
\step*{}{r_1\;:\;[\exists b:S.(b\varepsilon B\wedge x\;\varepsilon\; (b\cdot (C\cdot D)))]}{}
\introduce*{}{b:S\;|\;r_2:b\varepsilon B\wedge x\;\varepsilon\; (b\cdot (C\cdot D))}{}
\step*{}{a_{15}:=\wedge_1(r_2)\;:\;b\varepsilon B}{}
\step*{}{a_{16}:=\wedge_2(r_2)\;:\;x\;\varepsilon\; (b\cdot (C\cdot D))}{}

\step*{}{a_{17}:= v_1ba_{15}\;:\;b\varepsilon G}{}

\step*{}{a_{18}:=
term_{\ref{lemma:set-assoc}.6}(S,G,\cdot,e, ^{-1},u,C,D,v_2,v_3,b, a_{17})\;:\;(b\cdot C)\cdot D=b\cdot(C\cdot D)}{}

\step*{}{a_{19}:=\wedge_2(a_{18})xa_{16}
\;:\;x\;\varepsilon\;((b\cdot C)\cdot D)}{}

\step*{}{a_{19}\;:\;[\exists f:S.(f\;\varepsilon\;(b\cdot C)\wedge x\;\varepsilon\;(f\cdot D))]}{}
\introduce*{}{f:S\;|\;r_3:f\;\varepsilon\;(b\cdot C)\wedge x\;\varepsilon\;(f\cdot D)}{}
\step*{}{a_{20}:=\wedge_1(r_3)\;:\;f\;\varepsilon\;(b\cdot C)}{}
\step*{}{a_{21}:=\wedge_2(r_3)\;:\;x\;\varepsilon\;(f\cdot D)}{}

\step*{}{a_{22}:=\wedge(a_{15}, a_{20})
\;:\;b\varepsilon B\wedge f\;\varepsilon\;(b\cdot C)}{}
\step*{}{a_{23}:=\exists_1(\lambda t:S.(t\varepsilon B\wedge f\;\varepsilon\;(t\cdot C)),b, a_{22})\;:\;f\;\varepsilon\;(B\cdot C)}{}

\step*{}{a_{24}:=\wedge(a_{23},a_{21})
\;:\;f\;\varepsilon\;(B\cdot C)\wedge x\;\varepsilon\;(f\cdot D)}{}
\step*{}{a_{25}:=\exists_1(\lambda t:S.(t\;\varepsilon\;(B\cdot C)\wedge x\;\varepsilon\;(t\cdot D)),f, a_{24})\;:\;x\;\varepsilon\;((B\cdot C)\cdot D)}{}

\conclude*{}{a_{26}:=\exists_3(a_{19},a_{25})\;:\;x\;\varepsilon\;((B\cdot C)\cdot D)}{}

\conclude*{}{a_{27}:=\exists_3(r_1,a_{26})\;:\;x\;\varepsilon\;((B\cdot C)\cdot D)}{}

\conclude*{}{a_{28}:=
\lambda x:S.\lambda r_1:x\;\varepsilon\;(B\cdot (C\cdot D)). a_{27}
\;:\;B\cdot(C\cdot D)\subseteq (B\cdot C)\cdot D}{}

\step*{}{\boldsymbol {term_{ \ref{lemma:set-assoc}.7}}(S,G,\cdot,e, ^{-1},u, B,C,D,v_1, v_2,v_3):=\wedge(a_{14}, a_{28})
\;:\;\boldsymbol{(B\cdot C)\cdot D=B\cdot(C\cdot D)}}{}
\end{flagderiv}

\subsection*{Proof of Proposition \ref{lemma:coset}}
1) - 3) are proven in the following diagram.

\begin{flagderiv}
\introduce*{}{S: *_s\;|\;G:ps(S)\;|\;\cdot: S\rightarrow S\rightarrow S\;|\;e:S\;|\;^{-1}:S\rightarrow S\;|\;u:Group(S,G,\cdot,e,^{-1})}{}
\introduce*{}{H:ps(S)\;|\;v:H\leqslant G}{}
\step*{}{a_1:=\wedge_1(\wedge_1( \wedge_1(v)))\;:\;H\subseteq G}{}
\step*{}{a_2:=\wedge_2(\wedge_1( \wedge_1(v)))\;:\;e\varepsilon H}{}

\step*{}{a_3:=\wedge_2( \wedge_1(v))\;:\; Closure_1(S,H,^{-1})}{}

\step*{}{a_4:=\wedge_2(v)\;:\;Closure_2(S,H,\cdot)}{}

\introduce*{}{x:S\;|\;w:x\varepsilon H}{}
\step*{}{a_5:=a_1xw\;:\;x\varepsilon G}{}

\step*{}{a_6:=Gr_4(S,G,\cdot, e,^{-1},u,x,a_5)\;:\;x\cdot e=x}{}

\step*{}{a_7:=term_{ \ref{theorem:axiom_corollary}.6}(S,G,\cdot,e,^{-1},u,x,a_5)\;:\;(x^{-1})^{-1}=x}{}

\step*{}{a_8:=a_3xw\;:\;x^{-1}\varepsilon H}{}

\step*{}{a_9:= term_{\ref{lemma:mult-triv}.1}(S,\;^{-1},H,(x^{-1}),a_8)\;:\;(x^{-1})^{-1}\;\varepsilon\;H^{-1}}{}

\step*{}{a_{10}:=eq\mhyphen subs_1(\lambda t:S.(t\;\varepsilon\; H^{-1}), a_7,a_9)\;:\; x\;\varepsilon\;H^{-1}}{}

\step*{}{a_{11}:= term_{\ref{lemma:mult-triv}.4}(S,\cdot,H,H,x,e,w,a_2) \;:\;(x\cdot e)\;\varepsilon\;(H\cdot H)}{}

\step*{}{a_{12}:=eq\mhyphen subs_1(\lambda t:S.(t\;\varepsilon\; (H\cdot H)), a_6,a_{11})\;:\; x\;\varepsilon\;(H\cdot H)}{}

\conclude*{}{a_{13}:=\lambda x:S.\lambda w:x\varepsilon H.a_{10}\;:\;
H\subseteq H^{-1}}{}

\step*{}{a_{14}:=\lambda x:S.\lambda w:x\varepsilon H.a_{12}\;:\;
H\subseteq (H\cdot H)}{}

\step*{}{a_{15}:=term_{ \ref{lemma:subgroup}.1}(S,G,\cdot,e,^{-1},u, H, v)\;:\;Group(S,H,\cdot,e,^{-1})}{}

\step*{}{a_{16}:=\lambda x:S.\lambda w:x\varepsilon H.w\;:\;
H\subseteq H}{}

\step*{}{a_{17}:=term_{ \ref{lemma:mult}.3}(S,H,\cdot,e,^{-1},a_{15},H, a_{16})\;:\;H^{-1}\subseteq H}{}

\step*{}{a_{18}:=term_{ \ref{lemma:mult}.8}(S,H,\cdot,e,^{-1},a_{15},H,H, a_{16},a_{16})\;:\;(H\cdot H)\subseteq H}{}

\step*{}{\boldsymbol {term_{\ref{lemma:coset}.1}}(S,G,\cdot,e,^{-1},u, H, v):=\wedge(a_{17}, a_{13})\;:\;\boldsymbol{H^{-1}=H}}{\textsf{1) is proven}}

\step*{}{\boldsymbol {term_{\ref{lemma:coset}.2}}(S,G,\cdot,e,^{-1},u, H, v):=\wedge(a_{18}, a_{14})\;:\;\boldsymbol{H\cdot H=H}}{\textsf{2) is proven}}

\introduce*{}{D:ps(S)\;|\;w:D\leqslant G}{}
\step*{}{a_{19}:=\wedge_1( \wedge_1( \wedge_1(w)))\;:\;D\subseteq G}{}
\step*{}{a_{20}:=term_{ \ref{lemma:mult}.9}(S,G,\cdot,e,^{-1},u,H,D,a_1,a_{19})\;:\;(H\cdot D)^{-1}=D^{-1}\cdot H^{-1}}{}

\step*{}{a_{21}:=term_{ \ref{lemma:coset}.1}(S,G,\cdot,e,^{-1},u,H,v)\;:\;H^{-1}=H}{}
\step*{}{a_{22}:=term_{ \ref{lemma:coset}.1}(S,G,\cdot,e,^{-1},u,D,w)\;:\;D^{-1}=D}{}
\step*{}{a_{23}:=eq\mhyphen cong_1(\lambda Z:ps(S).(D^{-1}\cdot Z),a_{21})\;:\;D^{-1}\cdot H^{-1}=D^{-1}\cdot H}{}

\step*{}{a_{24}:=eq\mhyphen cong_1(\lambda Z:ps(S).(Z\cdot H), a_{22})\;:\;D^{-1}\cdot H=D\cdot H}{}

\step*{}{\boldsymbol {term_{\ref{lemma:coset}.3}}(S,G,\cdot,e,^{-1},u,H,D,v, w):=eq\mhyphen trans_1(eq\mhyphen trans_1(a_{20},a_{23}),a_{24})
\\\quad\quad
\;:\;\boldsymbol{(H\cdot D)^{-1}=D\cdot H}\hspace{6cm}\textsf{3) is proven}}{}
\end{flagderiv}

4) - 6) are proven in the following diagram.

\begin{flagderiv}
\introduce*{}{S: *_s\;|\;G:ps(S)\;|\;\cdot: S\rightarrow S\rightarrow S\;|\;e:S\;|\;^{-1}:S\rightarrow S\;|\;u:Group(S,G,\cdot,e,^{-1})}{}
\introduce*{}{H:ps(S)\;|\;v:H\leqslant G\;|\;g:S\;|\;w:g\varepsilon G}{}
\step*{}{a_1:=\wedge_1(\wedge_1(\wedge_1 (v)))\;:\;H\subseteq G}{}

\step*{}{a_2:=
term_{ \ref{lemma:coset}.1}(S,G,\cdot,e,^{-1},H,v)
\;:\;H^{-1}=H}{}
\step*{}{a_3:=
term_{ \ref{lemma:mult}.6}(S,G,\cdot,e,^{-1},H,a_1,g,w)
\;:\;(g\cdot H)^{-1}=H^{-1}\cdot g^{-1}}{}
\step*{}{a_4:=
term_{ \ref{lemma:mult}.7}(S,G,\cdot,e,^{-1},H,a_1,g,w)
\;:\;(H\cdot g)^{-1}=g^{-1}\cdot H^{-1}}{}

\step*{}{a_5:=eq\mhyphen cong_1(\lambda Z:ps(S).(Z\cdot g^{-1}),a_2)
\;:\;H^{-1}\cdot g^{-1}=H\cdot g^{-1}}{}
\step*{}{a_6:=eq\mhyphen cong_1(\lambda Z:ps(S).(g^{-1}\cdot Z),a_2)
\;:\;g^{-1}\cdot H^{-1}=g^{-1}\cdot H}{}

\step*{}{\boldsymbol {term_{\ref{lemma:coset}.4}}(S,G,\cdot,e,^{-1},u, H, v,g,w):=eq\mhyphen trans_1(a_3,a_5)
\;:\;\boldsymbol{(g\cdot H)^{-1}=H\cdot g^{-1}}\hspace{1cm} \textsf{4) is proven}}{}

\step*{}{\boldsymbol {term_{\ref{lemma:coset}.5}}(S,G,\cdot,e,^{-1},u, H, v,g,w):=eq\mhyphen trans_1(a_4,a_6)
\;:\;\boldsymbol{(H\cdot g)^{-1}=g^{-1}\cdot H}\hspace{1cm} \textsf{5) is proven}}{}

\step*{}{a_7:=\wedge_2
\wedge_1(\wedge_1(v)))\;:\;e\varepsilon H}{}
\step*{}{a_8:=Gr_3(S,G,\cdot, e,^{-1},u,g,w)\;:\;g^{-1}\varepsilon G}{}

\step*{}{a_9:= term_{\ref{lemma:mult}.4}(S,G,\cdot,e,^{-1},u,H, a_1, g^{-1},a_8)\;:\; (g^{-1}\cdot H)\subseteq G}{}

\step*{}{a_{10}:= term_{\ref{lemma:mult}.5}(S,G,\cdot,e,^{-1},u,H, a_1, g,w)\;:\; (H\cdot g)\subseteq G}{}

\step*{}{a_{11}:= term_{\ref{lemma:mult}.5}(S,G,\cdot,e,^{-1},u,(g^{-1}\cdot H), a_9,g,w)\;:\; (g^{-1}\cdot H\cdot g)\subseteq G}{}

\step*{}{a_{12}:=Gr_4(S,G,\cdot, e,^{-1},u,g^{-1},a_8)\;:\;g^{-1}\cdot e=g^{-1}}{}

\step*{}{a_{13}:=Gr_8(S,G,\cdot, e,^{-1},u,g,w)\;:\;g^{-1}\cdot g=e}{}

\step*{}{a_{14}:= term_{\ref{lemma:mult-triv}.2}(S,\cdot,H,g^{-1},e,a_7) \;:\;(g^{-1}\cdot e)\;\varepsilon\;(g^{-1}\cdot H)}{}

\step*{}{a_{15}:=eq\mhyphen subs_1(\lambda t:S.(t\;\varepsilon\;(g^{-1}\cdot H)), a_{12}, a_{14})\;:\; g^{-1}\;\varepsilon\;(g^{-1}\cdot H)}{}

\step*{}{a_{16}:= term_{\ref{lemma:mult-triv}.3}(S,\cdot,(g^{-1}\cdot H), g, g^{-1}, a_{15}) \;:\;(g^{-1}\cdot g)\;\varepsilon\;(g^{-1}\cdot H\cdot g)}{}

\step*{}{a_{17}:=eq\mhyphen subs_1(\lambda t:S.(t\;\varepsilon\;(g^{-1}\cdot H\cdot g)), a_{13}, a_{16})\;:\;e \;\varepsilon\;(g^{-1}\cdot H\cdot g)}{}

\step*{}{a_{18}:= term_{\ref{lemma:mult}.7}(S,G,\cdot,e,^{-1},u,(g^{-1}\cdot H), a_9,g,w)\;:\; (g^{-1}\cdot H\cdot g)^{-1}=g^{-1}\cdot (g^{-1}\cdot H)^{-1}}{}

\step*{}{a_{19}:= term_{\ref{lemma:coset}.4}(S,G,\cdot,e,^{-1},u,H,v,g^{-1},a_8)\;:\; (g^{-1}\cdot H)^{-1}=H\cdot (g^{-1})^{-1}}{}

\step*{}{a_{20}:= term_{\ref{theorem:axiom_corollary}.6}(S,G,\cdot, e,^{-1},u,g,w)\;:\;(g^{-1})^{-1}=g}{}

\step*{}{a_{21}:=eq\mhyphen cong_1(\lambda t:S.(H\cdot t), a_{20})\;:\;H\cdot (g^{-1})^{-1}=H\cdot g}{}

\step*{}{a_{22}:=eq\mhyphen trans_1(a_{19},a_{21})\;:\;(g^{-1}\cdot H)^{-1}=H\cdot g}{}
\step*{}{a_{23}:=eq\mhyphen cong_1(\lambda Z:ps(S).(g^{-1}\cdot Z), a_{22})\;:\; g^{-1}\cdot (g^{-1}\cdot H)^{-1}=g^{-1}\cdot (H\cdot g)}{}
\step*{}{a_{24}:= term_{\ref{lemma:set-assoc}.2}(S,G,\cdot,e,^{-1},u,H, a_1,g^{-1},g,a_8,w)\;:\; g^{-1}\cdot H\cdot g=g^{-1}\cdot (H\cdot g)}{}

\step*{}{a_{25}:=eq\mhyphen trans_2(eq\mhyphen trans_1(a_{18},a_{23}),a_{24})\;:\;(g^{-1}\cdot H\cdot g)^{-1}=g^{-1}\cdot H\cdot g}{}

\step*{}{a_{26}:=
term_{\ref{lemma:set-assoc}.5}(S,G,\cdot, e,^{-1},u,(g^{-1}\cdot H),(g^{-1}\cdot H\cdot g),a_9,a_{11},g,w) 
\\\quad\quad
 \;:\;(g^{-1} \cdot H\cdot g)\cdot(g^{-1} \cdot H\cdot g)=(g^{-1} \cdot H)\cdot(g\cdot(g^{-1} \cdot H\cdot g))}{}

\step*{}{a_{27}:=
term_{\ref{lemma:set-assoc}.2}(S,G,\cdot, e,^{-1},u,(g^{-1}\cdot H),a_9,g,g,w,w) \;:\; g\cdot(g^{-1} \cdot H)\cdot g=g\cdot(g^{-1} \cdot H\cdot g)}{}

\step*{}{a_{28}:=
term_{\ref{lemma:set-assoc}.3}(S,G,\cdot, e,^{-1},u,H,a_1,g, g^{-1},w,a_8) 
 \;:\;(g\cdot g^{-1})\cdot H =g\cdot (g^{-1}\cdot H)}{}

\step*{}{a_{29}:=eq\mhyphen cong_1(\lambda t:S.(t\cdot H),
a_{13})\;:\; (g\cdot g^{-1})\cdot H=e\cdot H}{}

\step*{}{a_{30}:=
term_{\ref{lemma:mult}.1}(S,G,\cdot, e,^{-1},u,H,a_1) 
 \;:\;e\cdot H=H}{}
 
\step*{}{a_{31}:=eq\mhyphen trans_1(eq\mhyphen trans_3(a_{28},a_{29}),a_{30})\;:\; g\cdot (g^{-1}\cdot H)=H}{} 
 
\step*{}{a_{32}:=eq\mhyphen cong_1(\lambda Z:ps(S).(Z\cdot g), a_{31})\;:\;(g\cdot (g^{-1}\cdot H))\cdot g=H\cdot g}{}
 
\step*{}{a_{33}:=eq\mhyphen trans_3(a_{27},a_{32})\;:\;g\cdot(g^{-1} \cdot H\cdot g)=H\cdot g}{} 
 
\step*{}{a_{34}:=eq\mhyphen cong_1(\lambda Z:ps(S).(g^{-1} \cdot H\cdot Z), a_{33})\;:\;(g^{-1} \cdot H)(g\cdot(g^{-1} \cdot H\cdot g))=
(g^{-1} \cdot H)\cdot(H\cdot g)}{} 
 
\step*{}{a_{35}:=
term_{\ref{lemma:set-assoc}.6}(S,G,\cdot, e,^{-1},u,H,(H\cdot g),a_1,a_{10}, g^{-1},a_8) 
 \;:\;(g^{-1} \cdot H)\cdot(H\cdot g)=g^{-1}\cdot(H\cdot(H\cdot g))}{}
 
\step*{}{a_{36}:=
term_{\ref{lemma:set-assoc}.4}(S,G,\cdot, e,^{-1},u,H,H, a_1,a_1,g, w) 
 \;:\;(H\cdot H)\cdot g =H\cdot (H\cdot g)}{}

\step*{}{a_{37}:=
term_{\ref{lemma:coset}.2}(S,G,\cdot, e,^{-1},u,H,v) 
 \;:\;H\cdot H=H}{}

\step*{}{a_{38}:=eq\mhyphen cong_1(\lambda Z:ps(S).(Z\cdot g), a_{37}) \;:\;
(H\cdot H)\cdot g=H\cdot g}{}

\step*{}{a_{39}:=eq\mhyphen trans_3(a_{36}, a_{38}) \;:\;
H\cdot(H\cdot g)=H\cdot g}{}

\step*{}{a_{40}:=eq\mhyphen cong_1(\lambda Z:ps(S).(g^{-1}\cdot Z),a_{39}) \;:\;
g^{-1}\cdot (H\cdot (H\cdot g))=g^{-1}\cdot (H\cdot g) }{}

\step*{}{a_{41}:=
term_{\ref{lemma:set-assoc}.2}(S,G,\cdot, e,^{-1},u,H,a_1, g^{-1},g,a_8,w) 
\;:\;g^{-1}\cdot H\cdot g=g^{-1}\cdot (H\cdot g)}{}

\step*{}{a_{42}:=eq\mhyphen trans_2(eq\mhyphen trans_1(eq\mhyphen trans_1(eq\mhyphen trans_1(a_{26}, a_{34}),a_{35}),a_{40}),a_{41})
\\\quad\quad
\;:\; (g^{-1}\cdot H\cdot g)\cdot(g^{-1}\cdot H\cdot g) =g^{-1}\cdot H\cdot g}{} 

\introduce*{}{x:S\;|\;r:x\;\varepsilon\; (g^{-1}\cdot H\cdot g)}{}

\step*{}{a_{43}:= term_{\ref{lemma:mult-triv}.1}(S,\cdot,(g^{-1}\cdot H\cdot g),x,r) \;:\;x^{-1}\;\varepsilon\;(g^{-1}\cdot H\cdot g)^{-1}}{}

\step*{}{a_{44}:=\wedge_1(a_{25})(x^{-1}) a_{43}\;:\;x^{-1}\;\varepsilon\;(g^{-1}\cdot H\cdot g)}{}

\conclude*{}{a_{45}:=
\lambda x:S.\lambda r:(x\;\varepsilon\;(g^{-1} \cdot H\cdot g)).a_{44}
\;:\;Closure_1(S,(g^{-1} \cdot H\cdot g),\;^{-1})}{}

\introduce*{}{x:S\;|\;r_1:x\;\varepsilon\; (g^{-1}\cdot H\cdot g)\;|\;y:S\;|\;r_2:y\;\varepsilon\; (g^{-1}\cdot H\cdot g)}{}

\step*{}{a_{46}:= term_{\ref{lemma:mult-triv}.4}(S,\cdot,(g^{-1}\cdot H\cdot g),(g^{-1}\cdot H\cdot g),x,y,  r_1,r_2)
 \;:\;(x\cdot y)\;\varepsilon\;((g^{-1}\cdot H\cdot g)\cdot (g^{-1}\cdot H\cdot g))}{}

\step*{}{a_{47}:=\wedge_1(a_{42})(x\cdot y)
a_{46}\;:\;(x\cdot y)\;\varepsilon\;(g^{-1}\cdot H\cdot g)}{}

\conclude*{}{a_{48}:=
\lambda x:S.\lambda r_1:(x\;\varepsilon\;(g^{-1} \cdot H\cdot g)).\lambda y:S.\lambda r_2:(y\;\varepsilon\;(g^{-1} \cdot H\cdot g)).a_{47}
\;:\;Closure_2(S,(g^{-1} \cdot H\cdot g),\;\cdot)}{}

\step*{}{\boldsymbol{
term_{\ref{lemma:coset}.6}}(S,G,\cdot, e,^{-1},u,H,v,g,w) 
:=\wedge(\wedge(\wedge(a_{11}, a_{17}),a_{45}),a_{48})
\;:\;\boldsymbol{
(g^{-1} \cdot H\cdot g)\leqslant G}\hspace{1cm} \textsf{6) is proven}}{}
\end{flagderiv}

7)
\begin{flagderiv}
\introduce*{}{S: *_s\;|\;G:ps(S)\;|\;\cdot: S\rightarrow S\rightarrow S\;|\;e:S\;|\;^{-1}:S\rightarrow S\;|\;u:Group(S,G,\cdot,e,^{-1})}{}

\introduce*{}{H:ps(S)\;|\;v:H\leqslant G\;|\;g:S\;|\;w:g\varepsilon G}{}
\step*{}{a_1:=\wedge_1(\wedge_1( \wedge_1(v)))\;:\;H\subseteq G}{}
\step*{}{a_2:=\wedge_2(\wedge_1( \wedge_1(v)))\;:\;e\varepsilon H}{}

\step*{}{a_3:=\wedge_2( \wedge_1(v))\;:\; Closure_1(S,H,^{-1})}{}

\step*{}{a_4:=\wedge_2(v)\;:\;Closure_2(S,H,\cdot)}{}

\step*{}{a_5:=Gr_3(S,G,\cdot, e,^{-1},u,g,w)\;:\;g^{-1}\varepsilon G}{}

\step*{}{a_6:=Gr_4(S,G,\cdot, e,^{-1},u,g,w)\;:\;g\cdot e=g}{}

\step*{}{a_7:=Gr_7(S,G,\cdot, e,^{-1},u,g,w)\;:\;g\cdot g^{-1}=e}{}

\step*{}{\text{Notation }R:=R_H\;:\; S\rightarrow S\rightarrow *_p}{}

\step*{}{a_8:= term_{\ref{lemma:eq-class}.1}(S,G,\cdot,e,^{-1}, u,H,v,g,w) \;:\;g\cdot H=Rg\cap G}{}

\assume*{}{r:(g\cdot H=H)}{}

\step*{}{a_9:= term_{\ref{lemma:mult-triv}.2}(S,\cdot,H,g,e,a_2) \;:\;(g\cdot e)\;\varepsilon\;(g\cdot H)}{}

\step*{}{a_{10}:=eq\mhyphen subs_1(\lambda t:S.(t\;\varepsilon\;(g\cdot H)),a_6,a_9) \;:\;
g\;\varepsilon\;(g\cdot H)}{}

\step*{}{a_{11}:=\wedge_1(r)ga_{10}\;:\;
g\varepsilon H}{}

\conclude*{}{a_{12}:=\lambda r:
(g\cdot H=H).a_{11}
\;:\;
(g\cdot H=H)\Rightarrow g\varepsilon H}{}

\assume*{}{r_1:g\varepsilon H}{}

\introduce*{}{z:S\;|\;r_2:z\;\varepsilon\;(g\cdot H)}{}

\step*{}{r_2:(\exists h:S.(h\varepsilon H\wedge z=g\cdot h))}{}

\introduce*{}{h:S\;|\;r_3:h\varepsilon H\wedge z=g\cdot h}{}

\step*{}{a_{13}:=\wedge_1(r_3)\;:\;h\varepsilon H}{}

\step*{}{a_{14}:=\wedge_2(r_3)\;:\;z=g\cdot h}{}

\step*{}{a_{15}:=a_4gr_1ha_{13} \;:\;(g\cdot h)\;\varepsilon \;H}{}

\step*{}{a_{16}:=eq\mhyphen subs_2(\lambda t:S.t\varepsilon H,a_{14},a_{15}) \;:\;
z\varepsilon H}{}

\conclude*{}{a_{17}:=\exists_3(r_2,a_{16}) \;:\;
z\varepsilon H}{}

\conclude*{}{a_{18}:=\lambda z:S.\lambda r_2:(z\;\varepsilon\;(g\cdot H)).a_{17}
\;:\;(g\cdot H)\subseteq H}{}

\introduce*{}{z:S\;|\;r_2:z\varepsilon H}{}

\step*{}{a_{19}:=a_1zr_2\;:\;z\varepsilon G}{}

\step*{}{a_{20}:=a_3gr_1\;:\;g^{-1}\varepsilon H}{}

\step*{}{a_{21}:= a_4(g^{-1})a_{20}zr_2\;:\;(g^{-1}\cdot z)\;\varepsilon \;H}{}

\step*{}{a_{21}\;:\;z\varepsilon Rg}{}

\step*{}{a_{22}:=\wedge(a_{21}, a_{19})\;:\;z\varepsilon (Rg\cap G)}{}

\step*{}{a_{23}:=\wedge_2(a_8)za_{22}\;:\;z\;\varepsilon\;(g\cdot H)}{}

\conclude*{}{a_{24}:=\lambda z:S.\lambda r_2:z\varepsilon H.a_{23}
\;:\;H\subseteq (g\cdot H)}{}

\step*{}{a_{25}:=\wedge(a_{18},a_{24})\;:\;(g\cdot H)=H}{}

\conclude*{}{a_{26}:=\lambda r_1:g\varepsilon H.a_{25}
\;:\;g\varepsilon H\Rightarrow
((g\cdot H)=H)}{}

\step*{}{\boldsymbol{
term_{\ref{lemma:coset}.7}}(S,G,\cdot, e,^{-1},u,H,v,g,w) 
:=\wedge(a_{12},a_{26})
\;:\;\boldsymbol{
(g\cdot H=H)\;\Leftrightarrow\;g\varepsilon H}}{}
\end{flagderiv}

\subsection*{Proof of Theorem \ref{theorem:permutable}}
\begin{flagderiv}
\introduce*{}{S: *_s\;|\;G:ps(S)\;|\;\cdot: S\rightarrow S\rightarrow S\;|\;e:S\;|\;^{-1}:S\rightarrow S\;|\;u:Group(S,G,\cdot,e,^{-1})}{}

\step*{}{a_1:=Gr_1(S,G,\cdot, e,^{-1},u)\;:\;e\varepsilon G}{}
\step*{}{a_2:=Gr_4(S,G,\cdot, e,^{-1},u,e,a_1)\;:\;e\cdot e=e}{}

\introduce*{}{B,C:ps(S)\;|\;v:B\leqslant G\;|\;w:C\leqslant G}{}
\step*{}{a_3:=\wedge_1(\wedge_1( \wedge_1(v)))\;:\;B\subseteq G}{}
\step*{}{a_4:=\wedge_2(\wedge_1( \wedge_1(v)))\;:\;e\varepsilon B}{}

\step*{}{a_5:=\wedge_1(\wedge_1( \wedge_1(w)))\;:\;C\subseteq G}{}
\step*{}{a_6:=\wedge_2(\wedge_1( \wedge_1(w)))\;:\;e\varepsilon C}{}
\step*{}{a_7:= term_{\ref{lemma:mult}.8}(S,G,\cdot,e,^{-1},u,B,C, a_3,a_5)\;:\;(B\cdot C)\subseteq G}{}

\step*{}{a_8:= term_{\ref{lemma:coset}.1}(S,G,\cdot,e,^{-1},u,B,v)\;:\;B^{-1}=B}{}

\step*{}{a_9:= term_{\ref{lemma:coset}.3}(S,G,\cdot,e,^{-1},u,B,C,v,w)\;:\;(B\cdot C)^{-1}=C\cdot B}{}

\step*{}{a_{10}:= term_{\ref{lemma:mult-triv}.4}(S,\cdot,B,C,e,e,a_4,a_6) \;:\;(e\cdot e)\;\varepsilon\;(B\cdot C)}{}

\step*{}{a_{11}:=eq\mhyphen subs_1(\lambda t:S.(t\;\varepsilon\; (B\cdot C)), a_2,a_{10})\;:\; e\;\varepsilon\;(B\cdot C)}{}

\assume*{}{r:(B\cdot C)\leqslant G}{}

\step*{}{a_{12}:= term_{\ref{lemma:coset}.1}(S,G,\cdot,e,^{-1},u,(B\cdot C), r)\;:\;(B\cdot C)^{-1}=B\cdot C}{}

\step*{}{a_{13}:=eq\mhyphen trans_3(a_{12},a_9)\;:\;B\cdot C=C\cdot B}{}

\conclude*{}{a_{14}:=\lambda 
r:((B\cdot C)\leqslant G).a_{13}
\;:\;(B\cdot C)\leqslant G\Rightarrow B\cdot C=C\cdot B}{}

\assume*{}{r_1:B\cdot C=C\cdot B}{}

\step*{}{a_{15}:=eq\mhyphen trans_2(a_9,r_1)\;:\;(B\cdot C)^{-1}=B\cdot C}{}

\step*{}{a_{16}:=
term_{\ref{lemma:set-assoc}.7}(S,G,\cdot, e,^{-1},u,B,B,C,a_3,a_3, a_5) 
 \;:\;(B\cdot B)\cdot C=B\cdot (B\cdot C)}{}

\step*{}{a_{17}:=
term_{\ref{lemma:set-assoc}.7}(S,G,\cdot, e,^{-1},u,B,C,C,a_3,a_5, a_5) 
 \;:\;(B\cdot C)\cdot C=B\cdot (C\cdot C)}{}

\step*{}{a_{18}:= term_{\ref{lemma:coset}.2}(S,G,\cdot,e,^{-1},u,B,v)\;:\;B\cdot B=B}{}

\step*{}{a_{19}:= term_{\ref{lemma:coset}.2}(S,G,\cdot,e,^{-1},u,C,w)\;:\;C\cdot C=C}{}

\step*{}{a_{20}:=eq\mhyphen cong_1(\lambda Z:ps(S).(Z\cdot C), a_{18}) \;:\;(B\cdot B)\cdot C=B\cdot C}{}

\step*{}{a_{21}:=eq\mhyphen cong_1(\lambda Z:ps(S).(B\cdot Z),a_{19}) \;:\;B\cdot(C\cdot C)=B\cdot C}{}

\step*{}{a_{22}:=eq\mhyphen trans_3(a_{16},a_{20})\;:\;B\cdot (B\cdot C)=B\cdot C}{}

\step*{}{a_{23}:=eq\mhyphen trans_1(a_{17},a_{21})\;:\;(B\cdot C)\cdot C=B\cdot C}{}

\step*{}{a_{24}:=eq\mhyphen cong_2(\lambda Z:ps(S).(Z\cdot C),r_1) \;:\;(C\cdot B)\cdot C=(B\cdot C)\cdot C}{}

\step*{}{a_{25}:=eq\mhyphen trans_1(a_{24},a_{23})\;:\;(C\cdot B)\cdot C=B\cdot C}{}

\step*{}{a_{26}:=
term_{\ref{lemma:set-assoc}.7}(S,G,\cdot, e,^{-1},u,C,B,C,a_5,a_3, a_5) 
 \;:\;(C\cdot B)\cdot C=C\cdot(B\cdot C)}{}

\step*{}{a_{27}:=eq\mhyphen trans_3(a_{26},a_{25})\;:\;C\cdot(B\cdot C)=B\cdot C}{}

\step*{}{a_{28}:=eq\mhyphen cong_1(\lambda Z:ps(S).(B\cdot Z),a_{27}) \;:\;B\cdot(C\cdot (B\cdot C))=B\cdot (B\cdot C)}{}

\step*{}{a_{29}:=
term_{\ref{lemma:set-assoc}.7}(S,G,\cdot, e,^{-1},u,B,C,( B\cdot C),a_3,a_5, a_7) 
 \;:\;(B\cdot C)\cdot (B\cdot C)=B\cdot (C\cdot(B\cdot C))}{}
\step*{}{a_{30}:=eq\mhyphen trans_1(eq\mhyphen trans_1(a_{29},a_{28}),a_{22})\;:\;(B\cdot C)\cdot (B\cdot C)=B\cdot C}{}
\introduce*{}{x:S\;|\;r_2:x\;\varepsilon\; (B\cdot C)}{}

\step*{}{a_{31}:= term_{\ref{lemma:mult-triv}.1}(S,\;^{-1},(B\cdot C),x,r_2) \;:\;x^{-1}\;\varepsilon\;(B\cdot C)^{-1}}{}

\step*{}{a_{32}:=\wedge_1( a_{15})xa_{31} \;:\; \;x^{-1}\;\varepsilon\;(B\cdot C)}{}

\conclude*{}{a_{33}:=\lambda x:S.\lambda r_2:(x\;\varepsilon\; (B\cdot C)).a_{32} \;:\;Closure_1(S,(B\cdot C),\,^{-1})
}{}

\introduce*{}{x:S\;|\;r_2:x\;\varepsilon\; (B\cdot C)\;|\;y:S\;|\;r_3:y\;\varepsilon\; (B\cdot C)}{}

\step*{}{a_{34}:= term_{\ref{lemma:mult-triv}.4}(S,\cdot,(B\cdot C),(B\cdot C),x,y, r_2,r_3) \;:\;(x\cdot y)\;\varepsilon\;(B\cdot C)\cdot (B\cdot C)}{}

\step*{}{a_{35}:=\wedge_1( a_{30})(x\cdot y)a_{34} \;:\;(x\cdot y)\;\varepsilon\;(B\cdot C)}{}

\conclude*{}{a_{36}:=\lambda x:S.\lambda r_2:(x\;\varepsilon\; (B\cdot C)).\lambda y:S.\lambda r_3:(y\;\varepsilon\; (B\cdot C)).a_{35} \;:\;Closure_2(S,(B\cdot C),\,\cdot)}{}

\step*{}{a_{37}:=\wedge(\wedge(\wedge(a_7, a_{11}),a_{33}),a_{36})\;:\;(B\cdot C)\leqslant G}{}

\conclude*{}{a_{38}:= \lambda r_1:(B\cdot C=C\cdot B).a_{37}\;:\;(B\cdot C=C\cdot B\;\Rightarrow\;
(B\cdot C)\leqslant G)}{}

\step*{}{\boldsymbol{
term_{\ref{theorem:permutable}}}(S,G,\cdot, e,^{-1},u,B,C,v, w):=\wedge(a_{14}, a_{38})
\;:\;\boldsymbol{(B\cdot C)\leqslant G\Leftrightarrow B\cdot C=C\cdot B}}{}
\end{flagderiv}

\subsection*{Proof of Proposition \ref{prop:conjugate}}
\begin{flagderiv}
\introduce*{}{S: *_s\;|\;G:ps(S)}{}

\step*{}{\text{Definition  }Subs(S,G,\cdot,e,^{-1}):=
\{B:ps(S)|B\subseteq G\}
\;:\;ps(ps((S))}{}

\step*{}{\text{Notation }M\;\text{ for }Subs(S,G,\cdot,e,^{-1})}{}

\introduce*{}{\cdot: S\rightarrow S\rightarrow S\;|\;^{-1}:S\rightarrow S}{}

\step*{}{\text{Definition  }R_c:=\lambda x,y:S.\exists g:S.(g\varepsilon G\wedge y=g^{-1}\cdot x\cdot g)\;:\;S\rightarrow S\rightarrow *_p}{}

\step*{}{\text{Definition  }R_s:=\lambda B,C:ps(S).\exists g:S.(g\varepsilon G\wedge C=g^{-1}\cdot B\cdot g)\;:\;ps(S)\rightarrow ps(S)\rightarrow *_p}{}

\introduce*{}{e:S}{}

\step*{}{\text{Definition  }Subg(S,G,\cdot,e,^{-1}):=
\{B:ps(S)\;|\;B\leqslant G\}
\;:\;ps(ps((S))}{}

\step*{}{\text{Notation }K\;\text{ for }Subg(S,G,\cdot,e,^{-1})}{}

\assume*{}{u:Group(S,G,\cdot,e,^{-1})}{}

\step*{}{a_1:=Gr_1(S,G,\cdot,e,^{-1},u)\;:\;e\varepsilon G}{}

\step*{}{a_2:=Gr_2(S,G,\cdot,e,^{-1},u)\;:\;Assoc(S,G,\cdot)}{}

\step*{}{a_3:=term_{ \ref{theorem:axiom_corollary}.8}(S,G,\cdot,e,^{-1},u)\;:\;e^{-1}=e}{}

\introduce*{}{x:S\;|\;v:x\varepsilon G}{}

\step*{}{a_4:=eq\mhyphen cong_1(\lambda t:S.(t\cdot x), a_3)
\;:\;e^{-1}\cdot x=e\cdot x}{}

\step*{}{a_5:=Gr_5(S,G,\cdot,e,^{-1},u,x,v)\;:\;e\cdot x=x}{}

\step*{}{a_6:=eq\mhyphen trans_1(a_4,a_5)
\;:\;e^{-1}\cdot x=x}{}

\step*{}{a_7:=eq\mhyphen cong_1(\lambda t:S.(t\cdot e), a_6)
\;:\;e^{-1}\cdot x\cdot e=x\cdot e}{}

\step*{}{a_8:=Gr_4(S,G,\cdot,e,^{-1},u,x,v)\;:\;x\cdot e=x}{}

\step*{}{a_9:=eq\mhyphen trans_1(a_7,a_8)
\;:\;e^{-1}\cdot x\cdot e=x}{}

\step*{}{a_{10}:=eq\mhyphen sym(a_9)
\;:\;x=e^{-1}\cdot x\cdot e}{}

\step*{}{a_{11}:=\wedge(a_1,a_9)
\;:\;e\varepsilon G\wedge x=e^{-1}\cdot x\cdot e}{}

\step*{}{a_{12}:=\exists_1(\lambda t:S.(t\varepsilon G\wedge x=t^{-1}\cdot x\cdot t),e,a_{11)}
 \;:\; R_cxx}{}

\conclude*{}{a_{13}:=\lambda x:S.\lambda v:x\varepsilon G.a_{12}
\;:\;refl(S,G,R_c)}{}

\introduce*{}{x:S\;|\;v_1:x\varepsilon G\;|\;y:S\;|\;v_2:y\varepsilon G\;|\;w:R_cxy}{}

\step*{}{w
\;:\;(\exists g:S.(g\varepsilon G\wedge y=g^{-1}\cdot x\cdot g))}{}

\introduce*{}{g:S\;|\;r:g\varepsilon G\wedge y=g^{-1}\cdot x\cdot g}{}

\step*{}{a_{14}:=\wedge_1(r)
\;:\;g\varepsilon G}{}

\step*{}{a_{15}:=\wedge_2(r)
\;:\;y=g^{-1}\cdot x\cdot g}{}

\step*{}{a_{16}:=Gr_3(S,G,\cdot,e,^{-1},u,g,a_{14})\;:\;g^{-1}\varepsilon G}{}

\step*{}{a_{17}:=eq\mhyphen cong_1(\lambda t:S.(g\cdot t), a_{15})
\;:\;g\cdot y=g\cdot(g^{-1}\cdot x\cdot g)}{}

\step*{}{a_{18}:=Gr_9(S,G,\cdot,e,^{-1},u,x,g,v_1,a_{14})\;:\;(x\cdot g)\;\varepsilon\; G}{}

\step*{}{a_{19}:=a_2ga_{14}(g^{-1})a_{16}(x\cdot g)a_{18}
\;:\;(g\cdot g^{-1})\cdot (x\cdot g)=g\cdot(g^{-1}\cdot (x\cdot g))}{}

\step*{}{a_{20}:=a_2(g^{-1})a_{16}xv_1ga_{14}
\;:\;g^{-1}\cdot x\cdot g= g^{-1}\cdot (x\cdot g)}{}

\step*{}{a_{21}:=eq\mhyphen cong_1(\lambda t:S.(g\cdot t),a_{20})
\;:\;g\cdot (g^{-1}\cdot x\cdot g)=g\cdot(g^{-1}\cdot (x\cdot g))}{}

\step*{}{a_{22}:=Gr_7(S,G,\cdot,e,^{-1},u,g,a_{14})\;:\;g\cdot g^{-1}=e}{}

\step*{}{a_{23}:=eq\mhyphen cong_1(\lambda t:S.(t\cdot (x\cdot g)), a_{22})
\;:\;g\cdot g^{-1}\cdot (x\cdot g)=e\cdot (x\cdot g)}{}

\step*{}{a_{24}:=Gr_5(S,G,\cdot,e,^{-1},u, (x\cdot g),a_{18})\;:\;e\cdot (x\cdot g)=x\cdot g}{}

\step*{}{a_{25}:=eq\mhyphen trans_1(eq\mhyphen trans_1(eq\mhyphen trans_2(eq\mhyphen trans_1(a_{17},a_{21}), a_{19}),a_{23}),a_{24})
\;:\;g\cdot y=x\cdot g}{}

\step*{}{a_{26}:=eq\mhyphen cong_1(\lambda t:S.(t\cdot g^{-1}), a_{25})
\;:\;g\cdot y\cdot g^{-1}=x\cdot g\cdot g^{-1}}{}

\step*{}{a_{27}:=a_2xv_1ga_{14}(g^{-1})a_{16}
\;:\;x\cdot g\cdot g^{-1}=x\cdot(g\cdot g^{-1})}{}

\step*{}{a_{28}:=eq\mhyphen cong_1(\lambda t:S.(x\cdot t), a_{22})
\;:\;x\cdot(g\cdot g^{-1})=x\cdot e}{}

\step*{}{a_{29}:=Gr_4(S,G,\cdot,e,^{-1},u,x,v_1)\;:\;x\cdot e=x}{}

\step*{}{a_{30}:=term_{ \ref{theorem:axiom_corollary}.6}(S,G,\cdot,e,^{-1}, u,g,a_{14})\;:\;(g^{-1})^{-1}=g}{}

\step*{}{a_{31}:=eq\mhyphen subs_2(\lambda t:S.(t\cdot y\cdot g^{-1}=x\cdot g\cdot g^{-1}), a_{30},a_{26})
\;:\;(g^{-1})^{-1}\cdot y\cdot g^{-1}=x\cdot g\cdot g^{-1}}{}

\step*{}{a_{32}:=eq\mhyphen trans_1(eq\mhyphen trans_1(eq\mhyphen trans_1(a_{31}, a_{27}), a_{28}),a_{29})
\;:\;(g^{-1})^{-1}\cdot y\cdot g^{-1}=x}{}

\step*{}{a_{33}:=eq\mhyphen sym(a_{32})
\;:\;x=(g^{-1})^{-1}\cdot y\cdot g^{-1}}{}

\step*{}{a_{34}:=\wedge(a_{16},a_{33})
\;:\;g^{-1}\varepsilon G\wedge x=(g^{-1})^{-1}\cdot y\cdot g^{-1}}{}

\step*{}{a_{35}:=\exists_1(\lambda t:S.(t\varepsilon G\wedge x=t^{-1}\cdot y\cdot t),g^{-1},a_{34})
 \;:\; R_cyx}{}

\conclude*{}{a_{36}:=\exists_3(w,a_{35})
 \;:\; R_cyx}{}

\conclude*{}{a_{37}:=
\lambda x:S.\lambda v_1:x\varepsilon G.\lambda y:S.\lambda v_2:y\varepsilon G.\lambda w:R_cxy.a_{36}
 \;:\;sym(S,G,R_c)}{}

\introduce*{}{x:S\;|\;v_1:x\varepsilon G\;|\;y:S\;|\;v_2:y\varepsilon G\;|\;z:S\;|\;v_3:z\varepsilon G\;|\;w_1:R_cxy\;|\;w_2:R_cyz}{}

\step*{}{w_1
\;:\;(\exists g:S.(g\varepsilon G\wedge y=g^{-1}\cdot x\cdot g))}{}

\introduce*{}{g:S\;|\;r_1:g\varepsilon G\wedge y=g^{-1}\cdot x\cdot g}{}

\step*{}{a_{38}:=\wedge_1(r_1)
\;:\;g\varepsilon G}{}

\step*{}{a_{39}:=\wedge_2(r_1)
\;:\;y=g^{-1}\cdot x\cdot g}{}

\step*{}{a_{40}:=Gr_3(S,G,\cdot,e,^{-1},u,g,a_{38})\;:\;g^{-1}\varepsilon G}{}

\step*{}{w_2
\;:\;(\exists h:S.(h\varepsilon G\wedge z=h^{-1}\cdot y\cdot h))}{}

\introduce*{}{h:S\;|\;r_2:h\varepsilon G\wedge z=h^{-1}\cdot y\cdot h}{}

\step*{}{a_{41}:=\wedge_1(r_2)
\;:\;h\varepsilon G}{}

\step*{}{a_{42}:=\wedge_2(r_2)
\;:\;z=h^{-1}\cdot y\cdot h}{}

\step*{}{a_{43}:=Gr_3(S,G,\cdot,e,^{-1},u,h,a_{41})\;:\;h^{-1}\varepsilon G}{}

\step*{}{a_{44}:=term_{ \ref{theorem:axiom_corollary}.5}(S,G,\cdot,e,^{-1}, u,g,h,a_{38},a_{41})\;:\;(g\cdot h)^{-1}=h^{-1}\cdot g^{-1}}{}

\step*{}{a_{45}:=eq\mhyphen cong_2(\lambda t:S.(t\cdot x), a_{44}) \;:\;h^{-1}\cdot g^{-1}\cdot x=(g\cdot h)^{-1}\cdot x}{}

\step*{}{a_{46}:= a_2(h^{-1})a_{43}(g^{-1}) a_{40}xv_1
\;:\;h^{-1}\cdot g^{-1}\cdot x=h^{-1}\cdot(g^{-1}\cdot x)}{}

\step*{}{a_{47}:= eq\mhyphen trans_3(a_{46},a_{45})
\;:\;h^{-1}\cdot(g^{-1}\cdot x)=(g\cdot h)^{-1}\cdot x}{}

\step*{}{a_{48}:=Gr_9(S,G,\cdot,e,^{-1},u,g^{-1},x,a_{40},v_1)\;:\;(g^{-1}\cdot x)\;\varepsilon \;G}{}

\step*{}{a_{49}:= a_2(h^{-1})a_{43}(g^{-1}\cdot x)a_{48}ga_{38}
\;:\;h^{-1}\cdot(g^{-1}\cdot x)\cdot g=h^{-1}\cdot (g^{-1}\cdot x\cdot g)}{}

\step*{}{a_{50}:=eq\mhyphen cong_2(\lambda t:S.(t\cdot h), a_{49}) \;:\;h^{-1}\cdot (g^{-1}\cdot x\cdot g)\cdot h= h^{-1}\cdot (g^{-1}\cdot x)\cdot g\cdot h}{}

\step*{}{a_{51}:=eq\mhyphen cong_1(\lambda t:S.(t\cdot g\cdot h), a_{47}) \;:\;h^{-1}\cdot (g^{-1}\cdot x)\cdot g\cdot h=(g\cdot h)^{-1}\cdot x\cdot g\cdot h}{}

\step*{}{a_{52}:=Gr_9(S,G,\cdot,e,^{-1},u,g,h,a_{38},a_{41})\;:\;(g\cdot h)\;\varepsilon \;G}{}

\step*{}{a_{53}:=Gr_3(S,G,\cdot,e,^{-1},u,(g\cdot h), a_{52})\;:\;(g\cdot h)^{-1}\;\varepsilon \;G}{}

\step*{}{a_{54}:=Gr_9(S,G,\cdot,e,^{-1},u,(g\cdot h)^{-1},x,a_{53},v_1)\;:\;((g\cdot h)^{-1}\cdot x)\;\varepsilon\; G}{}

\step*{}{a_{55}:= a_2((g\cdot h)^{-1}\cdot x)a_{54}ga_{38}ha_{41}
\;:\;(g\cdot h)^{-1}\cdot x\cdot g\cdot h=(g\cdot h)^{-1}\cdot x\cdot (g\cdot h)}{}

\step*{}{a_{56}:=eq\mhyphen subs_1(\lambda t:S.(z=h^{-1}\cdot t\cdot h), a_{39},a_{42})
\;:\;z=h^{-1}\cdot (g^{-1}\cdot
x\cdot g)\cdot h}{}

\step*{}{a_{57}:=eq\mhyphen trans_1(eq\mhyphen trans_1(eq\mhyphen trans_1(a_{56},a_{50}), a_{51}),a_{55})
\;:\;z=(g\cdot h)^{-1}\cdot 
x\cdot (g\cdot h)}{}

\step*{}{a_{58}:=\wedge(a_{52},a_{57})
\;:\;(g\cdot h)\;\varepsilon \;G\wedge z=(g\cdot h)^{-1}\cdot 
x\cdot (g\cdot h)}{}

\step*{}{a_{59}:=\exists_1(\lambda t:S.(t\varepsilon G\wedge z=t^{-1}\cdot x\cdot t),(g\cdot h),a_{58})
 \;:\; R_cxz}{}

\conclude*{}{a_{60}:=\exists_3(w_2,a_{59})
 \;:\; R_cxz}{}
 
\conclude*{}{a_{61}:=\exists_3(w_1,a_{60})
 \;:\; R_cxz}{}

\conclude*{}{a_{62}:=
\lambda x:S.\lambda v_1:x\varepsilon G.\lambda y:S.\lambda v_2:y\varepsilon G.\lambda z:S.\lambda v_3:z\varepsilon G.\lambda w_1:R_cxy.\lambda w_2:R_cyz.a_{61}
 \;:\;trans(S,G,R_c)}{}

\step*{}{\boldsymbol {term_{\ref{prop:conjugate}.1}}(S,G,\cdot,e,^{-1},u):= 
\wedge(\wedge(a_{13},a_{37}), a_{62})
\;:\;\boldsymbol{equiv\mhyphen rel(S,G,R_c)}\hspace{2cm}\textsf{1) is proven}}{}

\introduce*{}{B:ps(S)\;|\;v:B\varepsilon M}{}

\step*{}{v:B\subseteq G}{}

\step*{}{a_{63}:=term_{ \ref{lemma:mult}.1}(S,G,\cdot,e,^{-1},u,B,v)
\;:\;e\cdot B=B}{}

\step*{}{a_{64}:=term_{ \ref{lemma:mult}.2}(S,G,\cdot,e,^{-1},u,B,v)
\;:\;B\cdot e=B}{}

\step*{}{a_{65}:=eq\mhyphen cong_1(\lambda t:S.(t\cdot B), a_3)
\;:\;e^{-1}\cdot B=e\cdot B}{}

\step*{}{a_{66}:=eq\mhyphen trans_1(a_{65},a_{63})\;:\;e^{-1}\cdot B=B}{}
\step*{}{a_{67}:=eq\mhyphen cong_1(\lambda Z:ps(S).(Z\cdot e),a_{66})\;:\;e^{-1}\cdot B\cdot e=B\cdot e}{}

\step*{}{a_{68}:=eq\mhyphen trans_1(a_{67},a_{64})\;:\;e^{-1}\cdot B\cdot e=B}{}

\step*{}{a_{69}:=eq\mhyphen sym(a_{68})\;:\;B=e^{-1}\cdot B\cdot e}{}

\step*{}{a_{70}:=\wedge(a_1,a_{69})
\;:\;e\varepsilon G\wedge B=e^{-1}\cdot B\cdot e}{}

\step*{}{a_{71}:=\exists_1(\lambda t:S.(t\varepsilon G\wedge B=t^{-1}\cdot B\cdot t),e,a_{70})
 \;:\; R_sBB}{}

\conclude*{}{a_{72}:=\lambda B:ps(S).\lambda v:B\varepsilon M.a_{71}\;:\;refl(ps(S),M,R_s)}{}

\introduce*{}{B:ps(S)\;|\;v_1:B\varepsilon M\;|\;C:ps(S)\;|\;v_2:C\varepsilon M\;|\;w:R_sBC}{}

\step*{}{v_1
\;:\;B\subseteq G}{}

\step*{}{w\;:\; (\exists g:S.(g\varepsilon G\wedge C=g^{-1}\cdot B\cdot g))}{}

\introduce*{}{g:S\;|\;r:g\varepsilon G\wedge C=g^{-1}\cdot B\cdot g}{}

\step*{}{a_{73}:=\wedge_1(r)
\;:\;g\varepsilon G}{}

\step*{}{a_{74}:=\wedge_2(r)
\;:\;C=g^{-1}\cdot B\cdot g}{}

\step*{}{a_{75}:=Gr_3(S,G,\cdot,e,^{-1},u,g,a_{73})\;:\;g^{-1}\varepsilon G}{}

\step*{}{a_{76}:=eq\mhyphen cong_1(\lambda Z:ps(S).(g\cdot Z),a_{74})
\;:\;g\cdot C=g\cdot(g^{-1}\cdot B\cdot g)}{}

\step*{}{a_{77}:=term_{ \ref{lemma:mult}.5}(S,G,\cdot,e,^{-1},u,B,v_1,g,a_{73})\;:\;(B\cdot g)\;\subseteq\; G}{}

\step*{}{a_{78}:=term_{ \ref{lemma:set-assoc}.3}(S,G,\cdot,e,^{-1},u,(B\cdot g),a_{77},g,g^{-1},a_{73}, a_{75})
\;:\; (g\cdot g^{-1})\cdot (B\cdot g)=g\cdot(g^{-1}\cdot (B\cdot g))}{}

\step*{}{a_{79}:= term_{ \ref{lemma:set-assoc}.2}(S,G,\cdot,e,^{-1},u,B,v_1,g^{-1} ,g,a_{75}, a_{73})
\;:\;g^{-1}\cdot B\cdot g= g^{-1}\cdot (B\cdot g)}{}

\step*{}{a_{80}:=eq\mhyphen cong_1(\lambda Z:ps(S).(g\cdot Z),a_{79})
\;:\;g\cdot (g^{-1}\cdot B\cdot g)=g\cdot(g^{-1}\cdot (B\cdot g))}{}

\step*{}{a_{81}:=Gr_7(S,G,\cdot,e,^{-1},u,g,a_{73})\;:\;g\cdot g^{-1}=e}{}

\step*{}{a_{82}:=eq\mhyphen cong_1(\lambda t:S.(t\cdot (B\cdot g)),a_{81})
\;:\;g\cdot g^{-1}\cdot (B\cdot g)=e\cdot (B\cdot g)}{}

\step*{}{a_{83}:=term_{ \ref{lemma:mult}.1}(S,G,\cdot,e,^{-1},u,(B\cdot g), a_{77})
\;:\;e\cdot (B\cdot g)=B\cdot g}{}

\step*{}{a_{84}:=eq\mhyphen trans_1(eq\mhyphen trans_1(eq\mhyphen trans_2(eq\mhyphen trans_1(a_{76},a_{80}), a_{78}),a_{82}),a_{83})
\;:\;g\cdot C=B\cdot g}{}

\step*{}{a_{85}:=eq\mhyphen cong_1(\lambda t:S.(t\cdot g^{-1}),a_{84})
\;:\;g\cdot C\cdot g^{-1}=B\cdot g\cdot g^{-1}}{}

\step*{}{a_{86}:=term_{ \ref{lemma:set-assoc}.1}(S,G,\cdot,e,^{-1},u,B,v_1,g,g^{-1} ,a_{73}, a_{75})
\;:\;B\cdot g\cdot g^{-1}=B\cdot(g\cdot g^{-1})}{}

\step*{}{a_{87}:=eq\mhyphen cong_1(\lambda t:S.(B\cdot t), a_{81})
\;:\;B\cdot(g\cdot g^{-1})=B\cdot e}{}

\step*{}{a_{88}:=term_{ \ref{lemma:mult}.2}(S,G,\cdot,e,^{-1},u,B,v_1)\;:\;B\cdot e=B}{}

\step*{}{a_{89}:=term_{ \ref{theorem:axiom_corollary}.6}(S,G,\cdot,e,^{-1}, u,g,a_{73})\;:\;(g^{-1})^{-1}=g}{}

\step*{}{a_{90}:=eq\mhyphen subs_2(\lambda t:S.(t\cdot C\cdot g^{-1}=B\cdot g\cdot g^{-1}), a_{89},a_{85})
\;:\;(g^{-1})^{-1}\cdot C\cdot g^{-1}=B\cdot g\cdot g^{-1}}{}

\step*{}{a_{91}:=eq\mhyphen trans_1(eq\mhyphen trans_1(eq\mhyphen trans_1(a_{90},a_{86}), a_{87}),a_{88})
\;:\;(g^{-1})^{-1}\cdot C\cdot g^{-1}=B}{}

\step*{}{a_{92}:=eq\mhyphen sym(a_{91})
\;:\;B=(g^{-1})^{-1}\cdot C\cdot g^{-1}}{}

\step*{}{a_{93}:=\wedge(a_{75},a_{92})
\;:\;g^{-1}\varepsilon G\wedge B=(g^{-1})^{-1}\cdot C\cdot g^{-1}}{}

\step*{}{a_{94}:=\exists_1(\lambda t:S.(t\varepsilon G\wedge B=t^{-1}\cdot C\cdot t),g^{-1},a_{93}
) \;:\; R_sCB}{}

\conclude*{}{a_{95}:=\exists_3(w,a_{94})
 \;:\; R_sCB}{}

\conclude*{}{a_{96}:=
\lambda B:ps(S).\lambda v_1:B\varepsilon M.\lambda C:ps(S).\lambda v_2:C\varepsilon M.\lambda w:R_sBC.a_{95}
 \;:\;sym(ps(S),M,R_s)}{}

\introduce*{}{B:ps(S)\;|\;v_1:B\varepsilon M\;|\;C:ps(S)\;|\;v_2:C\varepsilon M\;|\;D:ps(S)\;|\;v_3:D\varepsilon M\;|\;w_1:R_sBC\;|\;w_2:R_sCD}{}

\step*{}{v_1:B\subseteq G}{}

\step*{}{w_1
\;:\;(\exists g:S.(g\varepsilon G\wedge C=g^{-1}\cdot B\cdot g))}{}

\introduce*{}{g:S\;|\;r_1:g\varepsilon G\wedge C=g^{-1}\cdot B\cdot g}{}

\step*{}{a_{97}:=\wedge_1(r_1)
\;:\;g\varepsilon G}{}

\step*{}{a_{98}:=\wedge_2(r_1)
\;:\;C=g^{-1}\cdot B\cdot g}{}

\step*{}{a_{99}:=Gr_3(S,G,\cdot,e,^{-1},u,g,a_{97})\;:\;g^{-1}\varepsilon G}{}

\step*{}{w_2
\;:\;(\exists h:S.(h\varepsilon G\wedge D=h^{-1}\cdot C\cdot h))}{}

\introduce*{}{h:S\;|\;r_2:h\varepsilon G\wedge D=h^{-1}\cdot C\cdot h}{}

\step*{}{a_{100}:=\wedge_1(r_2)
\;:\;h\varepsilon G}{}

\step*{}{a_{101}:=\wedge_2(r_2)
\;:\;D=h^{-1}\cdot C\cdot h}{}

\step*{}{a_{102}:=Gr_3(S,G,\cdot,e,^{-1},u,h,a_{100})\;:\;h^{-1}\varepsilon G}{}

\step*{}{a_{103}:=term_{ \ref{theorem:axiom_corollary}.5}(S,G,\cdot,e,^{-1}, u,g,h,a_{97},a_{100})\;:\;(g\cdot h)^{-1}=h^{-1}\cdot g^{-1}}{}

\step*{}{a_{104}:=eq\mhyphen cong_2(\lambda t:S.(t\cdot B), a_{103}) \;:\;h^{-1}\cdot g^{-1}\cdot B=(g\cdot h)^{-1}\cdot B}{}

\step*{}{a_{105}:= term_{ \ref{lemma:set-assoc}.3}(S,G,\cdot,e,^{-1},u,B,v_1,h^{-1} ,g^{-1},a_{102}, a_{99})
\;:\;h^{-1}\cdot g^{-1}\cdot B=h^{-1}\cdot(g^{-1}\cdot B)}{}

\step*{}{a_{106}:= eq\mhyphen trans_3(a_{105},a_{104})
\;:\;h^{-1}\cdot(g^{-1}\cdot B)=(g\cdot h)^{-1}\cdot B}{}

\step*{}{a_{107}:=
term_{ \ref{lemma:mult}.4}( S,G,\cdot,e,^{-1},u,B,v_1, g^{-1},a_{99})
\;:\;(g^{-1}\cdot B)\subseteq G}{}

\step*{}{a_{108}:= term_{ \ref{lemma:set-assoc}.2}(S,G,\cdot,e,^{-1},u,(g^{-1}\cdot B),a_{107},h^{-1}, g, a_{102}, a_{97})
\\\quad\quad
\;:\;h^{-1}\cdot(g^{-1}\cdot B)\cdot g=h^{-1}\cdot (g^{-1}\cdot B\cdot g)}{}

\step*{}{a_{109}:=eq\mhyphen cong_2(\lambda Z:ps(S).(Z\cdot h), a_{108}) \;:\;h^{-1}\cdot (g^{-1}\cdot B\cdot g)\cdot h= h^{-1}\cdot (g^{-1}\cdot B)\cdot g\cdot h}{}

\step*{}{a_{110}:=eq\mhyphen cong_1(\lambda Z:ps(S).(Z\cdot g\cdot h), a_{106}) \;:\;h^{-1}\cdot (g^{-1}\cdot B)\cdot g\cdot h=(g\cdot h)^{-1}\cdot B\cdot g\cdot h}{}

\step*{}{a_{111}:=Gr_9(S,G,\cdot,e,^{-1},u,g,h,a_{97},a_{100})\;:\;(g\cdot h)\;\varepsilon\;G}{}

\step*{}{a_{112}:=Gr_3(S,G,\cdot,e,^{-1},u,(g\cdot h), a_{111})\;:\;(g\cdot h)^{-1}\varepsilon G}{}

\step*{}{a_{113}:=term_{ \ref{lemma:mult}.4}(S,G,\cdot,e,^{-1},u,B,v_1,(g\cdot h)^{-1},a_{112})
\;:\;((g\cdot h)^{-1}\cdot B)\subseteq G}{}

\step*{}{a_{114}:= term_{ \ref{lemma:set-assoc}.1}(S,G,\cdot,e,^{-1},u,((g\cdot h)^{-1}\cdot B),a_{113},g,h,a_{97}, a_{100})
\\\quad\quad
\;:\;(g\cdot h)^{-1}\cdot B\cdot g\cdot h=(g\cdot h)^{-1}\cdot B\cdot (g\cdot h)}{}

\step*{}{a_{115}:=eq\mhyphen subs_1(\lambda Z:ps(S).(D=h^{-1}\cdot Z\cdot h), a_{98}, a_{101})
\;:\;D=h^{-1}\cdot (g^{-1}\cdot
B\cdot g)\cdot h}{}

\step*{}{a_{116}:=eq\mhyphen trans_1(eq\mhyphen trans_1(eq\mhyphen trans_1(a_{115},a_{109}), a_{110}),a_{114})
\;:\;D=(g\cdot h)^{-1}\cdot 
B\cdot (g\cdot h)}{}

\step*{}{a_{117}:=\wedge(a_{111},a_{116})
\;:\;(g\cdot h)\;\varepsilon \;G\wedge D=(g\cdot h)^{-1}\cdot 
B\cdot (g\cdot h)}{}

\step*{}{a_{118}:=\exists_1(\lambda t:S.(t\varepsilon G\wedge D=t^{-1}\cdot B\cdot t),(g\cdot h),a_{117})
 \;:\; R_sBD}{}

\conclude*{}{a_{119}:=\exists_3(w_2,a_{118})
 \;:\; R_sBD}{}
 
\conclude*{}{a_{120}:=\exists_3(w_1,a_{119})
 \;:\; R_sBD}{}

\conclude*{}{a_{121}:=
\lambda B:ps(S).\lambda v_1:B\varepsilon M.\lambda C:ps(S).\lambda v_2:C\varepsilon M.\lambda D:ps(S).\lambda v_3:D\varepsilon M.\lambda w_1:R_sBC.\lambda w_2:R_sCD.a_{120}
\\\hspace{10cm}
 \;:\;trans(ps(S),M,R_s)}{}

\step*{}{\boldsymbol {term_{\ref{prop:conjugate}.2}}(S,G,\cdot,e,^{-1},u):= 
\wedge(\wedge(a_{72},a_{96}), a_{121})
\;:\;\boldsymbol{equiv\mhyphen rel(ps(S),M,R_s)}\hspace{2cm}\textsf{2) is proven}}{}

\introduce*{}{B:ps(S)\;|\;v:B\varepsilon K}{}

\step*{}{v:B\leqslant G}{}

\step*{}{a_{122}:= 
\wedge_1(\wedge_1 (\wedge_1(v)))\;:\;B\subseteq G}{}

\step*{}{a_{122}\;:\;B\varepsilon M}{}

\step*{}{a_{123}:= 
a_{72}Ba_{122}\;:\;R_sBB}{}

\conclude*{}{a_{124}:= 
\lambda B:ps(S).\lambda v:
B\varepsilon K.a_{122}\;:\;K\subseteq M}{}

\step*{}{a_{125}:= 
\lambda B:ps(S).\lambda v:
B\varepsilon K.a_{123}\;:\;refl(ps(S),K,R_s)}{}

\introduce*{}{B:ps(S)\;|\;v_1:B\varepsilon K\;|\;C:ps(S)\;|\;v_2:C\varepsilon K\;|\;w:R_sBC}{}

\step*{}{a_{126}:= 
a_{124}Bv_1\;:\;B\varepsilon M}{}

\step*{}{a_{127}:= 
a_{124}Cv_2\;:\;C\varepsilon M}{}

\step*{}{a_{128}:= 
a_{96}Ba_{126}Ca_{127}w\;:\;R_sCB}{}

\conclude*{}{a_{129}:= 
\lambda B:ps(S).\lambda v_1:B\varepsilon  K.
\lambda C:ps(S).\lambda v_2:C\varepsilon  K.\lambda w:R_sBC. a_{128}\;:\;sym(ps(S),K,R_s)}{}

\introduce*{}{B:ps(S)\;|\;v_1:B\varepsilon K\;|\;C:ps(S)\;|\;v_2:C\varepsilon K\;|\;D:ps(S)\;|\;v_3:D\varepsilon K \;|\;w_1:R_sBC\;|\;w_2:R_sCD}{}

\step*{}{a_{130}:= 
a_{124}Bv_1\;:\;B\varepsilon M}{}

\step*{}{a_{131}:= 
a_{124}Cv_2\;:\;C\varepsilon M}{}

\step*{}{a_{132}:= 
a_{124}Dv_3\;:\;D\varepsilon M}{}

\step*{}{a_{133}:= 
a_{121}Ba_{130}Ca_{131}Da_{132}w_1w_2\;:\;R_sBD}{}

\conclude*{}{a_{134}:= 
\lambda B:ps(S).\lambda v_1:B\varepsilon  K.
\lambda C:ps(S).\lambda v_2:C\varepsilon  K.\lambda D:ps(S).\lambda v_3:D\varepsilon  K.\lambda w_1:R_sBC.\lambda w_2:R_sCD.
a_{133}
\\\quad\quad
\;:\;trans(ps(S),K,R_s)}{}

\step*{}{\boldsymbol {term_{\ref{prop:conjugate}.3}}(S,G,\cdot,e,^{-1},u):= 
\wedge(\wedge(a_{125},a_{129}), a_{134})
\;:\;\boldsymbol{equiv\mhyphen rel(ps(S),K,R_s)}\hspace{2cm}\textsf{3) is proven}}{}
\end{flagderiv}

\subsection*{Proof of Proposition \ref{prop:normal-triv}}

1)
\vspace{-0.2cm}

\begin{flagderiv}
\introduce*{}{S: *_s\;|\;G:ps(S)\;|\;\cdot: S\rightarrow S\rightarrow S\;|\;e:S\;|\;^{-1}:S\rightarrow S\;|\;u:Group(S,G,\cdot,e,^{-1})}{}

\step*{}{a_1:= term_{ \ref{lemma:triv-subgroups}.1}(S,G,\cdot,e,^{-1},u)
\;:\;G\leqslant G}{}

\introduce*{}{g:S\;|\;v:g\varepsilon G\;|\;h:S\;|\;w:h\varepsilon G}{}

\step*{}{a_2:=Gr_3(S,G,\cdot,e,^{-1},u,g,v)
\;:\;g^{-1}\;\varepsilon\; G}{}

\step*{}{a_3:=Gr_9(S,G,\cdot,e,^{-1},u,g^{-1},h,a_2,w)
\;:\;(g^{-1}\cdot h)\;\varepsilon\; G}{}

\step*{}{a_4:=Gr_9(S,G,\cdot,e,^{-1},u,(g^{-1}\cdot h),g,a_3,v)
\;:\;(g^{-1}\cdot h\cdot g)\;\varepsilon\; G}{}

\conclude*{}{a_5:=
\lambda g:S.\lambda v:g\varepsilon G.\lambda h:S.\lambda w:h\varepsilon G.a_4
\;:\;\forall g:S.[g\varepsilon G \Rightarrow \forall h:S.(h\varepsilon G \Rightarrow 
(g^{-1}\cdot h\cdot g)\;\varepsilon\; G)]}{}

\step*{}{\boldsymbol {term_{\ref{prop:normal-triv}.1}}(S,G,\cdot,e,^{-1},u):= 
\wedge(a_1,a_5)
\;:\;\boldsymbol{G\triangleleft G}}{}
\end{flagderiv}

2)
\vspace{-0.2cm}

\begin{flagderiv}
\introduce*{}{S: *_s\;|\;G:ps(S)\;|\;\cdot: S\rightarrow S\rightarrow S\;|\;e:S\;|\;^{-1}:S\rightarrow S\;|\;u:Group(S,G,\cdot,e,^{-1})}{}

\step*{}{\text{Notation }H:=\{x:S\;|\;x=e\}
\;:\;ps(S)}{}

\step*{}{a_1:= term_{ \ref{lemma:triv-subgroups}.2}(S,G,\cdot,e,^{-1},u)
\;:\;H\leqslant G}{}

\introduce*{}{g:S\;|\;v:g\varepsilon G\;|\;h:S\;|\;w:h\varepsilon H}{}

\step*{}{w\;:\;h=e}{}

\step*{}{a_2:=Gr_3(S,G,\cdot,e,^{-1},u,g,v)
\;:\;g^{-1}\;\varepsilon\; G}{}

\step*{}{a_3:=Gr_4(S,G,\cdot,e,^{-1},u,g^{-1},a_2)
\;:\;g^{-1}\cdot e=g^{-1}}{}

\step*{}{a_4:=Gr_8(S,G,\cdot,e,^{-1},u,g,v)
\;:\;g^{-1}\cdot g=e}{}

\step*{}{a_5:=eq\mhyphen cong_1(\lambda t:S.(g^{-1}\cdot t\cdot g),w)
\;:\;g^{-1}\cdot h\cdot g=g^{-1}\cdot e\cdot g}{}

\step*{}{a_6:=eq\mhyphen cong_1(\lambda t:S.(t\cdot g), a_3)
\;:\;g^{-1}\cdot e\cdot g=g^{-1}\cdot g}{}

\step*{}{a_7:=eq\mhyphen trans_1(eq\mhyphen trans_1(a_5,a_6),a_4)
\;:\;g^{-1}\cdot h\cdot g=e}{}

\step*{}{a_7\;:\;(g^{-1}\cdot h\cdot g)\;\varepsilon\;H}{}

\conclude*{}{a_8:=
\lambda g:S.\lambda v:g\varepsilon G.\lambda h:S.\lambda w:h\varepsilon H.a_7
\;:\;\forall g:S.[g\varepsilon G \Rightarrow \forall h:S.(h\varepsilon H \Rightarrow 
(g^{-1}\cdot h\cdot g)\;\varepsilon\; H)]}{}

\step*{}{\boldsymbol {term_{\ref{prop:normal-triv}.2}}(S,G,\cdot,e,^{-1},u):= 
\wedge(a_1,a_8)
\;:\;\boldsymbol{H\triangleleft G}}{}
\end{flagderiv}

\subsection*{Proof of Proposition \ref{prop:normal-abel}}

\begin{flagderiv}
\introduce*{}{S: *_s\;|\;G:ps(S)\;|\;\cdot: S\rightarrow S\rightarrow S\;|\;e:S\;|\;^{-1}:S\rightarrow S\;|\;u:Abelian\mhyphen group(S,G,\cdot,e,^{-1})}{}

\step*{}{a_1:=\wedge_1(u)
\;:\;Group(S,G,\cdot,e,^{-1})}{}

\step*{}{a_2:=\wedge_2(u)
\;:\;Commut(S,G,\cdot)}{}

\step*{}{a_3:=Gr_2(S,G,\cdot,e,^{-1},a_1)\;:\; Assoc(S,G,\cdot)}{}

\introduce*{}{H:ps(S)\;|\;v:H\leqslant G}{}

\step*{}{a_4:=\wedge_1(\wedge_1(\wedge_1(v)))
\;:\;H\subseteq G}{}

\introduce*{}{g:S\;|\;w_1:g\varepsilon G\;|\;h:S\;|\;w_2:h\varepsilon H}{}

\step*{}{a_5:=Gr_3(S,G,\cdot,e,^{-1},a_1,g,w_1)\;:\; g^{-1}\varepsilon G}{}

\step*{}{a_6:=a_4hw_2\;:\;h\varepsilon G}{}

\step*{}{a_7:= a_2(g^{-1})a_5ha_6\;:\;g^{-1}\cdot h=h\cdot g^{-1}}{}

\step*{}{a_8:=eq\mhyphen cong_1(\lambda t:S.(t\cdot g), a_7)
\;:\;g^{-1}\cdot h\cdot g=h\cdot g^{-1}\cdot g
}{}

\step*{}{a_9:= a_3ha_6(g^{-1})a_5gw_1\;:\;h\cdot g^{-1}\cdot g=h\cdot (g^{-1}\cdot g)}{}

\step*{}{a_{10}:=Gr_8(S,G,\cdot,e,^{-1},a_1,g,w_1)\;:\; g^{-1}\cdot g=e}{}

\step*{}{a_{11}:=eq\mhyphen cong_1(\lambda t:S.(h\cdot t), a_{10})
\;:\;h\cdot (g^{-1}\cdot g)=h\cdot e}{}

\step*{}{a_{12}:=Gr_4(S,G,\cdot,e,^{-1},a_1,h,a_6)\;:\; h\cdot e=h}{}

\step*{}{a_{13}:=eq\mhyphen trans_1(eq\mhyphen trans_1(eq\mhyphen trans_1(a_8, a_9),a_{11}),a_{12})
\;:\;g^{-1}\cdot h\cdot g=h}{}

\step*{}{a_{14}:=eq\mhyphen subs_2(\lambda t:S.t\varepsilon H,a_{13},w_2)
\;:\;(g^{-1}\cdot h\cdot g)\;\varepsilon \;H}{}

\conclude*{}{a_{15}:=
\lambda g:S.\lambda w_1:g\varepsilon G.\lambda h:S.\lambda w_2:h\varepsilon H.a_{14}
\;:\;\forall g:S.[g\varepsilon G\Rightarrow \forall h:S.(h\varepsilon H\Rightarrow 
(g^{-1}\cdot h\cdot g)\;\varepsilon \;H)]}{}

\step*{}{\boldsymbol {term_{\ref{prop:normal-abel}}}(S,G,\cdot,e,^{-1},u,H,v):= 
\wedge(v,a_{15})
\;:\;\boldsymbol{H\triangleleft G}}{}
\end{flagderiv}

\newpage
\subsection*{Proof of Proposition \ref{prop:normal-inters}}

\begin{flagderiv}
\introduce*{}{S: *_s\;|\;G:ps(S)\;|\;\cdot: S\rightarrow S\rightarrow S\;|\;e:S\;|\;^{-1}:S\rightarrow S\;|\;u:Group(S,G,\cdot,e,^{-1})}{}

\step*{}{a_1:=Gr_1(S,G,\cdot,e,^{-1},u)\;:\; e\varepsilon G}{}

\step*{}{a_2:=Gr_2(S,G,\cdot,e,^{-1},u)\;:\; Assoc(S,G,\cdot)}{}

\introduce*{}{H:ps(S)\;|\;v:H\leqslant G}{}

\step*{}{U:=\{Z:ps(S)\;|\;\exists x:S.(x\varepsilon G\wedge Z=x^{-1}\cdot H\cdot x)\}\;:\;ps(ps(S))}{}

\step*{}{N:=\cap U: ps(S)}{}

\step*{}{a_3:=\wedge_1(\wedge_1( \wedge_1(v)))\;:\;H\subseteq G}{}

\step*{}{a_4:=\wedge_2( \wedge_1(v))\;:\;Closure_1(S,H,^{-1})}{}

\step*{}{a_5:=\wedge_2(v)\;:\;Closure_2(S,H,\cdot)}{}

\step*{}{a_6:=
eq\mhyphen refl\;:\;e^{-1}\cdot H\cdot e=e^{-1}\cdot H\cdot e}{}

\step*{}{a_7:=\wedge(a_1,a_6)
\;:\;e\varepsilon G
\wedge e^{-1}\cdot H\cdot e=e^{-1}\cdot H\cdot e}{}

\step*{}{a_8:=\exists_1(\lambda t:S.(t\varepsilon G\wedge e^{-1}\cdot H\cdot e=t^{-1}\cdot H\cdot t),e, a_7)\;:\;(e^{-1}\cdot H\cdot e)\;\varepsilon\;U}{}

\step*{}{a_9:=\exists_1(\lambda Z:ps(S).Z\varepsilon U,(e^{-1}\cdot H\cdot e), a_8)\;:\;(\exists Z:ps(S).Z\varepsilon U)}{}

\introduce*{}{Z:ps(S)\;|\;w:Z\;\varepsilon\;U}{}

\step*{}{w\;:\;(\exists x:S.(x\varepsilon G\wedge Z=x^{-1}\cdot H\cdot x))}{}

\introduce*{}{x:S\;|\;r:x\varepsilon G\wedge Z=x^{-1}\cdot H\cdot x}{}

\step*{}{a_{10}:=\wedge_1(r)\;:\;x\varepsilon G}{}

\step*{}{a_{11}:=\wedge_2(r)\;:\;Z=x^{-1}\cdot H\cdot x}{}

\step*{}{a_{12}:=
term_{\ref{lemma:coset}.6}(S,G,\cdot,e ,^{-1},u,H,v,x,a_{10})
\;:\;(x^{-1}\cdot H\cdot x)\leqslant G}{}

\step*{}{a_{13}:=
eq\mhyphen subs_2(\lambda X:ps(S).X\leqslant G,a_{11}, a_{12})\;:\;Z\leqslant G}{}

\conclude*{}{a_{14}:=
\exists_3(w,a_{13})\;:\;Z\leqslant G}{}

\conclude*{}{a_{15}:=
\lambda Z:ps(S).\lambda w:Z\varepsilon U.a_{14}
\;:\;[\forall Z:ps(S).(Z\varepsilon U\Rightarrow Z\leqslant G)]}{}

\step*{}{a_{16}:=
term_{\ref{lemma:triv-subgroups}.4}(S,G,\cdot,e ,^{-1},u,U,a_9,a_{15})
\;:\;\cap U\leqslant G}{}

\step*{}{a_{16}\;:\;N\leqslant G}{}

\introduce*{}{g:S\;|\;w_1:g\varepsilon G\;|\;h:S\;|\;w_2:h\varepsilon N}{}

\step*{}{a_{17}:=Gr_3(S,G,\cdot,e,^{-1},u,g,w_1)\;:\; g^{-1}\varepsilon G}{}

\introduce*{}{Z:ps(S)\;|\;r_1:Z\;\varepsilon\;U}{}

\step*{}{r_1\;:\;(\exists x:S.(x\varepsilon G\wedge Z=x^{-1}\cdot H\cdot x))}{}

\introduce*{}{x:S\;|\;r_2:x\varepsilon G\wedge Z=x^{-1}\cdot H\cdot x}{}

\step*{}{a_{18}:=\wedge_1(r_2)\;:\;x\varepsilon G}{}

\step*{}{a_{19}:=\wedge_2(r_2)\;:\;Z=x^{-1}\cdot H\cdot x}{}

\step*{}{\text{Notation }\;
y:=x\cdot g^{-1}\;:\;S}{}

\step*{}{a_{20}:=Gr_9(S,G,\cdot,e,^{-1},u,x,g^{-1},a_{18}, a_{17})\;:\;y\varepsilon G}{}

\step*{}{a_{21}:=Gr_3(S,G,\cdot,e,^{-1},u,y, a_{20}) \;:\;y^{-1}\;\varepsilon\; G}{}

\step*{}{\text{Notation }\;
X:=y^{-1}\cdot H\cdot y\;:\;ps(S)}{}

\step*{}{a_{22}:=eq\mhyphen refl
\;:\;X=y^{-1}\cdot H\cdot y}{}

\step*{}{a_{23}:=\wedge(a_{20},a_{22})
\;:\;y\varepsilon G
\wedge X=y^{-1}\cdot H\cdot y}{}

\step*{}{a_{24}:=\exists_1(\lambda t:S.(t\varepsilon G\wedge X=t^{-1}\cdot H\cdot t),y, a_{23})\;:\;X\varepsilon U}{}

\step*{}{a_{25}:=w_2Xa_{24}\;:\;h\varepsilon X}{}

\step*{}{a_{25}\;:\;h\;\varepsilon\;(y^{-1}\cdot H\cdot y)}{}

\step*{}{a_{26}:=Gr_8(S,G,\cdot,e,^{-1},u,g,w_1) \;:\;g^{-1}\cdot g=e}{}

\step*{}{a_{27}:=a_2xa_{18}g^{-1}a_{17}gw_1
 \;:\;(x\cdot g^{-1})\cdot g=x\cdot (g^{-1}\cdot g)}{}

\step*{}{a_{28}:=eq\mhyphen cong_1(\lambda t:S.(x\cdot t), a_{26})
\;:\;x\cdot (g^{-1}\cdot g)=x\cdot e}{}

\step*{}{a_{29}:=Gr_4(S,G,\cdot,e,^{-1},u,x,a_{18}) \;:\;x\cdot e=x}{}

\step*{}{a_{30}:=eq\mhyphen trans_1(eq\mhyphen trans_1(a_{27}, a_{28}),a_{29})
\;:\;y\cdot g=x}{}

\step*{}{a_{31}:=
term_{\ref{theorem:axiom_corollary}.5}(S,G,\cdot,e ,^{-1},u, y,g,a_{20},w_1)
\;:\;(y\cdot g)^{-1}=g^{-1}\cdot y^{-1}}{}

\step*{}{a_{32}:=eq\mhyphen cong_1(^{-1},a_{30})
\;:\;(y\cdot g)^{-1}=x^{-1}}{}

\step*{}{a_{33}:=eq\mhyphen trans_3(a_{31}, a_{32})
\;:\;g^{-1}\cdot y^{-1}=x^{-1}}{}

\step*{}{a_{34}:=
term_{ \ref{lemma:mult}.4}(S,G,\cdot,e ,^{-1},u, H,a_3,y^{-1},a_{21})
\;:\;(y^{-1}\cdot H)\subseteq G}{}

\step*{}{a_{35}:=
term_{ \ref{lemma:mult}.5}(S,G,\cdot,e ,^{-1},u, (y^{-1}\cdot H), a_{34},y,a_{20})
\;:\;(y^{-1}\cdot H\cdot y)\subseteq G}{}

\step*{}{a_{36}:=
term_{ \ref{lemma:mult-triv}.2}(S,\cdot,(y^{-1}\cdot H\cdot y),g^{-1},h,a_{25})
\;:\;(g^{-1}\cdot h)\;\varepsilon\;(g^{-1}\cdot (y^{-1}\cdot H\cdot y))}{}

\step*{}{a_{37}:=
a_{35}ha_{25}
\;:\;h\varepsilon G}{}

\step*{}{a_{38}:=Gr_3(S,G,\cdot,e,^{-1},u,h,a_{37}) \;:\;h^{-1}\;\varepsilon\;G}{}

\step*{}{a_{39}:=
term_{ \ref{lemma:set-assoc}.3}(S,G,\cdot,e ,^{-1},u, H,a_3,g^{-1}, y^{-1},a_{17},a_{21})
\;:\;g^{-1}\cdot y^{-1}\cdot H =g^{-1}\cdot (y^{-1}\cdot H)}{}

\step*{}{a_{40}:=eq\mhyphen cong_1(\lambda Z:ps(S).(Z\cdot y), a_{39})
\;:\;g^{-1}\cdot y^{-1}\cdot H\cdot y=g^{-1}\cdot (y^{-1}\cdot H)\cdot y}{}

\step*{}{a_{41}:=
term_{ \ref{lemma:set-assoc}.2}(S,G,\cdot,e ,^{-1},u, (y^{-1}\cdot H), a_{34},g^{-1},y,a_{17},a_{20})
\\\quad\quad
\;:\;g^{-1}\cdot (y^{-1}\cdot H)\cdot y=g^{-1}\cdot (y^{-1}\cdot H\cdot y)}{}

\step*{}{a_{42}:=eq\mhyphen cong_1(\lambda t:S.(t\cdot H\cdot y), a_{33})
\;:\;g^{-1}\cdot y^{-1}\cdot H\cdot y=x^{-1}\cdot H\cdot y}{}

\step*{}{a_{43}:=eq\mhyphen trans_3(eq\mhyphen trans_1(a_{40}, a_{41}), a_{42})
\;:\;g^{-1}\cdot (y^{-1}\cdot H\cdot y)=x^{-1}\cdot H\cdot y}{}

\step*{}{a_{44}:=\wedge_1(a_{43})(g^{-1}\cdot h)a_{36}
\;:\;(g^{-1}\cdot h)\;\varepsilon\;(x^{-1}\cdot H\cdot y)}{}

\step*{}{a_{45}:=Gr_3(S,G,\cdot,e,^{-1},u,x,a_{18}) \;:\;x^{-1}\;\varepsilon\;G}{}

\step*{}{a_{46}:=
term_{ \ref{lemma:mult}.4}(S,G,\cdot,e ,^{-1},u,H,a_3, x^{-1}, a_{45})
\;:\;(x^{-1}\cdot H)\subseteq G}{}

\step*{}{a_{47}:=
term_{ \ref{lemma:mult-triv}.3}(S,\cdot,(x^{-1}\cdot H\cdot y),g,(g^{-1}\cdot h),a_{44})
\;:\;(g^{-1}\cdot h\cdot g)\;\varepsilon\;(x^{-1}\cdot H\cdot y\cdot g)}{}

\step*{}{a_{48}:=
term_{ \ref{lemma:set-assoc}.1}(S,G,\cdot,e ,^{-1},u,(x^{-1}\cdot H),a_{46},y,g, a_{20},w_1)
\;:\;x^{-1}\cdot H\cdot y\cdot g=x^{-1}\cdot H\cdot (y\cdot g)}{}

\step*{}{a_{49}:=eq\mhyphen cong_1(\lambda t:S.(x^{-1}\cdot H\cdot t), a_{30})
\;:\;x^{-1}\cdot H\cdot (y\cdot g)=x^{-1}\cdot H\cdot x}{}

\step*{}{a_{50}:=eq\mhyphen trans_1( a_{48},a_{49})
\;:\;x^{-1}\cdot H\cdot y\cdot g=x^{-1}\cdot H\cdot x}{}

\step*{}{a_{51}:=\wedge_1( a_{50})(g^{-1}\cdot h\cdot g)a_{47}
\;:\;(g^{-1}\cdot h\cdot g)\;\varepsilon\;(x^{-1}\cdot H\cdot x)}{}

\step*{}{a_{52}:=\wedge_2( a_{19})(g^{-1}\cdot h\cdot g)a_{51}
\;:\;(g^{-1}\cdot h\cdot g)\;\varepsilon\;Z}{}

\conclude*{}{a_{53}:=\exists_3(r_1,a_{52})
\;:\;(g^{-1}\cdot h\cdot g)\;\varepsilon\;Z}{}

\conclude*{}{a_{54}:=
\lambda Z:ps(S).\lambda r_1:Z\varepsilon U.a_{53}
\;:\;(g^{-1}\cdot h\cdot g)\;\varepsilon\;N}{}

\conclude*{}{a_{55}:=
\lambda g:S.\lambda w_1:g\varepsilon G.\lambda h:S.\lambda w_2:h\varepsilon N.a_{54}\;:\;
\forall g:S.[g\varepsilon G\Rightarrow \forall h:S.(h\varepsilon N\Rightarrow 
(g^{-1}\cdot h\cdot g)\;\varepsilon\;N)]}{}

\step*{}{\boldsymbol {term_{\ref{prop:normal-inters}}}(S,G,\cdot,e,^{-1},u,H,v):= 
\wedge(a_{16},a_{55})
\;:\;\boldsymbol{N\triangleleft G}}{}
\end{flagderiv}

\subsection*{Proof of Proposition \ref{prop:normal-criteria}}

\begin{flagderiv}
\introduce*{}{S: *_s\;|\;G:ps(S)\;|\;\cdot: S\rightarrow S\rightarrow S\;|\;e:S\;|\;^{-1}:S\rightarrow S\;|\;u:Group(S,G,\cdot,e,^{-1})}{}

\step*{}{a_1:=Gr_2(S,G,\cdot,e,^{-1},u)\;:\; Assoc(S,G,\cdot)}{}

\introduce*{}{g:S\;|\;v:g\varepsilon G}{}

\step*{}{a_2(g,v):=Gr_3(S,G,\cdot,e,^{-1},u,g,v)\;:\; g^{-1}\varepsilon G}{}

\step*{}{a_3(g,v):=Gr_8(S,G,\cdot,e,^{-1},u,g,v)\;:\; g^{-1}\cdot g=e}{}

\introduce*{}{F:ps(S)\;|\;w:F\subseteq G}{}

\step*{}{a_4:=eq\mhyphen cong_1(\lambda t:S.(F\cdot t), a_3(g,v))
\;:\;F\cdot (g^{-1}\cdot g)=F\cdot e}{}

\step*{}{a_5:=
term_{ \ref{lemma:mult}.2}(S,G,\cdot,e ,^{-1},u,F,w)\;:\;F\cdot e=F}{}

\step*{}{a_6:=
term_{ \ref{lemma:set-assoc}.1}(S,G,\cdot,e ,^{-1},u,F,w,g^{-1}, g, a_2(g,v), v)\;:\;F\cdot g^{-1}\cdot g=F\cdot (g^{-1}\cdot g)}{}

\step*{}{a_7(g,v,F,w):=eq\mhyphen trans_1(eq\mhyphen trans_1(a_6,a_4), a_5)\;:\;F\cdot g^{-1}\cdot g=F}{}

\done[2]

\introduce*{}{H:ps(S)\;|\;v:H\leqslant G}{}

\step*{}{\text{Notation }\;
A:=\forall g:S.(g\varepsilon G
\Rightarrow (g^{-1}\cdot H\cdot g)=H)\;:\;*_p}{}

\step*{}{\text{Notation }\;
B:=\forall g:S.(g\varepsilon G
\Rightarrow (g\cdot H\cdot g^{-1})=H)\;:\;*_p}{}

\step*{}{\text{Notation }\;
C:=\forall g:S.(g\varepsilon G
\Rightarrow g\cdot H=H\cdot g)\;:\;*_p}{}

\step*{}{\text{Notation }\;
D:=\forall g:S.(g\varepsilon G
\Rightarrow (g^{-1}\cdot H\cdot g)\subseteq H)\;:\;*_p}{}

\step*{}{a_8:=\wedge_1(\wedge_1( \wedge_1(v)))\;:\;H\subseteq G}{}

\introduce*{}{g:S\;|\;w:g\varepsilon G\;|\;h:S\;|\;r:h\varepsilon H}{}

\step*{}{a_9:=a_2(g,w)\;:\; g^{-1}\varepsilon G}{}

\step*{}{a_{10}:=
term_{ \ref{lemma:mult-triv}.2}(S,\cdot,H,g^{-1},h,r)\;:\;(g^{-1}\cdot h)\;\varepsilon\;(g^{-1}\cdot H)}{}

\step*{}{a_{11}(g,h,w,r):=
term_{ \ref{lemma:mult-triv}.3}(S,\cdot,(g^{-1}\cdot H),g, (g^{-1}\cdot h),a_{10})\;:\;(g^{-1}\cdot h\cdot g)\;\varepsilon\;(g^{-1}\cdot H\cdot g)}{}

\done

\assume*{}{w:H\triangleleft G\;|\;g:S\;|\;r_1:g\varepsilon G}{}

\assume*{}{x:S\;|\;r_2:x\;\varepsilon \;(g^{-1}\cdot H\cdot g)}{}

\step*{}{r_2\;:\;[\exists b:S.(b\;\varepsilon \;(g^{-1}\cdot H)\wedge x=b\cdot g)]}{}

\introduce*{}{b:S\;|\;r_3:b\;\varepsilon \;(g^{-1}\cdot H)\wedge x=b\cdot g}{}

\step*{}{a_{12}:=\wedge_1(r_3)\;:\;b\;\varepsilon \;(g^{-1}\cdot H)}{}

\step*{}{a_{13}:=\wedge_2(r_3)\;:\;x=b\cdot g}{}

\step*{}{a_{12}\;:\;(\exists h:S.(h\varepsilon H\wedge b=g^{-1}\cdot h)]}{}

\introduce*{}{h:S\;|\;r_4:h\varepsilon H\wedge b=g^{-1}\cdot h}{}

\step*{}{a_{14}:=\wedge_1(r_4)\;:\;h\varepsilon H}{}

\step*{}{a_{15}:=\wedge_2(r_4)\;:\;b=g^{-1}\cdot h}{}

\step*{}{a_{16}:=
eq\mhyphen cong_1(\lambda t:S.(t\cdot g),a_{15})\;:\;b\cdot g=g^{-1}\cdot h\cdot g}{}

\step*{}{a_{17}:=
eq\mhyphen trans_1(a_{13},a_{16})
\;:\;x=g^{-1}\cdot h\cdot g}{}

\step*{}{a_{18}:=
\wedge_2(w)gr_1ha_{14}
\;:\;(g^{-1}\cdot h\cdot g)\;\varepsilon\;H}{}

\step*{}{a_{19}:=
eq\mhyphen subs_2(\lambda t:S.t\varepsilon H,a_{17}, a_{18})\;:\;x\varepsilon H}{}

\conclude*{}{a_{20}:=
\exists_3(a_{12}, a_{19})\;:\;x\varepsilon H}{}

\conclude*{}{a_{21}:=
\exists_3(r_2, a_{20})\;:\;x\varepsilon H}{}

\conclude*{}{a_{22}:=
\lambda x:S.\lambda r_2:x\;\varepsilon \;(g^{-1}\cdot H\cdot g).a_{21}\;:\;(g^{-1}\cdot H\cdot g)\subseteq H}{}

\assume*{}{h:S\;|\;r_2:h\varepsilon H}{}

\step*{}{a_{23}:=
a_2(g,r_1)\;:\;g^{-1}\varepsilon G}{}

\step*{}{a_{24}:=
term_{\ref{theorem:axiom_corollary}.6}(S,G,\cdot,e ,^{-1},u,g,r_1)\;:\;(g^{-1})^{-1}=g}{}

\step*{}{a_{25}:=
a_8hr_2\;:\;h\varepsilon G}{}

\step*{}{x:=
g\cdot h\cdot g^{-1}}{}

\step*{}{a_{26}:=eq\mhyphen cong_1(\lambda t:S.(t\cdot h\cdot g^{-1}), a_{24})
\;:\;(g^{-1})^{-1}\cdot h\cdot g^{-1}=g\cdot h\cdot g^{-1}}{}

\step*{}{a_{26}
\;:\;(g^{-1})^{-1}\cdot h\cdot g^{-1}=x}{}

\step*{}{a_{27}:=\wedge_2(w)(g^{-1})a_{23}hr_2
\;:\;((g^{-1})^{-1}\cdot h\cdot g^{-1})\;\varepsilon\;H}{}

\step*{}{a_{28}:=eq\mhyphen subs_1(\lambda t:S.t\varepsilon H, a_{26},a_{27})
\;:\;x\varepsilon H}{}

\step*{}{a_{29}:=
a_{11}(g,x,r_1,a_{28})\;:\;(g^{-1}\cdot x\cdot g)\;\varepsilon\; (g^{-1}\cdot H\cdot g)}{}

\step*{}{a_{30}:=eq\mhyphen refl\;:\;x=g\cdot h\cdot g^{-1}}{}

\step*{}{a_{31}:=eq\mhyphen cong_1(\lambda t:S.(t\cdot g), a_{30})
\;:\;x\cdot g=g\cdot h\cdot g^{-1}\cdot g}{}

\step*{}{a_{32}:=Gr_9(S,G,\cdot,e,^{-1},u,g,h,r_1,a_{25}
)\;:\; (g\cdot h)\;\varepsilon\; G}{}

\step*{}{a_{33}:=a_1(g\cdot h) a_{32}(g^{-1})a_{23}gr_1
\;:\;g\cdot h\cdot g^{-1}\cdot g=g\cdot h\cdot (g^{-1}\cdot g)}{}

\step*{}{a_{34}:=a_3(g,r_1) 
\;:\;g^{-1}\cdot g=e}{}

\step*{}{a_{35}:=eq\mhyphen cong_1(\lambda t:S.(g\cdot h\cdot t), a_{34})
\;:\;g\cdot h\cdot (g^{-1}\cdot g)=g\cdot h\cdot e}{}

\step*{}{a_{36}:=Gr_4(S,G,\cdot,e,^{-1},u,(g\cdot h), a_{32})\;:\;g\cdot h\cdot e=g\cdot h}{}

\step*{}{a_{37}:=
eq\mhyphen trans_1(eq\mhyphen trans_1(eq\mhyphen trans_1(a_{31},a_{33}),a_{35}), a_{36})
\;:\;x\cdot g=g\cdot h}{}

\step*{}{a_{38}:=eq\mhyphen cong_1(\lambda t:S.(g^{-1}\cdot t), a_{37})
\;:\;g^{-1}(x\cdot g)=g^{-1}(g\cdot h)}{}

\step*{}{a_{39}:=a_8xa_{28}
\;:\;x\varepsilon G}{}

\step*{}{a_{40}:= a_1(g^{-1})a_{23}xa_{39}gr_1
\;:\;g^{-1}x\cdot g=g^{-1}(x\cdot g)}{}

\step*{}{a_{41}:= a_1(g^{-1})a_{23}gr_1ha_{25}
\;:\;g^{-1}g\cdot h=g^{-1}(g\cdot h)}{}

\step*{}{a_{42}:=eq\mhyphen cong_1(\lambda t:S.(t\cdot h), a_{34})
\;:\;g^{-1}\cdot g\cdot h=e\cdot h}{}

\step*{}{a_{43}:=Gr_5(S,G,\cdot,e,^{-1},u,h, a_{25})\;:\;e\cdot h= h}{}

\step*{}{a_{44}:=
eq\mhyphen trans_1(eq\mhyphen trans_1(eq\mhyphen trans_2(eq\mhyphen trans_1(a_{40},a_{38}), a_{41}), a_{42}),a_{43})
\;:\;g^{-1}\cdot x\cdot g=h}{}

\step*{}{a_{45}:=eq\mhyphen subs_1(\lambda t:S.(t\;\varepsilon\; (g^{-1}\cdot H\cdot g)), a_{44},a_{29})
\;:\;h\;\varepsilon\; (g^{-1}\cdot H\cdot g)}{}

\conclude*{}{a_{46}:=\lambda h:S.\lambda r_2:h\varepsilon H.a_{45}\;:\;H\subseteq (g^{-1}\cdot H\cdot g)}{}

\step*{}{a_{47}:=
\wedge(a_{22},a_{46})
\;:\;g^{-1}\cdot H\cdot g=H}{}

\conclude*{}{\boldsymbol {term_{\ref{prop:normal-criteria}.1}}(S,G,\cdot,e,^{-1},u,H,v):= 
\lambda w:H\triangleleft G.
\lambda g:S.\lambda r_1:g\varepsilon G.a_{47}\;:\;\boldsymbol {H\triangleleft G\Rightarrow A}\hspace{1cm}1)\Rightarrow 2)\textsf{  is proven}}{}

\assume*{}{w:A\;|\;g:S\;|\;r:g\;\varepsilon G}{}

\step*{}{a_{48}:=
a_2(g,r)\;:\;g^{-1}\varepsilon G}{}

\step*{}{a_{49}:=
w(g^{-1})a_{48}
\;:\;(g^{-1})^{-1}\cdot H\cdot g^{-1}=H}{}

\step*{}{a_{50}:=
term_{\ref{theorem:axiom_corollary}.6}(S,G,\cdot, e ,^{-1},u,g,r)\;:\;(g^{-1})^{-1}=g}{}

\step*{}{a_{51}:=eq\mhyphen subs_1(\lambda t:S.(t\cdot H\cdot g^{-1}=H), a_{50}, a_{49})
\;:\;g\cdot H\cdot g^{-1}=H}{}

\conclude*{}{\boldsymbol {term_{\ref{prop:normal-criteria}.2}}(S,G,\cdot,e,^{-1},u,H,v):= 
\lambda w:A.
\lambda g:S.\lambda r:g\varepsilon G.a_{51}\;:\;\boldsymbol {A\Rightarrow B}\hspace{2cm}2)\Rightarrow 3)\textsf{  is proven}}{}

\assume*{}{w:B\;|\;g:S\;|\;r:g\;\varepsilon G}{}

\step*{}{a_{52}:=
wgr\;:\;g\cdot H\cdot g^{-1}=H}{}

\step*{}{a_{53}:=eq\mhyphen cong_1(\lambda Z:ps(S).(Z\cdot g), a_{52})
\;:\;g\cdot H\cdot g^{-1}\cdot g=H\cdot g}{}

\step*{}{a_{54}:=
term_{ \ref{lemma:mult}.4}(S,G,\cdot,e ,^{-1},u,H, a_8,g,r)\;:\;(g\cdot H)\subseteq G}{}

\step*{}{a_{55}:=a_7(g,r,(g\cdot H),a_{54})
\;:\;g\cdot H\cdot g^{-1}\cdot g=g\cdot H}{}

\step*{}{a_{56}:=eq\mhyphen trans_3(a_{55},a_{53})
\;:\;g\cdot H=H\cdot g}{}

\conclude*{}{\boldsymbol {term_{\ref{prop:normal-criteria}.3}}(S,G,\cdot,e,^{-1},u,H,v):= 
\lambda w:B.
\lambda g:S.\lambda r:g\varepsilon G.a_{56}\;:\;\boldsymbol {B\Rightarrow C}\hspace{2cm}3)\Rightarrow 4)\textsf{  is proven}}{}

\assume*{}{w:C\;|\;g:S\;|\;r:g\;\varepsilon G}{}

\step*{}{a_{57}:=
a_2(g,r)\;:\;g^{-1}\;\varepsilon\; G}{}

\step*{}{a_{58}:=
w(g^{-1})a_{57}\;:\;g^{-1}\cdot H=H\cdot g^{-1}}{}

\step*{}{a_{59}:=eq\mhyphen cong_1(\lambda Z:ps(S).(Z\cdot g), a_{58})
\;:\;g^{-1}\cdot H\cdot g=H\cdot g^{-1}\cdot g}{}

\step*{}{a_{60}:=a_7(g,r,H,a_8)
\;:\;H\cdot g^{-1}\cdot g=H}{}

\step*{}{a_{61}:=eq\mhyphen trans_1(a_{59},a_{60})
\;:\;g^{-1}\cdot H\cdot g=H}{}

\step*{}{a_{62}:=\wedge_1(a_{61})
\;:\;(g^{-1}\cdot H\cdot g)\subseteq H}{}

\conclude*{}{\boldsymbol {term_{\ref{prop:normal-criteria}.4}}(S,G,\cdot,e,^{-1},u,H,v):= 
\lambda w:C.
\lambda g:S.\lambda r:g\varepsilon G.a_{62}\;:\;\boldsymbol {C\Rightarrow D}\hspace{2cm}4)\Rightarrow 5)\textsf{  is proven}}{}

\assume*{}{w:D}{}

\introduce*{}{g:S\;|\;r_1:g\;\varepsilon G\;|\;h:S\;|\;r_2:h\;\varepsilon H}{}

\step*{}{a_{63}:=
a_{11}(g,h,r_1,r_2)\;:\;(g^{-1}\cdot h\cdot g)\;\varepsilon\; (g^{-1}\cdot H\cdot g)}{}

\step*{}{a_{64}:=
wgr_1\;:\;(g^{-1}\cdot H\cdot g)\subseteq H}{}

\step*{}{a_{65}:=
a_{64}(g^{-1}\cdot h\cdot g)a_{63}\;:\;(g^{-1}\cdot h\cdot g)\;\varepsilon\; H}{}

\conclude*{}{a_{66}:=\lambda g:S.\lambda r_1:g\varepsilon G.\lambda h:S.\lambda r_2:h\varepsilon H.a_{65}\;:\;\forall g:S.[g\varepsilon G\Rightarrow\forall h:S.(h\varepsilon H\Rightarrow (g^{-1}\cdot h\cdot g) \;\varepsilon\;H)]}{}

\step*{}{a_{67}:=
\wedge(v,a_{66})\;:\;H\triangleleft G}{}

\conclude*{}{\boldsymbol {term_{\ref{prop:normal-criteria}.5}}(S,G,\cdot,e,^{-1},u,H,v):= 
\lambda w:D.a_{67}\;:\;\boldsymbol {D\Rightarrow H\triangleleft G}
\hspace{3.7cm}5)\Rightarrow 1)\textsf{  is proven}}{}
\end{flagderiv}

\subsection*{Proof of Corollary \ref{corol:normal}}

\begin{flagderiv}
\introduce*{}{S: *_s\;|\;G:ps(S)\;|\;\cdot: S\rightarrow S\rightarrow S\;|\;e:S\;|\;^{-1}:S\rightarrow S\;|\;u:Group(S,G,\cdot,e,^{-1})}{}

\introduce*{}{H:ps(S)}{}

\step*{}{\text{Notation }\;
A:=\forall g:S.(g\varepsilon G
\Rightarrow (g^{-1}\cdot H\cdot g)=H)\;:\;*_p}{}

\step*{}{\text{Notation }\;
B:=\forall g:S.(g\varepsilon G
\Rightarrow (g\cdot H\cdot g^{-1})=H)\;:\;*_p}{}

\step*{}{\text{Notation }\;
C:=\forall g:S.(g\varepsilon G
\Rightarrow g\cdot H=H\cdot g)\;:\;*_p}{}

\step*{}{\text{Notation }\;
D:=\forall g:S.(g\varepsilon G
\Rightarrow (g^{-1}\cdot H\cdot g)\subseteq H)\;:\;*_p}{}

\assume*{}{v:H\triangleleft G}{}

\step*{}{a_1:=\wedge_1(v)\;:\;H\leqslant G}{}

\step*{}{a_2:=
term_{ \ref{prop:normal-criteria}.1}(S,G,\cdot,e ,^{-1},u,H,a_1)v\;:\;A}{}

\step*{}{a_3:=
term_{ \ref{prop:normal-criteria}.2}(S,G,\cdot,e ,^{-1},u,H,a_1)a_2\;:\;B}{}

\step*{}{\boldsymbol {term_{\ref{corol:normal}.1}}(S,G,\cdot,e,^{-1},u,H,v):=
term_{ \ref{prop:normal-criteria}.3}(S,G,\cdot,e ,^{-1},u,H,a_1)a_3\;:\;
\boldsymbol {[\forall g:S.(g\varepsilon G
\Rightarrow g\cdot H=H\cdot g)]}
\\\hspace{13cm}\textsf{1) is proven}}{}

\done

\assume*{}{v:H\leqslant G\;|\;w:[\forall g:S.(g\varepsilon G\Rightarrow g\cdot H=H\cdot g)]}{}

\step*{}{w\;:\;C}{}

\step*{}{a_4:=
term_{ \ref{prop:normal-criteria}.4}(S,G,\cdot,e ,^{-1},u,H,v)w\;:\;D}{}

\step*{}{\boldsymbol {term_{\ref{corol:normal}.2}}(S,G,\cdot,e,^{-1},u,H,v,w):=
term_{ \ref{prop:normal-criteria}.5}(S,G,\cdot,e ,^{-1},u,H,v)a_4\;:\;\boldsymbol {H\triangleleft G}\\\hspace{13cm}\textsf{2) is proven}}{}
\end{flagderiv}

\subsection*{Proof of Lemma \ref{lemma:normal-subgroup}}

\begin{flagderiv}
\introduce*{}{S: *_s\;|\;G:ps(S)\;|\;\cdot: S\rightarrow S\rightarrow S\;|\;e:S\;|\;^{-1}:S\rightarrow S\;|\;u:Group(S,G,\cdot,e,^{-1})}{}

\introduce*{}{H:ps(S)\;|\;v:H\triangleleft G}{}

\step*{}{a_1:=\wedge_1(v)\;:\;H\leqslant G}{}

\step*{}{a_2:=\wedge_1(\wedge_1( \wedge_1(a_1)))\;:\;H\subseteq G}{}

\step*{}{a_3:=
term_{ \ref{lemma:coset}.1}(S,G,\cdot,e ,^{-1},u,H,a_1)\;:\;H^{-1}=H}{}

\step*{}{a_4:=
term_{ \ref{lemma:coset}.2}(S,G,\cdot,e ,^{-1},u,H,a_1)\;:\;H\cdot H=H}{}

\introduce*{}{x:S\;|\;w:x\varepsilon G}{}

\step*{}{a_5:=
term_{ \ref{lemma:coset}.4}(S,G,\cdot,e ,^{-1},u,H,a_1,x,w)\;:\;(x\cdot H)^{-1}=H\cdot x^{-1}}{}

\step*{}{a_6:=Gr_3(S,G,\cdot,e,^{-1},u,x,w)\;:\;x^{-1}\varepsilon G}{}

\step*{}{a_7:=
term_{ \ref{corol:normal}.1}(S,G,\cdot,e ,^{-1},u,H,v)x^{-1}a_6\;:\;x^{-1}\cdot H=H\cdot x^{-1}}{}

\step*{}{\boldsymbol {term_{\ref{lemma:normal-subgroup}.1}}(S,G,\cdot,e,^{-1},u,H,v,x,w):= 
eq\mhyphen trans_2(a_5,a_7)\;:\;\boldsymbol {(x\cdot H)^{-1}=x^{-1}\cdot H}}{\textsf{1) is proven}}

\done

\introduce*{}{x,y:S\;|\;w_1:x\varepsilon G\;|\;w_2:y\varepsilon G}{}

\step*{}{a_8:=
term_{ \ref{corol:normal}.1}(S,G,\cdot,e,^{-1},u,H,v) yw_2\;:\;y\cdot H=H\cdot y}{}

\step*{}{a_9:= 
eq\mhyphen cong_1(\lambda Z:ps(S).(x\cdot H\cdot Z),a_8)\;:\;x\cdot H\cdot(y\cdot H)=
x\cdot H\cdot(H\cdot y)}{}

\step*{}{a_{10}:=
term_{ \ref{lemma:mult}.4}(S,G,\cdot,e ,^{-1},u,H,a_2,x,w_1)\;:\;(x\cdot H)\subseteq G}{}

\step*{}{a_{11}:=
term_{ \ref{lemma:set-assoc}.4}(S,G,\cdot,e ,^{-1},u,(x\cdot H),H,a_{10},a_2,y,w_2) \;:\;x\cdot H\cdot H\cdot y=
x\cdot H\cdot (H\cdot y)}{}

\step*{}{a_{12}:=
term_{ \ref{lemma:set-assoc}.6}(S,G,\cdot,e ,^{-1},u,H,H,a_2, a_2,x,w_1) \;:\;x\cdot H\cdot H=x\cdot (H\cdot H)}{}

\step*{}{a_{13}:= 
eq\mhyphen cong_1(\lambda Z:ps(S).(x\cdot Z),a_4)\;:\;x\cdot (H\cdot H)=
x\cdot H}{}

\step*{}{a_{14}:= 
eq\mhyphen trans_1(a_{12},a_{13})\;:\;x\cdot H\cdot H=x\cdot H}{}

\step*{}{a_{15}:= 
eq\mhyphen cong_1(\lambda Z:ps(S).(Z\cdot y),a_{14})\;:\;x\cdot H\cdot H\cdot y=x\cdot H\cdot y}{}

\step*{}{a_{16}:=
term_{ \ref{lemma:set-assoc}.2}(S,G,\cdot,e ,^{-1},u,H,a_2, x,y,w_1,w_2) \;:\;x\cdot H\cdot y=x\cdot (H\cdot y)}{}

\step*{}{a_{17}:= 
eq\mhyphen cong_1(\lambda Z:ps(S).(x\cdot Z),a_8)\;:\;x\cdot (y\cdot H)=x\cdot (H\cdot y)}{}

\step*{}{a_{18}:=
term_{ \ref{lemma:set-assoc}.3}(S,G,\cdot,e ,^{-1},u,H,a_2, x,y,w_1,w_2) \;:\;x\cdot y\cdot H=x\cdot (y\cdot H)}{}

\step*{}{\boldsymbol {term_{\ref{lemma:normal-subgroup}.2}}(S,G,\cdot,e,^{-1},u,H,v,x,y,w_1, w_2):
\\= 
eq\mhyphen trans_2(
eq\mhyphen trans_2(
eq\mhyphen trans_1(
eq\mhyphen trans_1(
eq\mhyphen trans_2(a_9,a_{11}), a_{15}),a_{16}),a_{17}),a_{18})
\\\hspace{6cm}
\;:\;\boldsymbol {(x\cdot H)\cdot(y\cdot H)=x\cdot y\cdot H}\hspace{2cm}\textsf{2) is proven}}{}
\end{flagderiv}

\subsection*{Proof of Theorem \ref{theorem:product-subnormal}}

\begin{flagderiv}
\introduce*{}{S: *_s\;|\;G:ps(S)\;|\;\cdot: S\rightarrow S\rightarrow S\;|\;e:S\;|\;^{-1}:S\rightarrow S\;|\;u:Group(S,G,\cdot,e,^{-1})}{}

\introduce*{}{B,H:ps(S)\;|\;v:B\leqslant G\;|\;w:H\triangleleft G}{}

\step*{}{a_1:=\wedge_1(w)
\;:\;H\leqslant G}{}

\step*{}{a_2:=\wedge_2(w)
\;:\;\forall g:S.[g\varepsilon G\Rightarrow\forall h:S.(h\varepsilon H\Rightarrow (g^{-1}\cdot h\cdot g) \;\varepsilon\;H)]}{}

\step*{}{a_3:=\wedge_1(\wedge_1( \wedge_1(v)))\;:\;B\subseteq G}{}

\step*{}{a_4:=\wedge_2( \wedge_1(v))\;:\;Closure_1(S,B,^{-1})}{}

\step*{}{a_5:=\wedge_2(v)\;:\;Closure_2(S,B,\cdot)}{}

\step*{}{a_6:=
term_{ \ref{lemma:triv-subgroups}.3}(S,G,\cdot,e ,^{-1},u,B,H,v,a_1)\;:\;(B\cap H)\leqslant G}{}

\step*{}{a_7:=\lambda x:S.\lambda r:x\;\varepsilon\;(B\cap H).\wedge_1(r)
\;:\;(B\cap H)\subseteq B}{}

\step*{}{a_8:=
term_{ \ref{lemma:subgroup}.2}(S,G,\cdot,e ,^{-1},u,B,(B\cap H),v,a_6,a_7)\;:\;(B\cap H)\leqslant B}{}

\introduce*{}{b:S\;|\;r_1:b\varepsilon B\;|\;h:S\;|\;r_2:h\;\varepsilon\;(B\cap H)}{}

\step*{}{a_9:=\wedge_1(r_2)
\;:\;h\varepsilon B}{}

\step*{}{a_{10}:=\wedge_2(r_2)
\;:\;h\varepsilon H}{}

\step*{}{a_{11}:=a_3br_1
\;:\;b\varepsilon G}{}

\step*{}{a_{12}:=a_2ba_{11}ha_{10}
\;:\;(b^{-1}\cdot h\cdot b)\;\varepsilon\; H}{}

\step*{}{a_{13}:=a_4br_1
\;:\;b^{-1}\;\varepsilon\; B}{}

\step*{}{a_{14}:= a_5(b^{-1})a_{13}ha_9
\;:\;(b^{-1}\cdot h)\;\varepsilon\; B}{}

\step*{}{a_{15}:= a_5(b^{-1}\cdot h)a_{14}br_1
\;:\;(b^{-1}\cdot h\cdot b)\;\varepsilon\; B}{}

\step*{}{a_{16}:=\wedge(a_{15},a_{12})
\;:\;(b^{-1}\cdot h\cdot b)\;\varepsilon\;(B\cap H)}{}

\conclude*{}{a_{17}:=\lambda b:S.\lambda r_1:b\varepsilon B.\lambda h:S.\lambda r_2:h\;\varepsilon\;(B\cap H).a_{16}
\\\quad\quad
\;:\;[\forall b:S.[b\varepsilon B\Rightarrow\forall h:S.(h\;\varepsilon\;(B\cap H)\Rightarrow (b^{-1}\cdot h\cdot b) \;\varepsilon\;(B\cap H)]]}{}

\step*{}{\boldsymbol {term_{\ref{theorem:product-subnormal}.1}}(S,G,\cdot,e,^{-1},u,B,H,v,w):= 
\wedge(a_8,a_{17})\;:\;\boldsymbol {(B\cap H)\triangleleft B}\hspace{2cm}\textsf{1) is proven}}{}

\introduce*{}{x:S\;|\;r_1:x\;\varepsilon\;(B\cdot H)}{}

\step*{}{r_1:[\exists b:S.(b\varepsilon B\wedge x\;\varepsilon\;(b\cdot H))]}{}

\introduce*{}{b:S\;|\;r_2:b\varepsilon B\wedge x\;\varepsilon\;(b\cdot H)}{}

\step*{}{a_{18}:=\wedge_1(r_2)
\;:\;b\varepsilon B}{}

\step*{}{a_{19}:=\wedge_2(r_2)
\;:\;x\;\varepsilon\;(b\cdot H)}{}

\step*{}{a_{20}:=a_3ba_{18}
\;:\;b\varepsilon G}{}

\step*{}{a_{21}:=term_{ \ref{corol:normal}.1}(S,G,\cdot,e ,^{-1},u,H,w)ba_{20}
\;:\;b\cdot H=H\cdot b}{}

\step*{}{a_{22}:=\wedge_1(a_{21})xa_{19}
\;:\;x\;\varepsilon\;(H\cdot b)}{}

\step*{}{a_{23}:=\wedge(a_{18},a_{22})
\;:\;b\varepsilon B\wedge
x\;\varepsilon\;(H\cdot b)}{}

\step*{}{a_{24}:=\exists_1(\lambda t:S.(t\varepsilon B\wedge
x\;\varepsilon\;(H\cdot t)),b, a_{23})\;:\;x\;\varepsilon\;(H\cdot B)}{}

\conclude*{}{a_{25}:=\exists_3(r_1, a_{24})\;:\;x\;\varepsilon\;(H\cdot B)}{}

\conclude*{}{a_{26}:=
\lambda x:S.\lambda r_1:x\;\varepsilon\;(B\cdot H).a_{25}
\;:\;(B\cdot H)\subseteq(H\cdot B)}{}

\introduce*{}{x:S\;|\;r_1:x\;\varepsilon\;(H\cdot B)}{}

\step*{}{r_1:[\exists b:S.(b\varepsilon B\wedge x\;\varepsilon\;(H\cdot b))]}{}

\introduce*{}{b:S\;|\;r_2:b\varepsilon B\wedge x\;\varepsilon\;(H\cdot b)}{}

\step*{}{a_{27}:=\wedge_1(r_2)
\;:\;b\varepsilon B}{}

\step*{}{a_{28}:=\wedge_2(r_2)
\;:\;x\;\varepsilon\;(H\cdot b)}{}

\step*{}{a_{29}:=a_3ba_{27}
\;:\;b\varepsilon G}{}

\step*{}{a_{30}:=term_{ \ref{corol:normal}.1}(S,G,\cdot,e ,^{-1},u,H,w)ba_{29}
\;:\;(b\cdot H)=(H\cdot b)}{}

\step*{}{a_{31}:=\wedge_2(a_{30})xa_{28}
\;:\;x\;\varepsilon\;(b\cdot H)}{}

\step*{}{a_{32}:=\wedge(a_{27},a_{31})
\;:\;b\varepsilon B\wedge
x\;\varepsilon\;(b\cdot H)}{}

\step*{}{a_{33}:=\exists_1(\lambda t:S.(t\varepsilon B\wedge
x\;\varepsilon\;(t\cdot H)),b, a_{32})\;:\;x\;\varepsilon\;(B\cdot H)}{}

\conclude*{}{a_{34}:=\exists_3(r_1, a_{33})\;:\;x\;\varepsilon\;(B\cdot H)}{}

\conclude*{}{a_{35}:=
\lambda x:S.\lambda r_1:x\;\varepsilon\;(H\cdot B).a_{34}
\;:\;(H\cdot B)\subseteq(B\cdot H)}{}

\step*{}{a_{36}:=\wedge(a_{26},a_{35})\;:\;B\cdot H=H\cdot B}{}

\step*{}{\boldsymbol {term_{\ref{theorem:product-subnormal}.2}}(S,G,\cdot,e,^{-1},u,B,H,v,w):= 
a_{36}\;:\;\boldsymbol {B\cdot H=H\cdot B}\hspace{4.2cm}\textsf{2) is proven}}{}

\step*{}{\boldsymbol {term_{\ref{theorem:product-subnormal}.3}}(S,G,\cdot,e, ^{-1},u, B,H,v,w):= 
\wedge_2(term_{
\ref{theorem:permutable}}(S,G,\cdot,e, ^{-1},u, B,H,v, a_1))a_{36}
\\\hspace{10cm}
\;:\;\boldsymbol {(B\cdot H)\leqslant G}\hspace{2cm}\textsf{3) is proven}}{}
\end{flagderiv}

\subsection*{Proof of Theorem \ref{theorem:product-normal}}

\begin{flagderiv}
\introduce*{}{S: *_s\;|\;G:ps(S)\;|\;\cdot: S\rightarrow S\rightarrow S\;|\;e:S\;|\;^{-1}:S\rightarrow S\;|\;u:Group(S,G,\cdot,e,^{-1})}{}

\step*{}{a_1:=Gr_2(S,G,\cdot,e,^{-1},u)\;:\;Assoc(S,G,u)}{}

\introduce*{}{B,C:ps(S)\;|\;v:B\triangleleft G\;|\;w:C\triangleleft G}{}

\step*{}{a_2:=\wedge_1(v)
\;:\;B\leqslant G}{}

\step*{}{a_3:=\wedge_1(\wedge_1(\wedge_1(a_2)))
\;:\;B\subseteq G}{}

\step*{}{a_4:=\wedge_2(\wedge_1(a_2))
\;:\;Closure_1(S,B,^{-1})}{}

\step*{}{a_5:=\wedge_2(a_2)
\;:\;Closure_2(S,B,\cdot)}{}

\step*{}{a_6:=\wedge_1(w)
\;:\;C\leqslant G}{}

\step*{}{a_7:=\wedge_1(\wedge_1(\wedge_1(a_6)))
\;:\;C\subseteq G}{}

\step*{}{a_8:=\wedge_2(\wedge_1(a_6))
\;:\;Closure_1(S,C,^{-1})}{}

\step*{}{a_9:=\wedge_2(a_6)
\;:\;Closure_2(S,C,\cdot)}{}

\step*{}{a_{10}:=
term_{ \ref{theorem:product-subnormal}.3}(S,G,\cdot,e ,^{-1},u,B,C,a_2,w)\;:\;(B\cdot C)\leqslant G}{}

\introduce*{}{g:S\;|\;r:g\varepsilon G}{}

\step*{}{a_{11}:=
term_{\ref{lemma:set-assoc}.6}(S,G,\cdot,e ,^{-1},u,B,C,a_3,a_7,g,r)\;:\;g\cdot B\cdot C=g\cdot(B\cdot C)}{}

\step*{}{a_{12}:=
term_{ \ref{corol:normal}.1}(S,G,\cdot,e ,^{-1},u,B,v)gr\;:\;g\cdot B=B\cdot g}{}

\step*{}{a_{13}:=
term_{ \ref{corol:normal}.1}(S,G,\cdot,e ,^{-1},u,C,w)gr\;:\;g\cdot C=C\cdot g}{}

\step*{}{a_{14}:=
eq\mhyphen cong_1(\lambda Z:ps(S).(Z\cdot C), a_{12})
\;:\;g\cdot B\cdot C=B\cdot g\cdot C}{}

\step*{}{a_{15}:=
term_{\ref{lemma:set-assoc}.5}(S,G,\cdot,e ,^{-1},u,B,C,a_3,a_7,g,r)\;:\;B\cdot g\cdot C=B\cdot (g\cdot C)}{}

\step*{}{a_{16}:=
eq\mhyphen cong_1(\lambda Z:ps(S).(B\cdot Z), a_{13})
\;:\;B\cdot (g\cdot C)=B\cdot (C\cdot g)}{}

\step*{}{a_{17}:=
term_{ \ref{lemma:set-assoc}.4}(S,G,\cdot,e ,^{-1},u,B,C,a_3,a_7,g,r)\;:\;B\cdot C\cdot g=B\cdot (C\cdot g)}{}

\step*{}{a_{18}:=
eq\mhyphen trans_2(eq\mhyphen trans_1(eq\mhyphen trans_1(eq\mhyphen trans_3(a_{11}, a_{14}),a_{15}), a_{16}),a_{17})
\;:\;g\cdot (B\cdot C)=B\cdot C\cdot g}{}

\conclude*{}{a_{19}:=
\lambda g:S.\lambda r:g\varepsilon G.a_{18}
\;:\;[\forall g:S.(g\varepsilon G\Rightarrow g\cdot (B\cdot C)=(B\cdot C)\cdot g)]}{}

\step*{}{\boldsymbol {term_{\ref{theorem:product-normal}.1}}(S,G,\cdot,e,^{-1},u,B,C,v,w):= 
term_{ \ref{corol:normal}.2}(S,G,\cdot,e ,^{-1},u,(B\cdot C), a_{10}, a_{19})\;:\;\boldsymbol{(B\cdot C)\triangleleft G}
\\
\hspace{12cm} \textsf{1) is proven}}{}

\assume*{}{r_1:[\forall g:S.(
g\varepsilon B\Rightarrow g\varepsilon C\Rightarrow g=e)]}{}

\introduce*{}{b:S\;|\;r_2:b\varepsilon B\;|\;c:S\;|\;r_3:c\varepsilon C}{}

\step*{}{a_{20}:=a_3br_2
\;:\;b\varepsilon G}{}

\step*{}{a_{21}:=a_7cr_3
\;:\;c\varepsilon G}{}

\step*{}{a_{22}:=a_4br_2
\;:\;b^{-1}\varepsilon B}{}

\step*{}{a_{23}:=a_8cr_3
\;:\;c^{-1}\varepsilon C}{}

\step*{}{a_{24}:=a_3(b^{-1})
a_{22}
\;:\;(b^{-1})\;\varepsilon \;G}{}

\step*{}{a_{25}:=a_7(c^{-1})
a_{23}
\;:\;(c^{-1})\;\varepsilon \;G}{}

\step*{}{a_{26}:=\wedge_2(v)c
a_{21}br_2
\;:\;(c^{-1}\cdot b\cdot c)\;\varepsilon \;B}{}

\step*{}{a_{27}:=\wedge_2(w)b
a_{20}(c^{-1})a_{23}
\;:\;(b^{-1}\cdot c^{-1}\cdot b)\;\varepsilon \;C}{}

\step*{}{a_{28}:=a_5(b^{-1})
a_{22}(c^{-1}\cdot b\cdot c)a_{26}
\;:\;(b^{-1}\cdot (c^{-1}\cdot b\cdot c))\;\varepsilon \;B}{}

\step*{}{a_{29}:=a_9(b^{-1}\cdot c^{-1}\cdot b)a_{27}cr_3
\;:\;(b^{-1}\cdot c^{-1}\cdot b\cdot c)\;\varepsilon \;C}{}

\step*{}{a_{30}:=
term_{\ref{theorem:axiom_corollary}.5}(S,G,\cdot,e ,^{-1},u,c,b,a_{21},a_{20})\;:\;(c\cdot b)^{-1}=b^{-1}\cdot c^{-1}}{}

\step*{}{a_{31}:=a_1(b^{-1})
a_{24}(c^{-1})a_{25}ba_{20}
\;:\;b^{-1}\cdot c^{-1}\cdot b=b^{-1}\cdot (c^{-1}\cdot b)}{}

\step*{}{a_{32}:=
eq\mhyphen cong_1(\lambda t:S.(t\cdot c), a_{31})
\;:\;b^{-1}\cdot c^{-1}\cdot b\cdot c=b^{-1}\cdot (c^{-1}\cdot b)\cdot c}{}

\step*{}{a_{33}:=Gr_9(S,G,\cdot,e,^{-1},u,c^{-1},b, a_{25}, a_{20})\;:\;(c^{-1}\cdot b)\;\varepsilon\;G}{}

\step*{}{a_{34}:=a_1(b^{-1})
a_{24}(c^{-1}\cdot b)a_{33}ca_{21}
\;:\;b^{-1}\cdot (c^{-1}\cdot b)\cdot c=b^{-1}\cdot (c^{-1}\cdot b\cdot c)}{}

\step*{}{a_{35}:=
eq\mhyphen trans_1(a_{32},a_{34})
\;:\;b^{-1}\cdot c^{-1}\cdot b\cdot c=b^{-1}\cdot (c^{-1}\cdot b\cdot c)}{}

\step*{}{a_{36}:=
eq\mhyphen subs_2(\lambda t:S.t\varepsilon B, a_{35},a_{28})
\;:\;(b^{-1}\cdot c^{-1}\cdot b\cdot c)\;\varepsilon\;B}{}

\step*{}{a_{37}:=
r_1(b^{-1}\cdot c^{-1}\cdot b\cdot c)a_{36}a_{29}
\;:\;b^{-1}\cdot c^{-1}\cdot b\cdot c=e}{}

\step*{}{a_{38}:=
eq\mhyphen cong_1(\lambda t:S.(t\cdot b\cdot c), a_{30})
\;:\;(c\cdot b)^{-1}\cdot b\cdot c=b^{-1}\cdot c^{-1}\cdot b\cdot c}{}

\step*{}{a_{39}:=
eq\mhyphen trans_1(a_{38},a_{37})
\;:\;(c\cdot b)^{-1}\cdot b\cdot c=e}{}

\step*{}{a_{40}:=Gr_9(S,G,\cdot,e,^{-1},u,b,c,a_{20}, a_{21})\;:\;(b\cdot c)\;\varepsilon\;G}{}

\step*{}{a_{41}:=Gr_9(S,G,\cdot,e,^{-1},u,c,b,a_{21}, a_{20})\;:\;(c\cdot b)\;\varepsilon\;G}{}

\step*{}{\text{Notation }\;
x:=(c\cdot b)\;:\;S}{}

\step*{}{a_{41}\;:\;x\varepsilon G}{}

\step*{}{a_{39}\;:\;x^{-1}\cdot b\cdot c=e}{}

\step*{}{a_{42}:=Gr_3(S,G,\cdot,e,^{-1},u,x,a_{41})\;:\;x^{-1}\;\varepsilon\;G}{}

\step*{}{a_{43}:=Gr_7(S,G,\cdot,e,^{-1},u,x,a_{41})\;:\;x\cdot x^{-1}=e}{}

\step*{}{a_{44}:= a_1(x^{-1})
a_{42}ba_{20}ca_{21}
\;:\;x^{-1}\cdot b\cdot c=x^{-1}\cdot (b\cdot c)}{}

\step*{}{a_{45}:=
eq\mhyphen cong_1(\lambda t:S.(x\cdot t), a_{44})
\;:\;x\cdot (x^{-1}\cdot b\cdot c)=x\cdot (x^{-1}\cdot (b\cdot c))}{}

\step*{}{a_{46}:=a_1x
a_{41} (x^{-1})a_{42}(b\cdot c)a_{40}
\;:\;x\cdot x^{-1}\cdot (b\cdot c)=x\cdot (x^{-1}\cdot (b\cdot c))}{}

\step*{}{a_{47}:=
eq\mhyphen cong_1(\lambda t:S.(t\cdot(b\cdot c)), a_{43})\;:\; x\cdot x^{-1}\cdot (b\cdot c)=e\cdot (b\cdot c)}{}

\step*{}{a_{48}:= Gr_5(S,G,\cdot,e,^{-1}, u,(b\cdot c),a_{40})\;:\;e\cdot (b\cdot c)=b\cdot c}{}

\step*{}{a_{49}:= a_1(x^{-1})
a_{42}ba_{20}ca_{21}
\;:\;x^{-1}\cdot b\cdot c=x^{-1}\cdot (b\cdot c)}{}

\step*{}{a_{50}:=
eq\mhyphen cong_1(\lambda t:S.(x\cdot t), a_{49})
\;:\;x\cdot(x^{-1}\cdot b\cdot c)=x\cdot (x^{-1}\cdot (b\cdot c))}{}

\step*{}{a_{51}:=
eq\mhyphen cong_1(\lambda t:S.(x\cdot t), a_{39})
\;:\;x\cdot(x^{-1}\cdot b\cdot c)=x\cdot e}{}

\step*{}{a_{52}:= Gr_4(S,G,\cdot,e,^{-1},u,x,a_{41})\;:\;x\cdot e=x}{}

\step*{}{\boldsymbol {term_{\ref{theorem:product-normal}.2}}(S,G,\cdot,e,^{-1},u,B,C,v,w,r_1, b,c,r_2,r_3)
\\\quad := 
eq\mhyphen trans_1(eq\mhyphen trans_1(eq\mhyphen trans_2( eq\mhyphen trans_3(eq\mhyphen trans_1( a_{47},a_{48}),a_{46}), a_{50}),a_{51}),a_{52})
\;:\;\boldsymbol {b\cdot c=c\cdot b}
\\\hspace{12cm}
\textsf{2) is proven}}{}
\end{flagderiv}

\subsection*{Proof of Theorem \ref{theorem:quotient}}

\begin{flagderiv}
\introduce*{}{S: *_s\;|\;G:ps(S)\;|\;\cdot: S\rightarrow S\rightarrow S\;|\;e:S\;|\;^{-1}:S\rightarrow S\;|\;u:Group(S,G,\cdot,e,^{-1})}{}

\introduce*{}{H:ps(S)\;|\;v:H\triangleleft G}{}

\step*{}{a_1:= Gr_1(S,G,\cdot,e,^{-1},u)\;:\;e\varepsilon G}{}

\step*{}{a_2:= \wedge_1(v)\;:\;H\leqslant G}{}

\step*{}{a_3:=\wedge_1(\wedge_1( \wedge_1(a_2)))\;:\;H\subseteq G}{}

\step*{}{a_4:=\wedge_2(\wedge_1( \wedge_1(a_2)))\;:\;e\varepsilon H}{}

\step*{}{a_4:=eq\mhyphen refl\;:\; E=e\cdot H}{}

\step*{}{a_5:=\wedge(a_1,a_4)\;:\;
e\varepsilon G\wedge E=e\cdot H}{}

\step*{}{a_6:=\exists_1(\lambda t:S.(t\varepsilon G\wedge E=t\cdot H),e, a_5)\;:\;E\;\varepsilon\;(G/H)}{}

\introduce*{}{X:ps(S)\;|\;r_1:X\;\varepsilon\;(G/H)}{}

\step*{}{r_1\;:\;[\exists x:S.(x\varepsilon G\wedge X=x\cdot H)]}{}

\introduce*{}{x:S\;|\;r_2:x\varepsilon G\wedge X=x\cdot H}{}

\step*{}{a_7:=\wedge_1(r_2)\;:\;x\varepsilon G}{}

\step*{}{a_8:=\wedge_2(r_2)\;:\;X=x\cdot H}{}

\step*{}{a_9:=
term_{\ref{lemma:normal-subgroup}.2}(S,G,\cdot,e ,^{-1},u,H,v,x,e,a_7,a_1)
\;:\;(x\cdot H)\cdot(e\cdot H)=x\cdot e\cdot H}{}

\step*{}{a_{10}:=
term_{\ref{lemma:normal-subgroup}.2}(S,G,\cdot,e ,^{-1},u,H,v,e,x,a_1,a_7)
\;:\;(e\cdot H)\cdot(x\cdot H)=e\cdot x\cdot H}{}

\step*{}{a_{11}:=
Gr_4(S,G,\cdot,e, ^{-1},u,x, a_7)\;:\;x\cdot e=x}{}

\step*{}{a_{12}:=
Gr_5(S,G,\cdot,e, ^{-1},u,x,a_7)
\;:\;e\cdot x=x}{}

\step*{}{a_{13}:=
eq\mhyphen cong_1(\lambda t:S.(t\cdot H), a_{11})
\;:\;x\cdot e\cdot H=x\cdot H}{}

\step*{}{a_{14}:=
eq\mhyphen cong_1(\lambda t:S.(t\cdot H), a_{12})
\;:\;e\cdot x\cdot H=x\cdot H}{}

\step*{}{a_{15}:=
eq\mhyphen trans_1(a_9,a_{13})
\;:\;(x\cdot H)\cdot E=x\cdot H}{}

\step*{}{a_{16}:=
eq\mhyphen trans_1(a_{10},a_{14})
\;:\;E\cdot (x\cdot H)=x\cdot H}{}

\step*{}{a_{17}:=
eq\mhyphen subs_2(\lambda U:ps(S).(U\cdot E=U),a_8, a_{15})\;:\;X\cdot E=X}{}

\step*{}{a_{18}:=
eq\mhyphen subs_2(\lambda U:ps(S).(E\cdot U=U),a_8, a_{16})\;:\;E\cdot X=X}{}

\step*{}{a_{19}:=
\wedge(a_{17},a_{18})\;:\;
X\cdot E=X\wedge E\cdot X=X}{}

\step*{}{a_{20}:=
term_{\ref{lemma:normal-subgroup}.1}(S,G,\cdot,e ,^{-1},u,H,v,x,a_7)
\;:\;(x\cdot H)^{-1}=x^{-1}\cdot H}{}

\step*{}{a_{21}:=
eq\mhyphen cong_1(^{-1}, a_8)
\;:\;X^{-1}=(x\cdot H)^{-1}}{}

\step*{}{a_{22}:=
eq\mhyphen trans_1(a_{21},a_{20})\;:\;X^{-1}=x^{-1}\cdot H}{}

\step*{}{a_{23}:=
Gr_3(S,G,\cdot,e, ^{-1},u,x,a_7)
\;:\;x^{-1}\varepsilon G}{}

\step*{}{a_{24}:=
\wedge(a_{23},a_{22})\;:\;
x^{-1}\varepsilon G\wedge X^{-1}=x^{-1}\cdot H}{}

\step*{}{a_{25}:=\exists_1(\lambda t:S.(t\varepsilon G\wedge X^{-1}=t\cdot H),x^{-1}, a_{24})\;:\;X^{-1}\;\varepsilon\;(G/H)}{}

\step*{}{a_{26}:=
term_{\ref{lemma:normal-subgroup}.2}(S,G,\cdot,e ,^{-1},u,H,v,x,x^{-1},a_7 ,a_{23})\;:\;(x\cdot H)\cdot (x^{-1}\cdot H)=x\cdot x^{-1}\cdot H}{}

\step*{}{a_{27}:=
term_{\ref{lemma:normal-subgroup}.2}(S,G,\cdot,e ,^{-1},u,H,v,x^{-1},x, a_{23},a_7)\;:\;(x^{-1}\cdot H)\cdot (x\cdot H)=x^{-1}\cdot x\cdot H}{}

\step*{}{a_{28}:=
Gr_7(S,G,\cdot,e, ^{-1},u,x,a_7)
\;:\;x\cdot x^{-1}=e}{}

\step*{}{a_{29}:=
Gr_8(S,G,\cdot,e, ^{-1},u,x,a_7)
\;:\;x^{-1}\cdot x=e}{}

\step*{}{a_{30}:=
eq\mhyphen cong_1(\lambda t:S.(t\cdot H), a_{28})
\;:\;x\cdot x^{-1}\cdot H=e\cdot H}{}

\step*{}{a_{31}:=
eq\mhyphen cong_1(\lambda t:S.(t\cdot H), a_{29})
\;:\;x^{-1}\cdot x\cdot H=e\cdot H}{}

\step*{}{a_{32}:=
eq\mhyphen trans_1(a_{26},a_{30})\;:\;(x\cdot H)\cdot (x^{-1}\cdot H)=E}{}

\step*{}{a_{33}:=
eq\mhyphen trans_1(a_{27},a_{31})\;:\;(x^{-1}\cdot H)\cdot (x\cdot H)=E}{}

\step*{}{a_{34}:=
eq\mhyphen subs_2(\lambda U:ps(S).(x\cdot H\cdot U=E),a_{22}, a_{32})\;:\;x\cdot H\cdot X^{-1}=E}{}

\step*{}{a_{35}:=
eq\mhyphen subs_2(\lambda U:ps(S).(U\cdot(x\cdot H)=E),a_{22}, a_{33})\;:\;X^{-1}\cdot(x\cdot H)=E}{}

\step*{}{a_{36}:=
eq\mhyphen subs_2(\lambda U:ps(S).(U\cdot X^{-1}=E),a_8, a_{34})\;:\;X\cdot X^{-1}=E}{}

\step*{}{a_{37}:=
eq\mhyphen subs_2(\lambda U:ps(S).(X^{-1}\cdot U=E), a_8, a_{35})\;:\;X^{-1}\cdot X=E}{}

\step*{}{a_{38}:=
\wedge(a_{36},a_{37})\;:\;
Inverse_0(ps(S),\cdot,E,X,X^{-1})}{}

\introduce*{}{Y:ps(S)\;|\;r_3:Y\;\varepsilon\;(G/H)}{}

\step*{}{r_3\;:\;[\exists y:S.(y\varepsilon G\wedge Y=y\cdot H)]}{}

\introduce*{}{y:S\;|\;r_4:y\varepsilon G\wedge Y=y\cdot H}{}

\step*{}{a_{39}:=\wedge_1(r_4)\;:\;y\varepsilon G}{}

\step*{}{a_{40}:=\wedge_2(r_4)\;:\;Y=y\cdot H}{}

\step*{}{a_{41}:=
term_{\ref{lemma:normal-subgroup}.2}(S,G,\cdot,e ,^{-1},u,H,v,x,y,a_7,a_{39})
\;:\;(x\cdot H)\cdot(y\cdot H)=(x\cdot y)\cdot H}{}

\step*{}{a_{42}:=
Gr_9(S,G,\cdot,e, ^{-1},u,x,y,a_7,a_{39})
\;:\;(x\cdot y)\;\varepsilon \;G}{}

\step*{}{a_{43}:=
eq\mhyphen subs_2(\lambda U:ps(S).((x\cdot H)\cdot U=(x\cdot y)\cdot H), a_{40}, a_{41})\;:\;(x\cdot H)\cdot Y=(x\cdot y)\cdot H}{}

\step*{}{a_{44}:=
eq\mhyphen subs_2(\lambda U:ps(S).(U\cdot Y=(x\cdot y)\cdot H), a_8, a_{43})\;:\;X\cdot Y=(x\cdot y)\cdot H}{}

\step*{}{a_{45}:=
\wedge(a_{42},a_{44})\;:\;
(x\cdot y)\;\varepsilon\;G
\wedge X\cdot Y=(x\cdot y)\cdot H}{}

\step*{}{a_{46}:=\exists_1(\lambda t:S.(t\varepsilon G\wedge X\cdot Y=t\cdot H),(x\cdot y), a_{45})\;:\;(X\cdot Y)\;\varepsilon\;(G/H)}{}

\introduce*{}{Z:ps(S)\;|\;r_5:Z\;\varepsilon\;(G/H)}{}

\step*{}{r_5\;:\;[\exists z:S.(z\varepsilon G\wedge Z=z\cdot H)]}{}

\introduce*{}{z:S\;|\;r_6:z\varepsilon G\wedge Z=z\cdot H}{}

\step*{}{a_{47}:=\wedge_1(r_6)\;:\;z\varepsilon G}{}

\step*{}{a_{48}:=\wedge_2(r_6)\;:\;Z=z\cdot H}{}

\step*{}{a_{49}:=
term_{\ref{lemma:mult}.4}(S,G,\cdot,e ,^{-1},u,H,a_3,x,a_7)
\;:\;(x\cdot H)\subseteq G}{}

\step*{}{a_{50}:=
term_{\ref{lemma:mult}.4}(S,G,\cdot,e ,^{-1},u,H,a_3,y,a_{39})
\;:\;(y\cdot H)\subseteq G}{}

\step*{}{a_{51}:=
term_{\ref{lemma:mult}.4}(S,G,\cdot,e ,^{-1},u,H,a_3,z,a_{47})
\;:\;(z\cdot H)\subseteq G}{}

\step*{}{a_{52}:=
eq\mhyphen subs_2(\lambda U:ps(S).U\subseteq G, a_8, a_{49})\;:\;X\subseteq G}{}

\step*{}{a_{53}:=
eq\mhyphen subs_2(\lambda U:ps(S).U\subseteq G, a_{40}, a_{50})\;:\;Y\subseteq G}{}

\step*{}{a_{54}:=
eq\mhyphen subs_2(\lambda U:ps(S).U\subseteq G,a_{48}, a_{51})\;:\;Z\subseteq G}{}

\step*{}{a_{55}:=
term_{\ref{lemma:set-assoc}.7}(S,G,\cdot,e ,^{-1},u,X,Y,Z,a_{52}, a_{53},a_{54})
\;:\;(X\cdot Y)\cdot Z =X\cdot (Y\cdot Z)}{}

\conclude*{}{a_{56}:=
\exists_3(r_5,a_{55})
\;:\;(X\cdot Y)\cdot Z =X\cdot (Y\cdot Z)}{}

\conclude*{}{a_{57}:=
\lambda Z:ps(S).\lambda r_5:Z\;\varepsilon\;(G/H).a_{56}
\;:\;
[\forall Z:ps(S).(Z\;\varepsilon\;(G/H)\Rightarrow
(X\cdot Y)\cdot Z =X\cdot (Y\cdot Z))]}{}

\conclude*{}{a_{58}:=
\exists_3(r_3,a_{57})
\;:\;
[\forall Z:ps(S).(Z\;\varepsilon\;(G/H)\Rightarrow
(X\cdot Y)\cdot Z =X\cdot (Y\cdot Z))]}{}

\step*{}{a_{59}:=
\exists_3(r_3,a_{46})
\;:\;
(X\cdot Y)\;\varepsilon\;(G/H)}{}

\conclude*{}{a_{60}:=
\lambda Y:ps(S).\lambda r_3:Y\;\varepsilon\;(G/H).a_{59}
\;:\;[\forall Y:ps(S).(Y\;\varepsilon\;(G/H)\Rightarrow
(X\cdot Y)\;\varepsilon\;(G/H))]}{}

\step*{}{a_{61}:=
\lambda Y:ps(S).\lambda r_3:Y\;\varepsilon\;(G/H).a_{58}
\\\quad\quad
\;:\;
[\forall Y:ps(S).[Y\;\varepsilon\;(G/H)\Rightarrow
\forall Z:ps(S).(Z\;\varepsilon\;(G/H)\Rightarrow
(X\cdot Y)\cdot Z =X\cdot (Y\cdot Z))]]}{}

\conclude*{}{a_{62}:=
\exists_3(r_1,a_{61})
\\\quad\quad
\;:\;
[\forall Y:ps(S).[Y\;\varepsilon\;(G/H)\Rightarrow
\forall Z:ps(S).(Z\;\varepsilon\;(G/H)\Rightarrow
(X\cdot Y)\cdot Z =X\cdot (Y\cdot Z))]]}{}

\step*{}{a_{63}:=
\exists_3(r_1,a_{60})
\;:\;
[\forall Y:ps(S).(Y\;\varepsilon\;(G/H)\Rightarrow
(X\cdot Y)\;\varepsilon\;(G/H))]}{}

\step*{}{a_{64}:=
\exists_3(r_1,a_{25})
\;:\;
X^{-1}\;\varepsilon\;(G/H)}{}

\step*{}{a_{65}:=
\exists_3(r_1,a_{19})
\;:\;
X\cdot E=X\wedge E\cdot X=X}{}

\step*{}{a_{66}:=
\exists_3(r_1,a_{38})
\;:\;Inverse_0(ps(S),\cdot,E,X,X^{-1})}{}

\conclude*{}{a_{67}:=
\lambda X:ps(S).\lambda r_1:X\;\varepsilon\;(G/H).a_{62}
\;:\;Assoc(ps(S),(G/H),\cdot)}{}

\step*{}{a_{68}:=
\lambda X:ps(S).\lambda r_1:X\;\varepsilon\;(G/H).a_{63}
\;:\;Closure_2(ps(S),(G/H),\cdot)}{}

\step*{}{a_{69}:=
\lambda X:ps(S).\lambda r_1:X\;\varepsilon\;(G/H).a_{64}
\;:\;Closure_1(ps(S),(G/H), ^{-1})}{}

\step*{}{a_{70}:=
\lambda X:ps(S).\lambda r_1:X\;\varepsilon\;(G/H).a_{66}
\;:\;Inverse(ps(S),(G/H),\cdot,E,^{-1})}{}

\step*{}{a_{71}:=
\lambda X:ps(S).\lambda r_1:X\;\varepsilon\;(G/H).a_{65}
\;:\;
[\forall X:ps(S).(X\;\varepsilon\;(G/H)\Rightarrow
X\cdot E=X\wedge E\cdot X=X)]}{}

\step*{}{a_{72}:=
\wedge(a_6,a_{71})\;:\;
Identity(ps(S),(G/H),\cdot,E)}{}

\step*{}{a_{73}:=
\wedge(\wedge(a_{68},a_{67}), a_{72})\;:\;
Monoid(ps(S),(G/H),\cdot,E)}{}

\step*{}{a_{74}:=
\wedge(a_{69},a_{70})\;:\;
Inverse\mhyphen prop(ps(S),(G/H),\cdot,E,\,^{-1})}{}

\step*{}{\boldsymbol {term_{\ref{theorem:quotient}}}(S,G,\cdot,e,^{-1},u,H,v):= 
\wedge(a_{73},a_{74})
\;:\;\boldsymbol{Group(ps(S),(G/H),\cdot,E,\,^{-1})}}{}
\end{flagderiv}

\subsection*{Proof of Theorem \ref{theorem:correspondence}}
1)
\vspace{-0.2cm}

\begin{flagderiv}
\introduce*{}{S: *_s\;|\;G:ps(S)\;|\;\cdot: S\rightarrow S\rightarrow S\;|\;e:S\;|\;^{-1}:S\rightarrow S\;|\;u:Group(S,G,\cdot,e,^{-1})}{}

\introduce*{}{H,B:ps(S)\;|\;v_1:H\triangleleft G\;|\;
v_2:B\leqslant G
\;|\;w:H\subseteq B}{}

\step*{}{a_1:=\wedge_1(\wedge_1( \wedge_1(v_2)))\;:\;B\subseteq G}{}

\step*{}{a_2:=\wedge_2(\wedge_1( \wedge_1(v_2)))\;:\;e\varepsilon B}{}

\step*{}{a_3:=\wedge_2( \wedge_1(v_2))\;:\; Closure_1(S,B,^{-1})}{}

\step*{}{a_4:=\wedge_2(v_2)\;:\;Closure_2(S,B,\cdot)}{}

\step*{}{a_5:=
term_{\ref{theorem:product-subnormal}.1}(S,G,\cdot,e ,^{-1},u, B,H,v_2,v_1)
\;:\;(B\cap H)\triangleleft B}{}

\step*{}{a_6:=
\lambda x:s.\lambda r:x\varepsilon (B\cap H).\wedge_2(r)
\;:\;(B\cap H)\subseteq H}{}

\introduce*{}{x:S\;|\;r:x\varepsilon H}{}
\step*{}{a_7:=wxr\;:\;x\varepsilon B}{}

\step*{}{a_8:=\wedge(r,a_7)\;:\;x\;\varepsilon\;(B\cap H)}{}

\conclude*{}{a_9:=\lambda x:S.\lambda r:x\varepsilon H.a_8\;:\;H\subseteq(B\cap H)}{}

\step*{}{a_{10}:=\wedge(a_6,a_9)\;:\;B\cap H=H}{}

\step*{}{a_{11}:=
eq\mhyphen subs_1(\lambda Z:ps(S).Z\triangleleft B,a_{10},a_5)
\;:\;H\triangleleft B}{}

\introduce*{}{X:ps(S)\;|\;r_1:X\;\varepsilon\;(B/H)}{}

\step*{}{r_1\;:\;[\exists x:S.(x\varepsilon B\wedge X=x\cdot H)]}{}

\introduce*{}{x:S\;|\;r_2:x\varepsilon B\wedge X=x\cdot H}{}

\step*{}{a_{12}:=\wedge_1(r_2)\;:\;x\varepsilon B}{}
\step*{}{a_{13}:=\wedge_2(r_2)\;:\;X=x\cdot H}{}
\step*{}{a_{14}:=a_3xa_{12}\;:\;x^{-1}\varepsilon B}{}

\step*{}{a_{15}:=a_1xa_{12}\;:\;x\varepsilon G}{}

\step*{}{a_{16}:=\wedge( a_{15},a_{13})\;:\;x\varepsilon G\wedge X=x\cdot H}{}

\step*{}{a_{17}:=\exists_1(\lambda t:S.(t\varepsilon G\wedge X=t\cdot H), x, a_{16})\;:\;X\;\varepsilon\;(G/H)}{}

\step*{}{a_{18}:=
term_{\ref{lemma:normal-subgroup}.1}(S,G,\cdot,e ,^{-1},u,H,v_1,x,a_{15})
\;:\;(x\cdot H)^{-1}=x^{-1}\cdot H}{}

\step*{}{a_{19}:=
eq\mhyphen cong_1(^{-1},a_{13})
\;:\;X^{-1}=(x\cdot H)^{-1}}{}

\step*{}{a_{20}:=
eq\mhyphen trans_1(a_{19},a_{18})
\;:\;X^{-1}=x^{-1}\cdot H}{}
\step*{}{a_{21}:=\wedge( a_{14},a_{20})\;:\;x^{-1}\varepsilon B\wedge X^{-1}=x^{-1}\cdot H}{}

\step*{}{a_{22}:=\exists_1(\lambda t:S.(t\varepsilon B\wedge X^{-1}=t\cdot H),x^{-1}, a_{21})\;:\;X^{-1}\;\varepsilon\;(B/H)}{}

\introduce*{}{Y:ps(S)\;|\;r_3:Y\;\varepsilon\;(B/H)}{}

\step*{}{r_3\;:\;[\exists y:S.(y\varepsilon B\wedge Y=y\cdot H)]}{}
\introduce*{}{y:S\;|\;r_4:y\varepsilon B\wedge Y=y\cdot H}{}
\step*{}{a_{23}:=\wedge_1(r_4)\;:\;y\varepsilon B}{}
\step*{}{a_{24}:=\wedge_2(r_4)\;:\;Y=y\cdot H}{}

\step*{}{a_{25}:=a_1ya_{23}\;:\;y\varepsilon G}{}

\step*{}{a_{26}:=
eq\mhyphen cong_1(\lambda Z:ps(S).(Z\cdot Y),a_{13})
\;:\;X\cdot Y=x\cdot H\cdot Y}{}

\step*{}{a_{27}:=
eq\mhyphen cong_1(\lambda Z:ps(S).(x\cdot H\cdot Z), a_{24})
\;:\;x\cdot H\cdot Y=x\cdot H\cdot (y\cdot H)}{}

\step*{}{a_{28}:=
term_{\ref{lemma:normal-subgroup}.2}(S,G,\cdot,e ,^{-1},u,H,v_1,x,y,a_{15}, a_{25})
\;:\;(x\cdot H)\cdot(y\cdot H)=(x\cdot y)\cdot H}{}

\step*{}{a_{29}:=
eq\mhyphen trans_1(eq\mhyphen trans_1(a_{26},a_{27}),a_{28})
\;:\;X\cdot Y=(x\cdot y)\cdot H}{}

\step*{}{a_{30}:=
a_4xa_{12}ya_{23}
\;:\;(x\cdot y)\;\varepsilon \;B}{}

\step*{}{a_{31}:=
\wedge(a_{30},a_{29})
\;:\;(x\cdot y)\;\varepsilon \;B\wedge X\cdot Y=(x\cdot y)\cdot H}{}

\step*{}{a_{32}:=\exists_1(\lambda t:S.(t\varepsilon B\wedge X\cdot Y=t\cdot H),(x\cdot y), a_{31})\;:\;(X\cdot Y)\;\varepsilon\;(B/H)}{}

\conclude*{}{a_{33}:=\exists_3(r_3, a_{32})\;:\;(X\cdot Y)\;\varepsilon\;(B/H)}{}

\conclude*{}{a_{34}:=
\lambda Y:ps(S).\lambda r_3:Y\;\varepsilon\;(B/H)
.a_{33}\;:\;[\forall Y:ps(S).(Y\;\varepsilon\;(B/H)\Rightarrow(X\cdot Y)\;\varepsilon\;(B/H))]}{}

\conclude*{}{a_{35}:=\exists_3(r_1, a_{34})\;:\;[\forall Y:ps(S).(Y\;\varepsilon\;(B/H)\Rightarrow(X\cdot Y)\;\varepsilon\;(B/H))]}{}

\step*{}{a_{36}:=\exists_3(r_1, a_{22})\;:\;X^{-1}\;\varepsilon\;(B/H)}{}

\step*{}{a_{37}:=\exists_3(r_1, a_{17})\;:\;X\;\varepsilon\;(G/H)}{}

\conclude*{}{a_{38}:=
\lambda X:ps(S).\lambda r_1:X\;\varepsilon\;(B/H)
.a_{37}\;:\;(B/H)\subseteq(G/H)}{}

\step*{}{a_{39}:=
\lambda X:ps(S).\lambda r_1:X\;\varepsilon\;(B/H)
.a_{36}\;:\;Closure_1(ps(S),(B/H),^{-1})}{}

\step*{}{a_{40}:=
\lambda X:ps(S).\lambda r_1:X\;\varepsilon\;(B/H)
.a_{35}\;:\;Closure_2(ps(S),(B/H),\cdot)}{}

\step*{}{a_{41}:=
eq\mhyphen refl
\;:\;E=e\cdot H}{}

\step*{}{a_{42}:=
\wedge(a_2,a_{41})
\;:\;e\varepsilon \;B\wedge E=e\cdot H}{}

\step*{}{a_{43}:=\exists_1(\lambda t:S.(t\varepsilon B\wedge E=t\cdot H),e, a_{42})\;:\;E\;\varepsilon\;(B/H)}{}

\step*{}{a_{44}:=
\wedge(\wedge(\wedge( a_{43},a_{38}),a_{39}), a_{40})\;:\;
B/H\leqslant G/H}{}

\step*{}{\boldsymbol {term_{ \ref{theorem:correspondence}.1}}(S,G,\cdot,e,^{-1},u,H,B,v_1, v_2,w):= 
\wedge(a_{11}, a_{44})
\;:\;\boldsymbol{(H\triangleleft B)
\wedge (B/H\leqslant G/H)}}{}
\end{flagderiv}

2)
\vspace{-0.2cm}
\begin{flagderiv}
\introduce*{}{S: *_s\;|\;G:ps(S)\;|\;\cdot: S\rightarrow S\rightarrow S\;|\;e:S\;|\;^{-1}:S\rightarrow S\;|\;u:Group(S,G,\cdot,e,^{-1})}{}

\step*{}{a_1:= Gr_1(S,G,\cdot,e,^{-1},u)\;:\;e\varepsilon G}{}

\introduce*{}{H:ps(S)\;|\;v:H\triangleleft G\;|\;
C:ps(ps(S))
\;|\;w:C\leqslant G/H
}{}

\step*{}{\text{Notation }\;
B:=\{x:S\;|\;x\varepsilon G\wedge (x\cdot H)\;\varepsilon\;C\}\;:\;ps(S)}{}

\step*{}{a_2:= \wedge_1(v)\;:\;H\leqslant G}{}

\step*{}{a_3:=\wedge_1(\wedge_1( \wedge_1(a_2)))\;:\;H\subseteq G}{}

\step*{}{a_4:=\wedge_1(\wedge_1( \wedge_1(w)))\;:\;C\subseteq G/H}{}

\step*{}{a_5:=\wedge_2(\wedge_1( \wedge_1(w)))\;:\;E\varepsilon C}{}

\step*{}{a_6:=\wedge_2( \wedge_1(w))\;:\; Closure_1(ps(S),C,^{-1})}{}

\step*{}{a_7:=\wedge_2(w)\;:\;Closure_2(ps(S),C,\cdot)}{}

\step*{}{a_8:=\wedge(a_1,a_5)\;:\; e\varepsilon G\wedge (e\cdot H)\;\varepsilon \;C}{}

\step*{}{a_8\;:\; e\varepsilon B}{}

\step*{}{a_9:=
\lambda x:S.\lambda r:x \varepsilon B.\wedge_1(r)
\;:\; B\subseteq G}{}

\introduce*{}{x:S\;|\;r_1:x\varepsilon B}{}

\step*{}{a_{10}:=\wedge_1(r_1)\;:\;x\varepsilon G}{}

\step*{}{a_{11}:=\wedge_2(r_1)\;:\;(x\cdot H)\;\varepsilon\;C}{}

\step*{}{a_{12}:=
Gr_3(S,G,\cdot,e,^{-1},u,x,a_{10})
\;:\;x^{-1}\varepsilon G}{}

\step*{}{a_{13}:=a_6(x\cdot H)a_{11}\;:\;
(x\cdot H)^{-1}\;\varepsilon \;C}{}

\step*{}{a_{14}:=
term_{\ref{lemma:normal-subgroup}.1}(S,G,\cdot,e ,^{-1},u,H,v,x,a_{10})
\;:\;(x\cdot H)^{-1}=x^{-1}\cdot H}{}

\step*{}{a_{15}:=
eq\mhyphen subs_1(\lambda Z:ps(S).Z\varepsilon C,a_{14},a_{13})
\;:\;(x^{-1}\cdot H)\;\varepsilon\; C}{}

\step*{}{a_{16}:=
\wedge(a_{12},a_{15})
\;:\;x^{-1}\varepsilon G\wedge
(x^{-1}\cdot H)\varepsilon C}{}

\step*{}{a_{16}\;:\;
x^{-1} \varepsilon B}{}

\introduce*{}{y:S\;|\;r_2:y\varepsilon B}{}

\step*{}{a_{17}:=\wedge_1(r_2)\;:\;y\varepsilon G}{}

\step*{}{a_{18}:=\wedge_2(r_2)\;:\;(y\cdot H)\;\varepsilon\;C}{}

\step*{}{a_{19}:=
Gr_9(S,G,\cdot,e,^{-1},u,x,y,a_{10}, a_{17})\;:\;
(x\cdot y)\;\varepsilon\;G}{}

\step*{}{a_{20}:=a_7(x\cdot H)a_{11}(y\cdot H)a_{18}\;:\;
((x\cdot H)\cdot(y\cdot H))
\;\varepsilon \;C}{}

\step*{}{a_{21}:=
term_{\ref{lemma:normal-subgroup}.2}(S,G,\cdot,e ,^{-1},u,H,v,x,y,a_{10}, a_{17})\;:\;
(x\cdot H)\cdot(y\cdot H)=
(x\cdot y)\cdot H}{}

\step*{}{a_{22}:=
eq\mhyphen subs_1(\lambda Z:ps(S).Z\varepsilon C,a_{21},a_{20})
\;:\;((x\cdot y)\cdot H)\;\varepsilon\;C}{}

\step*{}{a_{23}:=
\wedge(a_{19},a_{22})
\;:\;(x\cdot y)\;\varepsilon\;G\wedge
((x\cdot y)\cdot H)\;\varepsilon\;C}{}

\step*{}{a_{23}\;:\;
(x\cdot y)\;\varepsilon\;B}{}

\conclude*[2]{}{a_{24}:=
\lambda x:S.\lambda r_1:x \varepsilon B.\lambda y:S.\lambda r_2:y \varepsilon B.a_{23}
\;:\;Closure_2(S,B,\cdot)}{}

\step*{}{a_{25}:=
\lambda x:S.\lambda r_1:x \varepsilon B.a_{16}
\;:\;Closure_1(S,B,^{-1})}{}

\step*{}{a_{26}:=
\wedge(\wedge(\wedge(a_9 ,a_8),a_{25}),a_{24})
\;:\;B\leqslant G}{}

\step*{}{a_{27}:=
term_{\ref{lemma:mult}.1}(S,G,\cdot,e ,^{-1},u,H,a_3)\;:\;
e\cdot H=H}{}

\introduce*{}{x:S\;|\;r:x\varepsilon H}{}

\step*{}{a_{28}:=
a_3xr
\;:\;x\varepsilon G}{}

\step*{}{a_{29}:=
term_{\ref{lemma:coset}.7}(S,G,\cdot,e ,^{-1},u,H,a_2, x,a_{28})\;:\;x\cdot H=H\;\Leftrightarrow\;x\varepsilon H}{}

\step*{}{a_{30}:=\wedge_2(a_{29})r
\;:\;x\cdot H=H}{}

\step*{}{a_{31}:=
eq\mhyphen trans_2(a_{27}, a_{30})
\;:\;E=x\cdot H}{}

\step*{}{a_{32}:=
eq\mhyphen subs_1(\lambda Z:ps(S).Z\varepsilon C,a_{31},a_5)
\;:\;(x\cdot H)\;\varepsilon\;C}{}

\step*{}{a_{33}:= 
\wedge(a_{28},a_{32})
\;:\;x\varepsilon G\wedge (x\cdot H)\;\varepsilon\;C}{}

\step*{}{a_{33}
\;:\;x\varepsilon B}{}

\conclude*{}{a_{34}:=
\lambda x:S.\lambda r:x\varepsilon H.a_{33}\;:\; H\subseteq B}{}

\step*{}{a_{35}:=\wedge_1( 
term_{ \ref{theorem:correspondence}.1}(S,G,\cdot,e,^{-1},u,H,B,v, a_{29},a_{34}))\;:\;H\triangleleft B}{}

\introduce*{}{X:ps(S)\;|\;r_1:X\varepsilon C}{}

\step*{}{a_{36}:=
a_4Xr_1
\;:\;X\;\varepsilon\;(G/H)}{}

\step*{}{a_{36}\;:\;[\exists x:S.(x\varepsilon G\wedge X=x\cdot H)]}{}
\introduce*{}{x:S\;|\;r_2:x\varepsilon G\wedge X=x\cdot H}{}
\step*{}{a_{37}:=\wedge_1(r_2)\;:\;x\varepsilon G}{}
\step*{}{a_{38}:=\wedge_2(r_2)\;:\;X=x\cdot H}{}

\step*{}{a_{39}:=
eq\mhyphen subs_1(\lambda Z:ps(S).Z\varepsilon C, a_{38},r_1)
\;:\;(x\cdot H)\;\varepsilon\;C}{}

\step*{}{a_{40}:=\wedge( a_{37},a_{39})\;:\;x\varepsilon G\wedge (x\cdot H)\;\varepsilon\;C}{}

\step*{}{a_{40}\;:\;x\varepsilon B}{}

\step*{}{a_{41}:=\wedge( a_{40},a_{38})\;:\;x\varepsilon B\wedge X=x\cdot H}{}

\step*{}{a_{42}:=\exists_1(\lambda t:S.(t\varepsilon B\wedge X=t\cdot H),x, a_{41})\;:\;X\;\varepsilon\;(B/H)}{}

\conclude*{}{a_{43}:=\exists_3(a_{36}, a_{42})\;:\;X\;\varepsilon\;(B/H)}{}

\conclude*{}{a_{44}:=
\lambda X:ps(S).\lambda r_1:X\varepsilon C
.a_{43}\;:\;C\subseteq(B/H)}{}

\introduce*{}{X:ps(S)\;|\;r_1:X\;\varepsilon\;(B/H)}{}

\step*{}{r_1\;:\;[\exists x:S.(x\varepsilon B\wedge X=x\cdot H)]}{}

\introduce*{}{x:S\;|\;r_2:x\varepsilon B\wedge X=x\cdot H}{}

\step*{}{a_{45}:=\wedge_1(r_2)\;:\;x\varepsilon B}{}

\step*{}{a_{46}:=\wedge_2(r_2)\;:\;X=x\cdot H}{}

\step*{}{a_{47}:=\wedge_2(a_{45})\;:\;(x\cdot H)\;\varepsilon\;C}{}

\conclude*{}{a_{48}:=
eq\mhyphen subs_2(\lambda Z:ps(S).Z\varepsilon C, a_{46},a_{47})
\;:\;X\varepsilon C}{}

\step*{}{a_{49}:=\exists_3(r_1, a_{48})\;:\;X\varepsilon C}{}

\conclude*{}{a_{50}:=
\lambda X:ps(S).\lambda r_1:X\;\varepsilon\;(B/H)
.a_{49}\;:\;(B/H)\subseteq C}{}

\step*{}{a_{51}:=\wedge(a_{44},a_{50})\;:\;C=B/H}{}

\step*{}{\boldsymbol {term_{ \ref{theorem:correspondence}.2}}(S,G,\cdot,e,^{-1},u,H,C,v,w)
:=\wedge(\wedge(a_{26}, a_{35}),a_{51})
\;:\;\boldsymbol{B\leqslant G\wedge (H\triangleleft B)
\wedge (C=B/H)}}{}
\end{flagderiv}

\bibliographystyle{plain}
\bibliography{farida}

\end{document}